\documentclass[reqno]{amsart}

\usepackage{epsfig}
\usepackage{amsmath}
\usepackage{amssymb}
\usepackage{amscd}
\usepackage{graphicx}
\usepackage{epstopdf} 
\usepackage{array}
\usepackage{verbatim}
\usepackage{bm}
\usepackage{wasysym}
\usepackage{stmaryrd}
\usepackage{yhmath}

\usepackage[mathscr]{euscript}

\usepackage{calc}         
\usepackage{graphicx}     
\usepackage{ifthen}       
\usepackage{subfigure}    
\usepackage{pst-all}      
\usepackage{pst-poly}     
\usepackage{multido}

\usepackage{multirow}
\usepackage{makecell}

\newcommand\mL{L\kern-0.08cm\char39}

\makeatletter
\def\l@section{\@tocline{1}{0pt}{1pc}{}{}}
\def\l@subsection{\@tocline{2}{0pt}{1pc}{4.6em}{}}
\def\l@subsubsection{\@tocline{3}{0pt}{1pc}{7.6em}{}}
\renewcommand{\tocsection}[3]{%
	\indentlabel{\@ifnotempty{#2}{\makebox[2.3em][l]{%
				\ignorespaces#1 #2.\hfill}}}#3}
\renewcommand{\tocsubsection}[3]{%
	\indentlabel{\@ifnotempty{#2}{\hspace*{2.3em}\makebox[2.3em][l]{%
				\ignorespaces#1 #2.\hfill}}}#3}
\renewcommand{\tocsubsubsection}[3]{%
	\indentlabel{\@ifnotempty{#2}{\hspace*{4.6em}\makebox[3em][l]{%
				\ignorespaces#1 #2.\hfill}}}#3}
\makeatother

\newtheorem{thm}{Theorem}[section]
\newtheorem*{main}{\bf Theorem A (Main Theorem)}
\newtheorem*{thmB}{\bf Theorem B}
\newtheorem*{thmC}{\bf Theorem C}
\newtheorem*{thmD}{\bf Theorem D}
\newtheorem{lem}[thm]{Lemma}

\newtheorem{prop}[thm]{Proposition}
\newtheorem{ques}[thm]{Question}
\newtheorem{cor}[thm]{Corollary}
\newtheorem{de}[thm]{Definition}

\newtheorem{example}[thm]{Example}

\theoremstyle{remark}
\newtheorem{rem}[thm]{Remark}

\numberwithin{equation}{section}
\numberwithin{figure}{section}

\newcommand \ra {\rightarrow}
\newcommand \lra {\longrightarrow}

\newcommand \N {\mathbb N}

\newcommand \Z {\mathbb Z}
\newcommand \R {\mathbb R}

\newcommand \Shom{S_{\rm{hom}}}
\newcommand \Smin{S_{\rm{min}}}

\newcommand{\Id}{\mathop{\rm Id }\nolimits}
\newcommand{\Fix}{\mathop{\rm Fix}}
\newcommand{\Per}{\mathop{\rm Per}}
\newcommand{\jump}{\mathop{\rm jump}}
\newcommand{\pre}{\mathop{\rm pre}}
\newcommand{\card}{\mathop{\rm card}}
\newcommand{\const}{\mathop{\rm const}}
\newcommand{\diam}{\mathop{\rm diam}}
\newcommand{\Int}{\mathop{\rm Int}}
\newcommand{\Ent}{\mathop{\rm Ent}}

\newcolumntype{?}{!{\vrule width 1pt}} 

\begin{document}
\title [Topology and Topological Sequence Entropy]
{Topology and Topological Sequence Entropy}
\author{\mL ubom\'\i r Snoha, Xiangdong Ye and Ruifeng Zhang}

\address{\hskip-\parindent
\mL ubom\'\i r Snoha, Department of Mathematics, Faculty of Natural
Sciences, Matej Bel University, Tajovsk\'eho 40, 974 01 Bansk\'a
Bystrica, Slovakia} 
\email{Lubomir.Snoha@umb.sk}
\address{\hskip-\parindent
Xiangdong Ye, Department of Mathematics, University of Science and
Technology of China, Hefei, Anhui 230026, P.R. China}
\email{yexd@ustc.edu.cn}
\address{\hskip-\parindent
Ruifeng Zhang, School of Mathematics, Hefei University of
Technology, Hefei Anhui 230009, P.R. China}
\email{rfzhang@hfut.edu.cn}

\date{}

\begin{abstract}
 Let $X$ be a compact metric space and $T:X\lra X$ be continuous. Let $h^*(T)$ be
		the supremum of topological sequence entropies of $T$ over all subsequences of
		$\Z_+$ and $S(X)$ be the set of the values $h^*(T)$ for all continuous maps $T$
		on $X$. It is known that $\{0\} \subseteq S(X)\subseteq \{0, \log 2, \log 3,
		\ldots\}\cup \{\infty\}$. Only three possibilities for $S(X)$ have been observed
		so far, namely $S(X)=\{0\}$, $S(X)=\{0,\log2, \infty\}$ and $S(X)=\{0, \log 2,
		\log 3, \ldots\}\cup \{\infty\}$.
		
		In this paper we completely solve the problem of finding all possibilities for
		$S(X)$ by showing that in fact for every set
		$\{0\} \subseteq A \subseteq \{0, \log 2, \log 3, \ldots\}\cup \{\infty\}$ there
		exists a one-dimensional continuum $X_A$ with $S(X_A) = A$. In the construction
		of $X_A$ we use Cook continua. This is apparently the first application of
		these very rigid continua in dynamics.
			
		 We further show that the same result is true if one considers only homeomorphisms rather than
con\-ti\-nuous maps. The problem for group actions is also addressed. For some class of group actions
(by homeomorphisms) we provide an analogous result, but in full generality this problem remains open.
The result works also for an analogous class of semigroup actions (by continuous maps).
\end{abstract}


\maketitle

\tableofcontents

\section{Introduction}\label{S:intro}

	\subsection{Supremum topological sequence entropy and main results}\label{SS:sup-ent}
	By a \emph{topological dynamical system} (t.d.s. for short)
	we mean a pair $(X,T)$, where $X$ is a nonempty compact metric space and
	$T:X\to X$ is a continuous map. A (metric) \emph{continuum} is a nonempty connected compact metric space.

	In 1958 Kolmogorov associated to any measure preserving
	system $(X,\mathcal{B},\mu,T)$ an
	isomorphic invariant, namely the measure theoretical entropy
	$h_\mu(T)$. Later on, in 1965, Adler, Konheim and McAndrew introduced in any
	t.d.s. an analogous concept, topological entropy $h(T)$. Systems with positive
	entropy are random in certain sense, and systems with zero entropy
	are said to be {\it deterministic} though they may exhibit
	complicated behaviors. There are several ways to distinguish between
	deterministic systems. One way to do this is to introduce the
	concept of entropy with respect to a (strictly increasing) subsequence of $\Z_+$
	(here we think of $\Z_+$ as the sequence $0,1,2, \dots$),
	see \cite{Ku} and \cite{Good1}, or the survey~\cite{CJS3}. Another way to do this is to
	investigate the so called complexity, see~\cite{BHM}.

	To study topological analogues of Kolmogorov systems, the authors in
	\cite{B} and \cite{HYis} introduced the notion of entropy pairs and
	entropy tuples. In \cite{HLSY} and \cite{HMY} the authors
	investigated sequence entropy pairs, and sequence entropy tuples
	and sequence entropy tuples for a measure, respectively. In
	Subsection~\ref{SS:seqentropy} we recall the definitions of a (sequence) entropy
	and a (sequence) entropy tuple. A~tuple is called \emph{intrinsic} if all its
	entries are pairwise different.
	
	In \cite{HYad} the authors defined a notion called maximal pattern entropy. One
	of equivalent definitions is that the \emph{maximal pattern entropy} $h^*(T)$ of
	a t.d.s. $(X,T)$ is the supremum of topological sequence entropies $h^A(T)$ of
	$T$ over all subsequences $A$ of $\Z_+$:
	\begin{equation}\label{Eq:hstarTdef}
	h^*(T)=\sup\{h^A(T) \colon \text{$A$ is a subsequence of $\Z_+$}\}~.
	\end{equation}
	This is why we will, throughout the paper, call the quantity $h^*(T)$ also the
	\emph{supremum topological sequence entropy of  $(X,T)$ or $T$} (some authors denote it
	by $h_{\infty}(T)$). In \cite{HYad}
	it is also showed that
	\begin{equation}\label{Eq:hstarT}
	h^*(T)=\sup\{\log n \colon \text{there is an intrinsic sequence entropy tuple of
		length}\ n\}
	\end{equation}
	(to be precise, we should speak on sequence entropy tuples for the map $T$; here and below we
	however abuse terminology if it is obvious which map is considered).
	Since $h^{k\Z_+}(T) = h(T^k) = kh(T)$, positive entropy implies infinite supremum topological sequence entropy
	(the converse is not true), and hence supremum topological sequence entropy is
	especially useful for zero entropy systems. By~\cite{HYad}, if $T$ is a homeomorphism then $h^*(T^n)=h^*(T)$ for
	all $n\neq 0$. The same proof gives that if $T$ is a continuous map then this is
	true for $n\geq 1$. If $X$ is a countable compact metric space, it is known that $h(T)=0$ for any continuous map $T$ from $X$ into itself.
    Note that this is not the case for the topological sequence entropy \cite{YZ} and our construction  heavily relies on this fact. See Subsection \ref{SS:seqentropy} for the definitions of sequence entropy and basic properties.

	The notion of the supremum topological sequence entropy fits especially well into the interval dynamics. To explain this,
	recall first the Sharkovsky ordering, see~\cite{Sh}, on the set $\mathbb N \cup \{2^{\infty}\}$:
	\begin{equation*}3\succ 5\succ 7\succ \dots \succ 2\cdot 3\succ 2\cdot
	5\succ 2\cdot 7\succ \dots \succ 4\cdot 3  \succ 4\cdot 5\succ 4\cdot 7\succ
	\dots \succ \dotsb \end{equation*}
	\begin{equation*}\succ 2^{n}\cdot 3 \succ 2^{n}\cdot 5\succ 2^{n}\cdot 7
	\succ \dots \succ \dots \succ 2^{\infty }\succ \dots \succ 2^{n} \succ \dots
	\succ 4\succ 2\succ 1.
	\end{equation*}
	We will also use the symbol $\succeq $ in the natural way.
	For $t\in \mathbb{N} \cup \{ 2^{\infty }\} $ we denote by $S(t)$ the set
	$\{k\in \mathbb{N} : t \succeq k \}$ ($S(2^{\infty })$ stands for the set
	$\{ 1,2,4, \dots , 2^{k}, \dots \}$). Denote by $C(I)$ the set of continuous selfmaps of a real compact interval $I$. For $T\in C(I)$
	let $\Per (T)$ be the set of periods of its periodic points. By Sharkovsky theorem, for every $T\in C(I)$ there exists a $t\in \mathbb{N}\cup \{2^{\infty }\}$
	with $\Per (T)=S(t)$ and, on the other hand, for every $t\in \mathbb{N}\cup \{2^{\infty }\}$ there exists a $T\in C(I)$ with
	$\Per (T)=S(t)$. 	
	
	If $\Per (T)=S(t)$, then $T$ is said to be of \emph{Sharkovsky type} $t$.  When
	speaking of types we consider them to be ordered by the Sharkovsky ordering. So if a map $T$ is of type $2^{\infty }$ or greater than
	$2^{\infty }$ or less than $2^{\infty }$, then, respectively,
	$\Per (T)=\{ 1,2,\dots , 2^{k}, \dots \}$ or $T$ has a periodic point
	with period not a power of $2$ or $\Per (T)=\{ 1,2,\dots , 2^{n} \}$ for some
	$n\in \Z_+$.
	
	The topological entropy of $T\in C(I)$ is positive if and only if $T$ is of type
	greater than $2^{\infty}$ (see \cite[Theorem 4.4.20]{ALM} and references therein).
	Further, if $T\in C(I)$ is of type greater than $2^{\infty}$ then it is Li-Yorke chaotic
	and if it is of type less than $2^{\infty}$ then it is not Li-Yorke chaotic, while among the maps of type $2^{\infty}$
	there are both Li-Yorke chaotic and non-chaotic maps, see~\cite{Smi}.
	By~\cite{FS}, $T\in C(I)$ is Li-Yorke chaotic if and only if $h^*(T)>0$. Due to~\cite{CJS}, there are only three possibilities for $T\in C(I)$,
	namely, $h^*(T)$ is either $0$ or $\log 2$ or $\infty$. These facts are shown in Table~\ref{T:interval}.

	\begin{table}[h]
	\begin{footnotesize}
		\begin{center}
		    \begin{tabular}{|c|c|c|c|}
			\hline
			$h(T)>0$ & \multicolumn{3}{|c|}{$h(T)= 0$}         \\ \hline
			type $\succ 2^{\infty}$ (chaotic)& type $2^{\infty}$, chaotic  &  type $2^{\infty}$, non-chaotic  & type $\prec 2^{\infty}$ (non-chaotic) \\ \hline
			$h^*(T)=\infty$ & $h^*(T)=\log 2$   & \multicolumn{2}{|c|}{$h^*(T)= 0$ } \\ \hline
			\end{tabular}
  	    \end{center}
		\end{footnotesize}
         \caption{Supremum topological sequence entropy for interval maps}\label{T:interval}
    \end{table}

    So, in the particular case $X=I$, we know the set
	\begin{equation}\label{Eq:defSX}
	S(X) = \{h^*(T):\, T \text{ is a continuous map } X\to X\}~.
	\end{equation}
	\index{$S(X)$}In the present paper we are interested in how this set can look like for nonempty compact metric spaces.
	By~(\ref{Eq:hstarT}), $h^*(T)=\log n$ for some $n\in \N$ or $h^*(T)=\infty =:
	\log \infty$. Further, if $T$ is a constant map or the identity then $h^*(T)=0$.
	Therefore
	$$
	\{0\} \subseteq S(X)\subseteq \{0, \log 2, \log 3, \ldots\} \cup \{\infty\}~.
	$$
	The set $S(X)$ depends on $X$.
	Table~\ref{T:knownSX} summarizes known results (here $\N$ is the set of positive
	integers, $\N^*=\N\cup\{\infty\}$, $\N_k=\{1,\ldots,k\}$,
	$\N_k^*=\{1, \ldots, k\}\cup\{\infty\}$ and  $\log M$ has obvious meaning for
	any $M\subseteq [1, \infty ]$):

	\begin{table}[h]
		\begin{center}
			\begin{tabular}{|m{1cm}||m{3cm}|m{38mm}|m{55mm}|}
				\hline
				$S(X)$ & $\log \N_1 = \{0\}$ & $\log \N^*_2 = \{0, \log 2\} \cup\{\infty\}$ &
				$\log \N^* = \{0, \log 2, \log 3, \dots\} \cup \{\infty\}$\\
				\hline
				$X$ & finite sets, \newline zero-dimensional spaces with finite derived sets
				~\cite{YZ} & interval~\cite{CJS}, \newline circle~\cite{CJS1}, \newline finite
				trees~\cite{TYZ}, \newline finite graphs~\cite{T} & zero-dimensional spaces with
				infinite derived sets~\cite{TYZ}, \newline some dendrites~\cite{TYZ}, \newline
				manifolds of dimension $\ge 2$~\cite{TYZ}\\
				\hline
			\end{tabular}
		\end{center}
		\caption{Known values of $S(X)$}\label{T:knownSX}
	\end{table}
	
	\noindent Since every dendrite is an absolute retract for the class of all
	compact metric spaces~\cite{Bor}, we in fact have $S(X) = \log \N^*$ not only
	for the manifolds with dimension $\ge 2$, but for any compact metric space $X$
	containing the dendrite from~\cite{TYZ}.
	
	Thus only three possibilities for $S(X)$ have been found so far: $\log \N_1$,
	$\log\N_2^*$ and $\log \N^*$. It is natural to ask whether these are the only
	possibilities for $S(X)$. In fact, at the beginning of our research we conjectured this.
	However, in the present paper we show that this is not the case.
	We completely solve the problem of finding all possibilities
	for $S(X)$ by proving the following theorem.
	
	\begin{main}\label{T:main}
		For every set $\{0\} \subseteq A \subseteq \log \mathbb N^*$ there exists a
		one-dimensional continuum $X_A\subseteq \mathbb R^3$ with $S(X_A) = A$.	
	\end{main}
	
	The information that $X_A$ exists in $\mathbb R^3$ is superfluous since by Menger-N\"obeling theorem
	every $m$-dimensional compact metrizable space can be embedded in $\mathbb R^{2m+1}$. The formulation we use
	just indicates that we construct the $1$-dimensional $X_A$ already as a subspace of $\mathbb R^3$.
	
	As mentioned in Table~\ref{T:knownSX}, only for some dendrites $X$ we know their set $S(X)$. However,
	it is obvious that for $\mathfrak c$-many sets $A$ ($\mathfrak c$ being the cardinality of the continuum),
	the required space $X_A$ does not exist among dendrites. In fact, since every dendrite $D$ can be retracted
	onto an interval $I\subseteq D$, we get that $S(D)\supseteq S(I) = \{0, \log 2, \infty\}$.
	
	The proof is long and complicated. In fact, Main Theorem itself is repeated in Subsection~\ref{S:proof} as Theorem~\ref{T:main again} and a relatively short proof is given there; however, all the previous sections are a preparation for the proof.
	
	Recall that the definition of $S(X)$, see~\eqref{Eq:defSX}, involves all continuous selfmaps of $X$.
	One can ask whether our Main Theorem still will be true if, instead of continuous maps, only \emph{homeomorphisms} are considered.
    The answer is positive. If we denote
    \[
    \Shom (X)=\{h^{*}(T)\colon T\text{ is a homeomorphism } X \rightarrow X \},
    \]
    then the following is true.

    \begin{thmB}\label{T:thmB}
    	For every set $\{0\} \subseteq A \subseteq \log \mathbb N^*$ there exists a one-dimensional continuum $\tilde{X}_A\subseteq \mathbb R^3$ with $\Shom(\tilde{X}_A) = A$.
    \end{thmB}

	The problem can be addressed also for \emph{group actions}. If $G$ is a topological group (with discrete topology) and $\Phi$ is an action (by homeomorphisms) of $G$ on $X$, let $h^*(X,G,\Phi)$ be the supremum topological sequence entropy of this action; for definitions see Subsection~\ref{S:group actions}. If we denote
	\[
	S_G(X)=\{h^*(X, G, \Phi) \colon \, \Phi\ \text{is an action of $G$ on $X$} \},
	\]
	then we have the following theorem.

	 \begin{thmC}\label{T:thmC}
	 	Let $G$ be a topological group such that there is a surjective group homomorphism $G\rightarrow \mathbb{Z}$. Then
	 	for every set $\{0\} \subseteq A \subseteq \log \mathbb N^*$ with $A$ finite or $\infty\in A$, there exists a one-dimensional continuum
	 	$\tilde{X}_A\subseteq \mathbb R^3$ with $S_G(\tilde{X}_A)= A$. If in addition $G$ is also finitely generated, then such a continuum exists
	 	for every set $\{0\} \subseteq A \subseteq \log \mathbb N^*$.
	 \end{thmC}

	 So, the problem is solved for groups which have $\mathbb Z$ as a quotient group (though some sets $A$ are not covered if the group is not finitely generated). In full generality this problem remains open.
	
	 Such a theorem is true also for an analogous class of \emph{semigroup actions} (by continuous maps). If $P$ is a topological semigroup with identity (and with discrete topology), let $S_P(X)$ be the set of supremum topological sequence entropies of all $P$-actions on $X$. Then we have the following analogue of the previous theorem.
	
	  \begin{thmD}\label{T:thmD}
	  	 Let $P$ be a topological semigroup with identity such that there is a surjective semigroup homomorphism $P\rightarrow \mathbb{Z}_+$. Then
	  	 for every set $\{0\} \subseteq A \subseteq \log \mathbb N^*$ with $A$ finite or $\infty\in A$, there exists a one-dimensional continuum
	  	 $X_A\subseteq \mathbb R^3$ with $S_P(X_A)= A$. If in addition $P$ is also finitely generated, then such a continuum exists for every set $\{0\} \subseteq A \subseteq \log \mathbb N^*$.
	  \end{thmD}
	
	  Section~\ref{S:questions} contains several open problems related to the topic of this paper.
	
	  \medskip
	
	  To work with supremum topological sequence entropy, we will use tools developed by David Kerr and Hanfeng Li in~\cite{KL} and tools from~\cite{HYad}. Namely, to compute $h^*(T)$, we will use the formula~\eqref{Eq:hstarT-IN} obtained by comparing the formula~\eqref{Eq:hstarT} with the characterization of sequence entropy tuples in terms of IN-tuples, which is due to Kerr and Li.
	  	
	  Last but not least, we wish to bring the attention of the reader to the remarkable fact that in the construction of the continua $X_A$ and $\tilde{X}_A$ we use amazing {\bf Cook continua} studied in continuum theory since 1960's. Our paper is apparently the first application of these very rigid continua in dynamics. We believe that the way how we use them may turn out to be useful for producing some other examples and counterexamples in dynamics and to solve some problems (including perhaps the problem of possible sets of values of topological entropy, see Question~\ref{Q:ent}).

	\subsection{Rigid spaces, Cook continua and their applications}\label{SS:Cook-appl}
	Roughly speaking, Cook continua are nondegenerate metric continua which are `everywhere' rigid for continuous maps (the definition is given below). They are difficult to construct and were motivated by previously known examples of spaces rigid for homeomorphisms.
	
	A nondegenerate space $X$ is called \emph{rigid for homeomorphisms} \index{rigid for homeomorphisms} if the only homeomorphism $X\to X$ is the identity. De Groot and Wille~\cite{GW} found such spaces in the class of one-dimensional Peano continua in 1958. An example of such a space is a disc with interiors of a dense family of propellers, with different numbers of blades, removed. Another example, attributed by them to de~Iongh, is a dendrite with a dense set of branching points of different orders.
	
	For a topological space $X$, let $H(X,X)$ be the group of all homeomorphisms of $X$ onto $X$, the group operation being the composition.
	Given an abstract group $G$, does there exist a \emph{topological group picture of~$G$}, i.e. a space $X$ such that $H(X,X)$ is isomorphic to $G$?
    Due to de Groot~\cite{G}, we know that the answer is affirmative. He showed that such a space $X$ can always be found in the class of
    connected, locally connected, complete metric spaces of any positive dimension, as well as in the class of compact, connected, Hausdorff spaces.
	Note that in general the topological group picture $X$ does not exist in the class of compact metric spaces because then the cardinality $\card X \leq c$ while there are groups with arbitrarily large cardinalities. Let us also mention that if the group $G$ is countable then $X$ exists in the class of Peano continua, as shown already by de Groot nad Wille~\cite{GW}.
	
	A nondegenerate space $X$ is \emph{rigid for continuous maps} \index{rigid for continuous maps}, or just rigid,\index{rigid} if every continuous map $X\to X$ is either the identity or a constant map. A rigid metric continuum was constructed by Cook~\cite{Cook} in 1967. For a topological space $X$, let $C(X,X)$ be the monoid of all continuous maps of $X$ into $X$, the semigroup operation being the composition (the unit element is the identity). Given an abstract monoid $M$, does there exist a space $X$ such that $C(X,X)$ is isomorphic to $M$? In general the answer is negative. The reason is that $C(X,X)$ contains all the constant maps (denote by $\const_a$ the constant map sending $X$ to the point $a\in X$) and the constant maps are left zeros (left absorbing elements) of the monoid $C(X,X)$, i.e. $\const_a \circ f = \const_a$. However, monoids with many left zeros are rather special. In particular, the monoid $(C(X,X), \circ)$ is never isomorphic to $(\mathbb Z, +)$ (such a space $X$ would be infinite, so $C(X,X)$ would have infinitely many left zeros, while $(\mathbb Z, +)$ has no left zero). What about the family $C(X,X) \setminus \{\const_a\colon a\in X\}$? In general it is not closed under composition. Nevertheless, in some cases it is a monoid. Given an abstract monoid $M$, does there exist a space $X$ such that $C(X,X) \setminus \{\const_a\colon a\in X\}$ is a monoid isomorphic to $M$?\footnote{This problem was posed by de Groot himself, at the Colloquium on Topology in Tihany, Hungary, in 1964.} Trnkov\'a proved that the answer is affirmative. She showed that such a space $X$ always exists in the class of metric spaces~\cite{Tr1} as well as in the class of compact Hausdorff spaces~\cite{Tr1, Tr2}. In particular, the choice of the trivial monoid $M$ gives the existence of a nondegenerate space rigid for continuous maps.
	
	We mentioned that the first example of a rigid metric continuum was given by Cook~\cite{Cook} in 1967. In fact he constructed in that paper what is now usually called a Cook continuum  (see the continuum $M_1$ in~\cite[Theorem 8]{Cook}, for a detailed description see~\cite[Appendix A]{PT}).
	
	\begin{de}\label{D:Cook}
		A \emph{Cook continuum} \index{Cook continuum} $\mathscr C$ is a \emph{nondegenerate} metric continuum such that, for every subcontinuum $K$ and every continuous
		map $f: K\to \mathscr C$, either $f$ is constant (i.e. $f(K)$ is a singleton) or $f(x)=x$ for all $x\in K$ (hence $f(K)=K$).
	\end{de}

	The Cook continuum constructed in~\cite{Cook} is one-dimensional (and hereditarily indecomposable) and so it can be embedded into $\mathbb R^3$,
	but it cannot be embedded into the plane because the construction uses solenoids, see~\cite[Note after Theorem~12]{Cook}.
	A~continuum is \emph{planar} if it is embeddable into the plane. Ma\'ckowiak constructed a Cook continuum in the plane,
	see~\cite[Corollary 6.2 and the discussion below it]{M}. His continuum is arc-like (i.e. chainable, hence planar and non-separating), and hereditarily decomposable.
	
	The fact that there exist spaces rigid for homeomorphisms and even for continuous maps (even Cook continua in the plane) is surprising.
	Such spaces are interested for topologists as, in a sense, pathological objects. They are used in the theory of topological representations of algebraic objects, see~\cite{PT}. Say, the de Groot's result~\cite{G} that every group is isomorphic to the group of homeomorphisms of a (nice) metric space, can be proved as follows. After realizing that every group $G$ is isomorphic to the automorphism group of some directed graph $\mathcal G$ (where by a directed graph we mean a set, possibly infinite, with binary relation), one can replace the edges of $\mathcal G$ by homeomorphic copies of a space $Y$ rigid for homeomorphisms, obtaining in such a way a space $X$ whose group $H(X,X)$ is isomorphic to $G$. Roughly speaking, we take copies of the same rigid space, place them along the edges of a directed graph and glue them at the points where two edges meet.
	
	Can rigid spaces be useful in dynamics? It seems that so far they have been used in dynamics only either in a trivial way or as an inspiration for constructing spaces which are `rigid-like' with respect to some dynamical property:	
	 \begin{itemize}
	 	\item \emph{Rigid spaces used as counterexamples to naive questions.} For instance, one can be interested in whether every nondegenerate space/continuum $X$ admits a continuous map $T\colon X\to X$ such that
	 	\begin{itemize}
	 		\item there exists $x\in X$ whose omega-limit set is not a singleton;
	 		\item there exists a scrambled pair for $T$ (a question really discussed in~\cite[p. 353]{BHS});
	 		\item the topological entropy of $T$ is positive.
	 	\end{itemize}
	 	The existence of rigid continua shows that in all three cases the answer is negative.
	 	
	 	\item \emph{Rigid spaces as inspiration for constructing spaces which are `rigid-like' with respect to some dy\-na\-mi\-cal property.} Say, in~\cite{DST} a continuum $X$ is constructed such that it is rigid-like with respect to minimality, meaning that $H(X,X) = \{T^n:\, n\in \mathbb Z\}$ where all $T^n$, $n\neq 0$, are minimal homeomorphisms and there is no other minimal continuous map $X\to X$. Being inspired by this, Akin and Rautio~\cite{AR} consider compact metric spaces $X$ such that every homeomorphism other than the identity is topologically transitive.	 	
\end{itemize}		
			
In particular, Cook continua are well known in the continuum theory as examples of `very rigid' continua, but apparently they have not been used in topological dynamics yet. The reason is, perhaps, that one can hardly imagine that they could be really useful in topological dynamics, say for the purpose of constructing spaces admitting an interesting nontrivial dynamics. In the present paper we will controvert this by using them to prove our Main Theorem. To obtain required spaces $X_A$, we will use infinitely many pairwise disjoint nondegenerate subcontinua of a planar Cook continuum. We will built infinitely many blocks, usually obtained by gluing together (copies of) infinitely many appropriately chosen Cook continua from that infinite family, and then we will glue together those blocks in an appropriate way (in fact some blocks will not really be `glued' with the rest of the space but, in spite of it, the whole space $X_A$ will be a continuum).

\subsection{Outline of the proof of Main Theorem}

To make the life of the reader easier, we write this outline of the proof of our Main Theorem. As we have mentioned, we use Cook continua, in fact a countably infinite family of pairwise disjoint subcontinua of a fixed planar Cook continuum. For a moment denote this family by $\mathscr F$. The whole proof can be divided into $4$ steps.
\begin{itemize}
	\item {\bf Step 1:}  We prove the theorem for $A=\{0,\infty\}$ (this is done in  Section~\ref{S:zero-infty} and Section~\ref{S:X}, see Proposition~\ref{P:zero inf}).
	\item {\bf Step 2:}  We prove the theorem for $A=\{0, \log 2\}$  (this is done in  Section~\ref{S:X1-T1-log2} and Section~\ref{S:cont zero-log2}, see Proposition~\ref{P:zero log2}).
	\item {\bf Step 3:}  We extend the result of Step 2 to the case $A=\{0, \log m\}$, for any integer $m \geq 3$ (see Section~\ref{S:generalization}, Proposition~\ref{P:zero logm}). This is just an analogue of Step 2.
\end{itemize}	
Each of the continua $X_A$, where $A$ has cardinality $2$ as above, is a disjoint union of two parts called the head and the snake of $X_A$. The snake has a first point and approaches the head (an analogue is the $\sin (1/x)$ continuum; of course, our continua $X_A$ are much more complicated). Moreover, each of the above constructions depends on the choice of the family $\mathscr F$, more precisely, on the choice of \emph{one} Cook continuum from $\mathscr F$ which is used in the construction of the head and on the way how all the \emph{other} Cook continua from $\mathscr F$ are arranged into an injective sequence; they are used in the construction of the snake. When we say here that a Cook continuum is used in the construction of a space, we have in mind that a copy or many copies of this continuum are subsets of that space (by a copy we mean a homeomorphic copy).

Moreover, our constructions are such that if we fix the same family $\mathscr F$ and we make the same choice of the mentioned one Cook continuum and the same way of arranging all the other Cook continua into an injective sequence, then the \emph{snakes} are homeomorphic (but not equivalently embedded in $\mathbb R^3$).

Now split $\mathscr F$ into infinitely many pairwise disjoint infinite families of continua. For a moment denote these families of Cook continua by $\mathscr F_{\infty}, \mathscr F_2, \mathscr F_3, \dots$ (for each of them also choose one Cook continuum and arrange all the other Cook continua into an injective sequence).
\begin{itemize}	
	\item {\bf Step 4:}  Once we know that Main Theorem is true whenever $A$ has cardinality $2$ (and is trivially true also when $A = \{0\})$, we finally prove it for arbitrary set $\{0\} \subseteq A \subseteq \log  \N^*$ (see Section~\ref{S:proof}, Theorem~\ref{T:main again}) as follows.  For every $k\in \N^*\setminus \{1\}$ put $A(k) = \{0, \log k\}$ and consider the continuum $X_{A(k)}$ constructed in one of the Steps 1,2,3, but now using building bricks from the family $\mathscr F_k$. Then choose, in $\mathbb R^3$, a `central' point and such copies of the continua $X_{A(k)}$ which are pairwise disjoint except that the chosen central point is the first point of the snakes of all considered continua $X_{A(k)}$. Moreover, we choose the copies of $X_{A(k)}$ such that the diameters of $X_{A(\infty)}, X_{A(1)}, X_{A(2)}, \dots$ tend to zero. Then the union of all $X_{A(k)}$ is a one-dimensional continuum which looks like a `flower' with infinitely many smaller and smaller `petals' (copies of the continua $X_{A(k)}$) intersecting at the chosen `central' point. Given $A$ with cardinality at least $3$, we obtain the required $X_A$ as the union of the petals which correspond to those $k\in \N^*\setminus \{1\}$ for which $\log k \in A$. The proof that $S(X_A)= A$ is not difficult and it takes only about one page. It uses the properties of the continua $X_{A(k)}$.
\end{itemize}

\medskip

So, the most important and difficult parts are Step 1 and Step 2. We are going to present ideas behind them. Assume that $\mathscr F$ is fixed, one Cook continuum $\mathscr K \in \mathscr F$ is chosen and $\mathscr K_1 , \mathscr K_2, \dots$ is an injective sequence containing all the other Cook continua in $\mathscr F$.
\medskip

In Step 1, to construct a continuum $X$ with $S(X)=\{0,\infty\}$ we first construct a system $(X_1,T)$ with
\[
X_1 = \mathscr K_0 \sqcup C \quad \text{with} \quad C = \{x_1, x_2, \dots\}, \quad
T x_i=x_{i+1}, \, i=1,2,\dots \quad \text{and} \quad T|_{\mathscr K_0}=Id|_{\mathscr K_0}~.
\]
Here $\mathscr K_0$ is a copy of $\mathscr K \in \mathscr F$; it will be the head of $X\supseteq X_1$. The trajectory $x_1, x_2, \dots$ of $T$ can be chosen in such a way that $h^*(T)=\infty$, see Section~\ref{S:zero-infty} for details. We think of $\mathscr K_0$ as lying in a vertical plane, the trajectory $x_1, x_2,\dots$ approaches it from the right. For every $m$, we connect $x_m$ and $x_{m+1}$ by a properly chosen continuum $D_m$. For each $m$, $D_m = D^*_m \sqcup \{x_{m+1}\}$ where the sets $D^*_m$ are of the form:
\begin{equation}
\begin{split}
D^*_1= & \, \text{copy of $\mathscr K_1$ $\sqcup$ copy of
	$\mathscr K_2$ $\sqcup$ copy of $\mathscr K_3$ $\sqcup$ copy of $\mathscr K_4$ $\sqcup$ } \dots, \notag \\
D^*_2= & \, \text{copy of $\mathscr K_2$ $\sqcup$ copy of
	$\mathscr K_4$ $\sqcup$ copy of $\mathscr K_6$ $\sqcup$ copy of $\mathscr K_8$ $\sqcup$ } \dots, \notag \\
D^*_3= & \, \text{copy of $\mathscr K_4$ $\sqcup$ copy of
	$\mathscr K_8$ $\sqcup$ copy of $\mathscr K_{12}$ $\sqcup$ copy of $\mathscr K_{16}$ $\sqcup$ } \dots, \notag \\
\dots & \, ,
\end{split}
\end{equation}
see Figure~\ref{F:3Dm}. The used Cook continua of the form `$\text{copy of $\mathscr K_i$}$' as well as the Cook continuum $\mathscr K_0$ are building `bricks' of the continuum $X$ defined by
\begin{equation*}
X=\mathscr K_0 \sqcup \bigcup_{m=1}^{\infty}D_m =  \mathscr K_0 \sqcup
\bigsqcup_{m=1}^{\infty}D_m^* ,
\end{equation*}
see Figure~\ref{F:spaceX}. The bricks in a set $D^*_m$ form a sequence, each of them has the `first point' and the `last point' (see Section~\ref{S:X} for details), two consecutive bricks intersect at one point (the last point of one of them and the first point of the next one), otherwise they are disjoint. The consecutive continua $D_m$ and $D_{m+1}$ intersect at the point $x_{m+1}$, non-consecutive ones are disjoint.

One can show that $X$, defined as the disjoint union of the `head' $\mathscr K_0$ and the `snake' $\bigcup_{i=1}^{\infty}D_m$, is a one-dimensional continuum, see Lemma~\ref{L:Xiscont}. It has the required property $S(X)= \{0,\infty\}$, see Proposition~\ref{P:zero inf}. The main steps in the proof of this fact are as follows.

First, in Lemma~\ref{L:map} we prove that if $B$ is a  brick then $F(B)$ is either a singleton or a brick homeomorphic to $B$, whenever $F: X\to X$ is a continuous map. Further, the definition of the sets $D_i$ is such that if $m<M$ are positive integers then there is exactly one continuous surjective map of $D_m$ onto $D_M$, but there is no surjective map of $D_M$ onto $D_m$, see Lemma~\ref{L:mapDD}, Figure~\ref{F:sur} and Remark~\ref{R:mapDD}. Already this fact and the fact that all the bricks in $X$ are Cook continua, indicate that the continuum $X$ does not admit `too many' continuous selfmaps.

To study the dynamics of all possible continuous selfmaps of $X$ we first define a particular continuous map $G: X\to X$, a continuous extension of the map $T: X_1\to X_1$, with $h^*(G)=\infty$ (see Lemma~\ref{L:existsG}; this map $G$ sends each continuum $D_m$ onto $D_{m+1}$). Further we show that if $F: X\to X$ is a non-constant continuous map, then there are only two possibilities, see Proposition~\ref{P:zero inf}. The first possibility is that the set of fixed points $\Fix (F)$ of the (non-constant) map $F$ is a `nice' subcontinuum of $X$ intersecting the snake and there is a positive integer $N$ such that $F^N(X) = X$, whence $h^*(F)= 0$. The second possibility is that the (non-constant) map $F$ has no fixed point in the snake and, for some positive integer $N$, the map $F$ coincides with $G^N$ on some neighbourhood of the head, which implies that $h^*(F) = \infty$.

\medskip

In Step 2, the construction of a continuum  $X \subseteq \mathbb R^3$ with $S(X)
= \{0,\log 2\}$  is to some extent similar to that in Step 1, but necessarily there are also important differences and the situation is dramatically more complicated than in Step 1.

Again we start with an auxiliary system $(X_1, T)$, now with $h^*(T)=\log 2$, in the form of a disjoint union of two parts. Now the first part is not a Cook continuum with identity as in Step 1. Instead, it is a countable set $A =\{a_i,i \in \Z\} \cup \{a_{\infty}\}$ in the vertical plane above the point $0$ on the horizontal axis, with $T(a_i)=a_{i+1}$, $i \in \Z$ and $T(a_{\infty})=a_{\infty}$, see Figure~\ref{fig:(A,T_1)}. The second part of $(X_1,T)$ is, similarly as in Step~1, just one trajectory $x_0,x_1,\dots$. This trajectory lies in $(0,1] \times A$ with $T(x_i)=x_{i+1}$, approaching the first part `from the right'. Thus,
\[
X_1 = A \sqcup \{x_0,x_1,\dots\}
\]
with the dynamics given by the map $T$ described above. By choosing the points $x_i$, $i=0,1,\dots$ carefully (this construction is long and complicated, see Section~\ref{S:X1-T1-log2} for the details), we get $h^*(T)=\log 2$, see Theorem~\ref{final-thm}. Here one can see the first complication when compared with Step 1. Even if we restrict ourselves to spaces $X_1$ in the form of the disjoint union of a first part of $X_1$ and just one trajectory approaching it, it is much more easier to construct the system $(X_1,T)$ in such a way that $h^*(T)$ is extremely large, i.e. equal to $\infty$, than to construct it with $h^*(T)$ being exactly $\log 2$, neither larger nor smaller.

Further, similarly as in Step 1, we add Cook continua to get $X$ in the form of the disjoint union of a head and a snake. In Step 1, the head of $X$ was the same as the first part of $X_1$, namely the Cook continuum $\mathscr K_0$. Now we denote the head of $X$ by $\mathscr A_0$ and  in Section~\ref{S:cont zero-log2} we construct it, in the vertical plane containing $A$, by joining the consecutive points of $A$ by copies of the chosen Cook continuum from $\mathscr F$, see shaded continua in Figure~\ref{fig:rhombusZ}. These copies are in fact taken such that each of them is obtained from any other one by a direct similitude, to keep the geometry of $X$ under control (this will be useful later). So, the head of $X$ looks like a necklace of similar copies of the same Cook continuum, together with the point~$a_{\infty}$. Topologically, the snake is obtained from $\{x_0,x_1,\dots\}$ in the same way as in Step 1, so the snake is homeomorphic to that from Step~1 and we have
\begin{equation*}
X= \text{head } \sqcup \text{ snake} = \mathscr A_0 \sqcup \text{ (snake homeomorphic to that from Step 1)}.
\end{equation*}
However, the two snakes have different positions in $\mathbb R^3$, meaning that they are not equivalently embedded in $\mathbb R^3$. We have to be careful with the construction because we want $X$ to be a compact space (i.e., we want that all the cluster points of the snake which lie in the vertical plane containing the head, belong to the head). With some care, the snake can indeed be placed in $\mathbb R^3$ in such a way that $X$ is compact (in fact a one-dimensional continuum), see Lemma~\ref{L:Xcpct}.

Similarly as in Step 1, in Subsection~\ref{SS:defG} we extend the map $T: X_1\to X_1$ to a particular map $G: X\to X$ (we of course want that $h^*(G)=\log 2$). The map $G$ on our snake is just the unique map topologically conjugate to the map $G$ we defined on the snake in Step 1. Further, $G$ sends each Cook continuum in the `necklace' $\mathscr A_0$ onto the next one. Unfortunately, while in Step 1 the particular map $G$ was trivially continuous on $X$ and $h^*(G)=\infty$, now we are facing two problems. First, the restrictions of $G$ to the head and to the snake are clearly continuous but contrary to Step 1, now these two continuous parts of $G$ need to fit together to produce a continuous map on $X$. Further, while it is trivial that $h^*(G)\geq h^*(T)=\log 2$, this inequality alone does not imply that $h^*(G) = \log 2$. (In Step 1 we were in better situation, since there $h^*(G)\geq h^*(T)= \infty$ trivially implied that $h^*(G) = \infty$.) To cope with these two problems, we specify the geometry of $X$ in more details, by adding some additional requirements to our construction of $X$. For the details see Subsection~\ref{SS:geometry}. Then finally $X$ is as we need. In Lemma~\ref{L:Gcont} we prove that  $G$ on such a continuum $X$ is continuous and in Lemma~\ref{L:existsG-log2}, using Lemma~\ref{L:G-log2-2case}, we prove that $h^*(G)=\log 2$. Though only snakes in Step 1 and Step 2 are homeomorphic (while the heads are not), analogous arguments as in Step 1 then easily give that $S(X)=\{0,\log 2\}$, see Proposition~\ref{P:zero log2}.

\bigskip

	\noindent{\bf Organization of the paper.} In Section~\ref{S:prelim} we introduce some
	notions and facts that will be used in the paper. In Section~\ref{S:zero-infty}
	we construct systems with zero topological entropy but infinite supremum
	topological sequence entropy.
	In Section~\ref{S:X}, based on an example from Section~\ref{S:zero-infty}, we
	construct a
	continuum $X$ with $S(X)=\{0,\infty\}$. Then in Sections~\ref{S:X1-T1-log2} and~\ref{S:cont zero-log2}
	we	construct a continuum $X$ with $S(X)=\{0,\log 2\}$. For a generalization of these two Sections
	from $\log 2$ to $\log m$ see Section~\ref{S:generalization}. Finally, the proof of our Main Theorem is
	given in Section~\ref{S:proof}. Also for the proofs of Theorems B, C and D see Section~\ref{S:proof}.
	Readers interested in open problems should consult Section~\ref{S:questions}.

\medskip

	\noindent{\bf Acknowledgments.} This project was started when the first
	two authors visited Max Planck Institute for Mathematics, Bonn, in 2009.
	Much of the work on the project was done also when the first author visited USTC, Hefei in 2011 and 2016 and when the third author visited Matej Bel University in 2018. The authors thank all these institutions for the warm hospitality and financial supports.
    
    Snoha was  supported by the Slovak Research and Development	Agency under the contract No.~APVV-15-0439
	and by VEGA grant 1/0786/15. Ye was supported by NNSF of	China (11371339  and 11431012).
    Zhang was supported by the NNSF of China (11871188 and 11671094).
    
	The authors thank Jerzy Krzempek for providing them with the
	reference~\cite{M},	where the existence of Cook continua in the plane is proved. This enabled the
	authors to simplify the geometry of their constructions.  
	
	Sincere thanks of the authors go to Hanfeng Li for his comments
    on the preliminary version of the paper, which resulted in the homeomorphism case in Subsection~\ref{S:main-homeo},
    the group actions and the semigroup actions cases in Subsection~\ref{S:group actions}, and an open problem in Subsection~\ref{SS:actions case}.


\section{Preliminaries from dynamics and topology}\label{S:prelim}

In this section we introduce some related notions from topological dynamics and topology.

	\subsection{Topological sequence entropy, sequence entropy
		tuples}\label{SS:seqentropy}
	
	Let $(X,T)$ be a t.d.s. Recall that a set $Y\subseteq X$ is called
	\emph{$T$-invariant} if $T(Y)\subseteq Y$.
	
	The topological entropy of $T$ \index{topological entropy}, the topological sequence
	entropy of $T$ \index{topological sequence entropy} with respect to a subsequence $A$ of $\Z_+$ and the
	maximal pattern entropy \index{maximal pattern entropy} of $T$ are denoted by $h(T)$, $h^A(T)$ and
	$h^*(T)$ respectively. To recall the definition of $h^A(T)$, let  $A=\{a_0 <
	a_1<\cdots \}$ be a sequence of nonnegative integers. Given an open cover
	$\mathcal{U}$ of $X$, define
	$$
	h^{A}(T,\mathcal{U})=\limsup_{n \rightarrow
		\infty}\frac{1}{n}\log \mathcal{N} \left (\bigvee_{i=0}^{n-1}T^{-a_i}(\mathcal{U})\right),
	$$
	$\mathcal{N}(\mathcal V)$ being the minimal possible cardinality of a subcover chosen from
	a cover $\mathcal V$. Then
	$$
	h^{A}(T)=\sup \{ h^{A}(T,\mathcal{U}):\, \mathcal{U}\text{ is an open cover of }
	X\}.
	$$
	Note that $h^A(T)$ becomes $h(T)$ for $A=\Z_+$. As already said in Introduction,
	for the maximal pattern entropy of $T$ we have $h^*(T)=\sup_{A}h^A(T)$ \index{$h^*(T)$}, where
	$A$
	is ranging over all subsequences of $\Z_+$, see \cite{HYad} and therefore we
	will call $h^*(T)$ also the supremum topological sequence entropy of $T$.
	
	Recall that a t.d.s. $(X,T)$ is {\it null} if the
	sequence entropy is zero for all subsequences, i.e. if $h^*(T)=0$,
	and it is \emph{tame} if its enveloping semigroup is Fr\'echet (a topological
	space is \emph{Fr\'echet} if for any $A\subseteq X$ and any $x\in \overline {A}$ there
	is a sequence $\{x_n\}$ with $A \ni x_n\lra x$).

Let us state some properties for null systems. For a measure preserving system $(X,\mathcal{B},\mu,T)$ we can define
entropy or sequence entropy with respect to $\mu$. It is a classical result \cite{Ku} that a measure preserving system $(X,\mathcal{B},\mu,T)$
is null if and only if it has discrete spectrum. It is known \cite{HLSY} that if $(X,T)$ is a minimal system and $h^*(T)=0$, then $\pi:X\rightarrow X_{eq}$ (the factor
map to the maximal equicontinuous factor) is an almost one-to-one extension, $(X,T)$ is uniquely ergodic and $\pi$ is also
an isomorphism (in the measurable sense). Recently, the structure of a minimal null system under group actions is determined by Glasner \cite{Eli-G}. See \cite{LR, GR, HLY}
for other results related to nullness and sequence entropy.
	
	The notion of a (sequence) entropy tuple of length $n$ \index{sequence entropy tuple} was defined in \cite{B,
		HYis, HLSY}. An $n$-tuple $(x_i)_{i=1}^{n} \in X^n, n \ge
	2$, is called  a {\it sequence entropy $n$-tuple} for $(X,T)$ if at least two
	points in $(x_i)_{i=1}^{n}$ are different and for any disjoint
	closed neighborhoods $\{U_1, \ldots, U_t\}$ of
	$\{x_{j_1},\ldots,x_{j_t}\}=\{x_1,\ldots,x_n\}$, there exists an
	increasing sequence $A \subseteq \Z_+$ ($A=\Z_+$ in case of entropy tuples) such
	that
	$$h^A(T,\{U_1^c,\ldots,U_t^c\})>0,$$ where $2\le t\le n$ and
	$x_{j_p}\not=x_{j_k}$ when $p\not=k$.
	
	The set of
	entropy tuples and sequence entropy tuples of length $n$ will be
	denoted by $E_n(X,T)$ and $SE_n(X,T)$ respectively. A tuple is said
	to be a pair if $n=2$. Instead of $E_2(X,T)$ and $SE_2(X,T)$ we write
	$E(X,T)$ and $SE(X,T)$, respectively. Note that the maximal zero entropy factor (resp. null factor)
    of a topological dynamical system is induced by the smallest closed invariant equivalence relation containing $E_2(X,T)$ (resp. $SE_2(X,T)$),
    see \cite{B} and \cite{HLSY}.
	
    Originally, entropy and sequence entropy tuples were defined using
	open covers. We will not use the original definitions but we are going
	to state equivalent definitions given in \cite{KL} using the following notion of
	independence.

	\begin{de} Let $(X,T)$ be a t.d.s. and $\tilde{A}=( A_1,\ldots,A_k )$
		be a tuple of subsets of $X$. We say that a subset $J\subseteq\Z_+$ is an
		independence set for $\tilde{A}$
		if for any nonempty finite subset $I\subseteq J$, we have
		$$
		\bigcap_{i\in I}T^{-i}A_{s(i)}\not=\emptyset
		$$
		for any $s\in \{1,\ldots,k\}^I$. If such a set $J$ is finite and has $p$ elements, we also say that it is an independence set, \index{independence set} or independence set of times, of length~$p$.
	\end{de}

	\begin{rem}\label{R:indep-subtuple}
		Note that $J$ in this definition is also an independence set for any sub-tuple
		of $\tilde{A}$. Also, a subset of an independence set for $\tilde{A}$ is again
		an independence set for $\tilde{A}$.
	\end{rem}

	\begin{rem}\label{R:indep-subset}
		In a natural way we will use this terminology also in the following slightly
		more general situation.
		Suppose that $Y\subseteq X$ and $S: Y\to X$. Let $\tilde{A}=(A_1,\ldots,A_k)$ be
		a tuple of subsets of $X$ and $J\subseteq\Z_+$. By saying that $J$ is an
		\emph{independence set for $\tilde{A}$ for the map $S$}, we will mean that for
		any nonempty finite subset $I\subseteq J$ and any map $s: I \to \{1,\ldots,k\}$
		we have $\bigcap_{i\in I}S^{-i}A_{s(i)}\not=\emptyset$, i.e. there is a point
		$y\in Y$ such that, for every $i\in I$, $S^i(y)\in X$ exists and $S^i(y)\in
		A_{s(i)}$. In particular, if such a map $S$ is then extended to a continuous map
		$T: X\to X$ (note that $Y$ is not necessarily $T$-invariant), obviously $J$ will
		be an independence set for $\tilde{A}$ for the t.d.s. $(X,T)$.
	\end{rem}

	In \cite{KL} the authors defined IE-tuples \index{IE-tuples}, IT-tuples \index{IT-tuples} and IN-tuples \index{IN-tuples}
	(here I, E, T and N stand for independent, entropy, tame and null,
	respectively). Though we use only IN-tuples, for completeness recall the three
	definitions.

	\begin{de}\label{D:IE-IT-IN}
		Consider a tuple $\tilde{x}=( x_1,\ldots,x_k)\in X^k$.
		If for every product neighborhood $U_1\times
		\ldots\times U_k$ of $\tilde{x}$ the tuple $(U_1,\ldots,U_k)$ has an
		independence set of \emph{positive density} or an
		\emph{infinite} independence set or \emph{arbitrarily long finite} independence
		sets, then the tuple $\tilde{x}$ is called an \emph{IE-tuple} or an
		\emph{IT-tuple} or an
		\emph{IN-tuple}, respectively.
	\end{de}

	Notice that for tuples we have $IE \Rightarrow IT \Rightarrow IN$.
	Recall that a tuple $(x_i)_{i=1}^n$ is called \emph{intrinsic} if
	$x_i\not=x_j$ when $i\not=j$ and \emph{diagonal} if all its entries
	are equal.

	The following theorem gives the promised equivalent definitions of
	entropy tuples and sequence entropy tuples from \cite{KL} (see also \cite{HYis}
	and \cite{H1}).
	
	\begin{thm}\label{KL-thm} Let $(X,T)$ be a t.d.s.
		\begin{enumerate}
			\item
			A tuple is an entropy tuple if and only if it is a non-diagonal IE-tuple. In
			particular, a system $(X,T)$ has zero entropy if and only if every
			IE-pair is diagonal.
			\item
			A tuple is a sequence entropy tuple if and only if it is a non-diagonal
			IN-tuple. In particular, a system $(X,T)$ is null if and only if
			every IN-pair is diagonal.
			\item
			$(X,T)$ is tame if and only if every IT-pair is diagonal.
		\end{enumerate}
	\end{thm}
	
	We remark that (1) was proved by Huang and Ye \cite{HYis} using the
	notion of an interpolating set (which was first used by Glasner and Weiss \cite{GlWe} in the symbolic setting), one half of (3) was proved by Huang in
	\cite{H1} using the notion of a scrambled pair and (1-3) were showed
	by Kerr and Li \cite{KL} using the notion of independence.
	
	Comparing~(\ref{Eq:hstarT}) with Theorem~\ref{KL-thm}(2), we get that
	\begin{equation}\label{Eq:hstarT-IN} \index{$h^*(T)$}
	h^*(T) = \sup\{\log k: \text{there is an intrinsic IN-tuple of length}\ k\}~.
	\end{equation}
	Throughout the paper, this formula will be used to compute $h^*(T)$. Just for
	completeness, we add one formula more. Recall that a set $A\subseteq X$ is
	called a \emph{sequence entropy set}
	if every non-diagonal tuple from $A$ is a sequence entropy tuple, and it is
	\emph{maximal}
	if it is maximal with respect to the inclusion. Comparing this terminology with
	~(\ref{Eq:hstarT}), we get that
	\begin{equation}\label{Eq:hstarT-set}
	h^*(T) = \sup\{\log \# A: \text{$A$ is a sequence entropy set}\}~.
	\end{equation}
	
	Let $(Y,S)$ and $(X,T)$ be t.d.s. If there is a surjective
	continuous map $\pi$ from $Y$ to $X$ such that $\pi\circ S=T\circ
	\pi$ then we say that $(X,T)$ is a \emph{factor} of $(Y,S)$ and
	$(Y,S)$ is an \emph{extension} of $(X,T)$.
	The direct product of $n$ copies of $X$ is denoted by
	$X^n$ and the direct product of $n$ copies of the map $T$ by $T^{(n)}$.
	Let $IE_n(X,T), IN_n(X,T)$ and $IT_n(X,T)$ be the families of all IE-tuples,
	IN-tuples and IT-tuples, respectively, of length $n$. Below, let $P_n(X,T)$
	be one of them, i.e. $P$ is one of $IE$, $IN$, $IT$. The set of non-wandering
	points of $T$ is denoted by $\Omega(T)$.
	
	\begin{prop}\label{P} Let $(X,T)$ and $(Y,S)$ be t.d.s. and let $\pi:(Y,S)\lra
		(X,T)$ be a factor map. Then
		\begin{itemize}
			\item[{(a)}] $P_n(X,T)$ is a closed
			$T^{(n)}$-invariant subset of $X^n$.
			
			\item[{(b)}]
			\begin{itemize}
				\item[{(1)}] If $( x_i)_{i=1}^n \in P_n(X,T)$, then for
				all $1\le i \le n$ there exists $y_i \in Y$ such that $\pi
				(y_i)=x_i$ and $( y_i)_{i=1}^n \in P_n(Y,S).$
				
				\item[{(2)}] If $( y_i)_{i=1}^n \in P_n(Y,S)$,
				then $( \pi (y_i))_{i=1}^n  \in P_n(X,T)$.
			\end{itemize}
			
			\item[{(c)}] If $( x_i)_{i=1}^n \in P_n(X,T)$ then $x_i\in
			\Omega(T)$ for each $1\le i\le n$.
			
			\item[{(d)}] $P_n(X,T)=P_n(X,T^k)$ for any $k\in \N$.
		\end{itemize}
	\end{prop}
	
	\begin{proof} (a) and (b) can be found in \cite{B},
		\cite{HYis} and \cite{KL}, (c) and (d) can be proved directly by
		definition.
	\end{proof}

	\begin{prop}\label{P2}
		Let $(X,T)$ be a t.d.s.
		\begin{itemize}
			\item [{(a)}] $h^*(T)=h^*(T|_{T(X)})$.
			\item [{(b)}] Let $T$ be an injective map and $( a_i)_{i=1}^n \in
			P_n(X,T)$.
			Then each of the points $a_i$ has a (unique) preimage and $(T^{-1}a_i)_{i=1}^n \in P_n(X,T)$.
			So, $P_n(X,T)$ is a $(T^{-1})^{(n)}$-invariant subset of $X^n$.
		\end{itemize}
	\end{prop}	
	
	\begin{proof} (a). Clearly, every intrinsic IN-tuple of $G=T|_{T(X)}$ is also an
		IN-tuple for $T$.
		Conversely, let $( x_1, \dots, x_n )$
		be an IN-tuple for $T$. We show that it is also an IN-tuple for $G$. Due to
		Proposition~\ref{P}(c),
		all the points $x_i$ lie in $T(X)$. For every $i$, let $V_i$ be a neighbourhood
		of $x_i$ in the topology of $T(X)$,
		i.e. $V_i=U_i \cap T(X)$ for some neighbourhood $U_i$ of $x_i$ in the topology
		of $X$. Given $k$, we are going to show that for the neighbourhoods
		$V_1,\dots,V_n$ there is an independence
		set $J$ of times of length $k$ for the map $G$.
		Since $( x_1, \dots, x_n )$ is an IN-tuple for $T$,
		there is an independence set $J^*$ of times of length $k+1$ for the
		neighbourhoods $U_1, \dots, U_n$ and the map $T$.
		This means that for any choice of indices $s(i)\in \{1,\dots, n\}$, $i\in J^*$,
		there is a point
		$x\in X$ such that $T^i(x)\in U_{s(i)}$. Now let $J$ be the set of (nonnegative)
		times of length $k$ obtained from $J^*$
		by removing the smallest element of $J^*$ and subtracting $1$ from every other
		element.
		Since $T(x)\in T(X)$ and $T^{i-1}(T(x))\in U_{s(i)}\cap T(X) = V_{s(i)}$ for any
		$i\in J^*$ which is greater than zero, we see that
		$J$ has the required properties.
		
		(b) Again, since $\Omega(T)\subseteq T(X)$, Proposition~\ref{P}(c) gives that
		$a_i\in T(X)$, $i=1,\dots,n$. So, the tuple $( T^{-1}a_i)_{i=1}^n$
		is well defined.
		
		Consider any tuple of open neighbourhoods $( U(T^{-1}a_i))_{i=1}^n$
		(here $U(T^{-1}a_i)$ is a neighbourhood of $T^{-1}a_i$).
		Since $T: X\to T(X)$, being a continuous bijection between compact metric spaces,
		is a homeomorphism, for every $i$ we get that the set $V(a_i):=T(U(T^{-1}a_i))$
		is an open neighbourhood of $a_i$ in the topology of $T(X)$ and so $V(a_i)=
		W(a_i)\cap T(X)$ for some open neighbourhood $W(a_i)$ of $a_i$ in the topology
		of $X$.
		
		So, for any tuple of open neighbourhoods $( U(T^{-1}a_i))_{i=1}^n$
		there is a tuple of open neighbourhoods $( W(a_i) )_{i=1}^n$ such
		that $T(U(T^{-1}a_i)) \subseteq W(a_i)$ for every $i$. No matter whether $P$ is
		$IE$, $IN$ or $IT$, to finish the proof of (b) it is sufficient to prove the
		following implication (cf. Definition~\ref{D:IE-IT-IN}):
		\begin{equation}\label{Eq:indsetT-1}
		\begin{aligned}
		&\text{$\{t(1)<t(2)<t(3)<\dots\}$ is a (finite or infinite) independence set of
			times for $( W(a_i) )_{i=1}^n$ $\Longrightarrow$} \\
		&\qquad \text{$\Longrightarrow$ $\{t(2)-1,t(3)-1,\dots\}$ is an independence set
			of times for $( U(T^{-1}a_i) )_{i=1}^n$}
		\end{aligned}
		\end{equation}
		So, fix such a set $J_W=\{t(1)<t(2)<t(3)<\dots\}$ (i.e., $t(1)$ is perhaps zero
		but $t(i)\geq 1$ for $i\geq 2$).
		To prove that $J_U=\{t(2)-1,t(3)-1,\dots\}$ is an independence set of times for
		$(U(T^{-1}a_i))_{i=1}^n$, fix a nonempty finite subset
		$I_U\subseteq J_U$ and a function $s: I_U \to \{1,\dots,n\}$.
		We want to find a point $x_s\in X$ such that for every $t(i)-1\in I_U$ (note
		that then $i\geq 2$)
		we have $T^{t(i)-1}(x_s) \in U(T^{-1} a_{s(t(i)-1)})$. Since $I_W = \{t(i):
		t(i)-1\in I_U\}$ is a finite subset of $J_W$
		and $J_W$ is an independence set of times for $( W(a_i) )_{i=1}^n$,
		there
		is a point $z\in X$ such that $T^{t(i)}(z) \in W(a_{s(t(i)-1)})$ for every
		$t(i)\in I_W$. However, for $t(i)\in I_W$
		we have $i\geq 2$, whence $t_i\geq 1$ and so $T^{t(i)}(z) \in T(X)$. The last
		two inclusions give that
		$T^{t(i)}(z) \in W(a_{s(t(i)-1)}) \cap T(X) =  V(a_{s(t(i)-1)}) =
		T(U(T^{-1}a_{s(t(i)-1)}))$. Hence
		$T^{t(i)-1}(z) \in U(T^{-1} a_{s(t(i)-1)})$ and so the choice $x_s=z$ finishes
		the proof of the implication~(\ref{Eq:indsetT-1}).
	\end{proof}

	\begin{rem}\label{R:careful}
		In connection with Proposition~\ref{P2}(a), be careful. It is not already true
		that $h^A(T)= h^A(T|_{\Omega(T)})$ and
		$h^*(T)=h^*(T|_{\Omega(T)})$. See also Remark~\ref{R:arcwise}.
	\end{rem}

	\subsection{Retracts, chains in connected spaces, continua, Cook continua in the
		plane}\label{SS:Cook}
	Throughout the paper, if a space $Y$ is homeomorphic to a space $X$, we say that
	$Y$ is a (homeomorphic) copy of $X$.
	Given a space $X$ and its subspace $Y \subseteq X$, a continuous
	map $r : X \to Y$ is called a \emph{retraction} \index{retraction} if the restriction
	$r|_Y$ is the identity. Then $Y= r(X)$  is called a \emph{retract} \index{retract}
	of $X$. A compact metric space $Y$ is called an \emph{absolute
		retract} for the class of all compact metric spaces (in what
	follows we will sometimes shortly say ``absolute retract") if for any
	compact metric space $Z$, whenever (a copy of) $Y$ is a subspace of $Z$,
	(this copy of) $Y$ is a retract of $Z$.
	
	A classical result of Borsuk \cite[Corollary 13.5, p. 138]{Bor} says that each
	dendrite
	is an absolute retract and that there are no other one-dimensional compact
	metric spaces which are absolute retracts. In particular, any arc
	is an absolute retract (other examples of absolute retracts are the
	$n$-dimensional
	cubes, $n\geq 1$, and the Hilbert cube).
	
	Every connected space has the property that every two points can be joined by an
	$\varepsilon$-chain of points \index{$\varepsilon$-chain}, see e.g.~\cite[p.13]{W}. We will use the
	following modification of this fact.
	
	\begin{lem}\label{L:chains}
		Let $M$ be a nondegenerate connected metric space with a metric $\varrho$, $D$
		be
		a dense subset of $M$ and $a, b \in M$. Let $\varepsilon >0$. Then there is a
		positive integer $n$ such that $a$ and $b$ can be joined by an
		$\varepsilon$-chain of points of length $n+2$ lying (with possible exceptions of
		the points $a$ and $b$) in $D$, i.e. there is a finite chain $a = c_0, c_1, c_2,
		\dots, c_n, c_{n+1} = b$ with $c_1, c_2, \dots, c_n$ in $D\setminus \{a,b\}$ and
		$\varrho (c_i, c_{i+1}) < \varepsilon$ for every $i = 0, \dots , n$. The set of
		such $n$'s is in fact cofinite.
	\end{lem}
	
	\begin{proof}
		Let $M_a$ be the set of all points $x\in M$ such that $a$
		can be joined with $x$ by such a chain (with $n$ depending on $x$).
		Since $D$ is dense, the set $M_a$ is nonempty (it contains $a$ because there is
		an $\varepsilon$-chain $a,c_1, a$ with $c_1\in D\setminus\{a\}$) and it is
		easily seen that $M_a$ is both open and closed. Hence $M_a = M$ and so $b\in
		M_a$. By repeating the point $c_n$ one can see that the lengths of
		$\varepsilon$-chains joining $a$ and $b$ form a cofinite set (by small
		perturbations, one can even construct such chains consisting of pairwise
		different points, with possible exception of the endpoints $a$ and $b$, provided
		they coincide).
	\end{proof}

	Recall that, throughout the paper, a \emph{continuum} \index{continuum} is a nonempty compact
	connected
	metric space. Thus a singleton is a continuum. However, we will be interested in
	nondegenerate continua. A nondegenerate continuum has cardinality $\mathfrak c$ (in fact,
	every nonempty perfect Polish space has cardinality $\mathfrak c$).

	The following is Boundary Bumping Theorem \index{Boundary Bumping Theorem}, see e.g.~\cite[Theorem 5.4]{N}.
	
	\begin{thm}\label{T:BBT}
		Let $X$ be a continuum, and let $U$ be a nonempty, proper, open subset of $X$.
		If $K$ is a component of $\overline U$, then $K$ intersects the boundary of $U$
		(equivalently, since $K\subseteq \overline U$ and $U$ is open, $K$ intersects
		$X\setminus U$).
	\end{thm}
	
	From this we get the following fact, see e.g.~\cite[Corollary 5.5]{N}.

	\begin{lem}\label{L:small continua}
		If $X$ is a nondegenerate metric continuum then for every open set $\emptyset
		\neq U\subseteq X$ there is a nondegenerate continuum $K\subseteq U$.
	\end{lem}

	We say that two continua are \emph{comparable by continuous maps} \index{comparable by continuous maps} if one of them
	can be continuously mapped onto the other. Otherwise, they are
	\emph{incomparable by continuous maps}.\index{incomparable by continuous maps}
	
	\begin{lem}\label{L:Cook-properties}
		\begin{enumerate}
			\item A nondegenerate metric continuum is a Cook continuum if and only if it has
			the property that no two different nondegenerate subcontinua of it are
			comparable by continuous maps.
			\item Every nondegenerate subcontinuum of a Cook continuum is a Cook continuum.
			\item If two Cook continua $\mathscr C_1$ and $\mathscr C_2$ are homeomorphic,
			then they are \emph{uniquely homeomorphic}, i.e. there exists a unique
			homeomorphism $\varphi \colon \mathscr C_1 \to \mathscr C_2$. Moreover,
			if $\sigma \colon \mathscr C_1\to \mathscr C_2$ is a continuous surjection then $\sigma = \varphi$.
		\end{enumerate}
	\end{lem}

	\begin{proof}
		(1) If $\mathscr C$ is a Cook continuum then it clearly has that property.
		Conversely, let a nondegenerate metric continuum $C$ have that property. Let $K$
		be a subcontinuum of $C$ and $f: K\to C$ be continuous such that the continuum
		$f(K)$ is not a singleton. By the property, $f(K)=K$. To prove that $C$ is a
		Cook continuum, we show that $g:=f|_K$ is the identity. From now on we work in
		the topology of $K$. Suppose, on the contrary, that a point $b\in K$ is not
		fixed for $g$. Choose a $g$-preimage $a$ of $b$. So, $a\neq b$ and $g(a)=b\neq
		g(b)$. By continuity and Lemma~\ref{L:small continua}, there is a continuum
		$A\subseteq K$ such that $A$ contains $a$ but not $b$ and $g(A)$ contains~$b$
		but not~$a$. Due to the property, the continuum $g(A)$ has to be degenerate, so
		$g(A)=\{b\}$. The closed set $g^{-1}(b)$ contains $A$ and is disjoint from
		$\{b\}$. Let $U$ be an open neighbourhood (recall that in the topology of $K$)
		of $g^{-1}(b)$ such that $b \notin \overline U$. Let $K^*$ be that component of
		$\overline U$ which contains the continuum $A$. By Theorem~\ref{T:BBT}, $K^*$
		contains a point $c\in K\setminus U$. Then the continuum $f(K^*)=g(K^*)$
		contains both $b$ and $g(c)\neq b$ and so is both different from $K^*$ and
		nondegenerate. This contradicts the property.
		
		(2) A subcontinuum of a continuum is a continuum. So, (2) follows from the
		definition or from (1).
		
		(3) If also $\psi: \mathscr C_1 \to \mathscr C_2$ is a homeomorphism, then
		$\psi^{-1} \circ \varphi$ is a homeomorphism of $\mathscr C_1$ onto $\mathscr
		C_1$ and since $\mathscr C_1$ is a Cook continuum, it is the identity, whence
		$\varphi = \psi$. If $\sigma \colon \mathscr C_1\to \mathscr C_2$ is a continuous surjection
		then $\varphi^{-1}\circ \sigma$ is a continuous surjection $\mathscr C_1 \to \mathscr C_1$ and, since
		$\mathscr C_1$ is a Cook continuum, it is the identity. Hence $\sigma = \varphi$.
	\end{proof}
	
Cook continua are known to be at most two-dimensional, hereditarily indecomposable Cook
continua are always one-dimensional, see e.g. the very end of~\cite{Krz} and references therein, cf.~\cite{Krz2}.
Though the planar Cook continuum constructed by Ma\'ckowiak~\cite{M} is hereditarily decomposable,
it is still one-dimensional. This is even because of the following trivial reason.

\begin{lem}\label{L:1-dim}
	Every Cook continuum $\mathscr C$ in $\mathbb R^2$
	(even every rigid space $\mathscr C \subseteq \mathbb R^2$) is one-dimensional.
\end{lem}	

\begin{proof}
	Let $\mathscr C \subseteq \mathbb R^2$ be rigid. Then $\mathscr C$ is not zero-dimensional,
	since a disconnected space obviously admits a continuous selfmap whose range consists of two points.	
	So, $\mathscr C$ is at least one-dimensional. However, it is not two-dimensional,
	otherwise, see e.g.~\cite[Theorem IV.3]{HW}, it properly contains
	a $2$-dimensional cube and since the cube is an absolute retract for the class of
	all compact metric spaces, there is a continuous retraction of $\mathscr C$ onto that
	cube, a contradiction with rigidity of $\mathscr C$.
\end{proof}

	\begin{lem}\label{L:Cook-family}
		Let $\mathscr Q$ be a planar Cook continuum.
		Then there exist planar Cook continua
		\begin{equation}\label{Eq:Cook-family}
		\mathscr{K}_0, \mathscr{K}_1, \mathscr{K}_2, \dots ,
		\end{equation}
		in fact pairwise disjoint subcontinua of $\mathscr Q$,
		with the following properties for every $i,j =0,1, \dots$.
		\begin{itemize}
			\item [(a)] Whenever $K\subseteq \mathscr{K}_i$ is a  continuum and $g: K\to
			\mathscr{K}_i$ is continuous, then $g$ is constant or identity.
			\item [(b)] Whenever $K\subseteq \mathscr{K}_i$ is a  continuum and $f: K\to
			\mathscr{K}_j$, $i\neq j$, is continuous, then $f$ is constant (in particular,
			$\mathscr K_i$ and $\mathscr K_j$ are incomparable by continuous maps and hence
			non-homeomorphic).
		\end{itemize}
	\end{lem}

	\begin{proof}
		Since there is an infinite family of  pairwise disjoint nonempty open sets in
		$\mathscr Q$, Lemma~\ref{L:small continua} shows that in $\mathscr Q$
		there are pairwise disjoint nondegenerate subcontinua $\mathscr{K}_0,
		\mathscr{K}_1, \mathscr{K}_2, \dots$. By Lemma~\ref{L:Cook-properties}(2) they
		are Cook continua, hence (a).
		To prove (b), fix nonnegative integers $i\neq j$, a nondegenerate subcontinuum
		$K\subseteq \mathscr{K}_i$ (if $K$ is a singleton, the claim is trivial) and a
		continuous map  $f: K\to \mathscr{K}_j$. Then $f$ can be viewed as
		a continuous map $K\to \mathscr Q$ and so $f$ is either constant or identity.
		However, the latter possibility is excluded, because $\mathscr{K}_i$ and
		$\mathscr{K}_j$ are disjoint.
	\end{proof}

	The following fact is now obvious, but it will be needed below and therefore we
	state it explicitly.
	
	\begin{cor}\label{C:Cook-family}
		Replace the continua $\mathscr{K}_0, \mathscr{K}_1, \mathscr{K}_2, \dots$
		in~(\ref{Eq:Cook-family})
		by homeomorphic copies $\widetilde{\mathscr{K}}_0, \widetilde{\mathscr{K}}_1,
		\widetilde{\mathscr{K}}_2, \dots$
		of them, respectively. Then these new continua are still Cook continua
		(embeddable into the plane) and
		they still have the following properties.
		\begin{itemize}
			\item [($\widetilde{a}$)] Whenever $\widetilde{K}\subseteq
			\widetilde{\mathscr{K}}_i$ is a  continuum and
			$\widetilde{g}: \widetilde{K}\to \widetilde{\mathscr{K}}_i$ is continuous, then
			$\widetilde{g}$ is constant or identity.
			\item [($\widehat{a}$)]
			Let $\widehat {\mathscr K}_i$ be a copy of $\mathscr K_i$, possibly different
			from $\widetilde{\mathscr{K}}_i$. Whenever $\widetilde{K}\subseteq
			\widetilde{\mathscr{K}}_i$ is a  continuum and
			$\widehat{g}: \widetilde{K}\to \widehat{\mathscr{K}}_i$ is continuous, then
			$\widehat{g}$ is either constant or a homeomorphism $\widetilde{K}\to
			\widehat{g}(\widetilde{K})$.
			\item [($\widetilde{b}$)] Whenever $\widetilde{K}\subseteq
			\widetilde{\mathscr{K}}_i$ is a  continuum and
			$\widetilde{f}: \widetilde{K}\to \widetilde{\mathscr{K}}_j$, $i\neq j$, is
			continuous, then $\widetilde{f}$ is constant (in particular,
			$\widetilde{\mathscr{K}}_i$ and $\widetilde{\mathscr{K}}_j$ are incomparable by
			continuous maps and hence non-homeomorphic).
		\end{itemize}
	\end{cor}
	
	The following simple fact will be used in Section~\ref{S:cont zero-log2}
	for Cook continua in the plane.

	\begin{lem}\label{L:arc in nbhd}
		Let $C$ be a continuum in $\mathbb R^n$ and $V$ an open set containing $C$.
		Then each two points of $C$ can be joined by a polygonal arc in $V$.
	\end{lem}
	
	\begin{proof}
		Each two points of a connected open set $U$ in $\mathbb R^n$ can be joined by a
		polygonal arc in $U$, see e.g.~\cite[Theorem 3-5]{HY}. However, $V$ does contain
		a connected open set $U$ such that $C\subseteq U \subseteq V$.
	\end{proof}


\section{Systems with zero entropy and infinite supremum sequence
		entropy}\label{S:zero-infty}

	As already mentioned in Introduction, $h(T)>0$ implies $h^*(T)=\infty$. What can
	be said on the supremum sequence entropy $h^*(T)$ if $h(T)=0$? As we know from
	\cite{CJS1,CJS,T,TYZ},  if the phase space $X$ is a unit interval or a unit
	circle or a finite graph, then $h(T)=0$ implies that $h^*(T)$ is \emph{finite}
	(zero or $\log 2$). However, in general a space $X$ can admit a continuous
	selfmap $T$ with $h(T)=0$ and $h^*(T)=\infty$. The next subsection
	is devoted to the construction of such an example. It will be helpful in the
	next section.
	
	If $A$ and $B$ are disjoint sets, $A\cup B$ will sometimes be denoted by
	$A\sqcup B$. We will also use the notations
	\[
	1/\mathbb N  = \{1/n:\, n\in \mathbb N\} \quad \text{and} \quad
	1/\mathbb N^{\uparrow k}  = \{1/n:\, n = k, k+1, \dots \}, \quad k\in \mathbb
	N~.
	\]\index{$1/\mathbb N^{\uparrow k}$} \index{$1/\mathbb N$}
	Further, $\Id_A$ \index{$\Id_A$} denotes the identity on $A$.

	\subsection{A map $T: X_1 \to X_1$ with $h(T)=0$ and
		$h^*(T)=\infty$}\label{SS:cont+orbit}
	
	We are going to define such a map on a space which differs from a prescribed
	continuum $\mathscr C_0$ by one orbit only.
	
	Start by fixing any nondegenerate metric continuum $\mathscr C$ with metric
	$\varrho$ and a countable dense set $E=\{e^1, e^2,\dots \} \subseteq \mathscr
	C$. Our space $X_1$ will be a subset of $\left(\{0\}\cup 1/\mathbb N \right)
	\times \mathscr C$ endowed with the maximum metric (we have the Euclidean metric
	in the first coordinate and the metric $\varrho$ in the second coordinate).
	For $n\in \mathbb N$ and $i\in \mathbb N$, put $\mathscr C_0 = \{0\}\times
	\mathscr C$, $\mathscr C_n = \{1/n\}\times \mathscr C$, $E_0 = \{0\}\times E$,
	$e^i_0 = \{0\}\times \{e^i\}$, $E_n = \{1/n\}\times E$, $e^i_n = \{1/n\}\times
	\{e^i\}$. Let $P_1: \bigsqcup_{n=0}^\infty \mathscr C_n \to \{0\}\cup 1/\mathbb
	N$ and $P_2: \bigsqcup_{n=0}^\infty \mathscr C_n \to \mathscr C$ be the
	projections onto the first and the second coordinates.
	
	The set $X_1\subseteq \mathscr C_0 \sqcup \bigsqcup_{n=1}^\infty \mathscr C_n$
	will be of the form
	\begin{equation}\label{Eq:X1C}
	X_1 = \mathscr C_0 \sqcup C \quad \text{with} \quad C = \{x_1, x_2, \dots\}
	\subseteq \left( 1/\mathbb N \right) \times E~.
	\end{equation}
	Define $T: X_1\to X_1$ by
	\begin{equation}\label{Eq:T}
	T|_{\mathscr C_0} = \Id_{\mathscr C_0} \quad \text{ and } \quad T(x_n) =
	x_{n+1}~.
	\end{equation}
	Then the set $C$ and the sequence $(x_n)_{n=1}^\infty$ are the orbit and the
	trajectory (under $T$) of the point $x_1$.
	
	So, we only need to choose the points $x_n$ for all `times' $n$ (we want
	compactness of $X_1$, continuity of $T$, $h(T)=0$ and $h^*(T)=\infty$). We will
	do it step by step, for \emph{blocks of times}
	\begin{equation}\label{Eq:blocks times}
	\{1,2,\dots, k_1\}, \,  \{k_1+1, k_1+2,\dots, k_2\}, \, \ldots, \, \{k_{n-1}+1,
	k_{n-1}+2,\dots, k_n\}, \, \ldots
	\end{equation}
	where the integers $1<k_1<k_2<\dots$ are not specified yet.  The pieces of
	trajectory corresponding to the blocks of times in~(\ref{Eq:blocks times}) (and
	only to these blocks of times) will be conveniently called \emph{blocks (of
		trajectory)}.
	For simplicity, our first requirements on the choice of the points $x_n$ are:
	\begin{enumerate}
		\item [(1)] $P_1(x_n) = 1/n$, $n=1,2,\ldots$ (in view of~(\ref{Eq:X1C}) this
		implies compactness of $X_1$).
		\item [(2)]
		$P_2(x_1)=P_2(x_{k_1})=P_2(x_{k_1+1})=P_2(x_{k_2})=P_2(x_{k_2+1})=\cdots=e^1$
		(so, the points $x_1, x_{k_1}, x_{k_1+1}, x_{k_2}, \dots$ are already
		determined and, by~(\ref{Eq:T}), $T(x_{k_n}) = x_{k_n+1}$, $n=1,2,\dots$).
	\end{enumerate}
	Since  the $T$-trajectory  of the point $x_1=\{1\}\times \{e^1\}$ will, due to
	(1), (\ref{Eq:X1C}) and~(\ref{Eq:T}), approach the
	set $\mathscr C_0$ on which $T$ is the identity, the `jumps'  performed by our
	trajectory have to tend to zero (in order not to destroy the continuity of $T$
	at the points of $\mathscr C_0$). Since the sequence $P_1(x_i)$ tends to zero,
	we only need that
	also $\varrho (P_2(x_i), P_2(x_{i+1})) \to 0$ for $i\to \infty$. To ensure this,
	it is sufficient to fix positive reals $\varepsilon_n\searrow 0$ and to choose
	$(x_i)_{i=1}^\infty$ such that if consecutive points $x_i$ and $x_{i+1}$ belong
	to the $n$-th block of the trajectory, then $\varrho (P_2(x_i), P_2(x_{i+1})) <
	\varepsilon_n$ (if they belong to different blocks, then $i=k_n$ for some $n$
	and so, by (2), we even have $\varrho (P_2(x_i), P_2(x_{i+1}))=0$). In other
	words, our requirement is (we put $k_0=0$):
	\begin{enumerate}
		\item [(3a)] for every $n=1,2,\dots$, the sequence
		$(P_2(x_i))_{i=k_{n-1}+1}^{k_n}$ is an $\varepsilon_n$-chain of points lying in
		the countable dense set $E =\{e^1, e^2, \dots\}\subseteq \mathscr C$ and, not to
		violate (2), it starts and ends with $e^1$. (This together with (1) imply that
		$T$ defined by~(\ref{Eq:T}) is continuous and $h(T) = h(T|_{\Omega (T)}) =
		h(T|_{\mathscr C_0}) = h(\Id_{\mathscr C_0}) = 0$.)
	\end{enumerate}
	We still need something more from the sequence $(x_i)_{i=1}^\infty$ because we
	want also $h^*(T)=\infty$. Therefore we add the following requirement on the
	choice of $(x_i)_{i=1}^\infty$:
	\begin{enumerate}
		\item [(3b)] for some positive integers $t_{1,1}$, $t_{2,1} < t_{2,2}$, $\dots$,
		$t_{n,1} < t_{n,2} < \dots < t_{n,n}$, $\dots$, we have the following:
		\begin{itemize}
			\item $\{0, t_{1,1}\}$ is an independence set of times for the map $T$ restricted
			to
			the first block $\{x_1,\dots, x_{k_1}\}$ and for the family of $2$ sets
			$(1/\mathbb N) \times \{e^j\}$, $j= 1, 2$ (in the sense of
			Remark~\ref{R:indep-subset}),
			\item $\{0, t_{2,1}, t_{2,2}\}$ is an independence set of times for the map $T$
			restricted to the second block $\{x_{k_1+1},\dots, x_{k_2}\}$ and for the family
			of $3$ sets $\left(1/\mathbb N^{\uparrow (k_1+1)}\right) \times \{e^j\}$, $j= 1,
			2, 3$,
			
			$\dots$
			
			\item $\{0, t_{n,1}, t_{n,2}, \dots, t_{n,n}\}$ is an independence set of times
			for the map $T$ restricted to the $n$-th block $\{x_{k_{n-1}+1},\dots,
			x_{k_n}\}$
			and for the family of $n+1$ sets $\left(1/\mathbb N^{\uparrow (k_{n-1}+1)}\right
			)\times \{e^j\}$, $j= 1, 2, \dots, n+1$.
			
			$\dots$
		\end{itemize}
	\end{enumerate}
	Realize that (3b) really implies $h^*(T)=\infty$. Indeed,
	for every $j$, every neighborhood  (in the topology of $X_1$) of the point
	$e^j_0 \in E_0 \subseteq \mathscr C_0$ contains the sets $\left(1/\mathbb
	N^{\uparrow (k_{n-1}+1)}\right )\times \{e^j\}$
	for all sufficiently large $n$. Hence we easily come to the conclusion that
	\emph{each finite subset}
	of the countable infinite set $E_0$ forms an IN-tuple for $T$ (for any choice of
	neighbourhoods of these points, the tuple of the neighbourhoods has arbitrarily
	long finite independence sets of times; see Remark~\ref{R:indep-subtuple}). In
	view of~(\ref{Eq:hstarT-IN}), it follows that $h^*(T)=\infty$. Note that our
	argument shows that even for every positive integer $r$, the set
	$S_{r}:= \mathscr C_0 \sqcup \{x_r, x_{r+1}, \dots \}$ (notice that this is a
	$T$-invariant subset of $X_1$) has the property
	\begin{equation}\label{Eq:hstarsub}
	h^*(T|_{S_{r}})=\infty~.
	\end{equation}
	
	Thus, it remains to prove that the choice of $(x_i)_{i=1}^\infty$ can be done in
	such a way that all the requirements (1), (2), (3a) and (3b) are fulfilled
	simultaneously. To prove that, notice that due to cofiniteness emphasized in
	Lemma~\ref{L:chains}, one can find a positive integer $t_{1,1}$, as large as we
	need/wish, such that in $\mathscr C$
	\begin{itemize}
		\item there are four $\varepsilon_1$-chains of the same length $t_{1,1}+1$ --
		namely chains from $e^1$ to $e^2$ (the chain starts at $e^1$ at time $0$ and
		ends at $e^{2}$ at time $t_{1,1}$), from $e^2$ to $e^1$, from $e^1$ to $e^1$ and
		from $e^2$ to $e^2$.
	\end{itemize}
	By concatenating these four chains and possibly some intermediate
	$\varepsilon_1$-chains in an appropriate order, we obtain an
	$\varepsilon_1$-chain $e^{i(1)}, \dots, e^{i(k_1)}$ in $E\subseteq \mathscr C$
	with some length $k_1$ (this will be the length of the first block
	in~(\ref{Eq:blocks times})) and $i(1) = i(k_1)=1$. Then we choose corresponding
	points $x_1, \dots, x_{k_1}$ with $P_2(x_n)=e^{i(n)}$ and, to fulfill (1),
	$P_1(x_n) =1/n$, $n=1,\dots,k_1$ (the points $x_1, x_{k_1}$ have in fact already
	been defined in this way, see (2)). We define $T$ at the first $k_1$ points by
	putting $T(x_i)=x_{i+1}$ for $i=1,\dots, k_1$. By the construction, $\{0,
	t_{1,1}\}$ is an independence set of times for the map $T$ restricted to the
	first block $\{x_1,\dots, x_{k_1}\}$, even for $T$ restricted to $\{x_1,\dots,
	x_{k_1-1}\}$, and for the family of $2$ sets $(1/\mathbb N) \times \{e^j\}$, $j=
	1, 2$. The first block of the trajectory of $T$ may look like in
	Figure~\ref{F:block1} (where, however, proper scales are ignored).

	\begin{figure} [h]
		\includegraphics[width=14cm]{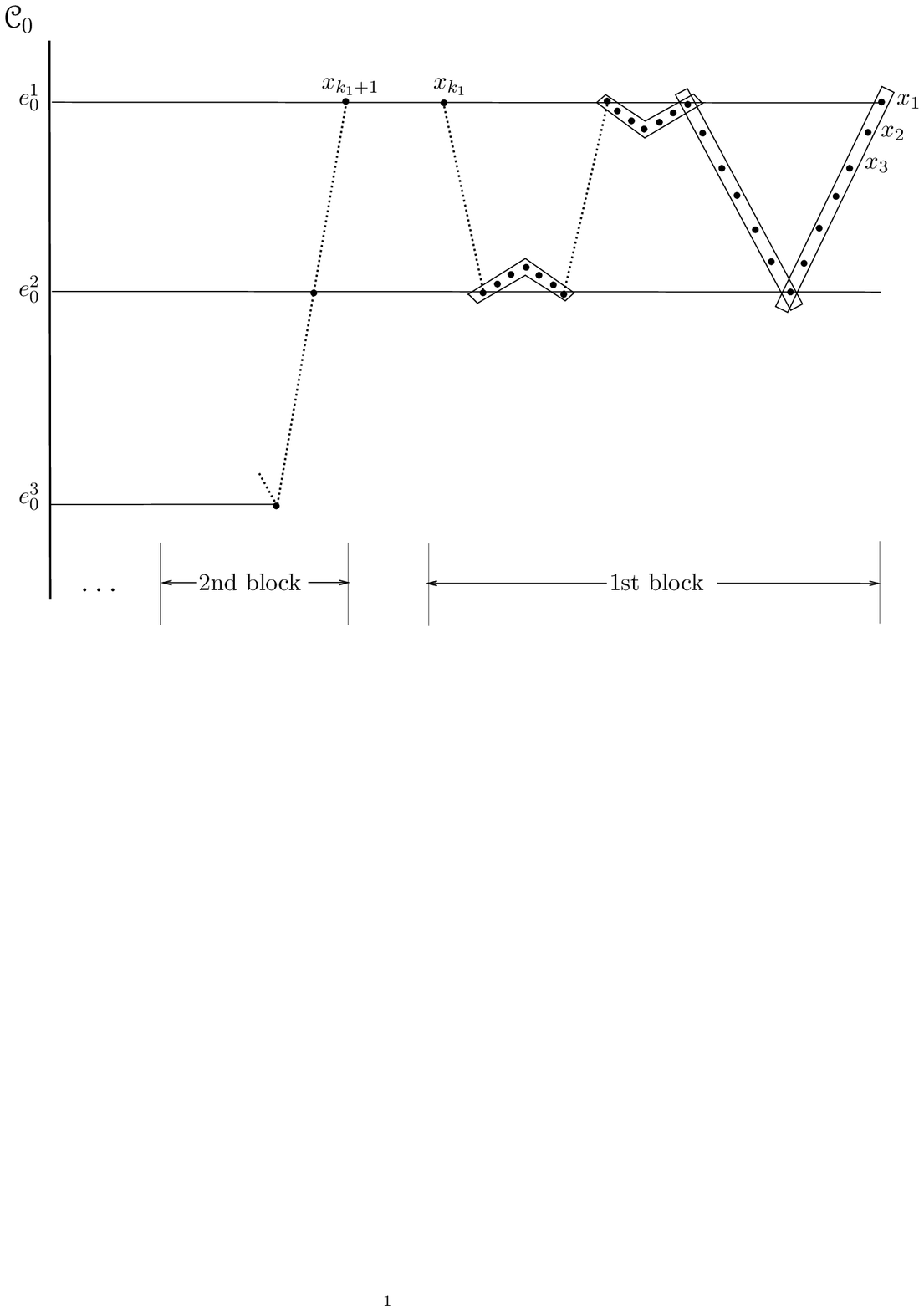}\\
		\caption{The first block (the parts corresponding to the four
			$\varepsilon_1$-chains of the same length are framed; we do not care
			whether they cover the whole block $\{x_1,\dots, x_{k_1}\}$ or not).}\label{F:block1}
	\end{figure}

	To define the second block of the trajectory, we proceed similarly. Again, using
	cofiniteness in Lemma~\ref{L:chains}, we find integers $0 < t_{2,1} < t_{2,2}$,
	perhaps very large ones, such that in $\mathscr C$ there are 27
	$\varepsilon_2$-chains such that
	\begin{itemize}
		\item for any of the 27 choices of not necessarily distinct integers $j_1, j_2,
		j_3$ in the set $\{1,2,3\}$, one of these $\varepsilon_2$-chains is such that it
		starts at $e^{j_1}$ at time $0$, then hits $e^{j_2}$ at time $t_{2,1}$ and ends
		at $e^{j_3}$ at time $t_{2,2}$.
	\end{itemize}
	By concatenating these 27 chains and possibly some intermediate
	$\varepsilon_2$-chains in an appropriate order, we obtain an
	$\varepsilon_2$-chain $e^{i(k_1+1)}, \dots, e^{i(k_2)}$ in $E\subseteq \mathscr
	C$ with some length $k_2-k_1$ (this will be the length of the second block
	in~(\ref{Eq:blocks times})) and $i(k_1+1) = i(k_2)=1$. Then we choose points
	$x_n$, $n=k_1+1,\dots, k_2$ with $P_2(x_n)=e^{i(n)}$ and $P_1(x_n) =  1/n$.
	Again, we put  $T(x_i)=x_{i+1}$ for $i=k_1+1,\dots, k_2$.
	
	Continuing this way, by induction we construct the whole set $C$ and the map
	$T|_C$. Then, since (1), (2), (3a) and (3b) are fulfilled, we know that $X_1$ is
	compact, $T$ is continuous, $h(T)=0$ and $h^*(T)=\infty$ as shown above. The set
	$E_0$ is a sequence entropy set and so the same is true for its closure
	$\mathscr C_0$. The set $\mathscr C_0$ is in fact the unique maximal sequence
	entropy set of the system.

	\begin{rem}\label{R:SX1}
	The space $X_1$ is quite simple and we know that $S(X_1) \supseteq \{0,\infty\}$. Unfortunately, this space admits
	too many continuous selfmaps and it is probable that $S(X_1)$ is larger than $\{0,\infty\}$. Moreover, we prefer to
	have a continuum, while $X_1$ is a disconnected space.
	\end{rem}
	
	Before going to a construction of a continuum $X$ with $S(X)=\{0, \infty\}$, it
	will be instructive to show how we can modify the construction from
	Subsection~\ref{SS:cont+orbit} if we wish the phase space to be a continuum.
	
	\subsection{Modifying $X_1$ to get a continuum}\label{SS:sine}
	In the construction of the space $X_1$ from the previous subsection, let
	$\mathscr C$ be just a straight line segment in the plane. Then the system
	$(X_1, T)$ lives in the plane. If we denote by $I_i$ the straight line segment
	with the endpoints $x_i$ and $x_{i+1}$, then the union of $X_1$ and all the
	segments $I_i$, $i\in \mathbb N$ is a continuum $X^{\Join}_1$.\footnote{In the
		next section we use this trick to get a continuum $X$ with $S(X)=\{0, \infty\}$
		but, instead of arcs, we will use much more complicated continua. Then the space
		$X$ will be complicated, but the dynamics of all possible continuous selfmaps of
		$X$ will be relatively simple and so we will be able to prove that
		$S(X)=\{0,\infty\}$.} The continuous map $T$ sends the endpoints of each $I_i$
	to the endpoints of $I_{i+1}$ and so it can be extended to a continuous map
	which sends $I_i$ onto $I_{i+1}$, $i=1,2,\dots$. What we get is obviously a map
	$T^{\Join} : X^{\Join}_1 \to X^{\Join}_1$ with $h(T^{\Join})=0$ and
	$h^*(T^{\Join})=\infty$.
	
	Moreover, the $\varepsilon_n$-chains in the construction can be chosen and the
	full orbit $\{x_1, x_2, \dots\}$  of the point $x_1$ (here $T(x_i)=x_{i+1}$ for
	every $i$) can be placed in the plane in such a way that the phase space of the
	system $(X^{\Join}_1, T^{\Join})$ is the topologist's sine curve.

	\begin{rem}\label{R:arcwise}
		Still slightly modifying the above construction, one can get a map (even a homeomorphism)
		$T^{\Join}\colon X^{\Join}_1 \to X^{\Join}_1$ with $h(T^{\Join})=0$ and
		$h^*(T^{\Join})=\infty$, where $X^{\Join}_1$ is the Warsaw circle.
		Notice that the Warsaw circle is a uniquely arcwise connected
		continuum, though not locally connected,  and still it admits a
		map with infinite supremum sequence entropy but with zero entropy. It
		is also worth noticing that in all the above examples $h^*(\mathscr T) =\infty$
		while $h^*(\mathscr T|_{\Omega(\mathscr T)})=0$, where $\mathscr T$ stands for
		$T$ or $T^{\Join}$.
	\end{rem}

\section{A continuum $X$ with $S(X) = \{0,\infty\}$}\label{S:X}
	
In this section, we will  construct a continuum $X \subseteq \mathbb R^3$ with $S(X) =
	\{0,\infty\}$.\footnote{We do not know whether $X$ could be found in the plane.}
	In the next subsection we outline the construction and then  we provide necessary
	details in the rest part of the section.
	
	\subsection{Bricks and outline of the construction of $X$}\label{SS:outline}
	
	Now we outline the construction.
	\begin{itemize}
		\item To construct $X$, we use Cook continua $\mathscr{K}_0, \mathscr{K}_1,
		\mathscr{K}_2, \dots \subseteq \R^2$ from Lemma~\ref{L:Cook-family} as building
		bricks (they are subcontinua of, say, the Cook continuum constructed in the
		plane by Ma\'ckowiak). To avoid cumbersome notations below, we will however
		denote them as
		\begin{equation}\label{Eq:bricks}
		\mathscr K, \mathscr{K}_1, \mathscr{K}_2, \dots
		\end{equation}
		(i.e. we omit the index in the notation of zeroth continuum).
		More precisely, our \emph{bricks} will be ho\-meo\-morphic copies of these
		continua, placed in $\mathbb R^3$. We will use only one copy of $\mathscr K$,
		namely in the form $\mathscr K_0 = \{0\}\times \mathscr K$ (zeroth brick), and a
		sequence of sequences of other bricks in $\mathbb R^3$:
		\begin{equation}\label{Kim}
		\begin{split}
		(\mathscr K_i^1)_{i=1}^\infty & = \text{copy of $\mathscr K_1$, copy of
			$\mathscr K_2$, copy of $\mathscr K_3$, copy of $\mathscr K_4$}, \dots \\
		(\mathscr K_i^2)_{i=1}^\infty & = \text{copy of $\mathscr K_2$, copy of
			$\mathscr K_4$, copy of $\mathscr K_6$, copy of $\mathscr K_8$}, \dots \\
		(\mathscr K_i^3)_{i=1}^\infty & = \text{copy of $\mathscr K_4$, copy of
			$\mathscr K_8$, copy of $\mathscr K_{12}$, copy of $\mathscr K_{16}$}, \dots \\
		& \dots
		\end{split}
		\end{equation}
		where copies of the same $\mathscr K_n$ (belonging to different sequences) will
		be different, even disjoint and possibly with different diameters. In fact all
		the bricks will be pairwise disjoint, except of that two consecutive bricks in
		one sequence will have one point in common.
		\item We start the construction with the system $(X_1,T)$ from
		Subsection~\ref{SS:cont+orbit}, where the arbitrary continuum $\mathscr C$ is
		now chosen to be the planar Cook continuum $\mathscr K$ from our
		list~(\ref{Eq:bricks}). Thus,
		\begin{equation}\label{Eq:X1K}
		X_1 = \mathscr K_0 \sqcup C \quad \text{with} \quad C = \{x_1, x_2, \dots\}
		\subseteq \left( 1/\mathbb N \right) \times E,
		\end{equation}
		$E$ being a countable dense subset of $\mathscr K$.
		The space $X_1$ lives in $\mathbb R^3$, we think of the Cook continuum
		$\mathscr K_0 =\{0\}\times \mathscr K$ as of a continuum lying in the
		vertical $yz$-plane (the $x$-axis going to the right).
		\item Then, by using the method from Subsection~\ref{SS:sine}, with the arcs
		$I_m$ (joining consecutive points $x_m$ and $x_{m+1}$ of the orbit $C$) replaced
		by appropriate continua $D_m$, we finally obtain our continuum $X\supseteq X_1$
		in $\mathbb R^3$. The continua $D_1, D_2, D_3, \dots$  will have the form of
		`infinite chains' from~(\ref{Kim}). More precisely, for every $m=1,2,\dots$ the
		continuum $D_m$ will be the closure of the union of all sets in the $m$-th
		sequence in~(\ref{Kim}) and $D_m$ and $D_{m+1}$ will have just one point in
		common, namely $x_{m+1}$.
		\item We will need to study continuous selfmaps of $X$. First recall that
		$T: X_1\to X_1$ is an injective map defined by
		\begin{equation}\label{Eq:Tagain}
		T|_{\mathscr K_0} = \Id_{\mathscr K_0} \quad \text{ and } \quad T(x_n) =
		x_{n+1}~.
		\end{equation}
		By Subsection~\ref{SS:cont+orbit}, for the system  $(X_1,T)$ we have $h(T)=0$
		and $h^*(T)=\infty$. In Lemma~\ref{L:existsG} we then extend  $T: X_1\to X_1$ to
		a continuous map $G:X \to X$. Clearly, $h^*(G)=\infty$. We also carefully study
		the dynamics of all non-constant continuous maps $F: X\to X$. We show that they
		are of two kinds. Some of them coincide with an iterate of $G$ on a substantial
		part of $X$, see Lemma~\ref{L:jumps}, and then $h^*(F)=\infty$. The others have
		quite a simple dynamics, see Corollary~\ref{C:FixF} and Lemmas~\ref{L:left}
		and~\ref{L:right}, and then $h^*(F)=0$. Hence $S(X) = \{0,\infty\}$, see
		Proposition~\ref{P:zero inf}.
	\end{itemize}

	\subsection{Notation for bricks and details of the construction of
		$X$}\label{SS:construction of X}
	As described above, we start with the space $X_1\subseteq \mathbb R^3$
	from~(\ref{Eq:X1K}), with the Cook continuum $\mathscr K_0 =\{0\}\times \mathscr
	K$ lying in the vertical $yz$-plane $\pi_0 = \{0\}\times \mathbb R^2$. Thus the
	zero-th brick from the list~(\ref{Eq:bricks}) has already been used. From now on
	we will use only bricks which are homeomorphic copies of the Cook continua in
	the set
	\begin{equation}\label{bricks in snake}
	D= \{\mathscr{K}_1, \mathscr{K}_2, \dots\}.
	\end{equation}
	
	Consider any two consecutive points $x_m, x_{m+1} \in C$. They are in two
	vertical planes, $x_m \in \pi_m = \{1/m\}\times \mathbb R^2$ and $x_{m+1} \in
	\pi_{m+1} = \{1/(m+1)\}\times \mathbb R^2$. In the straight line segment with
	the endpoints $x_m$ and $x_{m+1}$, choose a strictly monotone sequence of points
	starting with $x_m$ and converging to $x_{m+1}$.
	For any two consecutive points in this sequence one can choose a `two cones base
	to base' solid \index{solid}, i.e. the solid obtained by a rotation of a kite along its axis
	of symmetry (we shortly call it a `solid'), such that the two mentioned
	consecutive points are the two vertices of the solid.
	Thus, the points $x_m$ and $x_{m+1}$ are `joined' by a sequence of solids
	\begin{equation}
	\label{Eq:solids}
	S_1^m \ni x_m, S_2^m, S_3^m, \dots
	\end{equation}
	monotonically converging to $x_{m+1}$, see Figure~\ref{F:solids}.

	\begin{figure}[h]
		\begin{minipage}[t]{0.5\linewidth}
			\centering
			\includegraphics[width=8cm]{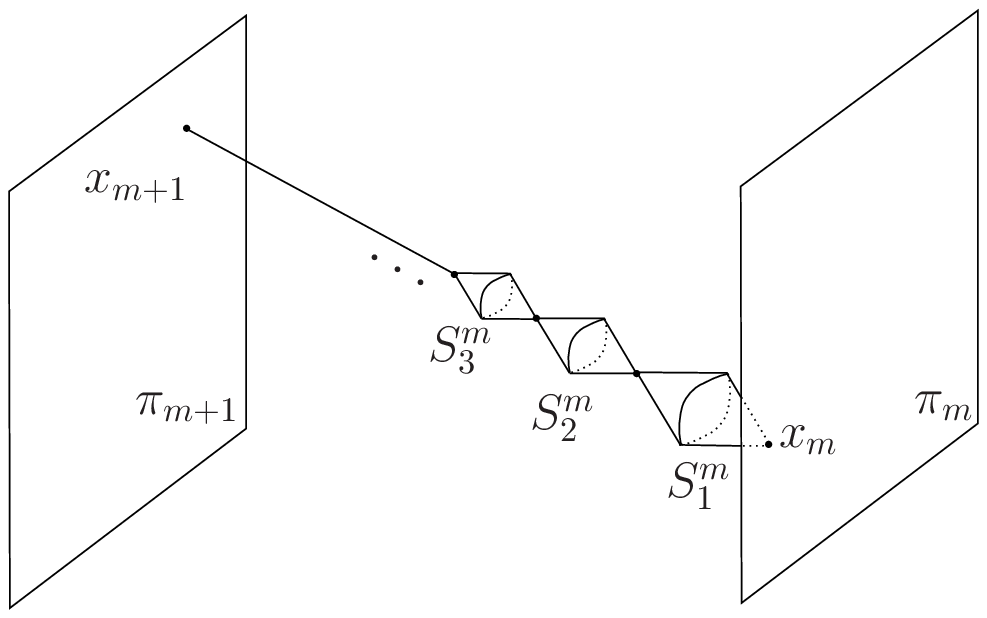}
			\caption{The sequence of solids `joining' $x_m$ and
				$x_{m+1}$.}\label{F:solids}
		\end{minipage}%
		\begin{minipage}[t]{0.5\linewidth}
			\centering
		    \includegraphics[width=8cm]{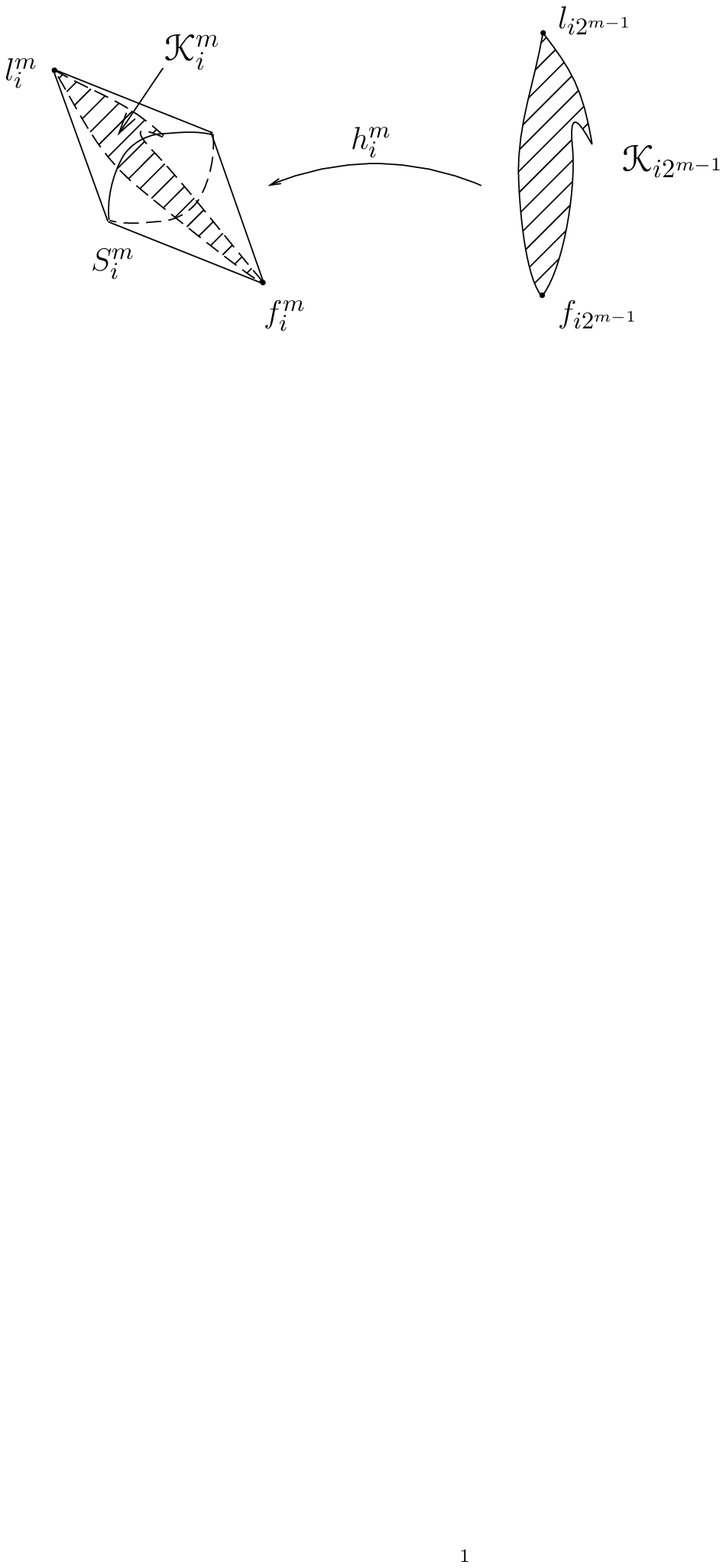}
			\caption{A Cook continuum $\mathscr K_i^m$ inside the solid
				$S_i^m$.}\label{F:in solid}
		\end{minipage}
	\end{figure}
	
	\noindent By choosing sufficiently small cone angles we may assume that:
	\begin{itemize}
		\item each solid has the diameter equal to the distance of its vertices (hence
		the diameters of the solids in the sequence~(\ref{Eq:solids}) tend to zero),
		\item two consecutive solids intersect only at their common vertex,
		\item given a vertex of a solid, no other point of the solid has the same
		$P_1$-projection onto the $x$-axis,
		\item the union of the solids in~(\ref{Eq:solids}) lies strictly between the
		vertical planes $\pi_m$ and $\pi_{m+1}$, with the exception of the point $x_m$
		(one of the vertices of the first solid $S_1^m$) which lies in $\pi_m$.
	\end{itemize}

	Now we are going to specify the choice of copies of $\mathscr K_i$
	in~(\ref{Kim}), i.e. the choice of all bricks
	different from~$\mathscr K_0$. First, for each $\mathscr K_i$, $i=1,2,\dots$
	fix, once and for all,  two \emph{extremal points} \index{extremal points} $f_i, \ell_i \in \mathscr
	K_i$, i.e. points whose (Euclidean) distance equals the diameter of $\mathscr
	K_i$. We will call them the \emph{first point}  and
	the \emph{last point} of $\mathscr K_i$ \index{the first point and last point of $\mathscr K_i$}, respectively.
	For $m=1,2,\dots$ and $i=1,2, \dots$,
	\begin{equation}\label{Eq:hcopy}
	\text{let $\mathscr K_i^m$ be a homeomorphic copy of  $\mathscr K_{i2^{m-1}}$}
	\end{equation}\index{$\mathscr K_i^m$}
	(this corresponds to~(\ref{Kim})) and consider a homeomorphism
	\begin{equation}\label{Eq: hmi}
	h_i^m: \mathscr K_{i2^{m-1}} \to \mathscr K_i^m
	\end{equation}
	(by Lemma~\ref{L:Cook-properties}(3), such a homeomorphism exists exactly one).
	Clearly, we can choose $\mathscr K_i^m$ with the following properties:
	\begin{itemize}
		\item $\mathscr K_i^m \subseteq S_i^m$,
		\item the $h_i^m$-images of extremal points of $\mathscr K_{i2^{m-1}}$, i.e. the
		points $f_i^m := h_i^m(f_{i2^{m-1}})$ and $\ell_i^m := h_i^m(\ell_{i2^{m-1}})$,
		coincide with the two vertices of the solid $S_i^m$, see Figure~\ref{F:in
			solid}.
	\end{itemize}

	\noindent So, when going from $x_m$ towards $x_{m+1}$, we meet the points
	$x_m=f_1^m, \ell_1^m = f_2^m, \ell_2^m = f_3^m, \dots$. The points $f_i^m$ and
	$\ell_i^m$ will be called \emph{the first and the last point (the extremal
		points) of~$\mathscr K_i^m$} \index{the first and the last point (the extremal
		points) of~$\mathscr K_i^m$}.\footnote{By the way, the terminology is in
		accordance with the fact that, due to the choice of $S_i^m$, these points are
		really extremal points of $\mathscr K_i^m$ in the sense of distance.} The other
	points of $\mathscr K_i^m$ are its non-extremal points.

	The sequence $(\mathscr K_i^m)_{i=1}^\infty$ `joins', in a sense, $x_m$ with
	$x_{m+1}$. We adopt the notations \index{$D_m$} \index{$D_m^*$}
	\begin{equation}\label{Eq: DD*}
	D_m= \bigcup_{i=1}^{\infty} \mathscr K_i^m \cup \{x_{m+1}\} \quad \text{ and }
	\quad D_m^* = \bigcup_{i=1}^{\infty} \mathscr K_i^m= D_m \setminus \{x_{m+1}\}.
	\end{equation}
	The first three sets $D_m$ can be seen in Figure~\ref{F:3Dm}; the bricks
	belonging to the same $D_m$ are pairwise non-homeomorphic (because we
	have~(\ref{Kim}) and the continua $\mathscr K_i$ in~(\ref{Eq:bricks}) are not
	homeomorphic).
	
	\begin{figure} [h]
		\includegraphics[width=16cm]{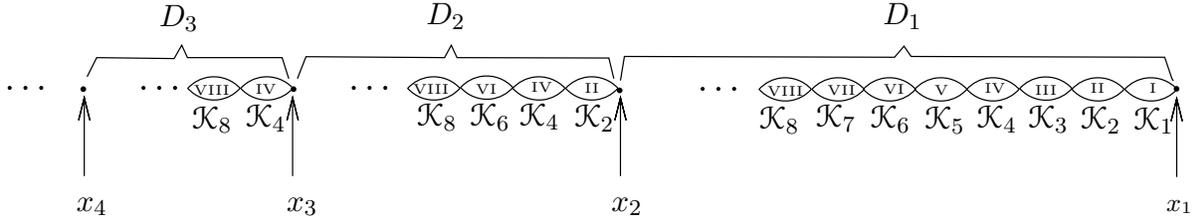}\\
		\caption{The first three sets $D_m$. Instead of ``copy of $\mathscr K_i$" we write just ``$\mathscr K_i$".}\label{F:3Dm}
	\end{figure}

	By the construction, we have the following obvious lemma (the convergence in (4)
	is the convergence in Hausdorff metric derived from the Euclidean metric in
	$\mathbb R^3$; we identify $\{x_{m+1}\}$ with $x_{m+1}$).
	
	\begin{lem}\label{L:block}
		For each $m =1,2,\dots,$ we have:
		\begin{enumerate}
			\item $x_m$ is the first point of $\mathscr K_1^m$, i.e. $x_m = f_1^m$;
			\item $x_{m+1} \not \in \bigcup_{i=1}^\infty \mathscr K_i^m = D_m^*$;
			\item $\mathscr K_i^m\cap \mathscr K_{i+1}^m = \{\ell_i^m\} = \{f_{i+1}^m\}$;
			\item $\mathscr K_n^m \to x_{m+1}$ when $n\rightarrow \infty$.
			\item $D_m= \overline{\bigcup_{i=1}^{\infty} \mathscr K_i^m} =
			\overline{D_m^*}$.
			\item The diameters of $D_m$ tend to zero when $m\to \infty$.
		\end{enumerate}
	\end{lem}
	
	Finally, we define
	\begin{equation}\label{X-definition}
	X=\mathscr K_0 \sqcup \bigcup_{m=1}^{\infty}D_m =  \mathscr K_0 \sqcup
	\bigsqcup_{m=1}^{\infty}D_m^* .
	\end{equation}
	Thus $X$ is the union of bricks (planar Cook continua) $\mathscr K_0$ and $\mathscr
	K_i^m$, $m=1,2,\dots$, $i=1,2,\dots$.
	
\begin{lem}\label{L:Xiscont}
	The space $X\subseteq \mathbb R^3$ defined by~\eqref{X-definition} is a one-dimensional continuum.	
\end{lem}

\begin{proof}
Since $X_1$ is compact and the diameters of $D_m$ tend to zero, see Lemma~\ref{L:block}(6), all those cluster points of the snake which are in $\pi_0$ belong to the head $\mathscr K_0$. Hence $X$ is compact. By Lemma~\ref{L:1-dim}, all the bricks $\mathscr K_0$ and $\mathscr
K_i^m$ are one-dimensional and since the union of countably many closed one-dimensional sets is one-dimensional, also $X$ is one-dimensional, see e.g.~\cite[Theorem III.2]{HW}.
\end{proof}	

The continuum $X$ is depicted in Figure~\ref{F:spaceX}. It lies in $\mathbb R^3$ and we are going to prove that $S(X)=\{0,\infty\}$; this will be done in Proposition~\ref{P:zero inf} below. However, the proof requires some preparation.
	
	\begin{figure}[h]
		\includegraphics[width=14cm]{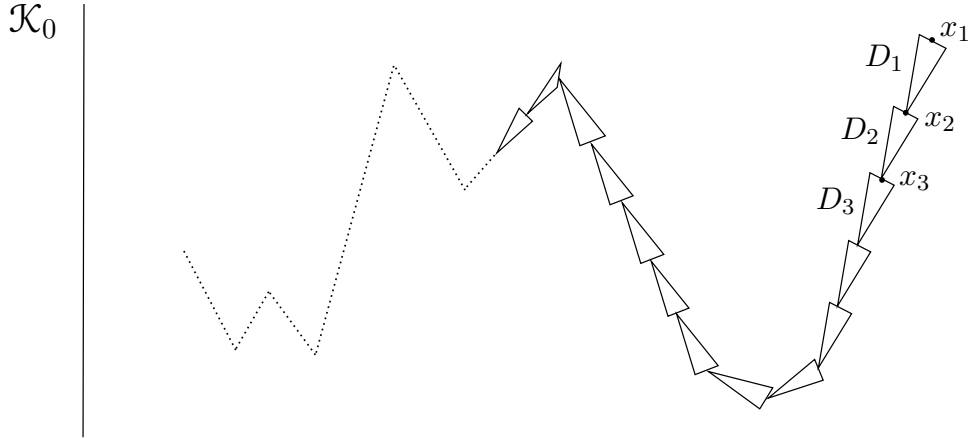}\\
		\caption{The continuum $X$ with $S(X)=\{0,\infty\}$; see also
			Figures~\ref{F:block1} and~\ref{F:3Dm}.}\label{F:spaceX}
	\end{figure}

	\begin{quotation}
		{\bf Standing notation for the rest of Section~\ref{S:X}:}
		In the rest of the section, $X$ denotes the space $X$ constructed above
		in~(\ref{X-definition})
		and $F$ denotes a continuous map $X\to X$.
	\end{quotation}

	\subsection{More terminology for parts of the space $X$. Properties of
		$X$}\label{SS:X is continuum}
	Call the brick $\mathscr K_0$ the \emph{head} \index{head} and the union of all other bricks
	in $X$
	the \emph{snake} \index{snake}. In view of~(\ref{X-definition}), $X$ is the disjoint union of
	the head and the snake.
	Let
	$$
	\Sigma = \{f_i^m, \ell_i^m:\, m, i =1,2,\dots\} = \{f_i^m:\, m, i=1,2,\dots\}
	$$
	be the set of extremal points of the bricks belonging to the snake.
	We will say that $a\in X$ is \emph{to the left of} $b\in X$  and we will write
	$a \prec b$  \index{to the left of} \index{$\prec$}if for their
	first projections we have $P_1(a)<P_1(b)$. In such a case we will
	also say that $b$ is \emph{to the right of} $a$ and write $b \succ a$ \index{to the right of} \index{$\succ$}. If
	$a\prec c \prec b$, we say that
	$c$ is \emph{between} $a$ and $b$ (or between $b$ and~$a$).
	Further, $a \preccurlyeq  b$ \index{$\preccurlyeq$} or $b \succcurlyeq a$ \index{$\succcurlyeq$} means that $a=b$ or $a\prec
	b$  (note that $a\preccurlyeq b$ is not equivalent with $P_1(a)\leq P_1(b)$).
	Clearly, $\preccurlyeq$ is a partial order on $X$. In the obvious meaning, the
	head $\mathscr K_0$ is \emph{on the left end} of $X$ and $x_1$ is
	\emph{the rightmost point} of $X$. Note that the last point $\ell_i^m$ of
	$\mathscr K_i^m$ is its leftmost
	point, similarly its first point $f_i^m$  is its rightmost point. So $\ell^m_i
	\prec f^m_i$ and not conversely.\footnote{The terminology ``first" and ``last"
		is related to the fact that we go along the snake towards the head, i.e. in the
		opposite direction as is the direction of the $x$-axis, while the terminology
		``left" and ``right" corresponds to the direction of the $x$-axis.}
	
	If $A,B$ are any subsets of $X$, we write $A\prec B$ (or $B\succ A$) if $a\prec
	b$ whenever $a\in A$ and $b\in B$.
	In such a case we say that $A$ is to the left of $B$ and $B$ is to the right of
	$A$. If also $C\subseteq X$ and
	$A\prec C \prec B$, we say that $C$ is between $A$ and $B$ (or between $B$ and
	$A$). We are not going to define $A\preccurlyeq B$ if $A,B$ are \emph{arbitrary}
	subsets of $X$. However, for \emph{special} subsets of $X$ we will do that in an
	appropriate, maybe not the most natural, way in Subsection~\ref{SS:induced}.

	If we restrict the partial order $\preccurlyeq$ to the set $\Sigma$, we get a
	linearly ordered set.
	The linearly ordered set $(\Sigma, \succcurlyeq)$, i.e. the one with reversed
	order, is clearly well ordered. Therefore we immediately get the following fact.
	
	\begin{lem}\label{L:antiwell1}
		Every nonempty subset of $(\Sigma, \preccurlyeq)$ has the largest element.
	\end{lem}
	
	If $a, b \in \Sigma$, we denote by
	$\langle \langle a, b \rangle \rangle = \langle \langle b, a \rangle \rangle$ \index{$\langle \langle a, b \rangle \rangle$}
	the union of the set $\{a,b\}$ with the set of all points in $X$
	which are between $a$ and $b$; in particular, $\langle \langle a, a \rangle
	\rangle =\{a\}$,
	$\mathscr K_i^m= \langle \langle f_i^m, \ell_i^m \rangle \rangle$ and
	$D_m=\langle \langle x_m,x_{m+1}\rangle \rangle$. We will also use the notation
	$\langle \langle \mathscr K_0, b \rangle \rangle$\index{$\langle \langle \mathscr K_0, b \rangle \rangle$} for
	the union of the set $\mathscr K_0 \cup \{b\}$ with the set of all points which
	are to the right of all points of
	$\mathscr K_0$ and to the left of $b$. Clearly, $\langle \langle \mathscr K_0,
	x_m \rangle \rangle = \mathscr K_0 \sqcup \bigcup _{i=1}^m D_i$.  Further, we
	put $((\mathscr K_0, b \rangle \rangle = \langle \langle \mathscr K_0, b \rangle
	\rangle \setminus \mathscr K_0$ \index{$((\mathscr K_0, b \rangle \rangle$}. The sets of the form $((\mathscr K_0,
	f^m_i\rangle \rangle$ are said to be \emph{sub-snakes} \index{sub-snake}.

	Let $a\preccurlyeq b$ be in $\Sigma$. Let $r: X\to \langle \langle a, b \rangle
	\rangle$ be  the map sending all the points which are to the left of $a$ to the
	point $a$ and all the points which are to the right of $b$ to the point $b$,
	leaving all the points from $\langle \langle a, b \rangle \rangle$ fixed. Then
	$r$ is obviously continuous, so it is a retraction of $X$ onto $\langle \langle
	a, b \rangle \rangle$. We will call it a \emph{monotone retraction}\index{monotone retraction}.\footnote{In
		fact there are no other retractions of $X$ onto $\langle \langle a, b \rangle
		\rangle$, but this fact is not important for us and so we will not prove it.}
	
	Recall that if $M$ is a connected set then $p\in M$ is called a
	\emph{cut point} \index{cut point} of $M$ if $M\setminus \{p\}$ is not connected.

	\begin{lem}\label{L:space}
		The space $X$ constructed above in~(\ref{X-definition}) has the following
		properties:
		\begin{enumerate}
			\item The sets $D_m^*$, the snake and the sub-snakes are connected.
			The sets $D_m$, $\langle\langle a, b \rangle\rangle$ and $\langle \langle
			\mathscr K_0, b \rangle \rangle$, where $a, b\in \Sigma$, $a\prec b$, are
			continua. The space $X$ is the closure of the snake. The head $\mathscr K_0$ is
			in the closure of any sub-snake.
			\item $X$ is a one-dimensional continuum in $\mathbb R^3$.
			\item Every point from $\Sigma \setminus \{x_1\}$ is a cut point of $X$.
			\item $\Sigma$ is totally disconnected, i.e. it does not contain any
			non-degenerate continuum.
			\item Every brick $\mathscr K_i^{m}$ as well as every set $D_m$ are monotone
			retracts of $X$.
			\item If a non-degenerate continuum $Q \subseteq X$ intersects a brick $\mathscr
			K_i^m$ in a point different from the extreme points of $\mathscr K_i^m$, then
			$Q\cap \mathscr K_i^m$ is a non-degenerate continuum.
		\end{enumerate}
	\end{lem}
	
	\begin{proof}
		(1-2) $X$ is one-dimensional by the discussion after~\eqref{X-definition}. We prove the rest.
		For every $m$, the bounded set $D_m^* = \bigcup_{i=1}^{\infty} \mathscr
		K_i^m$ is connected because consecutive continua $K_i^m$,  $K_{i+1}^m$
		intersect. Therefore the closure of this set, i.e. the set $D_m$, is a
		continuum. An analogous argument gives that every set $\langle\langle a, b
		\rangle\rangle$ is a continuum. Similarly, the snake $\bigcup_{m=1}^{\infty}D_m$
		is bounded and connected and so its closure is a continuum. For the same reason
		all the sub-snakes are connected.  Thus, to finish the proof of (1-2), it is
		sufficient to show that $X$ coincides with the closure of the snake (since the
		points of the snake which are outside of a given sub-snake have distance from
		the head larger than some positive constant, it is obvious that then $\mathscr
		K_0$ is in the closure of the sub-snake).
		
		Each finite union $\bigcup_{m=1}^{N}D_m$ is compact, so it is sufficient to show
		that the intersection of the closure of the snake with the vertical plane
		$\pi_0$ containing $\mathscr K_0$ is exactly the set $\mathscr K_0$. By
		construction from Subsection~\ref{SS:cont+orbit} this intersection contains the
		set $E_0$ dense in $\mathscr K_0$ and so it contains $\mathscr K_0$. To prove
		the converse inclusion, let $(y_n)_{n=1}^\infty$ be a sequence in the snake
		converging to a point $\omega \in \pi_0$. We need to prove that  $\omega \in
		\mathscr K_0$.
		
		Consider also the sequence $(x_n)_{n=1}^\infty$ of points forming the set $C$
		in~(\ref{Eq:X1K}). Recall that the distance from $x_n$ to $\pi_0$ is $1/n$ and
		if a subsequence of $x_n$ converges to some point, then this point belongs to
		the closure of $E_0$, i.e. is in $\mathscr K_0$. Let $\delta_n$ be the diameter
		of $D_n$. Each $y_n$ belongs to (at least one) $D_{k(n)}$. This means that the
		(Euclidean) distance in $R^3$ between $y_n$ and $x_{k(n)}$ is at most
		$\delta_{k(n)}$. Since $y_n$ converges to $\omega$ and, by Lemma~\ref{L:block},
		$\delta_{k(n)}$ converges to zero, then $x_{k(n)}$ converges to $\omega$. Hence
		$\omega \in \mathscr K_0$.
		
		(3) If $c \in \Sigma \setminus \{x_1\}$ then $X\setminus \{c\}$ is
		the disjoint union  of the set $U$ of all points of $X$ which are to
		the left of $c$ and the set $V$ of all points of $X$ which are to
		the right of $c$. Both $U$ and $V$ are nonempty and open in $X$ and so
		$X\setminus \{c\}$ is not connected.
		
		(4) It is obvious that $\Sigma$ is totally disconnected.
		
		(5) Even more has already been explained when defining the monotone retraction.
		
		(6) To simplify the notation, instead of $\mathscr K_i^m$, $f_i^m$ and
		$\ell_i^m$ write for a moment $\mathscr K$, $f$ and $\ell$. We have $X= L \sqcup
		\mathscr K \sqcup R$ where $L$ or $R$ are the sets of all those points from $X$
		which are to the left of $\ell$ or to the right of $f$, respectively.
		
		By the assumption, some point $b\in Q$ belongs to $\mathscr K\setminus \{\ell,
		f\}$. If $Q\subseteq \mathscr K$ then $Q\cap \mathscr K = Q$ and we are done. So
		assume that $Q$ intersects $L\cup R$. Without loss of generality suppose that
		there is a point $a \in Q\cap L$. Then, since $Q$ contains both $a$ and $b$ and,
		by the proof of (3), $\ell$ separates $a$ and $b$ in $X$, the continuum $Q$
		contains $\ell$. Thus the compact set $Q\cap (\mathscr K \sqcup R)$ contains
		$\ell$ and $b$. We claim that it is a continuum. Otherwise
		$Q\cap (\mathscr K \sqcup R) = A\sqcup B$ where $A,B$ are nonempty compact sets.
		One of them, say $A$, contains $\ell$.
		Then $Q= \left((Q\cap L)\cup A\right) \sqcup B$ is a disjoint union of two
		nonempty compacts sets, which contradicts the fact that it is a continuum.
		
		We have proved that $Q\cap (\mathscr K \sqcup R) = Q\setminus L$ is a
		non-degenerate continuum. If it is disjoint with $R$ then we are done because
		$Q\cap \mathscr K = (Q\setminus L) \setminus R = Q\setminus L$. Otherwise, the
		non-degenerate continuum $Q\setminus L$ (containing $b$ and $\ell$) contains a
		point $c\in R$ and by repeating an argument from above we get that $(Q\setminus
		L)\setminus R = Q\cap \mathscr K$ is a nondegenerate continuum.
	\end{proof}

	\subsection{Properties of continuous selfmaps of $X$}\label{SS:cont-bricks}
	
	The list of all bricks is: $\mathscr K_0, \mathscr K_1^1, \mathscr K_2^1, \dots,
	\mathscr K_1^2, \mathscr K_2^2, \dots, \dots$.
	It follows from the construction that two terms of this list are homeomorphic
	sets if and only if
	they are copies of the same of the continua $\mathscr K_1, \mathscr K_2, \dots$,
	i.e. if and only if they are of the form
	$\mathscr K_i^m$ and $\mathscr K_j^n$ with $i2^{m-1} = j2^{n-1}$.
	
	\begin{lem}\label{L:Kim to Kjn}
		Consider bricks in $X$ and in particular the brick $\mathscr K_0$ and the bricks
		$\mathscr K_i^m = h_i^m (\mathscr K_{i2^{m-1}})$
		and $\mathscr K_j^n = h_j^n (\mathscr K_{j2^{n-1}})$.
		\begin{itemize}
			\item [(a)] The only continuous selfmaps of a brick are constant maps and the
			identity.
			\item [(b)] The only continuous maps $\mathscr K_i^m \to \mathscr K_0$ are
			constant maps.
			\item [(c)] The only continuous maps $\mathscr K_0 \to \mathscr K_i^m$ are
			constant maps.
			\item [(d)] If the sets $\mathscr K_{i2^{m-1}}$ and $\mathscr K_{j2^{n-1}}$ are
			different terms of the sequence $\mathscr K_1, \mathscr K_2, \dots$
			(i.e., $i2^{m-1} \neq j2^{n-1}$), then the only continuous maps $\mathscr K_i^m
			\to \mathscr K_j^n$
			are constant maps.
			\item [(e)] If the sets $\mathscr K_{i2^{m-1}}$ and $\mathscr K_{j2^{n-1}}$ are
			equal (i.e., $i2^{m-1} = j2^{n-1}$) then there exists a non-constant continuous
			map $\varphi : \mathscr K_i^m \to \mathscr K_j^n$. Such a map exists only one,
			it is a homeomorphism, $\varphi (f_i^m) = f_j^n$ and $\varphi (\ell_i^m) =
			\ell_j^n$.
		\end{itemize}
	\end{lem}
	
	\begin{proof}
		(a) follows from Corollary~\ref{C:Cook-family}($\widetilde{a}$).
		(b), (c) and (d) follow from Corollary~\ref{C:Cook-family}($\widetilde{b}$).
		
		(e) Put $\mathscr K^*:= \mathscr K_{i2^{m-1}} = \mathscr K_{j2^{n-1}}$. The
		homeomorphism $H = h_j^n \circ (h_i^m)^{-1}: \mathscr K_i^m \to \mathscr K_j^n$
		obviously has all the required properties. Now let $\varphi : \mathscr K_i^m \to
		\mathscr K_j^n$ be any non-constant continuous map. Then $(h_j^n)^{-1} \circ
		\varphi \circ h_i^m$ is a non-constant continuous map $\mathscr K^*  \to
		\mathscr K^*$. Since $\mathscr K^*$ is a Cook continuum, the only non-constant
		continuous selfmap of $\mathscr K^*$ is the identity. Thus $(h_j^n)^{-1} \circ
		\varphi \circ h_i^m = \Id|_{\mathscr K^*}$ whence $\varphi = h_j^n \circ
		(h_i^m)^{-1} =H$ and so the proof of uniqueness is finished.
	\end{proof}

	\begin{lem}\label{L:map}
		If $B$ is a  brick then $F(B)$ is either a singleton or a brick homeomorphic to
		$B$.
	\end{lem}
	
	\begin{proof} The snake can be written in the form $\bigcup _{i=1}^\infty B_i$
		where $B_i$ are bricks. Fix a brick $B$. It is either the head $\mathscr K_0$ or
		one of the bricks $B_i$. Denote $F|_B$ by $G$. Assume that $G(B)$ is not a
		singleton.  Then it is a	non-degenerate continuum and so, by
		Lemma~\ref{L:space}(4), it
		cannot be a subset of $\Sigma$. To prove that it is a brick,
		distinguish three cases.
		
		\smallskip
		
		\emph{Case 1 : $G(B)$ is a non-degenerate subset of the head $\mathscr K_0$.} By
		Lemma~\ref{L:Kim to Kjn}(b), $B$ cannot be any of those $B_i$.
		Therefore $B= \mathscr K_0$. By Lemma~\ref{L:Kim to Kjn}(a), $G$ is the identity
		and so $G(B)= \mathscr K_0$.

		\emph{Case 2 : $G(B)$ is a non-degenerate subset of the snake.} Let $B_k$  be
		one of those bricks in the snake, which are intersected by $G(B)$ in a
		non-extremal point (i.e. in a point different from the first and last
		point of $B_k$). By Lemma~\ref{L:space}(6) the continuum $G(B)$
		intersects $B_k$ in a non-degenerate subcontinuum. Now let $r_k$ be the monotone
		retraction of $X$ onto $B_k$. The map $r_k\circ G : B\to B_k$
		is continuous and not constant. By Lemma~\ref{L:Kim to Kjn}, $B_k$ is a copy of
		$B$ and
		$r_k\circ G (B) = B_k$. Hence, taking into account the special form of the
		retraction $r_k$,  $G (B)$ contains all the points of $B_k$ with possible
		exceptions of the extremal points of $B_k$. However, $G(B)$ is a compact set,
		therefore $G(B) \supseteq B_k$.
		
		Let us summarize. We have proved that if $k$ is such that $G(B)$ intersects
		$B_k$ in a non-extremal point, then $B_k$ is a copy of $B$ (hence $B$ is not the
		head $\mathscr K_0$, because in the snake we do not have any brick homeomorphic
		to $\mathscr K_0$) and $G(B) \supseteq B_k$. We claim that such a brick $B_k$ is
		only one. Suppose, on the contrary, that $G(B)$ intersects two bricks in the
		snake in non-extremal points. Since the projection of the continuum $G(B)$ onto
		the $x$-axis is again a continuum, $G(B)$ intersects also two \emph{consecutive}
		bricks in the snake in non-extremal points. We know that each of these two
		consecutive bricks is a copy of $B$. However, no two consecutive bricks in the
		snake are homeomorphic, a contradiction.
		
		Thus $G(B)$ intersects only one brick $B_k$ in the snake in a non-extremal
		point. Then $G(B)\supseteq B_k$ and $G(B)$ does not intersect any other brick in
		a non-extremal point. It follows that $G(B)=B_k$ and the proof in this case is
		finished, because the fact that $B_k$ is a copy of $B$ has already been
		established.
		
		\emph{Case 3 : $G(B)$ intersects both the head and the snake.} So, the continuum
		$G(B)$ intersects both $\mathscr K_0$
		and some brick in the snake. Then
		each brick in the snake which is to the left of this one,
		is intersected by $G(B)$, even in non-extremal points. This is in particular
		true for some two consecutive bricks and in the same way as in Case 2 we get
		that these two consecutive bricks in the snake are homeomorphic. This
		contradiction shows that Case 3 is impossible.
	\end{proof}
	
	\begin{cor}\label{C:union}
		Let $\mathscr B$ be a subfamily of the family of all bricks and let $F(\bigcup
		\mathscr B) \supseteq C$ for some brick $C$. Then there is a brick $B\in
		\mathscr B$ such that $B$ is homeomorphic to $C$ and $F(B)=C$.
	\end{cor}
	
	\begin{proof}
		There are only countably many bricks in $\mathscr B$ and $C$ has cardinality
		$\mathfrak c$. Therefore there is a brick $B\in \mathscr B$ such that $F(B)$ intersects
		$C$ in more than one point. Hence, by Lemma~\ref{L:map}, $F(B)$ is a brick
		homeomorphic to $B$ and since different bricks have at most one point in common,
		$F(B)=C$.
	\end{proof}
	
	If $\mathscr K$ is a brick in the snake with extremal points $f$ and $\ell$, put
	$\mathscr K^\circ := \mathscr K \setminus \{f,\ell\}$ \index{$\mathscr K^\circ$}.\footnote{If the brick
		$\mathscr K$ is neither $\mathscr K_0$ nor $\mathscr K^1_1$, then $\mathscr
		K^\circ$ is the interior of $\mathscr K$ in the space $X$.}
	
	\begin{cor}\label{C:brick-brick}
		Let $B$ and $C$ be bricks in the snake such that $F(B)=C$. Then $F|_B: B\to C$
		is a homeomorphism sending the first and the last point of $B$ to the first and
		the last point of $C$, respectively. So, $F(B^\circ) = C^\circ$.	
	\end{cor}
	
	\begin{proof}
		By Lemma~\ref{L:map}, $C$ is homeomorphic to $B$. The rest follows from
		Lemma~\ref{L:Kim to Kjn}(a)(e).	
	\end{proof}
	
	The following immediately follows from Lemma~\ref{L:map} and
	Corollary~\ref{C:brick-brick}.
	
	\begin{cor}\label{C:hull}
		Let $\ell \prec f$ be consecutive points of $\Sigma$ mapped by $F$ to
		consecutive points $F(\ell) \prec F(f)$ of $\Sigma$. Then the brick
		$\langle \langle \ell, f \rangle \rangle$ is mapped onto the brick $\langle
		\langle F(\ell), F(f) \rangle \rangle$.
	\end{cor}
	
	\begin{lem}\label{L:const 1}
		\begin{itemize}
			\item [(a)] If a point $b$ of a brick $B$ is fixed for $F$ then either $b$ is
			the only fixed point in $B$ and then $F(B)=\{b\}$, or all points of $B$ are
			fixed.
			\item [(b)] If two points of a brick are fixed then all points of the brick
			are fixed.
			\item [(c)] If $B$ is a brick and two points of $B$ are mapped by $F$ to the
			same point, then $F(B)$ is a singleton.
			\item [(d)] If $B_1$ and $B_2$ are bricks in the snake (different or not,
			homeomorphic or not) and one of the extremal points of $B_1$ is mapped by $F$ to
			a non-extremal point $z\in B_2$ then $F(B_1)=\{z\}$.
			\item [(e)] If $B_1$ and $B_2$ are bricks in the same set $D_m$ and one of the
			points of $B_1$ is mapped by $F$ to a point $z\in B_2\setminus B_1$ then
			$F(B_1)=\{z\}$.
			\item [(f)] If $B$ is a brick in the snake and $F(B)$ intersects the head
			$\mathscr K_0$ then $F(B)$ is a singleton in $\mathscr K_0$.
		\end{itemize}
	\end{lem}

	\begin{proof}
		(a) By Lemma~\ref{L:map}, $F(B) =\{b\}$ or $F(B)$ is a brick homeomorphic to
		$B$ and clearly containing $b$. However, by construction of $X$, the bricks
		which are different from $B$ but intersect $B$ are not homeomorphic to $B$. Thus
		$F(B) =\{b\}$ or $F(B)=B$. In the latter case, all the points of $B$ are fixed
		by Lemma~\ref{L:Kim to Kjn}(a).
		
		(b) This follows from (a).
		
		(c) If not, then by Lemma~\ref{L:map} and Corollary~\ref{C:brick-brick} the map
		$F|_B:B\to F(B)$ is a homeomorphism, contradicting our assumption that it is not
		injective.
		
		(d) Just combine Lemma~\ref{L:Kim to Kjn} and Lemma~\ref{L:map}.
		
		(e) Since $F(B_1)\ni z$, by Lemma~\ref{L:map} either $F(B_1) = \{z\}$ or
		$F(B_1)$ is a brick homeomorphic to $B_1$. Suppose that we are in the latter
		case. Since $z\notin B_1$, the brick $F(B_1) \neq B_1$. Thus $B_1$ and $F(B_1)$
		are two different but homeomorphic bricks.  Hence, since the bricks in $D_m$ are
		pairwise non-homeomorphic, the brick $F(B_1)$ does not belong to $D_m$. However,
		the structure of $X$ is such that a brick which does not belong to $D_m$ never
		contains a point $z$ belonging to a brick in $D_m$. Thus, we have the former
		case $F(B_1) = \{z\}$.
		
		(f) Since no brick homeomorphic with $B$ intersects $\mathscr K_0$, by
		Lemma~\ref{L:map} we get that $F(B)$ is a singleton in~$\mathscr K_0$.
	\end{proof}
	
	The following lemma easily follows from the construction of $X$ and so we omit
	the proof.
	
	\begin{lem}\label{L:metalemma}
		Let $\mathscr P$ be a family of some of the bricks in the snake.  Assume that
		$\mathscr P$ has the following five properties.
		\begin{itemize}
			\item [(P1)] $\mathscr P$ contains at least one brick in the snake.
			\item [(P2)] If $\mathscr K^m_i \in \mathscr P$ then also $\mathscr K^m_{i+1}
			\in \mathscr P$.
			\item [(P3)] If $\mathscr K^m_i \in \mathscr P$ then also $\mathscr K^m_{i-1}
			\in \mathscr P$, provided $i\geq 2$.
			\item [(P4)] If $\mathscr K^m_i \in \mathscr P$ for all $i=1,2,\dots$, then
			also $\mathscr K^{m+1}_1 \in \mathscr P$.
			\item [(P5)] If $\mathscr K^m_i \in \mathscr P$ for all $i=1,2,\dots$, then
			also $\mathscr K^{m-1}_k \in \mathscr P$ for some $k\geq 1$, provided $m\geq 2$.
		\end{itemize}	
		Then all bricks in the snake belong to $\mathscr P$.
	\end{lem}
	
	\begin{lem}\label{L:const 2}
		\begin{itemize}
			\item [(a)] If the snake is not $F$-invariant then $F$ is constant.
			\item [(b)] If the set $\Sigma$ is not $F$-invariant then $F$ is constant.
			\item [(c)] If $F(\mathscr K_0) = \{z_0\}$ for some $z_0\in \mathscr K_0$,
			then $F(X)=\{z_0\}$ and so $F$ is constant.
		\end{itemize}
	\end{lem}

	\begin{proof}
		(a) By the assumption, there is a brick $B$ in the snake such that $F(B)$
		contains a point $z_0 \in \mathscr K_0$. By Lemma~\ref{L:const 1}(f) we get
		$F(B) =\{z_0\} \subseteq \mathscr K_0$. So the family $\mathscr P$ of all bricks
		in the snake whose $F$-image is $\{z_0\}$ is nonempty. We are going to show that
		$\mathscr P$ satisfies also (P2)-(P5) from Lemma~\ref{L:metalemma}.
		
		Indeed, (P2) and (P3) follow from Lemma~\ref{L:const 1}(f).
		Further, if $F(\mathscr K^m_i) = \{z_0\}$ for all $i=1,2,\dots$, then continuity
		gives $F(x_{m+1})=z_0$ and by Lemma~\ref{L:const 1}(f) we get $F(\mathscr
		K^{m+1}_1)=\{z_0\}$. So we have (P4). To prove (P5), let $m\geq 2$ and
		$F(\mathscr K^m_i) = \{z_0\}$ for all $i=1,2,\dots$. In particular,
		$F(x_{m})=z_0$.
		We have $z_0=\{0\}\times \{z\}$ where $z$ belongs to $\mathscr K$
		in~(\ref{Eq:bricks}) and $\mathscr K_0=\{0\}\times \mathscr K$.
		Let $\varepsilon > 0$ and let $V$ be an open neighbourhood of $z$ in $\mathscr
		K$, different from the whole $\mathscr K$. Then $W=X\cap ([0,\varepsilon) \times
		V)$ is an open neighbourhood of $z_0$ in $X$.
		The open set $F^{-1}(W)$ contains the point $x_m$ and so there exists $N$ such
		that $F^{-1}(W)$ contains also $\mathscr K^{m-1}_k$ for all $k\geq N$. Put $D=
		\{x_m\}\cup \bigcup_{k=N}^\infty \mathscr K^{m-1}_k$. Then $D$ is a continuum
		and so $F(D)\subseteq W$ is also a continuum and contains the point $z_0\in
		\mathscr K_0$.
		Since $V$ is not the whole $\mathscr K$, the set $W$ consists of $\{0\}\times V$
		and `pieces' of the snake, where each `piece' is a subset of the snake with
		positive distance from $\mathscr K_0$ (hence from $z_0$) and with positive
		distance from the union of all other `pieces'. Therefore the continuum $F(D)$ is
		necessarily a subset of $\{0\}\times V$. So, $F(D)\subseteq \mathscr K_0$. Then
		by Lemma~\ref{L:Kim to Kjn}(b), each brick in $D$ is mapped to a point and since
		the consecutive bricks in $D$ intersect, they are mapped to the same point. Thus
		$F(D)$ is a singleton and so $F(D)=\{z_0\}$. Hence (P5).
		
		By Lemma~\ref{L:metalemma}, all the bricks in the snake are mapped to $z_0$ and
		since $X$ is the closure of the snake, $F(X)=\{z_0\}$.

		(b) If a point from $\Sigma$ is mapped to the head, just use (a). Otherwise
		there is an extremal point of some brick $\mathscr K^r_s$ which is mapped  to a
		non-extremal point $z$ of some brick $\mathscr K^n_j$. By Lemma~\ref{L:const
			1}(d),  $F(\mathscr K^r_s) = \{z\}$.	So the family $\mathscr P$ of all bricks in
		the snake whose $F$-image is $\{z\}$ is nonempty. We claim that $\mathscr P$
		satisfies also (P2)-(P5) from Lemma~\ref{L:metalemma}. Since $z$ is a
		non-extremal point of $\mathscr K^n_j$, (P2) and (P3) follow from
		Corollary~\ref{L:const 1}(d). Further, if $F(\mathscr K^m_i) = \{z\}$ for all
		$i=1,2,\dots$, then continuity gives $F(x_{m+1})=z$ and by
		Corollary~\ref{L:const 1}(d) we get $F(\mathscr K^{m+1}_1)=\{z\}$. So we have
		(P4). To prove (P5), let $m\geq 2$ and $F(\mathscr K^m_i) = \{z\}$ for all
		$i=1,2,\dots$.
		In particular, we have $F(x_{m})=z$.
		Let $U(z) \subseteq (\mathscr K_j^n)^\circ$ be a small neighbourhood of $z$. By
		continuity, there is $N \in \N$ such that $F(\mathscr K_{j}^{m-1})\subseteq
		U(z)$ for all $j \ge N$.
		Then, say by Corollary~\ref{L:const 1}(d), for all $j\geq N$ there are points
		$z_j\in U(z)$ such that $F(\mathscr K_{j}^{m-1})=\{z_j\}$, $j\ge N$. Then the
		$F$-image of the continuum $\{x_{m}\} \cup \bigcup_{j=N}^\infty \mathscr
		K_j^{m-1}$ is the countable set $\{z\}\cup\{z_j:\, j\geq N\}$.
		Since this image has to be a continuum, the only possibility is that $z_j=z$ for
		all $j\geq N$. Hence $\mathscr K^{m-1}_j \in \mathscr P$ for all $j\geq N$ and
		(P5) is proved. Now, by Lemma~\ref{L:metalemma}, all the bricks in the snake are
		mapped to $z$. Since $X$ is the closure of the snake, we get $F(X)=\{z\}$.

		(c) Consider a neighbourhood $W$ of $z_0$ as in the proof of (a). Now the open
		set $F^{-1}(W)$ contains the set $\mathscr K_0$. Since $\mathscr K_0$ is a
		compact set in $X\subseteq [0,1]\times \mathscr K$, $F^{-1}(W)$ necessarily
		contains a whole sub-snake $S= ((\mathscr K_0, f^m_i \rangle\rangle$ for some
		$m$ and~$i$. Since $S$ is connected and has zero distance from $z_0$, its image
		$F(S)\subseteq W$ is connected and has zero distance from $F(z_0)=z_0$.
		Therefore, due to the structure of $W$ described in the proof of (a), $F(S)$ is
		necessarily a subset of $\{0\}\times V$. So, $F(S)\subseteq \mathscr K_0$. Then
		by Lemma~\ref{L:Kim to Kjn}(b), each brick in $S$ is mapped to a point and since
		the consecutive bricks intersect, also each set $D_m \subseteq S$ is mapped to a
		point. Since also consecutive sets $D_m$ intersect, it is easy to see that
		$F(S)$ is a singleton. Since $F(S)$ has zero distance from $z_0$, we get
		$F(S)=\{z_0\}$. Thus the snake is not $F$-invariant and (a) implies that $F$ is
		a constant map. Since $z_0$ is in its range, we get $F(X)=\{z_0\}$.
	\end{proof}

	\begin{lem}\label{L:ord-pres}
		Assume that $F$ is not constant. Then we have the following.
		\begin{itemize}
			\item [(a)] $F(\Sigma) \subseteq \Sigma$ and $F|_{\Sigma}: \Sigma \to \Sigma$
			is \emph{order-preserving}, i.e. $a\preccurlyeq b$ implies $F(a)\preccurlyeq
			F(b)$ whenever $a,b \in \Sigma$.
			\item [(b)] $F$ sends connected subsets of $\Sigma$ to connected subsets of
			$\Sigma$.\footnote{Here we use ``connected" in the sense of the theory of
				ordered sets, not topology.} In particular, two consecutive points of $\Sigma$
			are mapped to one point of $\Sigma$ or again to two consecutive points of
			$\Sigma$.
			\item [(c)] Let $a,b \in \Sigma$, $a\preccurlyeq b$. Then $F(a), F(b) \in
			\Sigma$, $F(a)\preccurlyeq F(b)$ and $F$ maps $\langle \langle a, b \rangle
			\rangle$ onto $\langle \langle F(a), F(b) \rangle \rangle$.
			\item [(d)] For every $m\in \mathbb N$ there is $k(m)\in \mathbb N$ such that
			$F(D_m^*) \subseteq D_{k(m)}^*$ (hence $F(D_m) \subseteq D_{k(m)}$).
			\item [(e)] If $F(D_m^*)$ contains the first point of $D_{k(m)}^*$ and has
			zero distance from the first point of $D_{k(m)+1}^*$, then $F(D_m^*) =
			D_{k(m)}^*$.
		\end{itemize}
	\end{lem}
	
	\begin{proof}
		(a-b) By Lemma~\ref{L:const 2}(b), $\Sigma$ is $F$-invariant. Assume that
		$\ell\preccurlyeq f$ are two consecutive points of $\Sigma$, i.e. the extremal
		points of a brick $B$ in the snake. We are going to prove that $F(\ell) = F(f)$
		or $F(\ell)$ and $F(f)$ are two consecutive points of $\Sigma$, i.e. the last
		point and the first point, respectively, of a brick in the snake (in particular,
		in either case we have $F(\ell)\preccurlyeq F(f)$). If $F(B)$ is a singleton,
		this is trivial. Otherwise, by Lemma~\ref{L:map}, $F(B)$ is a brick $C$
		homeomorphic to $B$. So, also $C$ is in the snake. By
		Corollary~\ref{C:brick-brick}, $\ell$ and $f$ are mapped to the last point and
		to the first point of $C$, respectively, and we are again done.
		
		Recall that $\Sigma = \bigsqcup _{m=1}^\infty \Sigma_m$ where $\Sigma_m:=\Sigma
		\cap D_m^*$. Fix $m$. Using the fact that on any two consecutive points of
		$\Sigma$ the map $F$ is order preserving in the special way proved above (i.e.
		the two consecutive points are mapped either to the  same point or to two
		consecutive points), we easily by induction get that
		$F$ is order preserving on $\Sigma_m$ and $F$ sends connected subsets of
		$\Sigma_m$ to connected subsets of $\Sigma_{k(m)}$ where $k(m)$ is a positive
		integer depending on $m$.
		
		The elements of $\Sigma_m$ form a sequence converging to the largest point
		$f_1^{m+1}$ of $\Sigma_{m+1}$. If $F(\Sigma_m)$ has the smallest element then,
		by continuity of $F$ at $f_1^{m+1}$ we get that $F$ sends  $f_1^{m+1}$ to that
		smallest element. If the set $F(\Sigma_m)$ does not have the smallest element
		then its elements form a sequence converging to $f_1^{k(m)+1}$ and continuity of
		$F$ at $f_1^{m+1}$ gives $F(f_1^{m+1}) = f_1^{k(m)+1}$. In either case, the
		facts that $F$ is order preserving both on $\Sigma_m$ and $\Sigma_{m+1}$ and
		sends connected subsets of these two sets again to connected sets, imply that
		$F$ is order preserving on the set $\Sigma_m \cup \Sigma_{m+1}$ and sends
		connected subsets of this set to connnected subsets of $\Sigma$.  Now it follows
		by induction that $F$ sends connected subsets of $\Sigma$ to connected subsets
		of $\Sigma$.
		
		(c) We know from (a) and (b) that $F(a) \preccurlyeq F(b)$ are in $\Sigma$ and
		$F$ maps $\Sigma \cap \langle \langle a, b \rangle \rangle$ onto $\Sigma \cap
		\langle \langle F(a), F(b) \rangle \rangle$ in the order-preserving way. Now it
		is sufficient to use Lemma~\ref{L:const 1}(c) and Corollary~\ref{C:hull}.
		
		(d) Fix $m$. The set $\Sigma \cap D_m^*$ consists of a sequence of points
		$f^m_1, f^m_2, f^m_3, \dots$. Since $\Sigma$ is $F$-invariant, for some $k(m)$
		we have $F(f^m_1)\in \Sigma  \cap D_{k(m)}^*$ and, by induction and using (b),
		we get that the whole set $\Sigma \cap D_m^*$ is mapped into
		$\Sigma \cap D_{k(m)}^*$. Now use Lemma~\ref{L:const 1}(c) and
		Corollary~\ref{C:hull}, or the result from (c).
		
		(e) This follows from the previous parts.
	\end{proof}

	\begin{lem}\label{L:mapDD}
		Let $m<M$ be positive integers. Then there is exactly one continuous surjective
		map of $D_m$ onto $D_M$.	
	\end{lem}
	
	\begin{proof}
		Let $M=m+k$, $k>0$. Let $B_1, B_2, B_3, \dots$ be the list of all bricks in
		$D_m^*$ (here $B_i:=\mathscr K^m_i$). Then $B'_{2^k}, B'_{2.2^{k}},
		B'_{3.2^{k}}, \dots$ is the list of all bricks in $D_M^*$. Any two of the bricks
		in the former list are non-homeomorphic and for every $i$, there is a
		homeomorphism $h_i: B_{i.2^{k}}\to B'_{i.2^{k}}$ sending the first point and the
		last point of $B_{i.2^{k}}$ to the first point and the last point of
		$B'_{i.2^{k}}$, respectively.
		
		Suppose that $\varPhi: D_m \to D_M$ is a continuous surjective map. It is the
		restriction of a continuous selfmap of $X$ (consider the monotone retraction of
		$X$ onto $D_m$ composed with $\varPhi$). So, we can apply Lemma~\ref{L:map} to
		get that, for every $j$, $\varPhi (B_j)$ is either a singleton or the brick
		$B'_{i.2^{k}}$, where the latter case is possible only if $j= i.2^{k}$.
		The remaining last point of $D_m$ is mapped to a point. Since countably many
		singletons cannot cover the whole brick $B'_{i.2^{k}}$, the surjective map
		$\varPhi$ necessarily sends $B_{i.2^{k}}$ onto $B'_{i.2^{k}}$, $i=1,2,\dots$. By
		Lemma~\ref{L:Kim to Kjn}(e), $\varPhi$ coincides on $B_{i.2^{k}}$ with the
		homeomorphism $h_i$, $i=1,2,\dots$. If $j$ is not a multiple of $2^k$, $\varPhi
		(B_j)$ is necessarily a singleton. By continuity (remembers also how $h_i$ maps
		extremal points of the bricks), there is no choice: $\varPhi$ necessarily sends
		the bricks $B_1,\dots B_{2^k-1}$ to the first point of $B'_{2^k}$, then the
		bricks $B_{2^k+1}, \dots, B_{2.2^k-1}$ to the point where $B'_{2^k}$ and
		$B'_{2.2^{k}}$ intersect, etc., see Figure~\ref{F:sur}. Also, by continuity,
		$\varPhi$ has to send the last point of $D_m$ to the last point of $D_M$, see
		Lemma~\ref{L:block}(4).

		\begin{figure}[h]
			\includegraphics[width=14cm]{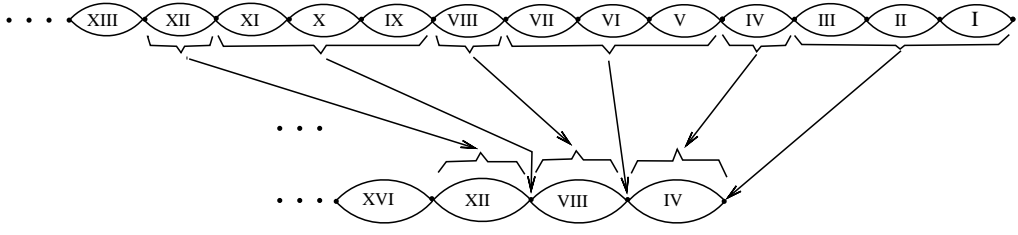}\\
			\caption{The unique continuous surjective map $D_1 \to D_3$.}\label{F:sur}
		\end{figure}

		So, there is at most one continuous surjective~$\varPhi$. Since the map acting
		as just described is obviously continuous and surjective, the proof is finished.
	\end{proof}

	\begin{rem}\label{R:mapDD}
		If $m<M$ then there are (countably) many non-constant continuous non-surjective
		maps of $D_m$ into $D_M$ (say, modify Figure~\ref{F:sur} by sending
		$\bigcup_{i=1}^7\mathscr K^1_i$ to the intersection of $\mathscr K^3_4$ and
		$\mathscr K^3_8$, or by sending the closure of $\bigcup_{i=13}^{\infty}\mathscr
		K^1_i$ to the intersection of $\mathscr K^3_{12}$ and $\mathscr K^3_{16}$, or
		both).
		What about non-constant continuous maps $D_M\to D_m$? There is no such
		surjective map (no brick in $D_M$ can be continuously mapped onto such a brick
		in $D_m$, say onto the first one, which is non-homeomorphic with any of the
		bricks in $D_M$). Non-constant non-surjective maps $D_M \to D_m$ do exist; one
		can show that each of them is the composition of a monotone retraction of $D_M$
		onto a brick $B$ in $D_M$ with the unique homeomorphism from $B$ onto the brick
		in $D_m$ homeomorphic to $B$.
	\end{rem}

	\subsection{Induced function $\widehat F$. The sets $\Fix (\widehat F)$ and
		$\Fix (F)$}\label{SS:induced}
	
	To prove that $S(X)=\{0, \infty\}$, we need to understand the dynamics of all
	continuous maps $F: X\to X$. To this end, it will be convenient first to replace
	every such map $F$ by what we will call an induced function $\widehat F$ \index{$\widehat F$}. It
	will be something like a discrete analogue of $F$, when we are basicly
	interested only in the images of bricks. To avoid a technical problem with the
	fact that two bricks may intersect, we partition the continuum~$X$ into the sets
	(recall that if $\mathscr K$ is a brick in the snake with extremal points $f$
	and $\ell$, then $\mathscr K^\circ := \mathscr K \setminus \{f,\ell\}$)
	\begin{equation}\label{Eq:partition}
	\begin{split}
	\mathscr K_0 & \prec \dots \\
	& \prec \dots \prec (\mathscr K^{m}_i)^\circ \prec \{f^{m}_i\} \prec
	\dots \prec (\mathscr K^m_2)^\circ \prec \{f^m_2\} \prec (\mathscr K^m_1)^\circ
	\prec \{f^m_1\} \\
	& \prec \dots \\
	& \prec \dots \prec  (\mathscr K^2_i)^\circ \prec \{f^2_i\} \prec  \dots
	(\mathscr K^2_2)^\circ \prec \{f^2_2\} \prec (\mathscr K^2_1)^\circ \prec
	\{f^2_1\} \\
	& \prec \dots \prec (\mathscr K^1_i)^\circ \prec \{f^1_i\}
	\prec \dots \prec (\mathscr K^1_2)^\circ \prec \{f^1_2\} \prec (\mathscr
	K^1_1)^\circ \prec \{f^1_1\}
	\end{split}
	\end{equation}
	where $A\prec B$ (or $B\succ A$) has the meaning from the beginning of
	Subsection~\ref{SS:X is continuum}, i.e.
	$a\prec b$ whenever $a\in A$ and $b\in B$. Denote the family of  all sets
	in~(\ref{Eq:partition}) by~$\widehat X$ \index{$\widehat X$}.
	If $A, B\in \widehat X$ then we introduce also the notation $A\preccurlyeq B$
	(or $B\succcurlyeq A$), By definition, for such \emph{special} subsets of $X$,
	this means that $A=B$ or $A\prec B$. Recall that for \emph{arbitrary} subsets of
	$X$ we do not define $A\preccurlyeq B$.\footnote{Note that if $A, B\in \widehat
		X$ then just defined relation $A\preccurlyeq B$ is not equivalent with
		$a\preccurlyeq b$, $a\in A$, $b\in B$. Try $A=B=(\mathscr K^m_i)^\circ$.}
	
	Clearly,
	$(\widehat X, \preccurlyeq)$ is a linearly ordered set and  analogously as in
	Lemma~\ref{L:antiwell1} we have the following obvious fact.
	
	\begin{lem}\label{L:antiwell2}
		Every nonempty subset of $(\widehat X, \preccurlyeq)$ has the largest element.
	\end{lem}
	
	In particular, every element of $\widehat X$ different from $\mathscr K_0$ has
	its \emph{left neighbour} (the largest among those elements which are to the
	left of it). Similarly, one can define the \emph{right neighbour} of an element
	of $X$; however, $\mathscr K_0$ and the elements $\{f^m_1\}$ do not have right
	neighbours.

	\begin{lem}\label{L:existsFhat}
		Let $A$ be any of the sets appearing in~(\ref{Eq:partition}), i.e., let $A\in
		\widehat X$. Then $F(A)\subseteq B$ for exactly one $B\in \widehat X$. So, the
		continuous map $F:X \to X$ naturally \emph{induces a function $\widehat F:
			\widehat X \to \widehat X$} defined by $\widehat F (A) = B$ when $F(A)\subseteq
		B$ for $A, B\in \widehat X$. 	
	\end{lem}
	
	\begin{proof}
		We only need to prove the existence of such a set $B$ (since $A$ is nonempty and
		$\widehat X$ is a partition of $X$, such $B$ is then unique). This is trivial if
		$A$ is a singleton $\{f^m_n\}$. If $A=\mathscr K_0$ then, by Lemma~\ref{L:map},
		$F(A)$ is either $\mathscr K_0$ or a singleton, and again such $B$ exists.
		Finally, let $A=(K^{m}_i)^\circ$ for some $m$ and $i$. Then, by
		Lemma~\ref{L:map} and Corollary~\ref{C:brick-brick}, $F(A)$ is a singleton or
		$F(A)= (K^n_j)^\circ$ for some $n$ and $j$, and again we are done.
	\end{proof}
	
	In view of Lemma~\ref{L:map}, $\widehat F (A) = B$ thus means that $F(A)$ is a
	singleton in $B$ or coincides with $B$ (or both, if $B\in \widehat X$ is a
	singleton). Note also that if $A, B \in \widehat X$ and $F(A)$ intersects $B$
	then the previous lemma and the fact that $\widehat X$ is a partition of $X$,
	imply that $F(A)\subseteq B$ and so $\widehat F(A)=B$.
	
	If $\psi$ is any function let $\Fix(\psi)$ denote the \emph{set of fixed points of $\psi$}.  We are going to
	study the set $\Fix (\widehat F)$. It will help us to describe the set $\Fix
	(F)$ which will play a crucial role in description of all possible dynamics on
	$X$.
	
	\begin{lem}\label{L:how br to br}
		Assume that $(\mathscr K^r_s)^\circ$ is a fixed point of $\widehat F$. Then we
		have one of the following.
		\begin{itemize}
			\item [(a)] $F((\mathscr K^r_s)^\circ) = (\mathscr K^r_s)^\circ$ and the
			restriction of $F$ to $\mathscr K^r_s$ is identity.
			\item [(b)] $F((\mathscr K^r_s)^\circ) = \{z\}$ for some $z\in (\mathscr
			K^r_s)^\circ$ and $F(X)=\{z\}$.
		\end{itemize}
		In particular, $F$ has a fixed point in $(\mathscr K^r_s)^\circ$.  	
	\end{lem}
	
	\begin{proof}
		By continuity, also the Cook continuum $\mathscr K^r_s$	is $F$-invariant. So,
		either $F$ is identity on $\mathscr K^r_s$ and we are in the case (a), or
		$F(\mathscr K^r_s)=\{z\}$ for some $z\in \mathscr K^r_s$. In the latter case
		$z\in (\mathscr K^r_s)^\circ$, because $(\mathscr K^r_s)^\circ$ is
		$F$-invariant. It follows that $\Sigma$ is not $F$-invariant whence $F(X)=\{z\}$
		by Lemma~\ref{L:const 2}(b). So we are in the case (b).
	\end{proof}
	
	\begin{cor}\label{C:gemini}
		Assume that $F$ is not constant and $(\mathscr K^r_s)^\circ \in \Fix (\widehat
		F)$. Then also $\{f^r_s\}$ and $\{\ell^r_s\} = \{f^r_{s+1}\}$ belong to $\Fix
		(\widehat F)$.	
	\end{cor}
	
	\begin{proof}
		Lemma~\ref{L:how br to br} implies that the restriction of $F$ to $\mathscr
		K^r_s$ is identity.	
	\end{proof}

	\begin{lem}\label{L:2xFix}
		$\Fix (F) \subseteq \bigcup \Fix (\widehat F)$ and $\Fix (F)$ intersects every
		set $A\in \Fix (\widehat F)$.
	\end{lem}
	
	\begin{proof}
		Let $x\in \Fix (F)$. Then $x \in A$ for some $A\in \widehat X$ and since $x$ is
		fixed for $F$, we have $F(A)\cap A\neq \emptyset$. Hence $\widehat F(A) = A$ and
		so $x\in A \in \Fix (\widehat F)$.
		
		Now let a set $A\in \widehat X$ be a fixed point of $\widehat F$. If $A=\{a\}$
		then $a$ is a fixed point of $F$. If $A$ is not a singleton then either $A=
		\mathscr K_0$ or $A= (\mathscr K^n_j)^\circ$ for some $n$ and $j$. In the former
		case we have $F(\mathscr K_0) \subseteq \mathscr K_0$ and the Cook continuum
		$\mathscr K_0=A$ contains a fixed point of $F$.
		In the latter case use Lemma~\ref{L:how br to br}.
	\end{proof}

	\begin{lem}\label{L:Darboux}
		Suppose that $A \prec B$ in $\widehat X$ are such that $\widehat F(A) \succ A$
		and $\widehat F(B) \prec B$. Then there exists $C$ between $A$ and $B$ such that
		$\widehat F(C)=C$.\footnote{The lemma works also with $B\prec A$, but we will
			not need it.} 	
	\end{lem}
	
	\begin {proof}
	Suppose, on the contrary, that there is no such $C$. We may assume that $A \neq
	\mathscr K_0$, since otherwise we can replace $A$ by $A^*\in \widehat X$ which
	is between $A$ and $B$ and sufficiently `close' to $\mathscr K_0$ (then
	continuity of $F$ ensures that $\widehat F(A^*)$ is to the right of $A^*$).
	Further, we may assume that not only $B$, but also every element  between $A$
	and $B$ (if $A$ and $B$ are not neighbours in $\widehat X$) is mapped by
	$\widehat F$ to the left. Otherwise we replace $A$ by the largest element
	between $A$ and $B$ which is mapped by $\widehat F$ to the right. So, it is
	sufficient to deduce a contradiction from the following assumptions: $\mathscr
	K_0 \prec A\prec B$, $A$ is mapped to the right and all $D$ with $A\prec D
	\preccurlyeq B$ are mapped to the left.
	
	First assume that $A$ has its right neighbour. Then either $A= (\mathscr
	K^m_i)^\circ$ is mapped to the right and its right neighbour $\{f^m_i\}$ to the
	left, or $A= \{f^m_i\}$ (with $i\geq 2$) is mapped to the right and its right
	neighbour $(\mathscr K^m_{i-1})^\circ$ to the left. Either case clearly
	contradicts the continuity of $F$.
	
	Now let $A=\{f^m_1\}$ for some $m\geq 2$ (since $A\prec B$, $m\neq 1$). Since
	$A$ is mapped to the right, the continuity of $F$ implies that all elements
	between $A$ and $B$ which are sufficiently `close' to $A$ are also mapped to the
	right, a contradiction with the fact that all $D$ with $A\prec D \preccurlyeq B$
	are mapped to the left.
\end{proof}	

Since the largest element of $\widehat X$ is either a fixed point of $\widehat
X$ or is mapped to the left, this lemma immediately gives the following
corollary.

\begin{cor}\label{C:arrow}
	If there is $A$ in $\widehat X$ which is mapped by $\widehat F$ to the right,
	then $\widehat F$ has a fixed point different from $\mathscr K_0$.
\end{cor}

A subset of $\widehat X$ is said to be \emph{connected}, if it contains all
elements $C$ between $A$ and $B$ whenever it contains $A$ and $B$.

\begin{lem}\label{L:FixFhat}
	The set $\Fix (\widehat F)$	is nonempty, has the smallest element and the
	largest element, and is connected. Moreover, if $\Fix (\widehat F)$ has more
	than one element, then $\Fix (F) = \bigcup \Fix (\widehat F)$.
\end{lem}

\begin{proof}
	Either $\mathscr K_0$ is a fixed point for $\widehat F$ or $\mathscr K_0 \prec
	\widehat F(\mathscr K_0)$ and then $\Fix (\widehat F) \neq \emptyset$ by
	Corollary~\ref{C:arrow}. From now on we will assume that $\Fix (\widehat F)$ has
	more than one element (and hence $F$ is not constant), otherwise the rest of the
	lemma is obvious.
	
	$\Fix (\widehat F)$ has the largest element by Lemma~\ref{L:antiwell2}. Suppose
	for a moment that $\Fix (\widehat F)$ has no smallest element.
	In view of Corollary~\ref{C:gemini} it means that either
	\begin{itemize}
		\item [(i)] for some $m$, $\Fix (\widehat F)$ does not contain $\{f^{m+1}_1\}$,
		while containing $\{f^m_i\}$ with arbitrarily large $i$, or
		\item [(ii)] $\Fix (\widehat F)$ does not contain $\mathscr K_0$, while
		containing $\{f^m_i\}$ with arbitrarily large $m$ ($i$~depends on $m$).
	\end{itemize}
	In case (i), the point $f^{m+1}_1$, being the limit of a sequence of fixed
	points of $F$, is obviously fixed for $F$. So $\{f^{m+1}_1\}$ is fixed for
	$\widehat F$, a contradiction. In case (ii), by considering a subsequence of fixed points
	converging to a point in the head, we see that $F$ has necessarily a fixed point
	also in $\mathscr K_0$ and so, by Lemma~\ref{L:map}, $\mathscr K_0 \in \Fix (\widehat
	F)$, a contradiction.

	To prove that $\Fix (\widehat F)$ is connected, let $A, B \in \Fix (\widehat F)$
	and  $A\prec C \prec B$. We need to prove that $C\in \Fix (\widehat F)$. By
	Corollary~\ref{C:gemini}, it is sufficient to consider two cases, namely
	\begin{itemize}
		\item [(I)] $A=\{f^{m+r}_j\} \prec \{x_{m+r}\} \prec \ldots \prec \{x_{m+1}\}
		\prec \{f^m_i\} = B$ where $r\geq 0$, and if $r=0$ then this reduces to
		$A=\{f^{m}_j\} 	\prec \{f^m_i\} = B$ where $j>i$,
		\item [(II)] $A=\mathscr K_0 \prec \{f^m_i\} = B$.
	\end{itemize}
	
	First consider the case (I). To prove that $C\in \Fix (\widehat F)$, it is
	clearly sufficient to show that the map $F$ is identity on $D_{AB}=\langle
	\langle f^{m+r}_j, f^m_i \rangle \rangle$. Let $\mathscr B_{AB}$ be the family
	of all bricks which are subsets of $D_{AB}$ (i.e. bricks `joining' $A$ and $B$).
	Then $F(D_{AB})$ is a sub-continuum of $X$ containing both $f^{m+r}_j$ and
	$f^m_i$. Therefore the first projection of $F(D_{AB})$ contains the whole
	interval whose endpoints are the first projections of $f^{m+r}_j$ and $f^m_i$,
	i.e
	\begin{equation}\label{Eq:proj}
	P_1(F(D_{AB}))\supseteq P_1(D_{AB})~.
	\end{equation}
	Now notice that, by Lemma~\ref{L:map}, if $B_1, B_2 \in \mathscr B_{A,B}$ then
	$F(B_1)$ is either disjoint with $B_2$ (then even their first projections are
	disjoint) or $F(B_1)$ is a singleton in $B_2$ or $F(B_1)=B_2$, where
	$F(B_1)=B_2$ is possible only if $B_1$ is homeomorphic to $B_2$. It follows that
	for every brick $B_2 \in \mathscr B_{AB}$ there is a brick $B_1 \in \mathscr
	B_{AB}$  homeomorphic to $B_2$ (perhaps $B_1=B_2$) such that $F(B_1)=B_2$ (one
	may notice that this strengthens~(\ref{Eq:proj}) to
	$F(D_{AB})\supseteq D_{AB}$). Indeed, otherwise the first projection of the
	$F$-image of each of the countably many bricks in $\mathscr B_{AB}$ covers at
	most one point of the uncountable first projection of $B_2$, a contradiction
	with~(\ref{Eq:proj}).
	
	However, as one can see from the construction of~$X$ (see~(\ref{Kim})
	and~(\ref{Eq: DD*})), every second
	brick among those `joining' $B$ and $x_{m+1}$ is such that no \emph{other} brick
	in $\mathscr B_{AB}$ is homeomorphic with it (if $r=0$ or if $r=1$ and $j=1$
	then necessarily even every brick, not only every second one, has this
	property). Therefore every such brick is mapped onto itself and so, by
	Lemma~\ref{L:Kim to Kjn}(a), $F$ is identity on it.
	
	So, not only $F$ fixes the singletons $A$ and $B$ but $F$ is identity also on
	every second brick joining $B$ and $x_{m+1}$ (joining $B$ and $f^m_j$ if $r=0$).
	It follows that $F$ is identity on $\langle \langle x_{m+1}, f^m_i \rangle
	\rangle$  (on $\langle \langle f^m_j, f^m_i \rangle \rangle$ if $r=0$), see
	Lemma~\ref{L:const 1}(b). If $r=0$ or $r=1$ and $j=1$ (i.e. if the whole
	$D_{AB}$ is a subset of $D_m$), we are already done. Otherwise, reasoning
	similarly but with $B$ replaced by $B'=\{x_{m+1}\} = \{f^{m+1}_1\}$, we get that
	$F$ is identity also on that part of $D_{AB}$ which lies in $D_{m+1}$.
	Continuing this way, after finitely many steps we finish the proof that $F$ is
	identity on the whole set $\langle \langle f^{m+r}_j, f^m_i \rangle \rangle$.
	
	Now consider the case (II). It is sufficient to prove that $F$ is identity on
	$D_{\mathscr K_0B}=\langle \langle \mathscr K_0, f^m_i \rangle \rangle$. The set
	$F(D_{\mathscr K_0B})$ is a sub-continuum of $X$
	containing both a point from $\mathscr K_0$ (because $\mathscr K_0 \in \Fix (\widehat
	F)$ and $\mathscr K_0$ has fixed point property) and the point $f^m_i$. By the same cardinality argument as in the case (I)
	we get that $F$ is identity on the sub-snake $((\mathscr K_0, f^m_i \rangle
	\rangle$ and so it is identity also on $\langle \langle \mathscr K_0, f^m_i
	\rangle \rangle$.
	
	Finally, let $\Fix (\widehat F)$ have at least two elements. In the proofs of the
	cases (I) and (II) we have shown that
	$F$ is identity on every set which is an element of $\Fix (\widehat F)$. This
	together with Lemma~\ref{L:2xFix} give $\Fix (F) = \bigcup \Fix (\widehat F)$.
\end{proof}

\begin{cor}\label{C:FixF}
	If $F$ is not constant, then the set $\Fix (F)$ is either $\mathscr K_0$ or of
	the form $\langle \langle \mathscr K_0, f^m_i \rangle \rangle$ or  $\langle
	\langle  f^n_j, f^m_i \rangle \rangle$ for some $f^n_j \preccurlyeq
	f^m_i$.\footnote{The case $f^n_j = f^m_i = f^1_1$ can in fact be excluded
		because if $\Fix(F) =\{f^1_1\}$ then $F$ can be shown to be constant.}
\end{cor}

\begin{proof}
	By Lemma~\ref{L:FixFhat}, $\Fix (\widehat F)\neq \emptyset$.
	First assume that $\mathscr K_0\in \Fix (\widehat F)$, i.e. $F (\mathscr
	K_0)\subseteq \mathscr K_0$. By Lemma~\ref{L:Kim to Kjn}(a), $F$ is constant or
	identity on $\mathscr K_0$. The former case implies, by Lemma~\ref{L:const
		2}(c), that $F$ is constant on $X$, a contradiction. Thus $F$ is identity on
	$\mathscr K_0$ and so $\Fix (F) \supseteq \mathscr K_0$. Then, combining
	Corollary~\ref{C:gemini} and Lemma~\ref{L:FixFhat} we get that either $\Fix (F)
	= \mathscr K_0$ or $\Fix (F) = \langle \langle \mathscr K_0, f^m_i \rangle
	\rangle$ for some $m$ and $i$.
	
	If some $(\mathscr K^r_s)^\circ \in \Fix (\widehat F)$, then by
	Corollary~\ref{C:gemini} the set $\Fix (\widehat F)$ has more than one element
	and, by using Lemma~\ref{L:FixFhat}, we get that $\Fix (F)$ is either $\langle
	\langle \mathscr K_0, f^m_i \rangle \rangle$ or  $\langle \langle  f^n_j, f^m_i
	\rangle \rangle$ for some $f^n_j \prec f^m_i$.
	
	Finally, let some $\{f^m_i\}\in \Fix (\widehat F)$, i.e. $f^m_i\in \Fix (F)$. We
	may assume that $\Fix (\widehat F)$ contains neither $\mathscr K_0$ nor any of
	the sets $(\mathscr K^r_s)^\circ$, since these cases have already been
	considered. Thus, since $\Fix (\widehat F)$ is connected, $\{f^m_i\}$ is the
	only element of $\Fix (\widehat F)$. By Lemma~\ref{L:2xFix}, $\Fix (F) =
	\{f^m_i\}$.
\end{proof}

\subsection{Proof that $S(X)=\{0,\infty\}$}\label{SS:dynamicsF}
We show that $X$ does admit more con\-ti\-nuous selfmaps than just the identity
and the constant maps, but not too many of them. Moreover, we show that besides
selfmaps with rather trivial dynamics, there is in a sense only one significant
selfmap of $X$. So, $X$ exhibits some degree of rigidity.

\begin{lem}\label{L:existsG}
	There is a continuous map $G\colon X\to X$, a continuous extension of the map $T\colon X_1\to X_1$,
	 with the following properties.
	\begin{itemize}
		\item [(a)] $h^*(G)=\infty$.
		\item [(b)] For every $r\in \N$, the set $S^{\Join}_r= \langle \langle \mathscr
		K_0, x_r  \rangle \rangle$ is $G$-invariant and $h^*(G|_{S^{\Join}_r})=\infty$.
		\item [(c)] For every $m\in \N$, $G(D_m^*) =D_{m+1}^*$.
	\end{itemize}
\end{lem}

\begin{proof}
	Recall that by the construction of $X$ from Subsection~\ref{SS:outline} we have
	$X\supseteq X_1$ where $X_1$ is the space from Subsection~\ref{SS:cont+orbit},
	with $\mathscr C_0:=\mathscr K_0$ (the arbitrary continuum $\mathscr C_0$ has
	been replaced by the Cook continuum $\mathscr K_0$). For the map $T: X_1 \to
	X_1$ we have $h^*(T)=\infty$, even for every positive integer $r$ we have
	$h^*(T|_{S_r})=\infty$ for the $T$-invariant set $S_r$, see~(\ref{Eq:hstarsub}).
	Note that $S_r \subseteq S_1 = X_1$, $S^{\Join}_r \subseteq S^{\Join}_1 = X$ and
	$S_r \subseteq S^{\Join}_r$. To prove the existence of $G$ with (a) and (b), it
	is therefore sufficient to extend $T: X_1 \to X_1$ to a continuous map $G: X\to
	X$ such that $S^{\Join}_r$ is $G$-invariant. Indeed, $G|_{X_1}=T$ and
	$h^*(T|_{S_r})=\infty$
	trivially imply $h^*(G|_{S^{\Join}_r})=\infty$.
	
	To find such $G$, recall the homeomorphisms $h_i^m: \mathscr K_{i2^{m-1}} \to
	\mathscr K_i^m$ used in the construction of $X$, see~(\ref{Eq: hmi}).
	Define $G: X\to X$ as follows:
	\begin{enumerate}
		\item $G|_{X_1}=T$;
		\item for every $m\in \N$, $G$ maps $D_m^*$ onto $D_{m+1}^*$ (hence (c)) in
		such a way that for every $i \in \N$,
		$$G(\mathscr K_{2i}^{m}) = \mathscr K_{i}^{m+1} \ \text{ with} \ G(x) =
		h^{m+1}_i\circ (h^m_{2i})^{-1} \ \text{ whenever}
		\ x\in \mathscr K_{2i}^{m},$$
		and $G(\mathscr K_{2i-1}^{m})= \{f_i^{m+1}\}$.
	\end{enumerate}
	(One can see that the restriction of $G$ to $\mathscr K_{2i}^{m}$ defined above
	is a homeomorphism and it is in fact the only non-constant continuous
	map $\mathscr K_{2i}^{m} \to \mathscr K_{i}^{m+1}$, see Lemma~\ref{L:Kim to Kjn}(e),
	cf. Lemma~\ref{L:Cook-properties}(3).)  Clearly, $G: X\to X$ is a
	continuous extension of $T$ and every set $S^{\Join}_r$ is $G$-invariant.
\end{proof}

\begin{lem}\label{L:left}
	Assume that $F$ is not constant and $p\in \Sigma$ is the smallest element of
	$\Fix (F)$. Then for every $x\in X$ such that $x\preccurlyeq p$ we have $F^2(x)
	=p$ (but in general not $F(x)=p$).	
\end{lem}

\begin{proof}
	We have $p= f_j^n$ for some $n$ and $j$. The brick $\mathscr K_j^n$ contains no
	fixed point of $F$ different from $p$. Therefore, by Lemma~\ref{L:const 1}(a),
	$F(\mathscr K_j^n) = \{p\}$. Then, by Lemma~\ref{L:ord-pres}(c), the sub-snake
	$((\mathscr K_0, p\rangle \rangle$
	is $F$-invariant. It is sufficient to prove that the $F^2$-image of this
	sub-snake is just $\{p\}$; since the head is in the closure of this sub-snake,
	the lemma follows. There are three possibilities.
	
	\emph{Case 1 : $F(((\mathscr K_0, p\rangle \rangle) = \{p\}$.} Then there is
	nothing to prove.

	\emph{Case 2 : $F(((\mathscr K_0, p\rangle \rangle) = \mathscr K^n_j$.} Since
	$F(\mathscr K_j^n) = \{p\}$, we get $F^2(((\mathscr K_0, p\rangle \rangle) =
	\{p\}$ as required. To show that here we cannot in general replace $F^2$ by $F$,
	see Figure~\ref{F:3Dm} and consider the following map. Let $p= f^2_2$ (i.e., $p$
	is the first point of the brick $\mathscr K^2_2$, which is a copy of $\mathscr
	K_4$) and let $\Fix (F) = \langle \langle p, x_1 \rangle \rangle$. Now, send the
	points of $D_2$ which are to the left of $p$ into $p$, the brick $\mathscr
	K^3_1$ (which is also a copy of $\mathscr K_4$) homeomorphically onto $\mathscr
	K^2_2$, and all those points of $X$ which are to the left of $\mathscr K^3_1$
	into the last point of $\mathscr K^2_2$.
	
	\emph{Case 3 : $F(((\mathscr K_0, p\rangle \rangle) \supseteq  \mathscr
		K^n_{j+1} \cup \mathscr K^n_j$.} We show that this is impossible. Indeed, by
	Corollary~\ref{C:union}, in this case there are bricks $B_1$ and $B_2$ in
	$((\mathscr K_0, p\rangle \rangle$ such that $F(B_1) = \mathscr K^n_j$ and
	$F(B_2) = \mathscr K^n_{j+1}$.  Since $F(\mathscr K_j^n) = \{p\}$, also
	$F(\mathscr K^n_{j+1})=\{p\}$ (note that $\mathscr K^n_{j+1}$ cannot be mapped
	onto $\mathscr K^n_{j}$ because the bricks $\mathscr K^n_{j+1}$ and $\mathscr
	K^n_{j}$ are not homeomorphic, see Lemma~\ref{L:map}). So, $B_1$ and $B_2$ are
	in $((\mathscr K_0, f^n_{j+2}\rangle \rangle$, $B_1$ is homeomorphic to
	$\mathscr K^n_j$ and  $B_2$ is homeomorphic to $\mathscr K^n_{j+1}$. However, it
	follows from~(\ref{Kim}), cf. Figure~\ref{F:3Dm}, that in $((\mathscr K_0,
	f^n_{j+2}\rangle \rangle$ there may exist such a brick $B_1$ or such a brick
	$B_2$, but not both. The reason is that no two bricks in $D_n$ are homeomorphic
	and when we construct the bricks in the sets $D_{n+1}, D_{n+2}, \dots$, then
	they are always homeomorphic copies of every second of the bricks in the
	previous set. However, $K^n_{j}$ and $K^n_{j+1}$ are two consecutive bricks of
	$D_n$.	
\end{proof}

\begin{lem}\label{L:right}
	Assume that $F$ is not constant and $q \in \Sigma$ is the largest element of
	$\Fix (F)$. Then there exists $N\in \mathbb N$ such that
	$F^N(x) = q$ for all points $x\succcurlyeq q$.	
\end{lem}

\begin{proof}
	We have $q= f_j^n$ for some $n$ and $j$. Let $q\prec x_1$, otherwise there is
	nothing to prove.
	
	First assume that $j\geq 2$, i.e. $\mathscr K_j^n$ is not the first brick of
	$D_n$. Since $q$ is fixed, we deduce from Lemma~\ref{L:ord-pres} that the
	$F$-images of the bricks $\mathscr K_{j-1}^n, \dots \mathscr K^n_1$ are subsets
	of $\langle \langle q, x_n \rangle \rangle$. Since these bricks are pairwise
	non-homeomorphic, Lemma~\ref{L:map} shows that they are mapped to singletons.
	Due to continuity of $F$ we then have
	\begin{equation}
	F(\langle \langle q, x_n \rangle \rangle)= \{q\} \notag
	\end{equation}
	and (if $n\geq 2$), using again Lemma~\ref{L:ord-pres} and continuity of $F$,
	$F(D_{n-1}^*) \subseteq \langle \langle q, x_n \rangle \rangle$. The last two
	formulas give $F^2(D_{n-1}^*)= \{q\}$. In particular the point $x_{n-1}$ is
	mapped by $F^2$ to $q$ and so (if $n\geq 3$), by Lemma~\ref{L:ord-pres}  applied
	to the continuous map $F^2$, we get $F^2(D_{n-2}^*) \subseteq \langle \langle q,
	x_n \rangle \rangle$. Hence $F^3(D_{n-2}^*) = \{q\}$. Continuing this way (if
	$n\geq 4$) we get $F^4(D_{n-3}^*) = \{q\}, \dots, F^n (D_1^*) = \{q\}$. So, we
	see that for $N=n$, $F^N(x)=q$ for all points $x\succcurlyeq q$.
	
	Now assume that $j=1$, i.e. $q= f_1^n = x_n$. Since $q\prec x_1$, we have $n\geq
	2$. By Lemma~\ref{L:ord-pres}, either $F(D_{n-1}^*) = \{q\}$ or $F(D_{n-1}^*)
	\subseteq D_{n-1}^*$. The latter case is impossible because all the bricks in
	$D_{n-1}^*$ are pairwise non-homeomorphic and so, by Lemma~\ref{L:map}, each of
	them would be mapped to a singleton or onto itself. However, if a brick in
	$D_{n-1}^*$ is mapped onto itself then the restriction of $F$ to such a brick is
	the identity, contradicting the assumption that $q$ is the maximal element of
	$\Fix (F)$. Therefore we have
	\begin{equation}\label{Eq:Dn-1 to q}
	F(D_{n-1}^*)=\{q\}.
	\end{equation}
	If $n\geq 3$, consider the next set $D_{n-2}^*$. Since $F(x_{n-1}) = q$ and $F$
	is continuous, by Lemma~\ref{L:ord-pres} we get that either $F(D_{n-2}^*)
	=\{q\}$ or $F(D_{n-2}^*) \subseteq D_{n-1}^*$. Using~(\ref{Eq:Dn-1 to q}) we get
	that in either case $F^2(D_{n-2}^*) =\{q\}$. If $n\geq 4$, consider the set
	$D_{n-3}^*$. Since $F^2(x_{n-2}) = q$, by Lemma~\ref{L:ord-pres} applied to
	$F^2$ we get that either $F^2(D_{n-3}^*) =\{q\}$ or $F^2(D_{n-3}^*) \subseteq
	D_{n-1}^*$. Hence $F^3(D_{n-3}^*) =\{q\}$. Continuing this way (if $n\geq 5$) we
	get $F^4(D_{n-4}^*) = \{q\}, \dots, F^{n-1} (D_1^*) = \{q\}$. So, we see that
	for $N=n-1$, $F^N(x)=q$ for all points $x\succcurlyeq q$.
\end{proof}

\begin{lem}\label{L:jumptoanotherblock}
	Assume that $F$ is not constant and has no fixed point in the snake. Let $m\in
	\N$. Then $F(D_m^*) \subseteq D_{k(m)}^*$ for some $k(m)>m$.
\end{lem}

\begin{proof}
	By Lemma~\ref{L:ord-pres}(d), $F(D_m^*) \subseteq D_{k(m)}^*$ for some $k(m)\in
	\N$. No element $A\in \widehat X \setminus \{\mathscr K_0\}$ is a fixed point
	of~$\widehat F$, otherwise Lemma~\ref{L:2xFix} gives the existence of a fixed
	point of $F$ in the snake, a contradiction. Therefore, by
	Corollary~\ref{C:arrow}, every element  $A\in \widehat X \setminus \{\mathscr
	K_0\}$ is mapped by $\widehat F$ to the left.
	It follows that $k(m)\geq m$. To prove strict inequality, suppose on the
	contrary that $F(D_m^*) \subseteq D_{m}^*$. Consider the point $x_{m+1}$. Since
	it is in the closure of $D_m^*$ and $F(D_m^*) \subseteq D_{m}^*$, it is mapped
	by $F$ to the closure of $D_m^*$. By the assumption, it is not a fixed point of
	$F$ and so it is in fact mapped by $F$ to $D_m^*$. This contradicts the fact
	that every element  $A\in \widehat X \setminus \{\mathscr K_0\}$ is mapped by
	$\widehat F$ to the left, in particular $\widehat F (\{x_{m+1}\}) \prec
	\{x_{m+1}\}$.
\end{proof}

If $F(D_m^*) \subseteq D_{k(m)}^*$, then we say that the \emph{jump of $D_m^*$}
has length $k(m)-m$ and we write $\jump (m) = k(m)-m$. Under the assumptions of
the previous lemma, all the sets $D_m^*$ have jumps of positive length (i.e.,
they `jump' to the left under the action of $F$).

\begin{lem}\label{L:jumps}
	Assume that $F$ is not constant and has no fixed point in the snake.
	\begin{itemize}
		\item [(a)] For every $m$, $\jump(m+1) \in \{\jump(m)-1, \ \jump(m)\}$.
		\item [(b)] The sequence $\jump(1), \jump(2), \dots$ is eventually constant.
		\item [(c)] There exist positive integers $r$ and $N$ such that on
		$S^{\Join}_r= \langle \langle \mathscr K_0, x_r  \rangle \rangle$ we have
		$$
		F|_{S^{\Join}_r} = G^N|_{S^{\Join}_r}
		$$
		where $G$ is the map from Lemma~\ref{L:existsG}.
	\end{itemize}
\end{lem}

\begin{proof}
	(a) Fix $m$. By Lemma~\ref{L:ord-pres}, we have $F(D_m^*) \subseteq D_{k(m)}^*$
	and $F(D_{m+1}^*) \subseteq D_{k(m+1)}^*$.
	Since $F|_{\Sigma}$ preserves the order $\preccurlyeq$ on $\Sigma$, and
	$D_{m+1}^*$ is to the left of $D_{m}^*$, either $D_{k(m+1)}^*$ is to the left of
	$D_{k(m)}^*$ or $D_{k(m+1)}^* = D_{k(m)}^*$. Thus, either $k(m+1) > k(m)$ or
	$k(m+1) = k(m)$. On the other hand, we cannot have $k(m+1) \geq k(m) +2$ because
	then there would exists at least one set $D_r^*$ lying between
	the sets $D_{k(m+1)}^*$ and $D_{k(m)}^*$ and so the $F$-image of the connected
	set $D_{m}^* \cup D_{m+1}^*$ would be disconnected. Therefore either $k(m+1) =
	k(m) +1$ or $k(m+1) = k(m)$. In the former case $\jump(m+1) = k(m+1)-(m+1) =
	k(m)- m = \jump(m)$ and in the latter case
	$\jump(m+1) = k(m+1)-(m+1) = k(m)-m-1 = \jump(m) -1$.
	
	(b) This follows from (a) since, under the assumptions, the jumps are positive
	integers.
	
	(c) By (b), there exist positive integers $r\geq 2$ and $N$ such that $\jump(i)
	= N$ for $i\geq r-1$. Thus  $F(D_{i}^*) \subseteq D_{i+N}^*$ for every $i\geq
	r-1$. Here $F(D_{r-1}^*)$ may be a proper subset of $D_{r-1+N}^*$ (cf.
	Remark~\ref{R:mapDD}). However, we claim that $F(D_{i}^*) = D_{i+N}^*$ for
	$i\geq r$. Indeed, suppose that there exists $i \geq r$ with $F(D_{i}^*)
	\subsetneq D_{i+N}^*$. Then, by Lemma~\ref{L:ord-pres}(e), the set $F(D_{i}^*)$
	either does not contain the first point of $D_{i+N}^*$ or has positive distance
	from the first point of $D_{i+N+1}^*$. It contradicts the fact that the
	$F$-images of connected sets $D_{i}^*\cup D_{i-1}^*$ and $D_{i}^*\cup D_{i+1}^*$
	have to be connected.
	
	Once we know that $F(D_{i}^*) = D_{i+N}^*$ for $i\geq r$,
	Lemma~\ref{L:ord-pres}(e) implies that even
	$F(D_{i}) = D_{i+N}$, $i\geq r$. Since $G^N(D_{i}) = D_{i+N}$ for every $i$,
	Lemma~\ref{L:mapDD} gives that $F$ and $G^N$ coincide on $D_i$, $i\geq r$. Hence
	they coincide on the sub-snake $(( \mathscr K_0, x_r  \rangle \rangle$ and, by
	continuity, also on	$S^{\Join}_r$.
\end{proof}

We finally get the following result.

\begin{prop}\label{P:zero inf}
	For the one-dimensional continuum $X\subseteq \mathbb R^3$ constructed above
	in~(\ref{X-definition}) we have $S(X)=\{0,\infty\}$.
	Moreover, if $F: X\to X$ is a continuous map then $h^*(F) = \infty$ if $F$ is
	non-constant and has
	no fixed point in the snake, otherwise $h^*(F) = 0$.
\end{prop}

\begin{proof}
	Let $F: X\to X$ be a continuous map. If $F$ is constant then $h^*(F)=0$. Now let
	$F$ be non-constant.	
	
	First assume that $F$ has a fixed point in the snake. By Corollary~\ref{C:FixF},
	$\Fix (F)$ is either of the form $\langle \langle \mathscr K_0, f^m_i \rangle
	\rangle$ or  $\langle \langle  f^n_j, f^m_i \rangle \rangle$ for some $f^n_j
	\preccurlyeq f^m_i$. Then, by Lemmas~\ref{L:left} and~\ref{L:right}, there
	exists a positive integer $N$ such that $F^N(X)=\Fix(F)$. This clearly implies
	that $h_A(F)=0$ for any sequence $A$ and so $h^*(F)=0$ (alternatively, use
	Proposition~\ref{P2}(a)).
	
	Now assume that $F$ has no fixed point in the snake. Then, by
	Lemma~\ref{L:jumps}(c), there exist positive integers $r$ and $N$ such that on
	$S^{\Join}_r= \langle \langle \mathscr K_0, x_r  \rangle \rangle$ we have
	$$
	F|_{S^{\Join}_r} = G^N|_{S^{\Join}_r}
	$$
	where $G$ is the map from Lemma~\ref{L:existsG}. So,
	$h^*(G|_{S^{\Join}_r})=\infty$ and since $h^*(T) = h^*(T^N)$ for every $T$,
	we have also $h^*(G^N|_{S^{\Join}_r})=\infty$. Then $h^*(F) \geq
	h^*(F|_{S^{\Join}_r}) = h^*(G^N|_{S^{\Join}_r}) = \infty$.
	
	We have shown that, for every continuous map $F$ on $X$, either $h^*(F)=0$ or
	$h^*(F)=\infty$ and so the proposition is proved.
\end{proof}


\section{A map $T: X_1 \to X_1$ with $h^*(T)=\log
	2$}\label{S:X1-T1-log2}

Recall that in the construction of the continuum $X$ with $S(X)=\{0,\infty\}$ we
started with the auxiliary system $(X_1, T)$ such that $h^*(T)=\infty$. The
(disconnected) space $X_1$ consisted of two parts. The first part was a Cook
continuum (with identity on it), the second one was just one orbit $x_0, x_1,
\dots$ approaching the first part. Having defined the system $(X_1, T)$, for
each $m$ we have `joined' the points $x_m$ and $x_{m+1}$ by a continuum $D_m$ in
the form of a sequence of Cook continua and we obtained in such a way the
continuum $X$ with the required property $S(X)=\{0,\infty\}$.

In this section and the next one, we are going to construct a continuum  $X \subseteq \mathbb R^3$ with $S(X)
= \{0,\log 2\}$  in a similar way; we start with an auxiliary system
$(X_1, T)$ and then we add sequences of Cook continua. However, now the
construction is more complicated. In fact, it is much easier to make all `nontrivial'
continuous selfmaps $T$ to have $h^*(T)$ extremely large, i.e. equal to
$\infty$, than to have it equal exactly to $\log 2$, neither larger nor smaller.

In this section, we will first construct a system $(X_1,T)$ with $h^*(T)=\log 2$ and
then show that $h^*(T)=\log 2$.

\subsection{Construction of a system $(X_1,T)$}
The first part of our system $(X_1, T)$ in $\mathbb R^3$ will be the one point
compactification of $n\mapsto n+1$ on $\mathbb Z$.
Let its countable phase space $A$ be a subset of the unit circle $\mathbb S^1$
lying in the vertical plane $\pi_0$ (containing
the second and third axes), the first axis going to the right.\footnote{We think
	of $\mathbb R^3$ as the product $(-\infty, \infty)\times \pi_0$, each point
	$z\in \mathbb R^3$ being uniquely determined by its projections $P_1(z)$ and
	$P_2(z)$ into $(-\infty, \infty)$ and $\pi_0$, respectively.}

So, let $a_i, i \in \Z$ and $a_{\infty}$ be different points in $\mathbb S^1$,
with $\lim_{i\rightarrow \infty}a_i=\lim_{i\rightarrow
	\infty}a_{-i}=a_{\infty}$, see  Figure~\ref{fig:(A,T_1)}, and put $A=\{a_i| i \in \Z \} \cup\{a_{\infty}\}$.
Define the restriction of $T$ to the set $A$ by putting $T(a_i)=a_{i+1}$ and
$T(a_{\infty})=a_{\infty}$.
\begin{figure}[h]
	\centering
	\includegraphics[width=8cm]{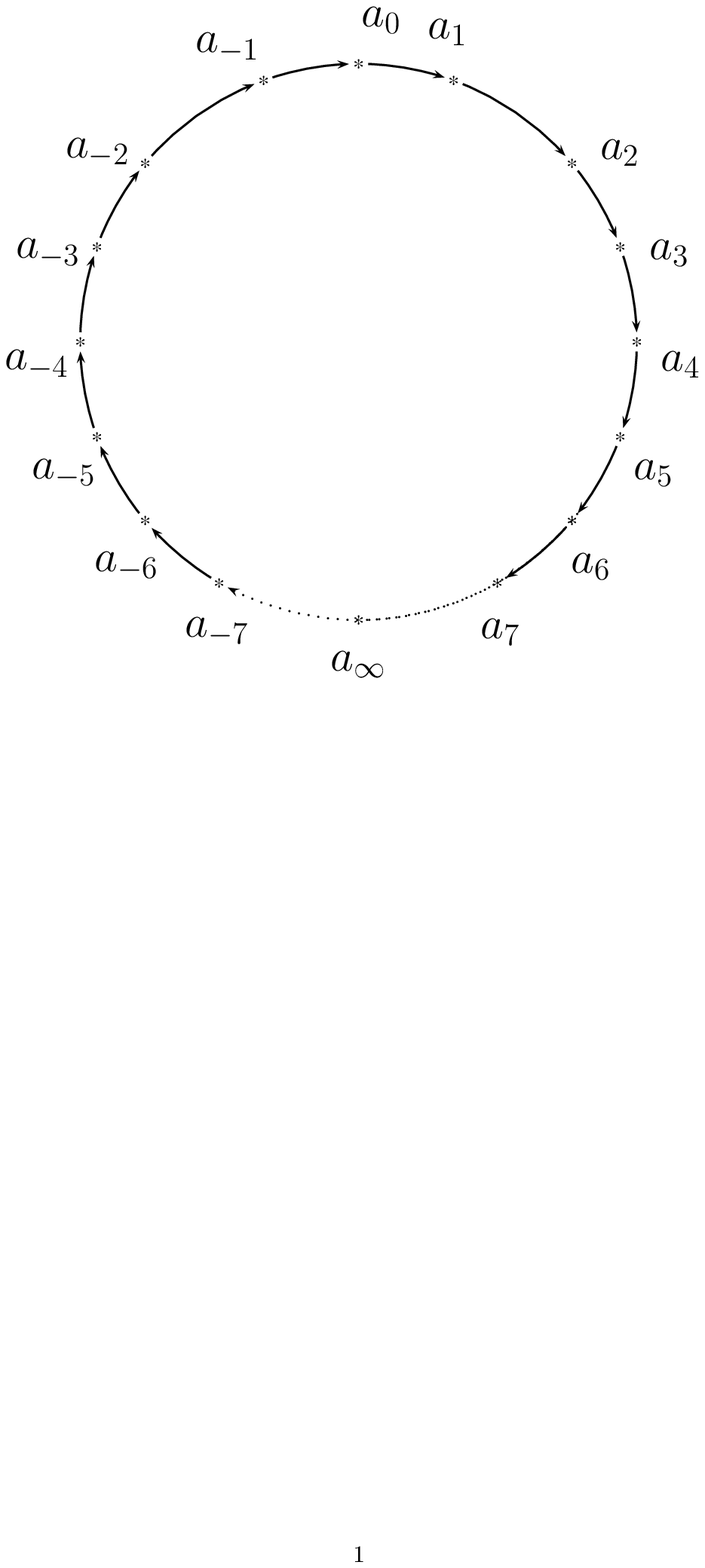}\\
	\caption{The system $(A,T|_A)$ lying in the vertical plane
		$\pi_0$}\label{fig:(A,T_1)}
\end{figure}

To keep under control the continuity of the map $T: X_1\to X_1$, as well as
for investigating IN-tuples, we will work
with neighbourhoods (in $\mathbb R^3$) of the points $a_i, i\in \mathbb Z$ and
$a_{\infty}$.
For each $k\in \mathbb N$ fix a family of \emph{pairwise disjoint} open sets
$U^k(a_i) \ni a_i$, $i\in \mathbb Z$ such that
$U^k(a_i)$ is the product of the interval $(-1/3k,1/3k)$ and a small open disk
in $\pi_0$ centered at $a_i$, such that
$$
\diam (U^k(a_i))< 1/k \text{ \ \ and \ \ } U^{k+1}(a_i) \subseteq U^k(a_i)
\text{ \ \ for every \ } k  \text{ \ \and \ \ } i.
$$
Choose also a neighbourhood $U^{1}(a_{\infty})$ of $a_{\infty}$ such that
\begin{equation}\label{eq:U_0U_infty}
U^{1}(a_{\infty})\cap U^{1}(a_j)= \emptyset \text{ \ \ for \ \ } |j| \le 3.
\end{equation}

The second part of $(X_1, T)$ will be formed by just one trajectory $x_0, x_1,
\dots$ lying in $(0,1] \times A$, with $T(x_i)=x_{i+1}$, approaching our set
$A$. In particular, we choose
\begin{equation}\label{Eq:approachpi0}
\text{(the distance of $x_j$ from $\pi_0$) $\searrow 0 \quad $ as $\quad j\to \infty$.}
\end{equation}
Since the set of all accumulation points of this trajectory will be the set
$A$, all sequence entropy tuples of $(X_1, T)$, if any, will consist only of
points of the set $A$. So, we claim that to get $h^*(T)=\log 2$, it will be
sufficient to fulfill the following list of requirements:
\begin{enumerate}
	\item $(a_0, a_1)$ is an IN-pair for $T$,
	\item $(a_0, a_j)$ is not an IN-pair for $T$ for any $|j|\geq 2$,
	\item $(a_0, a_{\infty})$ is not an IN-pair for $T$.
\end{enumerate}

Indeed, (1) implies $h^*(T)\geq \log 2$. On the other hand, (2) and (3) imply
that there is no intrinsic IN-tuple of length $3$. In fact, suppose that such a
tuple of length $3$ exists. It does not contain $a_{\infty}$, otherwise there is
also an IN-pair of the form $(a_i, a_{\infty})$ for some $i\in \mathbb Z$, but
then, using Proposition~\ref{P}(a) or Proposition~\ref{P2}(b), also $(T^{-i}(a_i), T^{-i}(a_{\infty})) =
(a_0, a_{\infty})$ is an intrinsic IN-pair, contradicting (3).
So, if an intrinsic IN-tuple of length $3$ exists, it is of the form $(a_i, a_j,
a_k)$, where $i<j<k$ are integers. By Proposition~\ref{P}(a) or Proposition~\ref{P2}(b),  $(T^{-i}(a_i),
T^{-i}(a_k)) = (a_0, a_{k-i})$ is an intrinsic IN-pair and since $|k-i|\geq
2$, we have a contradiction with (2).

To fulfill the three above requirements, it is obviously sufficient to fulfill
the following two requirements:

\begin{enumerate}
	\item [(R1)] For every $k$, the tuple $(U^k(a_0), U^k(a_1))$ has an
	independence set of times of cardinality $k+1$ (this is equivalent to (1)).
	\item [(R2)] The tuple $(U^1(a_0), U^1(a_j))$ does not have an independence set
	of times of cardinality $5$ whenever $|j|\geq 2$ or $j=\infty$ (this implies (2)
	and (3)).
\end{enumerate}

Now we are going to describe the sequence $x_0, x_1, \dots$, i.e. the trajectory
of $x_0$ under $T$. It will be a concatenation of infinitely many finite
sequences, some of them will be called \emph{blocks} \index{blocks}, the others will be called
\emph{outer gaps} \index{outer gaps}:
\begin{equation}\label{Eq:block-outgap}
x_0, x_1, \dots = \text{1st block}, \,  \text{1st outer gap}, \, \text{2nd
	block}, \, \text{2nd outer gap}, \, \dots~. 
\end{equation}
Further, for every $k$, the $k$-th block will be a concatenation of finite
sequences called \emph{pieces} \index{pieces} and \emph{inner gaps} \index{inner gaps}:
\begin{equation}\label{Eq:pieces-ingaps}
\text{k-th block} = \text{1st piece}, \, \text{1st inner gap}, \, \text{2nd
	piece}, \, \dots, \text{$(2^{k+1}-1)$st inner gap}, \, \text{$(2^{k+1})$-th
	piece}. 
\end{equation}
A proper choice of pieces will ensure (R1) and proper choices of inner and outer
gaps will ensure (R2).

The $k$-th block and the $k$-th outer gap will be denoted by $B(k)$ \index{$B(k)$} and $OG(k)$ \index{$OG(k)$},
respectively, $k=1,2,\dots$. The $k$-th outer gap $OG(k)$ will have length
$og(k)$ \index{$og(k)$}. The pieces in the $k$-th block (denoted by $P(k,l), 1 \le l\le 2^{k+1}$)\index{$P(k,l)$} will be
finite sequences of the same lengths $n_k^k+1$. The $t$-th inner gap in $k$-th block
(denoted by  $IG(k,t), 1 \le t \le 2^{k+1}-1$)\index{$IG(k,j)$} will have length $ig(k,t)$.
(The detailed definitions of $n_k^k, ig(k,t), og(k), P(k,l), IG(k,t),
OG(k)$ will be given later.) Thus, the structure of the trajectory is as in
Table~\ref{T:BPIGOG}, the 3rd row shows the lengths.

\begin{table}[h]
	\begin{footnotesize}
		\begin{center}
			\begin{tabular}{|c|c|c|c|c|c|c|c|c|c|c|c|c|}
				\hline
				\multicolumn{7}{|c|}{$B(1)$}  & $OG(1)$ & $\dots$ &
				\multicolumn{3}{|c|}{$B(k)$} & $\dots$ \\
				\hline
				$P(1,1)$ & $IG(1,1)$ & $P(1,2)$ & $IG(1,2)$ & $P(1,3)$ & $IG(1,3)$ & $P(1,4)$
				&  & $\dots$ & $P(k,1)$  &$ \dots$ & $P(k,2^{k+1}) $  & $\dots$ \\
				\hline
				$n_1^1+1$ & $ig(1,1)$ & $n_1^1+1$ & $ig(1,2)$& $n_1^1+1$ & $ig(1,3)$
				&$n_1^1+1$ & $og(1)$ & $\dots$ & $n_k^k+1$  &$ \dots$ & $n_k^k+1$  & $\dots$ \\
				\hline
			\end{tabular}
		\end{center}
	\end{footnotesize}
	\caption{Blocks, pieces, gaps and their lengths}\label{T:BPIGOG}
\end{table}

For each $k$, the $2^{k+1}$ pieces in the $k$-th block will be designed so that
(R1) be fulfilled for this $k$. Denote
\begin{equation}\label{eq:notFk}
F(k)=\{0,1\}^{\{0,1,2,\cdots,k\}}=\{s_l: 1 \le l \le 2^{k+1}\}, \quad
s_1=(0,0,\dots,0),\, \dots, \, s_{2^{k+1}}=(1,1,\dots,1).
\end{equation}
So, for each $k$ we fix, once and for all, a choice of $2^{k+1}$ functions $s_l$.
They can be considered as $(k+1)$-tuples of zeros and ones. Say, let them be
ordered from $s_1$ to $s_{2^{k+1}}$ lexicographically (then $s_1$ and
$s_{2^{k+1}}$ are constant, as written above).\footnote{Since the functions
	$s_l$ depend both on $k$ and $l$, we should in fact write $s_{(k,l)}$ rather than just $s_l$. We abuse
	notation here hoping that no misunderstanding will arise.}

For each $l$, the piece $P(k,l)$ will be a finite sequence of length $n_k^k+1$
of the form
\begin{equation}\label{eq:Pkl}
P(k,l) = x_j, x_{j+1}, \dots, x_{j+n_1^k},  \dots, x_{j+n_2^k},  \dots,
x_{j+n_k^k}
\end{equation}
where $j=j(k,l) \geq 0$ will depend on both $k$ and $l$, but $n_1^k, \dots,
n_k^k$ only on $k$ and not on $l$,\footnote{Therefore we write $n_i^k$ rather
	than  $n_i^{(k,l)}$.} such that
\begin{equation}\label{eq: pieces requirment}
x_j \in U^{k}(a_{s_l(0)}), \,\, T^{n_i^k}x_j=x_{j+n_i^k} \in
U^{k}(a_{s_l(i)}), \quad 1 \le i \le k~.
\end{equation}
Equivalently, if we put $n_0^k:=0$, this means that $x_j \in
\bigcap_{i=0}^{k}T^{-n_i^k}U^k(a_{s_l(i)})$.
Since for each $l=1,\dots, 2^{k+1}$ there will be a piece $P(k,l)$ corresponding to $s_l$,
the set
\begin{equation}\label{eq:notNk}
N(k)=\{n_0^k=0, n_1^k, \dots, n_k^k\}
\end{equation}
will be an independence set of times of length $k+1$ for $(U^k(a_0), U^k(a_1))$.

We see from~(\ref{eq:Pkl}) that the piece $P(k,l)$ consists of $k$ shorter
sequences, called \emph{winds} \index{wind}:
\begin{equation}\label{eq:winds}
W^{(k,l)}_1= x_j (= x_{j+n_0^k}), x_{j+1}, \dots, x_{j+n_1^k}, \quad  \dots,
\quad W^{(k,l)}_k = x_{j+n_{k-1}^k}, \dots, x_{j+n_k^k}~.
\end{equation}
The wind $W^{(k,l)}_i$ starts in $x_{j+n_{i-1}^k}\in U^k(a_{s_l(i-1)})$ and ends
in $x_{j+n_{i}^k}\in U^k(a_{s_l(i)})$.
Put $w^{k}_i= |W^{(k,l)}_i|$.\footnote{The wind  $W^{(k,l)}_i$ depends on $k$,
	$l$ and $i$ but its length does not depend on $l$, therefore we write  $w^{k}_i$
	rather than  $w^{(k,l)}_i$.}
The following table shows the structure of $P(k,l)$. The second row shows the
function (tuple) to which the piece corresponds, the last row contains the
lengths of the winds.

\begin{table}[h]
	\begin{footnotesize}
		\begin{center}
			\begin{tabular}{|c|c|c|c|c|}
               \hline
				$P(1,1)$ & $P(1,2)$ & $P(1,3)$ & $P(1,4)$ & $\dots$ 				 \\
				\hline
				$s_1=(0,0)$ & $s_2=(0,1)$ & $s_3=(1,0)$ & $s_4=(1,1)$ & $\dots$  \\
				\hline
				$W^{(1,1)}_1$ & $W^{(1,2)}_1$ & $W^{(1,3)}_1$  & $W^{(1,4)}_1$ & $\dots$ \\
				\hline
				$w^1_1 = n_1^1+1$ & $w^1_1$ & $w^1_1$ & $w^1_1$ & $\dots$  \\
				\hline
				\hline
               \multicolumn{4}{|c|}{$P(k,l) = x_j, \dots, x_{j+n_1^k},  \dots, x_{j+n_2^k},
					\dots, x_{j+n_k^k}$} & $\dots$ \\
				\hline
             	\multicolumn{4}{|c|}{$s_l = (s_l(0), s_l(1), \dots, s_l(k))$}  & $\dots$  \\
				\hline
				$W^{(k,l)}_1$ & $W^{(k,l)}_2$ & $\dots$ & $W^{(k,l)}_k$ & $\dots$ \\
				\hline
				 $w^{k}_1 = n_1^k- n_0^k +1 = n_1^k+1$ & $w^{k}_2 = n_2^k - n_1^k +1$ & $\dots$ & $w^{k}_k=n_k^k
				- n_{k-1}^k +1$ & $\dots$ \\
				\hline
            \end{tabular}
		\end{center}
	\end{footnotesize}
	\caption{Winds and their lengths}\label{T:winds}
\end{table}

Note that each of the elements $x_{j+n_1^k}, \dots, x_{j+n_{k-1}^k}$ belongs to
two of these winds.
So, the winds in $P(k,l)$ are not disjoint, two consecutive winds have one point
in common. However, we will abuse terminology and we still will say that
$P(k,l)$ is a concatenation of these winds. The fact that $P(k,l)$ consists of
$k$ winds (though not pairwise disjoint) will also be expressed by saying that
$P(k,l)$ `winds $k$-times around the set $A$'.

We have seen that already the mere concatenation of the pieces
$$
P(1,1), \, \dots, \,  P(1,4), \, P(2,1),\, \dots, \, P(2,8), \, \dots, \,
P(k,1), \, \dots, \, P(k,2^{k+1}),\, \dots
$$
ensures (R1). Unfortunately, then there is a risk that $(U^{1}(a_0),U^{1}(a_j))$
for some $|j|\ge 2$ or $(U^{1}(a_0),U^{1}(a_{\infty}))$
will have an independence set of times of length $5$. Therefore, to be sure that
also (R2) is fulfilled, we are going to add the gaps as indicated above. Since
the proof that the gaps will really have the required effect will need computations,
we are going to describe the trajectory $x_0, x_1, \dots$ in more details. We
provide also pictures.

Our space $X_1$ will be a subset of $[0,1] \times A \subseteq [0,1] \times
\pi_0$,
so let $P_1$ and $P_2$ be the projections of $[0,1] \times \pi_0$ onto $[0,1]$
and $\pi_0$, respectively.
The trajectory $x_0,x_1,\dots$ will be chosen such that
\begin{equation}\label{Eq:projections}
P_2(x_i) \in A,\, i=0,1,2, \dots, \quad \text{and} \quad P_1(x_i) \searrow 0
\end{equation}
and so all the accumulation points of the trajectory $x_0,x_1,\dots$ will be in
$A$.

Recall that, to ensure (R1), the first block $B(1)$ has $2^2$ pieces. Each piece
$P(1,l)$, corresponding to an $s_l \in \{0,1\}^{\{0,1\}}$, satisfies
\eqref{eq: pieces requirment} for $k=1$. We have $n_0^1=0$ and choose $n_1^1=7$.
Further, $s_1=(0,0)$, $s_2=(0,1)$, $s_3=(1,0)$ and $s_4=(1,1)$.
Figures~\ref{fig:p11}-\ref{fig:og1} show the four pieces and the three inner
gaps of the first block, and the first outer gap. We are going to explain how
the pictures should be understood.

\begin{figure}
	\begin{minipage}{0.45\linewidth}
		\centering
		\includegraphics[width=2.8in]{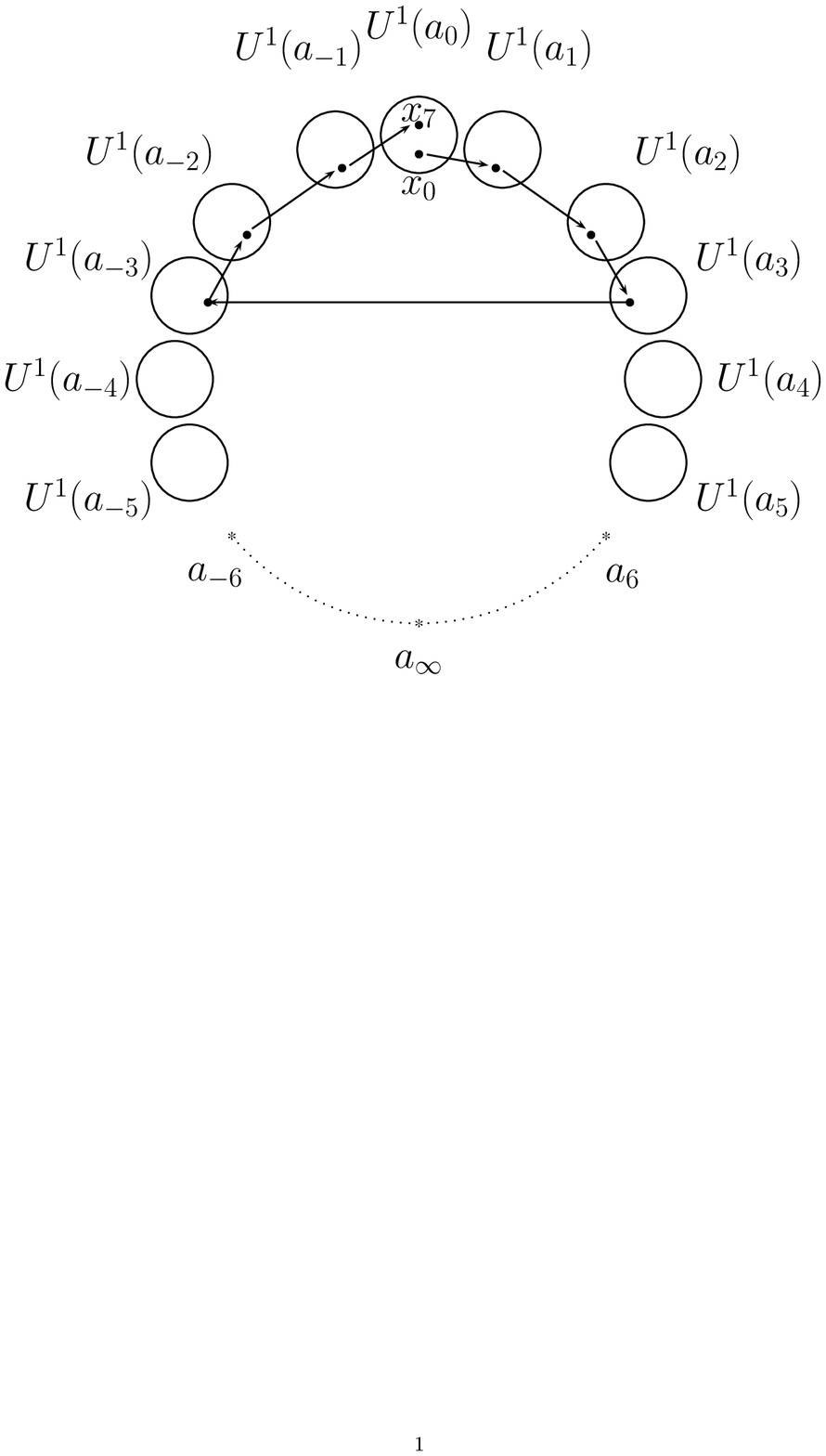}
		\caption{$P(1,1)=\{x_0,x_1,\dots,x_7\}$}
		\label{fig:p11}
	\end{minipage}%
	\begin{minipage}{0.45\linewidth}
		\centering
		\includegraphics[width=2.8in]{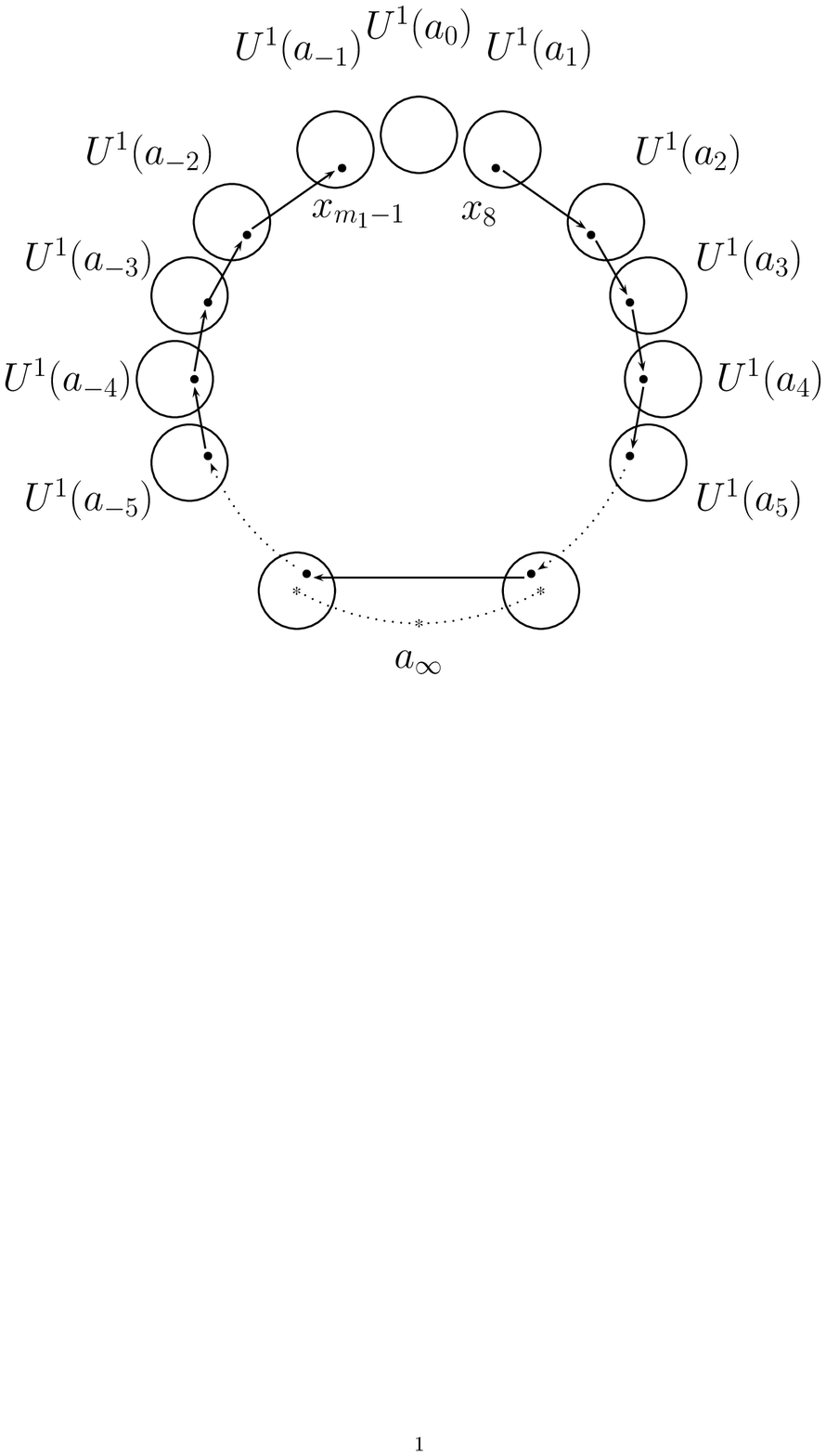}
		\caption{$IG(1,1)=\{x_8,x_9,\dots,x_{m_1-1}\}$ with $m_1-8=ig(1,1)$}
		\label{fig:ig11}
	\end{minipage}
\end{figure}

\begin{figure}
	\begin{minipage}{0.45\linewidth}
		\centering
		\includegraphics[width=2.8in]{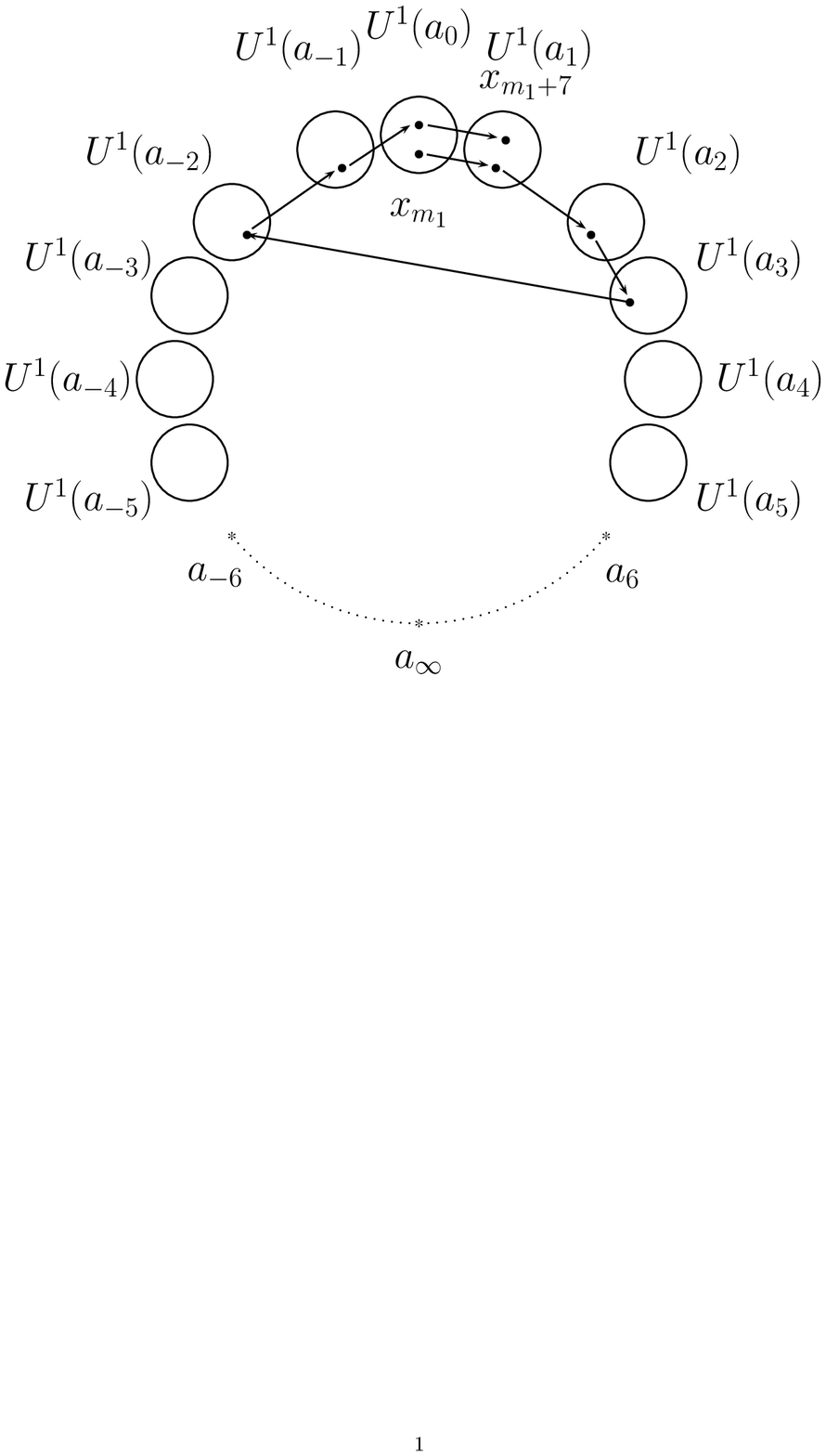}
		\caption{$P(1,2)=\{x_{m_1},x_{m_1+1},\dots,x_{m_1+7}\}$}
		\label{fig:p12}
	\end{minipage}
	\begin{minipage}{0.45\linewidth}
		\centering
		\includegraphics[width=2.8in]{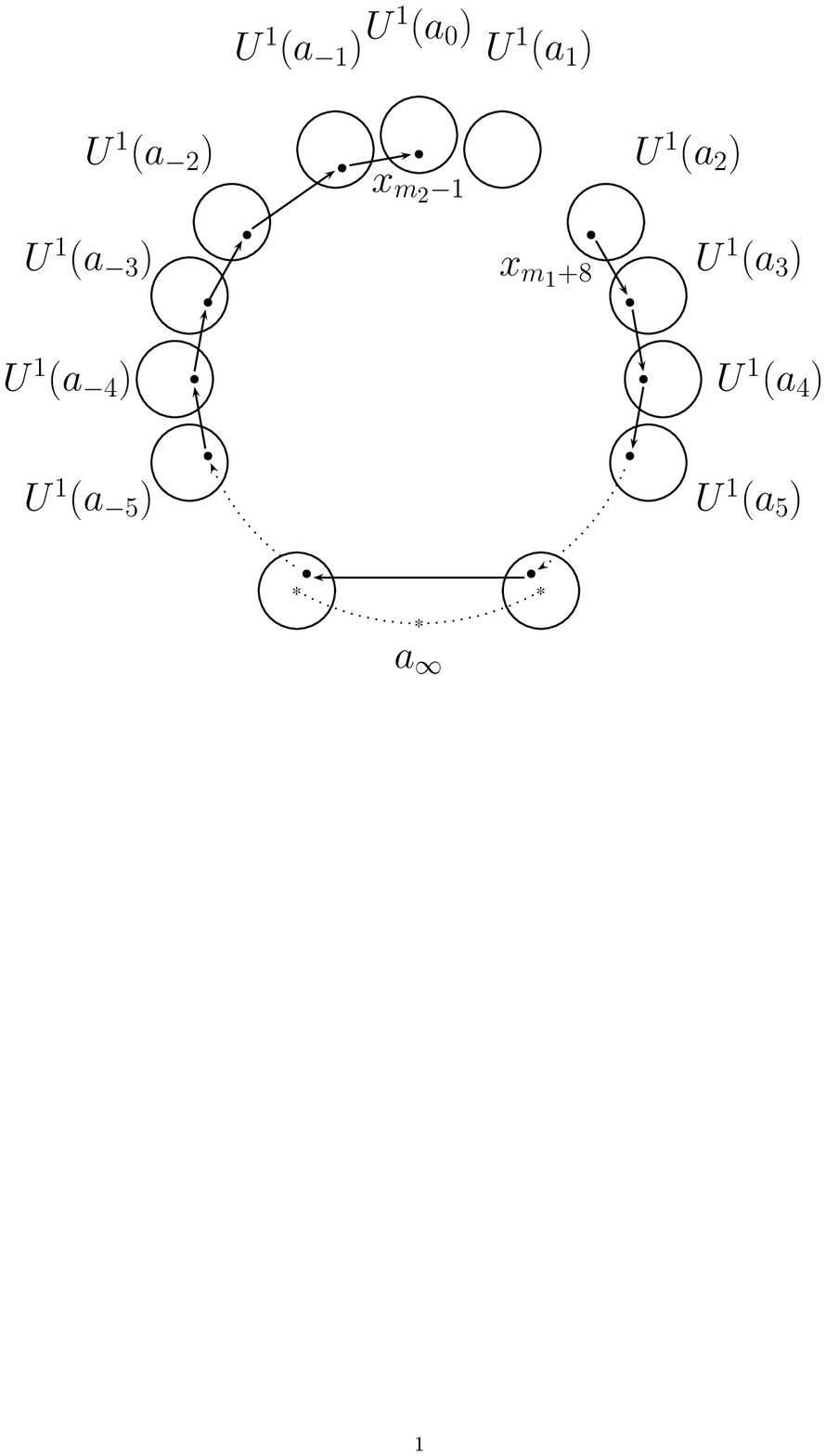}
		\caption{$IG(1,2)=\{x_{m_1+8},x_{m_1+9},\dots,x_{m_2-1}\}$ with
			$m_2-m_1-8=ig(1,2)$}
		\label{fig:ig12}
	\end{minipage}
\end{figure}

\begin{figure}
	\begin{minipage}[h]{0.45\linewidth}
		\centering
		\includegraphics[width=2.8in]{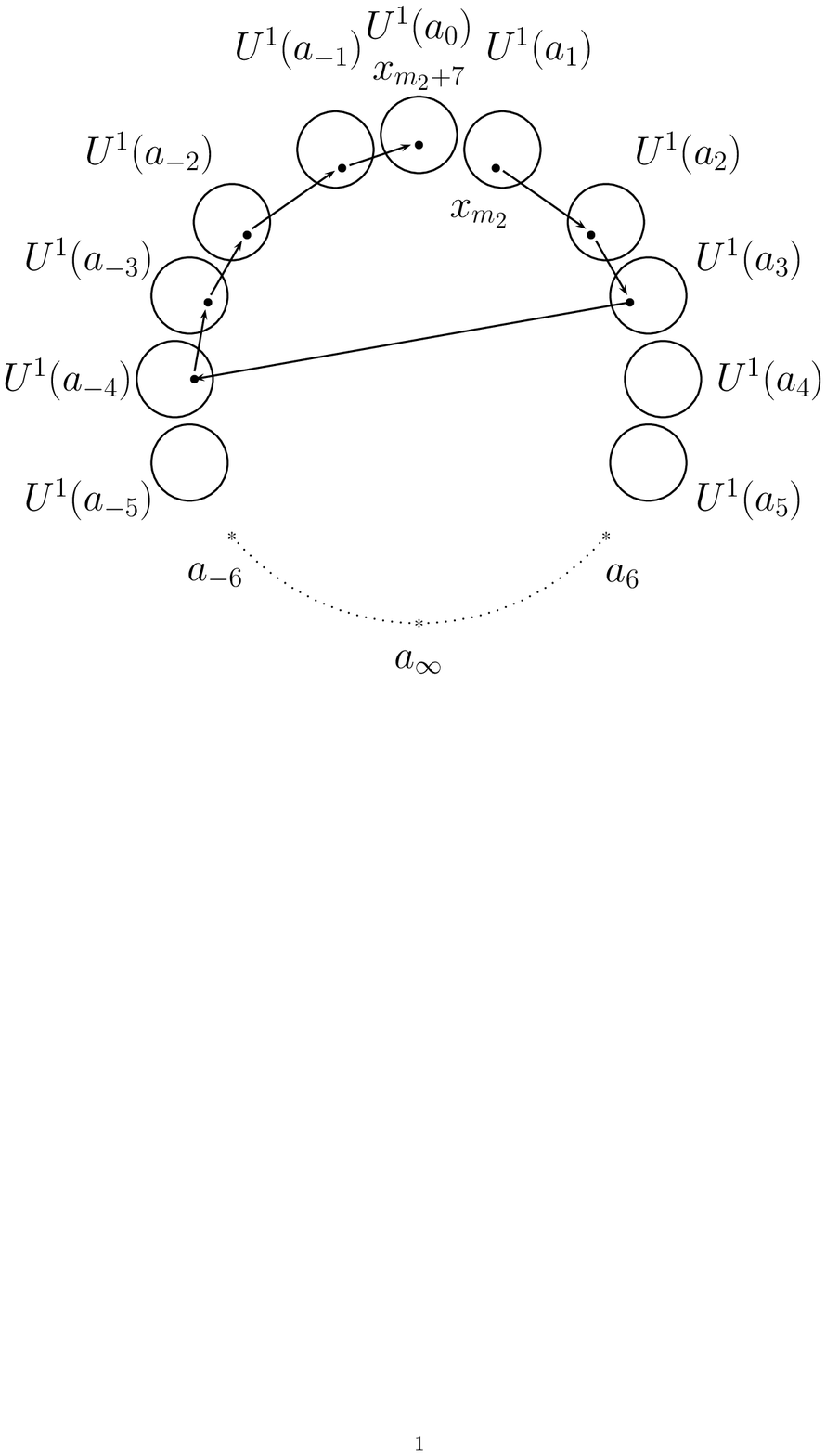}
		\caption{$P(1,3)=\{x_{m_2},x_{m_2+1},\dots,x_{m_2+7}\}$}
		\label{fig:p13}
	\end{minipage}
	\begin{minipage}[h]{0.45\linewidth}
		\centering
		\includegraphics[width=2.8in]{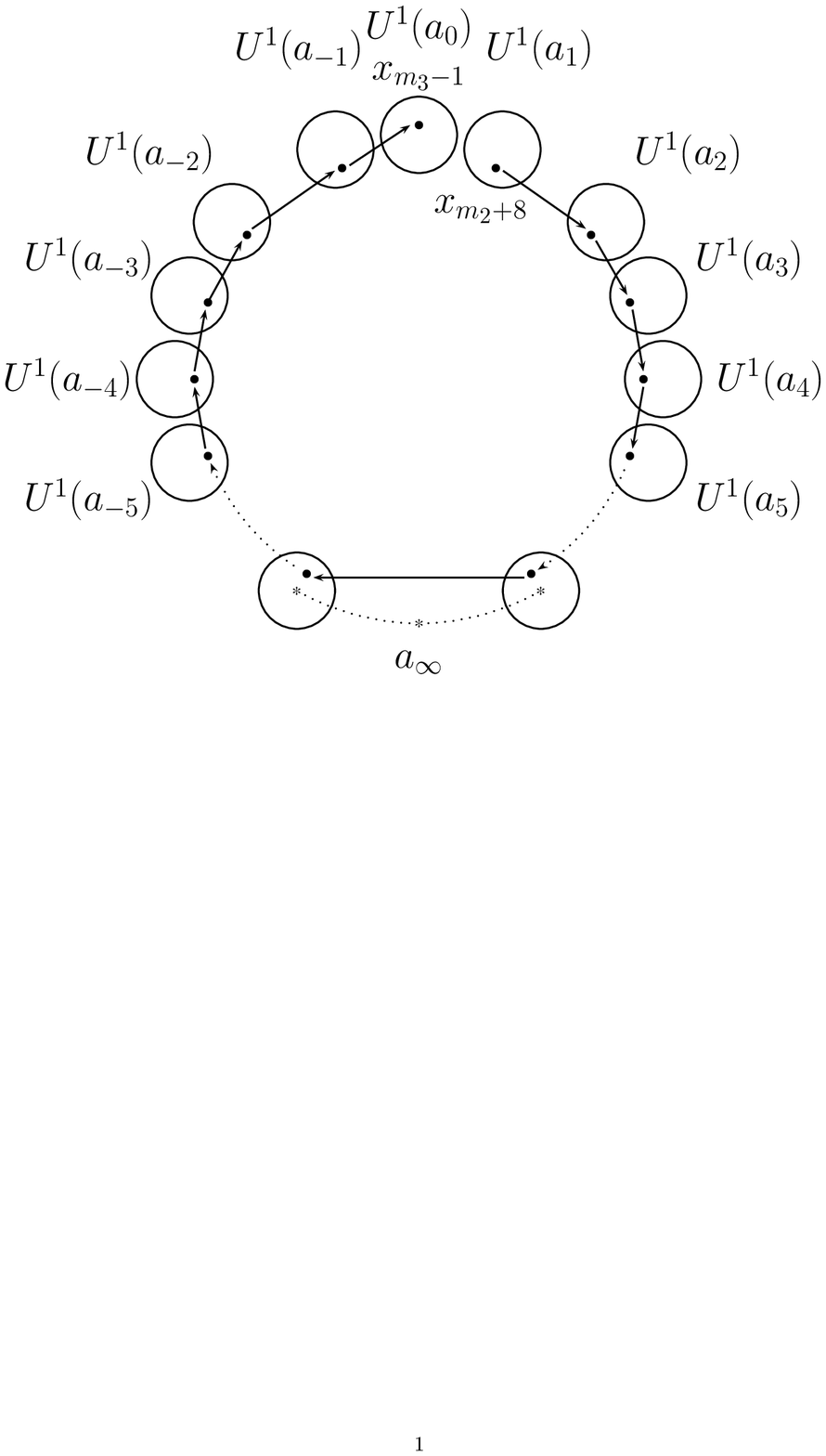}
		\caption{$IG(1,3)=\{x_{m_2+8},x_{m_2+9},\dots,x_{m_3-1}\}$ with
			$m_3-m_2-8=ig(1,3)$}
		\label{fig:ig13}
	\end{minipage}
\end{figure}

\begin{figure}
	\begin{minipage}[h]{0.45\linewidth}
		\centering
		\includegraphics[width=2.8in]{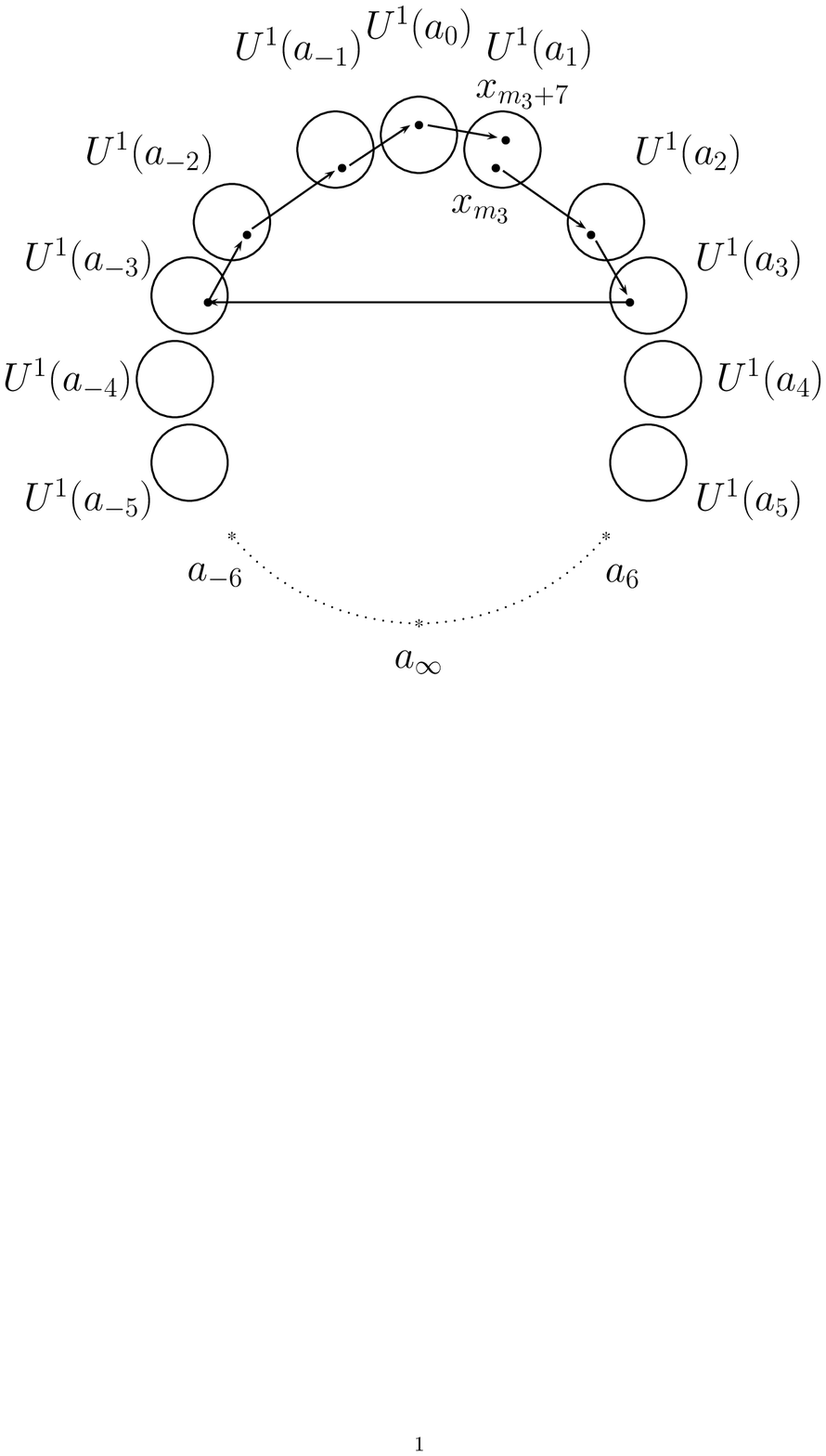}
		\caption{$P(1,4)=\{x_{m_3},x_{m_3+1},\dots,x_{m_3+7}\}$}
		\label{fig:p14}
	\end{minipage}
	\begin{minipage}[h]{0.45\linewidth}
		\centering
		\includegraphics[width=2.8in]{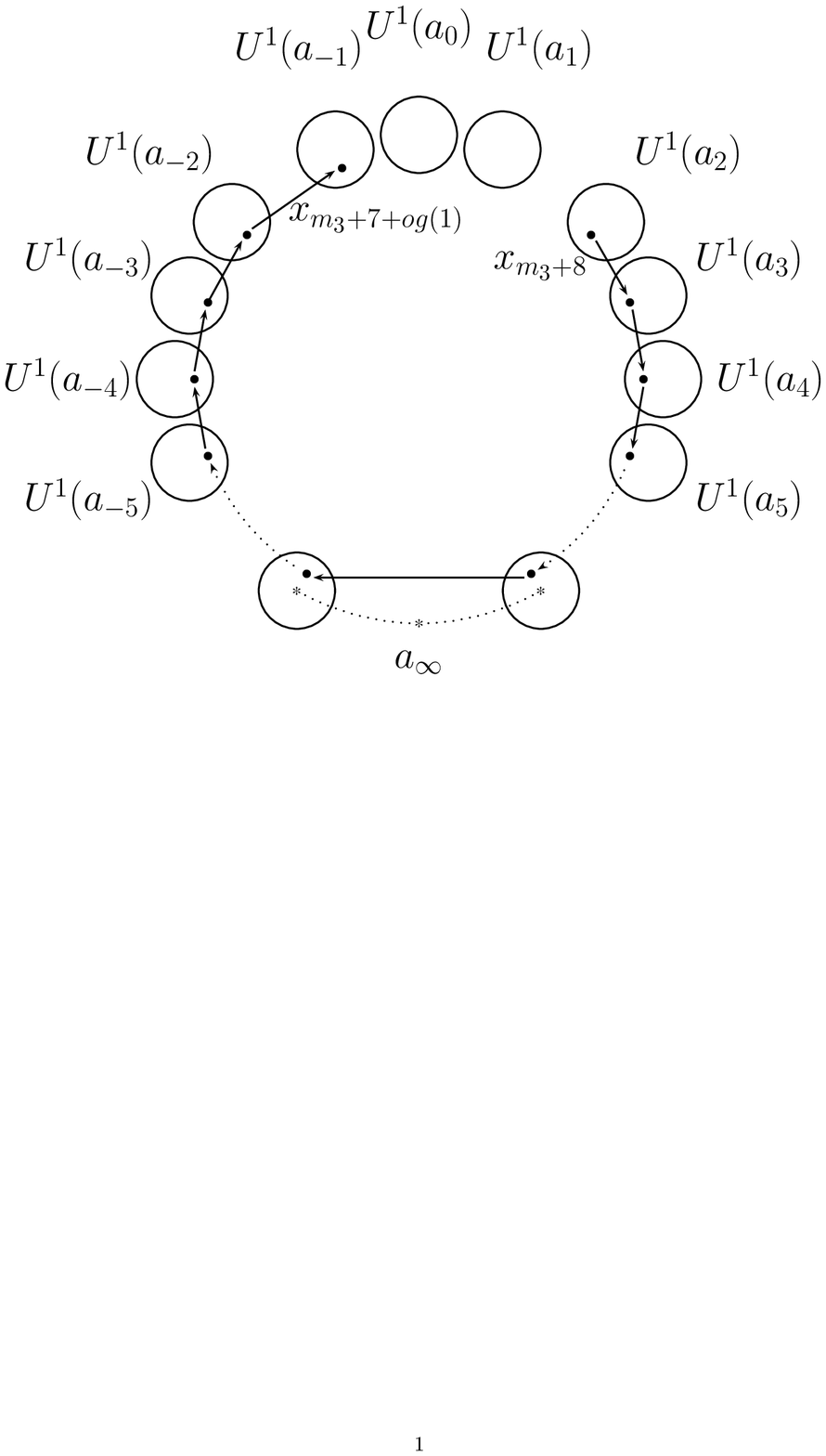}
		\caption{$OG(1)=\{x_{m_3+8},x_{m_3+9},\dots,x_{m_3+7+og(1)}\}$}
		\label{fig:og1}
	\end{minipage}
\end{figure}

Let us start with Figure~\ref{fig:p11} which shows the piece $P(1,1) = \{x_0,
x_1, \dots, x_7\}$. We are looking at our system along the first axis (with the
eyes at the point $+\infty$) and so what we see in that picture is in fact the
projection onto the vertical plane $\pi_0$. Thus, though we consider our
`cylindrical' neighbourhoods $U^1(a_i)$ of the points $a_i$, on the picture we
can see only their projections onto $\pi_0$, i.e. disks. However, we keep the
notation $U^1(a_i)$ rather than $P_2(U^1(a_i))$. Similarly, we consider points
$x_0, x_7 \in (0,1]\times \{a_0\}$ but we can see only their projections onto
$\pi_0$, still keeping the notations $x_0$ and $x_7$. Since $P_2(x_0) = P_2(x_7)
= a_0 = P_2(a_0)$, the projections of both points $x_0$ and $x_7$ should be in
the center of the corresponding disk. However, our  rule when drawing the
picture is that if the projections of points we consider are exactly the centers
of the disks, we draw them as different points in that disk, not to loose the
lucidity of the picture. So, Figure~\ref{fig:p11} means: $x_0, x_7 \in
(0,1]\times \{a_0\}$, $x_1 \in (0,1]\times \{a_1\}$, $\dots$,  $x_4 \in
(0,1]\times \{a_{-3}\}$, $\dots$. Further, arrows of course mean that
$T(x_i)=x_{i+1}$ for $i=0,1,\dots,6$. Recall also that $P_1(x_0)> \dots >
P_1(x_7)$, see~(\ref{Eq:projections}).

The first piece $P(1,1)$ is concatenated with the first inner gap $IG(1,1) =
\{x_8,x_9,\dots,x_{m_1-1}\}$, see Figure~\ref{fig:ig11}. Here $x_8\in
(0,1]\times \{a_1\}$. Of course $T(x_7)=x_8$, but we do not draw the
corresponding arrow (since we even do not draw $x_7$ in Figure~\ref{fig:ig11}).
We continue this way, alternating pieces and inner blocks until the first block
is completed, see Figures~\ref{fig:p12}-\ref{fig:p14}. Finally, before starting
the construction of the second block, we add the first outer gap
$OG(1)=\{x_{m_3+8},x_{m_3+9},\dots,x_{m_3+7+og(1)}\}$, see Figure~\ref{fig:og1}.

The pictures show the role of those finite sequences whose concatenation gives
the $T$-trajectory of $x_0$ and they also show why $T$ is continuous. Let us
go to details.

The pieces $P(1,1)$, $P(1,2)$, $P(1,3)$ and $P(1,4)$ give the existence of
points $x_0$, $x_{m_1}$, $x_{m_2}$ and $x_{m_3}$, respectively, such that $x_0$
is in $U^1(a_0)$ at time $n_0^1 =0$ and again in $U^1(a_0)$ at time $n_1^1 = 7$,
$x_{m_1}$ is in $U^1(a_0)$ at time $n_0^1 =0$ and in $U^1(a_1)$ at time $n_1^1 =
7$, $x_{m_2}$ is in $U^1(a_1)$ at time $n_0^1 =0$ and in $U^1(a_0)$ at time
$n_1^1 = 7$, and $x_{m_3}$ is in $U^1(a_1)$ at time $n_0^1 =0$ and again in
$U^1(a_1)$ at time $n_1^1 = 7$.

If $k>1$, the structure of the pieces in $B(k)$ is more complicated than in the
case $k=1$. To make the life of the reader easier, we add one more picture. For
instance, consider $k=2$. Let the corresponding independence set of times be,
say, $\{n_0^2=0, n_1^2 = 9, n_2^2 = 23\}$ (this is just for an illustration, later in
fact the numbers $n_1^2$ and $n_2^2$ will be chosen much larger). Further,
consider for instance $l=6$. The piece $P(2,6)$, see~(\ref{eq:Pkl}) and
Figure~\ref{fig:p26}, corresponds to $s_6 = (1,0,1)$. So, at time $n_0^2=0$ the
piece starts at $U^2(a_1)$ (because $s_6(0)=1$), at time $n_1^2 = 9$ it visits
$U^2(a_0)$ (because $s_6(1)=0$) and at time $n_2^2 = 23$ it ends in $U^2(a_1)$
(because $s_6(2)=1$).

\begin{figure}[h]
	\centering
	\includegraphics[width=10cm]{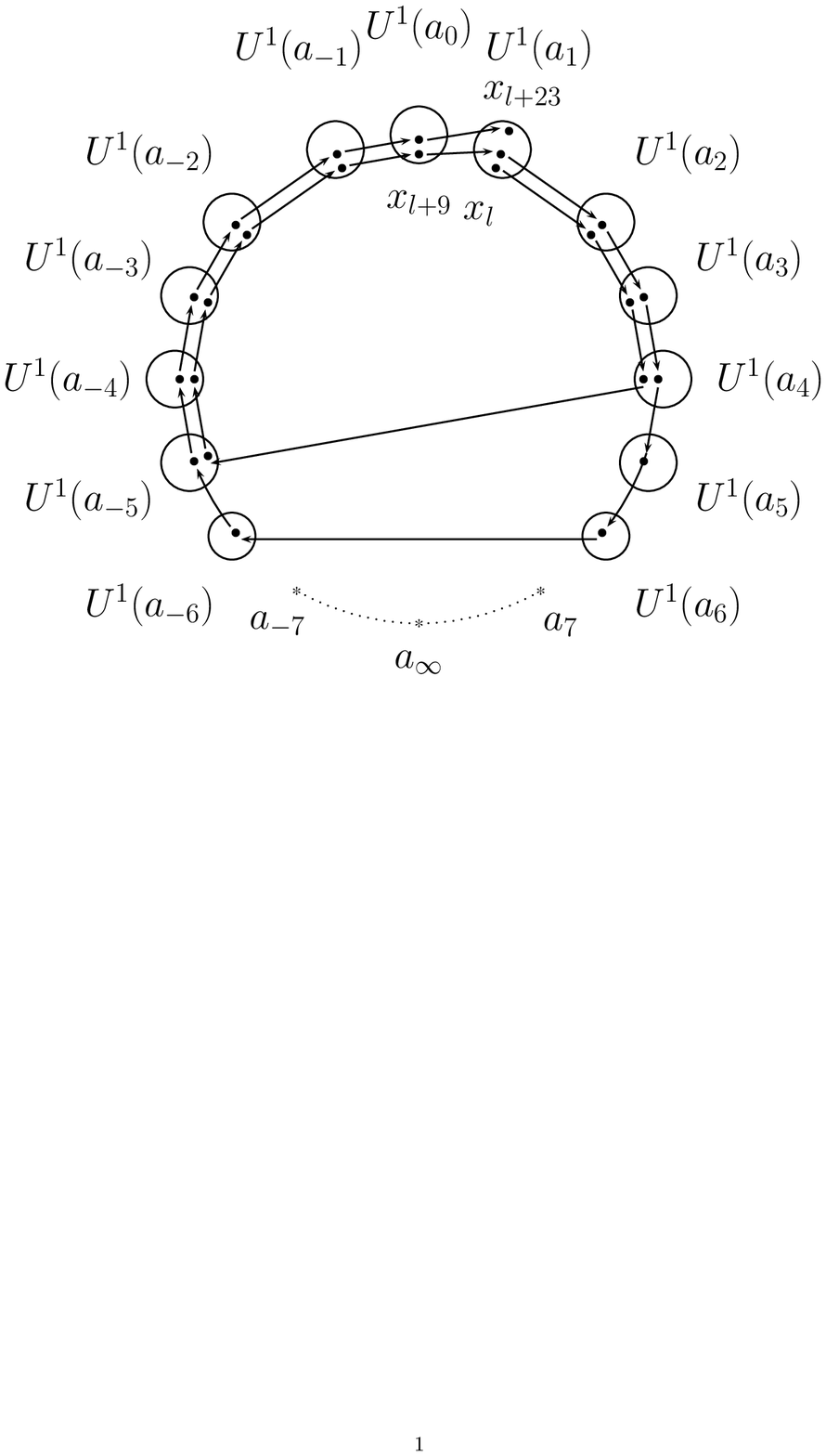}\\
	\caption{The piece $P(2,6)$ corresponding to $s_6 = (1,0,1)$, with $\{n_0^2=0,
		n_1^2 = 9, n_2^2 = 23\}$. It consists of two winds $W^{(2,6)}_1$ and
		$W^{(2,6)}_2$ with lengths $w^2_1 = 10$ and $w^2_2 =15$.}\label{fig:p26}
\end{figure}

We want that $T\colon X_1\to X_1$ be continuous. Therefore, when constructing the
trajectory $x_0,x_1,\dots$ approaching the set $A$ (where $T$ has already been
defined), we use the general rule that the $P_2$-projections of consecutive
points $x_i, x_{i+1}$ are the consecutive points of $A$ (consecutive in the
sense of $T|_A$-dynamics). In rare exceptions from this rule the points $x_i,
x_{i+1}$ are either `far away' from the plane $\pi_0$ (such as $x_3$ and $x_4$ in
Figure~\ref{fig:p11}) or they both are `close' to the fixed point $a_{\infty}$.
Always, when $x_i, x_{i+1}$ is such an exception, the point $x_i$ is on the `right side'
of $a_{\infty}$ and $x_{i+1}$ is on the `left side' of $a_{\infty}$ (see e.g. the
`jump' from $x_3$ on the `right side' of $a_{\infty}$ to $x_4$ on the `left side'
of $a_{\infty}$ in Figure~\ref{fig:p11} or two such `jumps' in Figure~\ref{fig:p26}).
In fact we will have infinitely many exceptions from what we call the general rule,
but they will be closer and closer to the fixed point $a_{\infty}$. Therefore the
continuity of $T$ will not be violated.

The inner gap $IG(1,l)$ is inserted between the pieces $P(1,l)$ and $P(1,l+1)$
and so, following the just mentioned general rule, we take care that
$P_2(\text{first point of }IG(1,l)) = T|_A(P_2(\text{last point of }P(1,l)))$
and
$P_2(\text{first point of }P(1,l+1)) = T|_A(P_2(\text{last point of
}IG(1,l)))$ (check this in our pictures).

The outer gap $OG(1)$ is used for linking the blocks $B(1)$ and $B(2)$. So, the
first point in $OG(1)$ is $x_{m_3+8}$, the first point in $B(2)$ is
$x_{m_3+og(1)+8}$ and when choosing their positions we keep our general rule
mentioned above, cf. Figures~\ref{fig:p14} and~\ref{fig:og1}.

The constructions of $B(2)$ and $OG(2)$ and, generally, $B(k)$ and $OG(k)$ go in
an analogous way. However, note that in $B(1)$
each of the pieces winds around $A$ only once (i.e., it consists of just one
wind). Each of the inner gaps in the first block, as well as
the first outer gap, also wind around $A$ once.
When constructing $B(k)$ and $OG(k)$, all the gaps will still wind around $A$
only once (therefore we do not introduce the notion of a wind in a gap),
but each of the pieces in $B(k)$ will $k$ times wind around $A$. So, each piece
in $B(k)$ consists of $k$ winds.
This is because the role of the pieces in $B(k)$ is to ensure that there exists
an independence set $\{n_0^k=0, n_1^k, \dots, n_k^k\}$ of times
of length $k+1$ for $(U^k(a_0), U^k(a_1))$.
Our way how to do that is as follows. Given $s_l \in \{0,1\}^{\{0,1,\dots,k\}}$,
the piece $P(k,l)$ starts in $U^k(a_{s_l (0)})$ at time $n^k_0=0$, then winds
once around $A$ to come to $U^k(a_{s_l (1)})$ at time $n^k_1$ (this is the last
point of the wind $W^{(k,l)}_1$ and the first point of the wind $W^{(k,l)}_2$),
again winds once around $A$ to come to $U^k(a_{s_l (2)})$ at time $n^k_2$ (this
was the wind $W^{(k,l)}_2$), and so on, finally winds around $A$ to come to
$U^k(a_{s_l (k)})$ at time $n^k_k$ (this was the wind $W^{(k,l)}_k$).
For an example see Figure~\ref{fig:p26} where the piece $P(2,6)$ consists of two
winds.

\medskip

Now the reader has a global picture of the system $(X_1,T)$.

\medskip

Before going to a more detailed specification of parameters in our construction,
we introduce some notation. Let $Y_k$ be the concatenation of $B(k)$ and
$OG(k)$. If we abuse the notation
(making no difference between a finite sequence and its set of values), then
$$
Y_k = B(k) \cup OG(k)
$$
and the second part of $X_1$, i.e. the orbit of $x_0$, is the set $Y =
\bigcup_{k=1}^{\infty} Y_k$.
For $k=1,2,\dots$, the set (or the finite sequence) $Y_k$ \index{$Y_k$} will be called the
\emph{$k$-th level of $Y$} \index{$k$-th level of $Y$}.
So, by saying that a point is in the  $k$-th level we mean that it belongs to
$Y_k$.
Further, let $\mathcal S$ be a finite subsequence of the trajectory
$x_0,x_1,\dots$
(say, $\mathcal S$ is a block, a piece, a gap or a wind). Consider the (finite)
set
$$
\pre (\mathcal S) = \{x_i:\, 0\leq i < s, \, \text{where $x_s$ is the first
	element of $\mathcal S$}\}~.
$$
We will call it \emph{the set of all predecessors of $\mathcal S$} \index{$\pre (\mathcal S)$}. If $s\geq 1$, the point
$x_{s-1}$ will be called
\emph{the immediate predecessor of $\mathcal S$} \index{immediate predecessor}.

Let us also adopt the convention that
\begin{itemize}
	\item $a \gg b$, or $b \ll a$  \index{$a \gg b$} \index{$b \ll a$}, means that $a >100\times b$, and
	\item $(a_i)_{i=1}^N \upuparrows$ \index{$(a_i)_{i=1}^N \upuparrows$} and $(a_i)_{i=1}^{\infty} \upuparrows$ \index{$(a_i)_{i=1}^{\infty} \upuparrows$} mean
	that in the corresponding
	sequence we have $a_i \ll a_{i+1}$ for each considered $i$ (except of $i=N$ in the first
	case). Moreover, in each of
	these cases we will sometimes write just $(a_i) \upuparrows$.
\end{itemize}

If $S$ is a finite set or a finite injective sequence, we will denote by $|S|$
the number of points in $S$.

For proving that $h^*(T)$ is $\log 2$ (and not higher), we still need
to specify appropriately some of the pa\-ra\-me\-ters, namely, for $k \in \N$,
the \emph{independence set of times} $N(k)=\{n_0^k=0,n_1^k,n_2^k\cdots,
n_{k}^k\}$ \index{$N(k)$}, the \emph{outer gap length} $og(k)$ \index{$og(k)$} and the \emph{inner gap lengths}
$ig(k,j)$ \index{$ig(k,j)$}, $ 1 \le j \le 2^{k+1}-1$. When imaging how we construct the
trajectory $x_0, x_1, \dots$, cf. Table~\ref{T:BPIGOG} and~\eqref{eq:Pkl}, we see that we just need to specify the numbers
$n_1^1$, $ig(1,1)$, $ig(1,2)$, $ig(1,3)$, $og(1)$, $n_1^2$, $n_2^2$, $ig(2,1),
\dots$, $ig(2, 7)$, $og(2)$, $n^3_1$, $n^3_2$, $n^3_3$, $ig(3,1),
\dots$~. Notice that to specify the numbers $n^k_i$ is equivalent, cf.
Table~\ref{T:winds}, with specifying the lengths of all winds. However, note
that the lengths of the winds in the pieces $P(k,l)$ do not depend on $l$ and so
we will, for each $k$, specify only the lengths of the winds $W^{(k,1)}_1,
\dots, W^{(k,1)}_k$, which will be equivalent with specification of the numbers
$n^k_1, \dots, n^k_k$. Thus, our task is to specify
$w_1^1$, $ig(1,1)$, $ig(1,2)$, $ig(1,3)$, $og(1)$, $w_1^2$, $w_2^2$, $ig(2,1),
\dots$, $ig(2, 7)$, $og(2)$, $w^3_1$, $w^3_2$, $w^3_3$, $ig(3,1), \dots$.
We need continuity of $T$ and $h^*(T) = \log 2$ but in spite of this we have
much freedom in specification of those numbers.

Inductively, we require only the following (for $k=1,2,\dots$):

\begin{enumerate}
	\item [(L1)] $w^1_1 =8$ or, equivalently, $n_1^1=7$.
	\item [(L2)] $|IG| \gg |\pre (IG)|$	whenever $IG$ is an inner gap.
	\item [(L3)] $|OG| \gg |\pre (OG)|$	whenever $OG$ is an outer gap.
	\item [(L4)] $|W| \gg  |\pre (W)|$	whenever $W$ is a wind in the first piece of
	a block.
	(Recall that the lengths of the winds in $P(k,l)$ are the same as those in
	$P(k,1)$.)
\end{enumerate}

The pieces and the gaps are formed by consecutive points of the trajectory
$x_0,x_1,\dots$, so we may for instance write
$P(k,l)$ as $x_j,x_{j+1},\dots, x_{j+|P(k,l)|-1}$.
Here $j$ depends on $k$ and $l$ but we do not need to know the exact formula for
it.
We simply choose separately the pieces, the inner gaps and the outer gaps, then
concatenate them as in  \eqref{Eq:block-outgap} and \eqref{Eq:pieces-ingaps},
and denote the sequence we obtain as $x_0,x_1,\dots$. So, it will be convenient
to adopt an alternative notation. Instead of ``$x$", we will use ``$y$" for
elements in pieces, ``$z$" for elements in inner gaps and ``$\omega$" for
elements in outer gaps. Moreover, the indices will indicate to which block or
gap the element belongs to. The alternative notation can be seen in the next
table.

\medskip
\begin{table}[h]
	\label{table:alternative notations}
	\begin{footnotesize}
		\begin{center}
			
			\begin{tabular}{|c|c|c|}
				\hline
				$P(1,1)$ & $IG(1,1)$ & $P(1,2)$  \\
				\hline
				$x_0,x_1,\dots,x_{7}$ &$x_8,x_9,\dots,x_{m_1-1}$ &
				$x_{m_1},x_{m_1+1},\dots,x_{m_1+7}$  \\
				\hline
				$y_{(1,1),0},y_{(1,1),1},\dots,y_{(1,1),n_1^1}$ &
				$z_{(1,1),1},z_{(1,1),2},\dots,z_{(1,1),ig(1,1)}$ &
				$y_{(1,2),0},y_{(1,2),1},\dots,y_{(1,2),n_1^1}$  \\
				\hline
				\hline
				$IG(1,2)$ & $P(1,3)$ & $IG(1,3)$  \\
				\hline
				$x_{m_1+8},x_{m_1+8},\dots,x_{m_2-1}$
				&$x_{m_2},x_{m_2+1},\dots,x_{m_2+7}$ &
				$x_{m_2+8},x_{m_2+9},\dots,x_{m_3-1}$   \\
				\hline
				$z_{(1,2),1},z_{(1,2),2},\dots,z_{(1,2),ig(1,2)}$ &
				$y_{(1,3),0},y_{(1,3),1},\dots,y_{(1,3),n_1^1}$
				& $z_{(1,3),1},z_{(1,3),2},\dots,z_{(1,3),ig(1,3)}$ \\
				\hline
				\hline
				$P(1,4)$ & $OG(1)$ & \dots\\
				\hline
				$x_{m_3},x_{m_3+1},\dots,x_{m_3+7}$ &
				$x_{m_3+8},x_{m_3+9},\dots,x_{m_3+7+og(1)}$ & \dots \\
				\hline
				$y_{(1,4),0},y_{(1,4),1},\dots,y_{(1,4),n_1^1}$ &
				$\omega_{1,1},\omega_{1,2},\dots,\omega_{1,og(1)}$& \dots \\
				\hline
			\end{tabular}
		\end{center}
	\end{footnotesize}
	\caption{Alternative notation}
\end{table}

So, $z_{(k,l),t}$ is the $t$-th point in the $l$-th inner gap of the $k$-th
block and $\omega_{k,t}$ is the $t$-th point in the $k$-th outer gap. In both
cases $t$ starts to run from $1$. In contrast, the piece $P(k,l)$ with length
$n_k^k+1$ has its elements denoted by $y_{(k,l),t}$ where $t$ runs from $0$ to
$n^k_k$. Then~(\ref{eq: pieces requirment}) becomes
\begin{equation}\label{eq: pieces req alt}
y_{(k,l),0}\in U^{k}(a_{s_l(0)}) \text{ \ \ and \ \ }
T^{n_i^k}(y_{(k,l),0})=y_{(k,l),n_i^k} \in U^{k}(a_{s_l(i)}), \, \, 1 \le i \le k.
\end{equation}

Recall that all the elements of pieces and gaps are in $(0,1] \times A$ and so
it will be convenient to adopt the convention (in fact used already above, say in
Figures~\ref{fig:p11}-\ref{fig:og1}) that
\begin{equation}\label{eg:pt-centre}
b \in U^{k}(a_j) \, \, \text{ means } \, \, b \in U^{k}(a_j) \text{ \ and \  }
P_2(b)=a_j.
\end{equation}

Recall that $P(k,l)$ (corresponding to $s_l \in \{0,1\}^{\{0,1,2,\cdots,k\}}$)
will wind $k$-times around $A$, i.e. it will `jump' from the `right side' of
$a_{\infty}$ to the `left side' of $a_{\infty}$ $k$-times. As we mentioned
above, when $k\to \infty$, the jumps will be performed closer and closer to
$a_{\infty}$ to ensure the continuity of $T$. For instance, in $P(1,2)$, see
Figure~\ref{fig:p12}), the jump is from  $U^1(a_{3})$ to $U^1(a_{-2})$ and we
put $j^{(1,2)}_1=3$ and $p^{(1,2)}_1=2$.
In general, given $k$ and $l$, the piece $P(k,l)$ performs $k$ jumps and so we
will have $k$ pairs of positive \emph{jump numbers} \index{jump numbers}, namely
$j^{(k,l)}_q, p^{(k,l)}_q$ for $q=1,\dots,k$. The meaning is that
\begin{equation}\label{eq:qth jump}
\text{the $q$-th jump in $P(k,l)$ is performed from $U^k(a_{j^{(k,l)}_q})$ to
	$U^k(a_{-p^{(k,l)}_q})$},
\end{equation}
i.e., this jump starts at a point $x_m$ and ends at a point $x_{m+1}$, with
$P_2(x_m)=a_{j^{(k,l)}_q}$ and $P_2(x_{m+1}) = a_{-p^{(k,l)}_q}$.
In Figure~\ref{fig:p26} we see the jump numbers $j^{(2,6)}_1 = 4,
p^{(2,6)}_1 =5$ and $j^{(2,6)}_2 = 6, p^{(2,6)}_2 = 6$.

The notion of jump numbers can analogously be introduced also for inner and
outer gaps, but we will not need a notation for them.
Notice that in all the above pictures of the inner gaps and outer gaps the jumps
from the `right' to the `left' are `horizontal', meaning that they are performed
from $U^k(a_r)$ to $U^k(a_{-r})$ (the both jump  numbers are the same, equal to
some $r$), and in all the above pictures of pieces we have `horizontal' or
`almost horizontal' jumps, meaning that $|j^{(k,l)}_q - p^{(k,l)}_q| \leq 1$.
Similarly as in the pictures, after fixing the lengths of all gaps and winds,
see (L1-4), the jumps can be chosen in such a way that all the jumps (in pieces
and in inner and outer gaps) are almost horizontal.

Now take into account that, by (L1-4), the lengths of the winds in the pieces of
a block $B(k)$ tend to infinity when $k\to \infty$. Since all the jumps are
almost horizontal, this clearly implies that the jump numbers $j^{(k,l)}_q$ and
$p^{(k,l)}_q$, for $1 \le l \le k$, tend to $\infty$ when $k\to \infty$.
Similarly, by (L1-4), the lengths $ig(k,l)$ of inner gaps as well as the lengths
$og(k,l)$ of outer gaps tend to $\infty$ when $k\to \infty$. Since the gaps wind
around $A$ only once and the jumps in the gaps are almost horizontal, these
jumps are performed closer and closer to $a_{\infty}$ when $k \to \infty$. It
follows that $T$ is continuous.

\medskip

\subsection{Proof of $h^*(T)=\log 2$}
We have finished the construction of $(X_1, T)$ and we already know that
$(X_1, T)$ fulfills (R1). It remains to show that also (R2) is fulfilled.

\medskip

We embark on a complicated proof that (R2) is fulfilled. To find an idea, suppose, on the contrary,
that $l_1<l_2<l_3<l_3<l_5$ is an independence set of times of length $5$ for $(U^{1}(a_0),U^{1}(a_j))$.
What does it mean for $\{l_1,l_2,l_3,l_4,l_5\}$? Among others, there exists a point $z\in \{x_0,x_1,\dots\}$ such that
\[
T^{l_i}z \in U^{1}(a_0), \quad i=1,2,3,4,5~.
\]
However, below we show that there are severe restrictions for the differences
$l_{i+1}-l_i$, $i=1,2,3,4$. These are iterative distances, with respect to $T$,
between the above points (the points in the snake are denoted as
$x_i=T^i(x_0)$ and then the \emph{iterative distance} \index{iterative distance} between $x_s$ and
$x_{s+t}$, $t\geq 0$, is defined to be $t$).

\medskip

Thus, we are going to investigate the iterative distances between elements of the snake
lying in $U^{1}(a_0)$. Notice the following. It follows from the construction of
$x_0,x_1,\dots$, cf. Figures~\ref{fig:p11}-\ref{fig:p26},
that if $x_i \in U^{1}(a_0)$, then either $x_i$ belongs to some piece $P(k,l)$
or, if it belongs to a gap, it is the last point of an inner gap and the next
piece starts in $U^{1}(a_1)$.
It cannot be the last point of an outer gap because $P(k,1)$ always starts in
$U^{1}(a_0)$ (in fact, $s_1 = (0,\dots)$ and not $(1,\dots)$,
see~(\ref{eq:notFk})).
Said in a different way, by the construction the outer gaps do not contain
points from $U^{1}(a_0)$.  Therefore it will be convenient to introduce the
following notion.

For any $k \in \N$ and $1 \le l \le 2^{k+1}$, we denote by $\tilde{P}(k,l)$ \index{$\tilde{P}(k,l)$} the
\emph{$l$-th part in the $k$-th level} \index{$l$-th part in the $k$-th level} defined as follows:
\begin{itemize}
	\item If $y_{(k,l),0} \in U^{1}(a_0)$, then set $\tilde{P}(k,l)=P(k,l) $;
	\item If $y_{(k,l),0} \in U^{1}(a_1)$, then set $\tilde{P}(k,l)=P(k,l) \cup
	\{x_{j-1}\}$ where $x_{j} =y_{(k,l),0}$.
\end{itemize}
The point $x_{j-1}$ with  $x_{j} =y_{(k,l),0}$ does not exist if $k=l=1$ (i.e.
if $j=0$). However, by the construction, $y_{(1,1),0} \in U^{1}(a_0)$ and so
$\tilde{P}(1,1)=P(1,1)$ is defined. More generally, since $P(k,1)$ always starts
in $U^{1}(a_0)$, we have
\begin{equation}\label{eq:tildePk1=Pk1}
\tilde{P}(k,1)=P(k,1), \quad k=1,2,\dots .
\end{equation}
As already shown above,
\begin{equation}\label{eq:no0outsidepart}
\text{if $x_i \in U^{1}(a_0)$, then $x_i \in \tilde{P}(k,l)$ for some $k$ and
	$l$.}
\end{equation}

\begin{lem}\label{L: distance of two 0}
	Let $k>0,l\in \{1,\dots, 2^{k+1}\}$.
	\begin{enumerate}
		\item The piece $P(k,l)$ is of the form
			$$
			P(k,l) = x_j, x_{j+1}, \dots, x_{j+n_1^k},  \dots, x_{j+n_2^k},  \dots,
			x_{j+n_k^k}~.
			$$
		The list of all points from $\tilde{P}(k,l)$ which belong to
		$U^{1}(a_0)$ is then
		\begin{equation}\label{eq:list}
		\text{one of $x_{j+n_0^k}, x_{j+n_0^k -1}$,\,  one of $x_{j+n_1^k},
			x_{j+n_1^k-1}, \dots$, \,  one of  $x_{j+n_k^k},  x_{j+n_k^k-1}$}
		\end{equation}
		(here $n_0^k=0$ and ``one" means ``exactly one"). If $l=1$ then in fact $x_{j+n_0^k}=x_j$ belongs to $U^{1}(a_0)$. If $s_l\in F(k)$ is the function corresponding to $P(k,l)$, then we can write
		\begin{equation}\label{eq:listnew}
		\tilde{P}(k,l) \cap U^{1}(a_0) =\{x_{j+n_0^k-s_l(0)}, x_{j+n_1^k -s_l(1)}, \dots, x_{j+n_k^k-s_l(k)}\}.		\end{equation}
		\item If two points in $\tilde{P}(k,l)\cap U^{1}(a_0)$ have iterative distance
		$t>0$, then
		\begin{equation*}
		t \in \{n_c^k-n_d^k-1, \, n_c^k-n_d^k, \, n_c^k-n_d^k+1\} \quad \text{for
			some} \quad 0\le d<c \le k
		\end{equation*}
		and no other pair of points in $\tilde{P}(k,l)\cap U^{1}(a_0)$ has the same iterative
		distance $t$.
	\end{enumerate}
\end{lem}

\begin{proof}
	(1) As in~(\ref{eq:Pkl}), we have $P(k,l) = x_j, x_{j+1}, \dots, x_{j+n_1^k},  \dots, x_{j+n_2^k},  \dots,
	x_{j+n_k^k}.$
	It immediately follows from the construction of the winds in $P(k,l)$ that the
	list of all points from $\tilde{P}(k,l)$ which belong to
	$U^{1}(a_0)$ is~(\ref{eq:list}). If $l=1$ then~(\ref{eq:no0outsidepart}) and~(\ref{eq:tildePk1=Pk1})
	exclude the point $x_{j+n_0^k -1}$ from the list. To be more precise, for each $c\in \{0,1,\dots, k\}$
	we have $x_{j+n^k_c} \in U^1(a_{s_l(c)})$ and so
	$$
	x_{j+n^k_c-s_l(c)} \in U^1(a_0).
	$$
	This means that the list of all elements of $\tilde{P}(k,l) \cap U^{1}(a_0)$ is~(\ref{eq:listnew}).

	(2) So, if $t>0$ and $x_s, x_{s+t}$ are two points in $\tilde{P}(k,l)\cap
	U^{1}(a_0)$, there are	
	$0\le d<c \le k$ such that $x_s$ is one of $x_{j+n_d^k}, x_{j+n_d^k-1}$ and
	$x_{s+t}$ is one of $x_{j+n_c^k}, x_{j+n_c^k-1}$.
	Hence the three possibilities for $t$.
	
	Thus, in each of the pairs $x_{j+n_d^k}, x_{j+n_d^k-1}$ and $x_{j+n_c^k},
	x_{j+n_c^k-1}$, \emph{exactly} one of the points is in $U^1(a_0)$ and their
	iterative distance is $t$. Therefore, if $p<q$ and $x_p, x_q$ is \emph{another}
	pair of points in $\tilde{P}(k,l)\cap U^{1}(a_0)$, i.e. in the
	list~(\ref{eq:list}), then either
	$x_p \notin \{x_{j+n_d^k}, x_{j+n_d^k-1}\}$ or $x_{q} \notin \{x_{j+n_c^k},
	x_{j+n_c^k-1}\}$. Since by (L4)
	the lengths of the winds in $P(k,l)$ satisfy the inequalities
	$$
	n^k_1+1 \ll n^k_2 -n^k_1 +1 \ll \dots \ll n^k_k-n^k_{k-1}+1~,
	$$
	the iterative distance $q-p$ of $x_q$ and $x_p$ is different from $t$.
\end{proof}

\medskip

Notice that, in notation from Lemma~\ref{L: distance of two 0},
$$
E(k,l) = \{x_{j}, x_{j+n_1^k},  \dots, x_{j+n_2^k}\}
$$\index{$E(k,l)$}
is the set of the \emph{endpoints of the winds} in $P(k,l)$.
By Lemma~\ref{L: distance of two 0}(1), the points of $U^1(a_0) \cap \tilde{P}(k,l)$ \emph{almost coincide} \index{almost coincide} with the endpoints of the winds.
By saying that two points almost coincide, we mean that their iterative distance is at most one.

Due to the construction, the independence is `caused' by the trajectory of
$x_0$,
the points from the head $A$ are not those points which visit a tuple of
neighbourhoods
in prescribed times.
Assume again that $\{l_1,l_2,l_3\}$, with $l_1<l_2<l_3$, is an independence set
of times of lengths $3$
for $(U^{1}(a_0),U^{1}(a_j))$.
Then we will have $2^3$ different segments of trajectory with lengths
$l_3-l_1+1$
and corresponding to different choices of functions saying which of the sets
$U^{1}(a_0)$, $U^{1}(a_j)$
are to be visited in the times $l_1,l_2,l_3$ by a point appropriate for that
choice.
Since these segments have the same length,
each one of them can be viewed as a `shift' (along the trajectory $x_0,
x_1,\dots$) of any other one.

Now it will be useful to think of the trajectory $x_0, x_1,\dots$ as a sequence of points going from left to right, along the real line,
with the distance $1$ between every two consecutive elements of the trajectory. The reason is that then for $t>0$ the iterative distance
between $x_s$ and $x_{s+t}$, which is $t$, is the same as the euclidean distance between them. Say, then a gap or a block
is long if and only if it is long also in the sense of the euclidean metric on the real line. Another advantage is that we can
speak on inner and outer gaps to the right of some block, or on all those elements of the trajectory which lie
in $U^1(a_0)$ and are to the left of some element, and the like. Say, $\pre (\mathcal S)$ is the set of all elements which are to the left of $\mathcal S$ (meaning, of course, to the left of
every element of $\mathcal S$).

Notice that if $x_m \in U^1(a_0)$ then it has its right neighbour (immediate successor) in $U^1(a_0)$ and if $m>0$ then it has also its left neighbour (immediate predecessor) in $U^1(a_0)$. If two points $x_i, x_j$ belong to a finite subsequence $\mathcal F$ of the trajectory $x_0, x_1, \dots$, we say
that the pair $x_i, x_j$ lies in $\mathcal F$ (instead of saying that the pair $(x_i,x_j)$ is an element of some cartesian product).

Let $x_i=T^i(x_0)$ and  $x_{j}=T^{j}(x_0)$, $i\neq j$, be two different elements of the trajectory $x_0, x_1, \dots$.
Then we say that the pair $x_i, x_j$ is \emph{$U^1(a_0)$-shiftable} \index{shiftable}, or just that the two points are $U^1(a_0)$-shiftable,
if both $x_i$ and $x_j$ belong to $U^1(a_0)$ and there is $m\neq 0$ such that also both $x_{i+m}$ and $x_{j+m}$ belong
to $U^1(a_0)$. More precisely, if this is true for some $m<0$ or $m>0$ we say that the pair $x_i, x_j$ is \emph{$U^1(a_0)$-left shiftable} \index{left shiftable}
or \emph{$U^1(a_0)$-right shiftable} \index{right shiftable}, respectively. The function sending each point $x_k$ to the point $x_{k+m}$ is called the \emph{shift by $m$} \index{shift by $m$}(right shift or left shift, depending on whether $m>0$ or $m<0$, respectively).

We are going to study the space $\{x_0,x_1,\dots\}$ from the point of view of shiftability.

\begin{lem}\label{L:notL-shiftable}
		Let $s \ge 0$, $t>0$ and $x_s, x_{s+t} \in U^1(a_0)$. If the points $x_s$ and $x_{s+t}$ belong to different blocks,
		then they are not $U^1(a_0)$-left shiftable.
\end{lem}

\begin{proof}
Let $x_s \in B(p)$ and $x_{s+t} \in B(q)$ for some $p\neq q$. Of course, $p < q$ and so at least the whole outer gap $OG(q-1)$ lies between $x_s$ and $x_{s+t}$. 	Suppose on the contrary that there is $m< 0$ such that also $x_{s+m}, x_{s+t+m} \in U_1(a_0)$. This in particular means that $s+m \geq 0$.
Let $q_{-}>0$ be the iterative distance of $x_{s+t}$ from its left neighbour in $U^1(a_0)$, i.e. the least positive integer such that
$x_{s+t-q_{-}} \in U^{1}(a_0)$. By (L4) (no matter whether $x_{s+t}$ is the first point of $B(q)$ or not) and (L3) we know that, respectively,
$$
q_{-}>og(q-1) \qquad \text{and} \qquad og(q-1) > s~.
$$
Since $x_{s+t+m} \in U_1(a_0)$,  by the definition of $q_{-}$ we then have
$|m|=-m \geq q_{-}>og(q-1) >s$. Hence $s+m <0$, a contradiction. 	
\end{proof}

\begin{lem}\label{L:notL-shiftable2}
Let $s \ge 0$, $t>0$ and $x_s, x_{s+t} \in U^1(a_0) \cap \tilde{P}(k,l)$. Then there is no $m<0$ such that $x_{s+m}\in U^{1}(a_0) \cap \bigcup_{i=1}^{k-1}B(i)$ and $x_{s+t+m} \in U^{1}(a_0) \cap P(k,1)$.	
\end{lem}

\begin{proof}
	Let $k\geq 2$, otherwise there is nothing to prove.
	Suppose, on the contrary, that there is $m<0$ with that property (hence $s+m\geq 0$).
	We have, as in~(\ref{eq:Pkl}),
	$$
	P(k,l) = x_{j(l)}, x_{j(l)+1}, \dots, x_{j(l)+n_1^k},  \dots, x_{j(l)+n_2^k},  \dots, x_{j(l)+n_k^k},
	$$
	where the set $E(k,l) = \{x_{j(l)}, x_{j(l)+n_1^k},  \dots, x_{j(l)+n_k^k}\}$ is the set of the \emph{endpoints of the winds} in $P(k,l)$. Similarly,
	$$
	P(k,1) = x_{j(1)}, x_{j(1)+1}, \dots, x_{j(1)+n_1^k},  \dots, x_{j(1)+n_2^k},  \dots,	x_{j(1)+n_k^k},
	$$
	where $E(k,1) =\{x_{j(1)}, x_{j(1)+n_1^k},  \dots, x_{j(1)+n_k^k}\}$ is the set of the endpoints of the winds in $P(k,1)$.
	
	Since $x_s, x_{s+t} \in U^1(a_0) \cap\tilde{P}(k,l)$, by Lemma~\ref{L: distance of two 0}(1) we know that they almost coincide with the $p$-th and the $q$-th elements in $E(k,l)$, for some $p<q$ (recall that by saying that two points almost coincide we mean that their iterative distance is at most one).  The left shift from $x_s, x_{s+t}$ to $x_{s+m}, x_{s+t+m}$ can be performed as the composition of two shorter left shifts. First, we shift $x_s, x_{s+t}$ to points $x_{s+\sigma}, x_{s+t+\sigma}$ which almost coincide with the $p$-th and the $q$-th elements in $E(k,1)$ (this is possible because the lengths of winds in $P(k,l)$ are the same as in $P(k,1)$). Then the point $x_{s+\sigma}$ is either in $P(k,1)$ or it is the last point of $OG(k-1)$. So we need to shift $x_{s+\sigma}, x_{s+t+\sigma}$ still to the left, now finally to $x_{s+m}, x_{s+t+m}$. Since $OG(k-1)$ does not contain points from $U^1(a_0)$, this shift has to be at least as long as it is the length of $OG(k-1)$, which is much larger than $1$. Since $x_{s+t+m}$ has to be in $U^1(a_0)$ and to the left of $x_{s+t+\sigma}$, this second shift (whose length is much larger than $1$) is of course at least as long as the iterative distance between $(q-1)$-st and $q$-th elements in $E(k,1)$ (see Lemma~\ref{L: distance of two 0}(1)), meant in approximative sense, i.e. an error, now definitely not greater than $2$, is possible when we claim this. This iterative distance is, by (L4), much larger than $|\pre (W)|$ where $W$ is the wind whose endpoints are the $(q-1)$-st and $q$-th elements in $E(k,1)$. Since $x_{s+\sigma}$ almost coincides with the $p$-th element of $E(k,1)$ and $p\leq q-1$, we get that $s+m<0$, a contradiction.
\end{proof}

\begin{lem}\label{L:location}
	Let $s \ge 0$, $t>0$ and let the points $x_s, x_{s+t} \in U^1(a_0)$ be
	$U^1(a_0)$-shiftable, i.e. there exists an $m \neq 0$  such that also
	$x_{s+m} \in U^{1}(a_0)$ and $x_{s+t+m} \in U^{1}(a_0)$. Then the following is true.
	\begin{enumerate}
		\item If $x_s, x_{s+t} \in \tilde{P}(k,i_1)$ for some $k$ and $i_1$, then $x_{s+m}, x_{s+t+m} \in \tilde{P}(k,i_2)$ for some $i_2\neq i_1$.
		\item If $x_s \in \tilde{P}(k,i_1)$, $x_{s+t} \in \tilde{P}(k,i_2)$
		for some $k$ and $i_1<i_2$, then
		$x_{s+m} \in \tilde{P}(k,i_1)$, $x_{s+t+m}\in \tilde{P}(k,i_2)$.
		\item The points $x_s, x_{s+t}, x_{s+m}, x_{s+t+m}$ belong to the same block $B(k)$, for some $k$.
		\item If  $x_s, x_{s+t} \in \tilde{P}(k,i)$ for some $k$ and $i$,
		then
		$x_s, x_{s+m}$ are in the ``similar positions" \index{similar positions}, meaning that if we write, as in~(\ref{eq:listnew}), the point $x_s$ in the form
		$$
		x_s= x_{r+n^k_{c}-s_{i}(c)} \in \tilde{P}(k,i) \quad \text{for some} \quad 0 \le c \le k,
		$$
		then there exists $i'$ such that
		$$
		x_{s+m} =x_{r'+n^k_{c}-s_{i'}(c)} \in \tilde{P}(k,i') \quad \text{with the same} \quad 0 \le c \le k.
		$$
		Here $r,r' \ge 0$ are such that $x_{r}=y_{(k,i),0}$,
		$x_{r'}=y_{(k,i'),0}$ are the first points of the pieces $P(k,i), P(k,i')$, respectively, and $s_{i}, s_{i'} \in F(k)$ are
		the functions corresponding to the pieces $P(k,i), P(k,i')$, respectively.
		\item If $x_s \in \tilde{P}(k,i_1)$ and $x_{s+t} \in \tilde{P}(k,i_2)$
		for some $k$ and $i_1<i_2$,
		then $x_s, x_{s+t}$ are in the ``similar positions".
	\end{enumerate}
\end{lem}

\begin{proof}
By~(\ref{eq:no0outsidepart}) and~(\ref{eq:tildePk1=Pk1}), the points from $U^1(a_0)$ belong to $\bigcup_{k=1}^\infty B(k)$. Moreover, since they belong to the union of all parts, Lemma~\ref{L: distance of two 0}(1) shows their location more precisely. Namely, each of them almost coincides with an endpoint of a wind.

	$(1)$ Since $x_s, x_{s+t} \in \tilde{P}(k,i_1)$, we have $t < |\tilde{P}(k,i_1)|$. On the other hand, by combining (L2) and (L4), each inner gap in $B(k)$ is much longer than $|\tilde{P}(k,i_1)|$, hence much longer than $t$. Therefore
	\begin{equation}\label{eq:not diff parts}
	x_{s+m}, x_{s+t+m} \quad \text{cannot lie in different parts of} \quad B(k)~.
	\end{equation}
	Further, we claim that, for $m\neq 0$,
	\begin{equation}\label{eq:not i1}
	x_{s+m}, x_{s+t+m} \quad \text{cannot both lie in} \quad \tilde{P}(k,i_1)~.
	\end{equation}
	This is because by Lemma~\ref{L: distance of two 0}(2), different pairs of point in $\tilde{P}(k,i_1) \cap U^1(a_0)$ have different iterative distances.
	
	By (L2-4), all (inner and outer) gaps which are to the right of $x_{s+t}$ are much longer than $t$. So, if $m>0$ then the points $x_{s+m}, x_{s+t+m}$ cannot be separated by such a gap, i.e. they are in the same part. This cannot be a part to the right of $B(k)$ because, by (L4), every wind in such a part is much longer than $t$ and so the iterative distance of $x_{s+m}, x_{s+t+m}$ would be longer than $t$, a contradiction. We conclude that
	\begin{equation}\label{eq:in same part}
	\text{if $m>0$\, then \, $x_{s+m}, x_{s+t+m} \in \tilde{P}(k,i_2)$ \, for some \, $i_1 < i_2$}
	\end{equation}
	(we have excluded the equality $i_1=i_2$ by~(\ref{eq:not i1})).
	
	Now let $m<0$. Since $x_s, x_{s+t} \in \tilde{P}(k,i_1)$, $t$ is at least as large as the length of the shortest wind in $P(k,i_1)$. Hence, by (L4), $t \gg |\pre (B(k))|$. Therefore, since the iterative distance between $x_{s+m}$ and $x_{s+t+m}$ is $t$, the point $x_{s+t+m}$ cannot be in $\pre (B(k))$. Using this fact as well as~(\ref{eq:not diff parts}) and ~(\ref{eq:not i1}), we get two cases:
	\begin{enumerate}
		\item [(i)]  $x_{s+m}\in \bigcup_{i=1}^{k-1}B(i)$ and $x_{s+t+m} \in \tilde{P}(k,i_2)$ for some $i_2\leq i_1$, or
		\item [(ii)]  $x_{s+m}, x_{s+t+m} \in \tilde{P}(k,i_2)$ for some $i_2< i_1$.
	\end{enumerate}
	
	In the case (i) we have $i_2=1$, otherwise the inner gap $IG(k,1)$ is between $x_{s+m}$ and $x_{s+t+m}$ and so the iterative distance of these two points is $t>ig(k,1)$. However, this contradicts the facts that, by (L2), $ig(k,1)$ is much larger than $|P(k,1)|=|P(k,i_1)|$ and $\tilde{P}(k,i_1)$ contains $x_s$ and $x_{s+t}$.  So, (i) can be replaced by (note that $\tilde{P}(k,1)=P(k,1)$)
	\begin{enumerate}
		\item [(i')]  $x_{s+m}\in \bigcup_{i=1}^{k-1}B(i)$ and $x_{s+t+m} \in P(k,1)$.
	\end{enumerate}
	However, (i') is excluded by Lemma~\ref{L:notL-shiftable2}. So, only the case (ii) is possible and then we are done, just combine it with~(\ref{eq:in same part}).

	\medskip
	
	$(2)$ The inner gap $IG(k,i_2-1)$ lies between $x_s$ and $x_{s+t}$ (with possible exception of the last point of $IG(k,i_2-1)$ which may belong to $\tilde{P}(k,i_2)$ and to be equal to $x_{s+t}$). Hence $t\geq ig(k,i_2-1)$ and so, by~(L2),
	\begin{equation}\label{eq:tggP}
	t\gg |P(k,i_1)| = |P(k,l)|, \quad l=1,2,\dots, 2^{k+1}
	\end{equation}
	(all the pieces in a block have the same length).
	
	We claim that
	\begin{equation}\label{eq:2left}
	\text{$x_{s+t+m}$ \, cannot be to the left of \, $\tilde{P}(k,i_2)$}.
	\end{equation}
	Indeed, $x_{s+t+m}$ cannot be in $IG(k,i_2-1)\setminus \tilde{P}(k,i_2)$, because this set is disjoint with $U^1(a_0)$. Also, $x_{s+t+m}$ cannot be to the left of $IG(k,i_2-1)$, otherwise $|m|\geq ig(k,i_2-1)$ and since $ig(k,i_2-1) \gg |\pre(IG(k,i_2-1))|$ and $x_s \in \pre(IG(k,i_2-1))$, we would get $|m|\gg s$ whence $s+m <0$, a contradiction. 	
	
	Further, we claim that
	\begin{equation}\label{eq:2right}
	\text{$x_{s+t+m}$ \, cannot be to the right of \, $\tilde{P}(k,i_2)$}.
	\end{equation}
	We prove this. First suppose that $x_{s+t+m}\in \tilde{P}(k,i_3)$ for some $i_3>i_2$. By~(\ref{eq:tggP}), $x_{s+m}\notin \tilde{P}(k,i_3)$. Then
$t=(s+t+m)-(s+m) >ig (k,i_3-1)$.  On the other hand, since $x_s\in \tilde{P}(k,i_1)$ and $x_{s+t} \in \tilde{P}(k, i_2)$, (L2) gives $t\ll ig(k, i_3-1)$, a contradiction.

	Now suppose that $x_{s+t+m}\in B(L)$ for some $L>k$. If also $x_{s+m}\in B(L)$ then, since $x_{s+t+m}, x_{s+m} \in U^1(a_0)$, $t$ is at least as long as the first wind in $B(L)$. By (L4), and using that $x_s, x_{s+t} \in \pre (B(L))$, we get $t \gg |\pre(B(L))| \geq t$, a contradiction.
	So assume that $x_{s+m}\notin B(L)$, i.e. $x_{s+m}\in \pre (B(L))$. Since $OG(L-1)$ is disjoint with $U^1(a_0)$, the whole gap $OG(L-1)$ is between $x_{s+m}$ and $x_{s+t+m}$. Hence, using also (L3) and the fact that $x_s, x_{s+t} \in \pre (OG(L-1))$, we get $t>OG(L-1) \gg \pre(OG(L-1)) \geq t$, again a contradiction. We have proved~(\ref{eq:2right}).
	
	By~(\ref{eq:2left}) and ~(\ref{eq:2right}), $x_{s+t+m} \in \tilde{P}(k,i_2)$. Since both $x_{s+t}$ and $x_{s+t+m}$ belong to the same part $\tilde{P}(k,i_2)$, by the already proved claim (1) we get that their shifts by $-t$, i.e. the points $x_{s}$ and $x_{s+m}$,
	belong to the same part $\tilde{P}(k,l)$, for some $l\neq i_2$. However, $x_s\in \tilde{P}(k,i_1)$ by the assumption, therefore $l=i_1$ and we are done.

	\medskip
	
	$(3)$ First we  prove that the points $x_s, x_{s+t}$ are in the same block.
	
	Suppose, on the contrary, that there are $p < q$ such that $x_s \in B(p)$ and $x_{s+t} \in B(q)$. In view of Lemma~\ref{L:notL-shiftable}, these two $U^1(a_0)$-shiftable points are $U^1(a_0)$-right shiftable. So, there is $m>0$ such that also $x_{s+m}, x_{s+t+m} \in U_1(a_0)$. Hence, also these two points belong to the union of blocks. We see that they are $U^1(a_0)$-left shiftable (shift by $-m$ sends them to $x_s$ and $x_{s+t}$). So, by Lemma~\ref{L:notL-shiftable}, $x_{s+m}$ and $x_{s+t+m}$ belong to the same block $B(r)$ for some $q\leq r$.
		
		Suppose that $q<r$. Then $x_{s+m}$ and $x_{s+t+m}$, being two elements of $U^1(a_0)$ in the block $B(r)$, have their iterative distance $t$ at least as large as the length of the first wind of the first piece in $B(r)$ (see (L4)). However, by (L4), this length is larger than the length of $\pre (B(r))$. In particular, $t$ is larger than the iterative distance between $x_s, x_{s+t} \in \pre (B(r))$, a contradiction.
		
		So,  $q=r$. We have four points in $U^1(a_0)$, namely $x_s\in B(p)$, $x_{s+t}, x_{s+m}, x_{s+t+m} \in B(q)$, where $p<q$.
		Since $x_{s+m}, x_{s+t+m} \in B(q)$, then either
		$$
		x_{s+m}, \, x_{s+t+m} \in \tilde{P}(q,i)
		$$
		for some $i \in \N$ or
		$$
		x_{s+m} \in \tilde{P}(q,i_1) \quad \text{and} \quad x_{s+t+m} \in \tilde{P}(q,i_2)
		$$
		for
		some $i_1<i_2 \in \N$. Then, by applying either $(1)$ or $(2)$ to the points $x_{s+m}$ and $x_{s+t+m}$ and the shift by $-m$, we get that
		$$
		x_s= x_{(s+m)-m} \in B_q
		$$
		and this contradicts the assumption that $x_s\in B(p)$.	
		
		So, we already know that $x_s, x_{s+t}$ are in the same block $B(k)$.	They are either in the same part or in different parts of this block. In either case, by (1) or (2), the other two points $x_{s+m}$ and $x_{s+t+m}$ are also in that block $B(k)$. The proof is finished.

	\medskip
	
	$(4)$ By (1), $x_{s+m}\in \tilde{P}(k,i')$ for some $i'$. Denote the firts points of the pieces $P)(k,i)$ and $P(k,i')$ by
	$x_r=y_{(k,i),0}$ and $x_{r'}=y_{(k,i'),0}$, respectively. Since $x_s\in \tilde{P}(k,i)\cap U^1(a_0)$ and $x_{s+m}\in \tilde{P}(k,i')\cap U^1(a_0)$, we can write, as in~(\ref{eq:listnew}),
	$$
	x_{s}= x_{r+n^k_c -s_i(c)}, \quad \text{for some} \quad 0\leq c \leq k
	$$
	and
	$$
	x_{s+m}= x_{r'+n^k_{d} -s_i(d)}, \quad \text{for some} \quad 0\leq d\leq k.
	$$
	We need to show that $d=c$. Using that $x_{s+t} \in \tilde{P}(k,i)$ and, by (1), $x_{s+t+m}\in \tilde{p}(k,i')$, we also have
	$$
	x_{s+t}= x_{r+n^k_\gamma -s_{i'}(\gamma)}, \quad \text{for some} \quad 0\leq c < \gamma \leq k
	$$
	and
	$$
	x_{s+t+m}= x_{r'+n^k_{\delta} -s_{i'}(\delta)}, \quad \text{for some} \quad 0\leq d < \delta \leq k.
	$$
    So, since $t=(s+t)-s = (s+t+m)-(s+m)$, we get
    $$
    t=n^k_{\gamma}-n^k_{c}-(s_i(\gamma)-s_i(c))= n^k_{\delta}-n^k_{d}-(s_{i'}(\delta)-s_{i'}(d)).
    $$
	Recall that $n^k_{\gamma}-n^k_{c}$ and $n^k_{\delta}-n^k_{d}$ are much larger than $s_i(\gamma)-s_i(c)$ and $s_{i'}(\delta)-s_{i'}(d)$. Also, by (L4),
	in the finite sequence $n^k_0, n^k_1, \dots, n^k_k$ we have $n^k_{j+1}\gg n^k_j$. It follows that $c=d$ and $\gamma= \delta$.

	\medskip
	
	$(5)$ By $(2)$, we have $x_{s+m} \in \tilde{P}(k,i_1)$ and $x_{s+t+m} \in \tilde{P}(k,i_2)$. Then
	$$
	x_s, x_{s+m} \in \tilde{P}(k,i_1) \text{\ \ and \ \ } x_{s+t}, x_{s+m+t} \in \tilde{P}(k,i_2).
	$$
	Then by $(4)$,
	$x_s, x_{s+t}$ are in the ``similar positions".
\end{proof}

We are finally able to prove the main result of this section. As already said, it is sufficient to prove (R1) and (R2) from the beginning of this section.

\begin{thm}\label{final-thm} The system $(X_1, T)$ has the following properties.
	\begin{enumerate}
		\item $(a_0,a_1)$ is an IN-pair.
		
		\item $(a_0,a_j)$ is not an IN-pair for any $|j|\geq 2$.
		
		\item $(a_0,a_{\infty})$ is not an IN-pair.
		
	\end{enumerate}
	Hence $h^*(T)=\log 2$.
\end{thm}

\begin{proof}
	(1) This is by the construction. As already said above, see e.g.~\eqref{eq: pieces requirment}, the set $N(k)$ from~\eqref{eq:notNk} is an independence set of times of length $k+1$ for $(U^k(a_0), U^k(a_1))$. Hence (1).
	
	\medskip

	(2) It is sufficient to show that $(a_0,a_j)$ is not an IN-pair whenever $j\geq 2$ because then, for $j\geq 2$, neither $(a_0,a_{-j})$ is an IN-pair. Indeed, if for some $j\geq 2$ the pair $(a_0,a_{-j})$ were an IN-pair, then by Proposition~\ref{P}(a) also $(a_0,a_j)$ would be an IN-pair.
	
    Fix $j\geq 2$ and suppose, on the contrary, that $(a_0,a_j)$ is an IN-pair. It follows that $(U^1(a_{0}),U^1(a_{j}))$ has an independence set of times of length $5$, i.e. there are pairwise distinct positive integers $l_{-1}<l_0 < l_1 < l_2 <l_3$ such that $\{l_{-1},l_0, l_1, l_2, l_3\}$ is an independence set of times for $(U^1(a_{0}),U^1(a_{j}))$. Then, in particular, there exist pairwise distinct $m_1, m_2, m_3, m_4 \in \N$ such that (notice that in the underlined inclusions we have $a_j$ rather than $a_0$)
	\begin{equation}\label{eq:r1}
	x_{m_1+l_{-1}} \in U^1(a_{0}), \, x_{m_1+l_0} \in U^1(a_{0}), \, x_{m_1+l_1} \in U^1(a_{0}),\ x_{m_1+l_2} \in U^1(a_{0}), \ x_{m_1+l_3} \in U^1(a_{0}),
	\end{equation}
	\begin{equation}\label{eq:r2}
	x_{m_2+l_{-1}} \in U^1(a_{0}), \, x_{m_2+l_0} \in U^1(a_{0}), \, x_{m_2+l_1} \in U^1(a_{0}),\ \underline{x_{m_2+l_2} \in U^1(a_{j})}, \ x_{m_2+l_3} \in U^1(a_{0}),
	\end{equation}
	\begin{equation}\label{eq:r3}
	x_{m_3+l_{-1}} \in U^1(a_{0}), \, x_{m_3+l_0} \in U^1(a_{0}), \, \underline{x_{m_3+l_1} \in U^1(a_{j})},\ x_{m_3+l_2} \in U^1(a_{0}), \ \ x_{m_3+l_3} \in U^1(a_{0}),
	\end{equation}
	\begin{equation}\label{eq:r4}
	x_{m_4+l_{-1}} \in U^1(a_{0}), \, \underline{x_{m_4+l_0} \in U^1(a_{j})}, \, x_{m_4+l_1} \in U^1(a_{0}),\ x_{m_4+l_2} \in U^1(a_{0}), \ \ x_{m_4+l_3} \in U^1(a_{0}).
	\end{equation}

	\noindent Look at the eight points in the first and last columns. If we now take a pair (made of them) lying in one row,
	then it is $U^1(a_0)$-shiftable to a pair (made of them) lying in any other row. Then, by Lemma~\ref{L:location}$(3)$,
	there exists $k \in \N$ such that these eight points lie in $B(k)$. Hence, since $l_{-1} < l_0 < l_1 < l_2 <l_3$, in each row
	the five points are in the same block (i.e., in that block $B(k)$). So,
	$$
	\{x_{m_i+l_j}:\, 1\le i \le 4, \, -1 \le j \le 3\} \subseteq B(k).
	$$
	From now on we are interested in the nine points in the following smaller $3\times 3$ table:
		\begin{equation*}
		x_{m_1+l_1} \in U^1(a_{0}),\ x_{m_1+l_2} \in U^1(a_{0}), \  x_{m_1+l_3} \in U^1(a_{0}),
		\end{equation*}
		\begin{equation*}
		x_{m_2+l_1} \in U^1(a_{0}),\ \underline{x_{m_2+l_2} \in U^1(a_{j})}, \ x_{m_2+l_3} \in U^1(a_{0}),
		\end{equation*}
		\begin{equation*}
		\underline{x_{m_3+l_1} \in U^1(a_{j})},\ x_{m_3+l_2} \in U^1(a_{0}), \ x_{m_3+l_3} \in U^1(a_{0}).
		\end{equation*}
	We show that if two points in the first row of this smaller table are in the same part $\tilde{P}(k,i_1)$ then all the three points in this
	row are in $\tilde{P}(k,i_1)$. Suppose that this is not the case. To get a contradiction, we will use only these three points and the last three points in~\eqref{eq:r4}. All these six points are in $U^1(a_0)$. Therefore, due to this symmetry, we can assume that for instance $x_{m_1+l_1},x_{m_1+l_2} \in \tilde{P}(k,i_1)$ and $x_{m_1+l_3} \in \tilde{P}(k,i_2)$ with $i_1 \neq i_2$. Since $x_{m_1+l_1},x_{m_1+l_3}$ are $U^1(a_0)$-shiftable to $x_{m_4+l_1},x_{m_4+l_3}$, by Lemma~\ref{L:location}(2) we get $x_{m_4+l_1} \in \tilde{P}(k,i_1)$. Also,  $x_{m_1+l_1},x_{m_1+l_2}$ are $U^1(a_0)$-shiftable to $x_{m_4+l_1},x_{m_4+l_2}$ and so, by Lemma~\ref{L:location}(1) we get
	$x_{m_4+l_1} \in \tilde{P}(k,i_3)$ with $i_3 \neq i_1$, a contradiction.
	
	Therefore we have two cases: either all three points in the first row are in the same part, or they are in three different parts.
	
	\medskip
	
	\emph{Case 1: $x_{m_1+l_1}, x_{m_1+l_2},  x_{m_1+l_3}
	\in \tilde{P}(k,i)$ for some $i$.}
	
	In this case, by Lemma~\ref{L:location}$(1)$, we have
	\begin{alignat}{6}
		x_{m_1+l_1}, & \,\, x_{m_1+l_2},  & \,\, x_{m_1+l_3} & \in  & \,\, & \tilde{P}(k,i)   \label{3cases-1} \\
		x_{m_2+l_1}, &                    & \,\, x_{m_2+l_3} & \in  & \,\, & \tilde{P}(k,i')  \label{3cases-2} \\
	                 & \,\, x_{m_3+l_2},  & \,\, x_{m_3+l_3} & \in  & \,\, & \tilde{P}(k,i'') \label{3cases-3}
	\end{alignat}
	where $i,i',i''$ are pairwise different (note that all these seven points are in $U^1(a_0)$).
	According to Lemma~\ref{L:location}$(4)$,  $T^{m_2+l_1}x_0$ and
	$T^{m_1+l_1}x_0$ are in the ``similar positions", i.e.
	$$
	x_{m_1+l_1}= x_{r+n^k_{d}-s_{i}(d)} \quad \text{and} \quad x_{m_2+l_1}= x_{r'+n^k_{d}-s_{i'}(d)}
	$$
	for some $0\leq d\leq k$ (here $x_r$ and $x_{r'}$ are the first points of the pieces $P(k,i)$ and $P(k,i')$, and $s_i$ and $s_{i'}$ are the corresponding functions in $F(k)$). Further, since $x_{m_1+l_2} \in \tilde{P}(k,i)$, we have
	\begin{equation}\label{eq:m1l2}
	x_{m_1+l_2}= x_{r+n^k_{c}-s_{i}(c)}
	\end{equation}
    for some $0\leq c \leq k$ (here $d < c$ and so $n_d^k<n_c^k$, because $m_1+l_1 < m_1+l_2$, but we do not use this property). Then
	$$
	l_2-l_1 = m_1+l_2 - (m_1 +l_1) =n_c^k-n_d^k-(s_i(c)-s_i(d)).
	$$
We are interested in the point $x_{m_2+l_2}$, so let us compute
	\begin{equation}\label{lastbut1}
	\begin{split}
		m_2+l_2 & =   m_2+l_1+(l_2-l_1)= r'+n^k_{d}-s_{i'}(d) + (l_2-l_1)  \\
		        & =   r'+n^k_c-s_{i'}(d)-(s_i(c)-s_i(d))  \\
		        & =   r'+n^k_c-s_{i'}(c)+(s_{i'}(c)-s_{i'}(d))-(s_i(c)-s_i(d))   \\
		        & \in \{r'+n^k_c-s_{i'}(c)+t: -2\le t \le 2\}.
	\end{split}
	\end{equation}
	Note that $x_{r'+n^k_c} \in U^{1}(a_{s_{i'}(c)})$, whence
	$x_{r'+n^k_c-s_{i'}(c)} \in U^{1}(a_0)$. Therefore
	\begin{equation}\label{eq:|j|>2 not}
	x_{m_2+l_2} \in U^1(a_{-2})\cup U^1(a_{-1})\cup U^1(a_{0})\cup
	U^1(a_{1})\cup U^1(a_{2}).
	\end{equation}
	On the other hand, by ~\eqref{eq:r2}, $x_{m_2+l_2} \in U^1(a_j)$.
	So, when $j \ge 3$, we have a contradiction.
	
	If $j=2$, $x_{m_2+l_2} \in U^1(a_j)$ is specified as $x_{m_2+l_2} \in U^1(a_2)$. Then, using~\eqref{lastbut1} and the fact that $x_{r'+n^k_c-s_{i'}(c)} \in U^1(a_0)$, we get
	$(s_{i'}(c)-s_{i'}(d))-(s_i(c)-s_i(d))=2$ and thus
	\begin{equation}\label{eq:l2-l1}
	l_2-l_1=n_c^k-n_d^k-(s_i(c)-s_i(d))=n_c^k-n_d^k+1.
	\end{equation}
	Let $r''$ be the first point of $P(k,i'')$ and let $s_{i''}$ be the function from $F(k)$ which corresponds to $P(k,i'')$.
	Then
	$$
	x_{r''+n^k_u-s_{i''}(u)} \in U^{1}(a_0) \text{ \ \ for \ \ }  0 \le u 	\le k.
	$$
	Recall that, by~\eqref{eq:r1} and ~\eqref{eq:r3},
	$$
	x_{m_1+l_2},  x_{m_1+l_3} \in  U^{1}(a_0) \text{ \ \ and \ \ }
	x_{m_3+l_2},  x_{m_3+l_3} \in U^{1}(a_0).
	$$
	Therefore, by Lemma~\ref{L:location}$(4)$, $x_{m_1+l_2}$ and
	$x_{m_3+l_2}$ are in the ``similar positions". In view of ~\eqref{eq:m1l2},
	$
	x_{m_3+l_2}=x_{r''+n_c^k-s_{i''}(c)}.
	$
	Using this and~\eqref{eq:l2-l1} and taking into account that the values of $s_{i''}$ are just $0$ and $1$, we have
	\begin{equation*}
	\begin{split}
	m_3+l_1&=m_3+l_2-(l_2-l_1)\\
	&=r''+n_c^{k}-s_{i''}(c)-(n_c^k-n_d^k+1)\\
	&=r''+n_d^{k}-s_{i''}(d)+(s_{i''}(d)-s_{i''}(c))-1\\
	&\in \{r''+n^k_d-s_{i''}(d)+t: -2\le t \le 0\}.
	\end{split}
	\end{equation*}
	This implies that
	$$x_{m_3+l_1} \in U^{1}(a_{-2})\cup U^{1}(a_{-1})\cup U^{1}(a_0),$$
	which contradicts that, by~\eqref{eq:r3}, $x_{m_3+l_1} \in U^{1}(a_2)$.
	
	\medskip
	
	\emph{Case 2: $x_{m_1+l_t} \in \tilde{P}(k,i_t)$,  $t=1,2,3$, for pairwise different $i_1, i_2, i_3$.}
	
	In this case, using  Lemma~\ref{L:location}$(2)$, we have
	\begin{alignat}{6}
		x_{l_1+m_1}, & \,\, x_{l_1+m_2}  &                      & \in  & \,\, & \tilde{P}(k,i_1)  \label{3cases-1-1} \\
		x_{l_2+m_1}, &                   & \,\, x_{l_2+m_3}     & \in  & \,\, & \tilde{P}(k,i_2)  \label{3cases-2-2} \\
		x_{l_3+m_1}, & \,\, x_{l_3+m_2}, & \,\, x_{l_3+m_3}     & \in  & \,\, & \tilde{P}(k,i_3)  \label{3cases-3-3}
	\end{alignat}
		
	We have received a system of inclusions very similar to that from Case 1, see~\eqref{3cases-1}-\eqref{3cases-3}. Recall that in Case 1 we have obtained a contradiction by considering two $U^1(a_0)$-shifts. Namely, we shifted $x_{m_1+l_1}, x_{m_1+l_3}$ to $x_{m_2+l_1}, x_{m_2+l_3}$, and
	$x_{m_1+l_2}, x_{m_1+l_3}$ to $x_{m_3+l_2}, x_{m_3+l_3}$. Then we were looking for the possible positions of the points $x_{m_2+l_2}$ and $x_{m_3+l_1}$ (i.e. the points missing in ~\eqref{3cases-1}-\eqref{3cases-3}). This led to contradictions.
	To finish the proof in Case 2, it is sufficient to proceed analogously as in Case 1, but now we consider the following
	$U^1(a_0)$-shifts: the shift of $x_{l_3+m_1}, x_{l_3+m_3}$ to $x_{l_2+m_1}, x_{l_2+m_3}$  and then the shift of
	$x_{l_3+m_1}, x_{l_3+m_2}$ to $x_{l_1+m_1}, x_{l_1+m_2}$. Then contradictions will be obtained by looking for positions of the points
	$x_{m_2+l_2}$ and $x_{m_3+l_1}$. We leave the details to the reader.
	
   	\medskip

	(3) The proof is completely analogous to that of (ii). Replacing always $U^1(a_j)$ by $U^1(a_{\infty})$, but otherwise repeating the proof word by word, one again gets (see~\eqref{eq:|j|>2 not})
	\[
	x_{m_2+l_2} \in U^1(a_{-2})\cup U^1(a_{-1})\cup U^1(a_{0})\cup U^1(a_{1})\cup U^1(a_{2})
	\]
	and simultaneously $x_{m_2+l_2} \in U^1(a_{\infty})$. In view of~\eqref{eq:U_0U_infty}, this already immediately gives a contradiction and finishes the proof of Case 1, so the proof is even easier than in (2). Analogously in Case 2.
\end{proof}

\begin{rem}\label{R:SX1logN}
Theorem~\ref{final-thm} gives a space $X_1$ with $S(X_1) \supseteq \{0,\log 2\}$. Since $X_1$ is a zero-dimensional space with infinite derived set, we in fact have $S(X_1)=\log \N^{*}$, see Table~\ref{T:knownSX}.
\end{rem}


\section{A continuum $X$ with $S(X)=\{0,\log 2\}$}\label{S:cont zero-log2}
In this section, we will construct a continuum $X$ with $S(X)=\{0,\log 2\}$. The main idea of the construction is similar to that of the construction of a continuum $X$ with $S(X)=\{0,\infty\}$ in Section~\ref{S:X}. However, now the construction is more subtle and
the technical details are more complicated.

\subsection{Outline of the construction of $X$}\label{SS:outline 0-log2}

We start with the space $X_1\subseteq \mathbb R^3$ defined in
Section~\ref{S:X1-T1-log2}. So,
\begin{equation}\label{Eq:X1log2}
	X_1 = A \sqcup \{x_0,x_1,\dots\}
\end{equation}
with the set $A=\{a_i:
i\in \Z\}\cup \{a_{\infty}\}$ lying in the circle $\mathbb S^1$ in the vertical plane $\pi_0$ and
$\{x_0,x_1,\dots\} \subseteq  (0,1] \times A$. Recall that we have also
constructed a continuous map
\begin{equation}\label{Eq:T1log2}
	T: X_1 \to X_1 \quad \text{with} \quad h^*(T)=\log 2~.
\end{equation}

We use the same tools as in Section~\ref{S:X}. For terminology, the reader is referred to that section.

While the head in Section~\ref{S:X} (i.e. in the case $S(X)=\{0,\infty\}$) was just one planar Cook continuum $\mathscr K_0$, now (i.e. in the case $S(X)=\{0,\log 2\}$) the head is more complicated. We denote it by $\mathscr A_0$ and we construct it by joining the consecutive points of $A$ by some Cook continua as follows.

We may assume that $a_0$ or $a_{\infty}$ is the north or south pole of  $\mathbb S^1$, respectively, and that the points $a_k$ and $a_{-k}$ are symmetric with respect to the vertical diameter of $\mathbb S^1$, see Figure~\ref{fig:(A,T_1)}. By the arc $\wideparen{a_k a_{k+1}}$  of the circle $\mathbb S^1$ we mean that arc of the circle, which has endpoints $a_k$ and $a_{k+1}$ and does not contain any other point of $A$. Let the lengths of these circle arcs be decreasing for $k=0,1,2, \dots$ and, due to symmetry, also for $k=-1,-2,\dots$. Thus, $\wideparen{a_0 a_{1}}$ and $\wideparen{a_{-1}a_{0}}$ are the longest of them.  We may also assume that the set $A$ is such that the lengths of these arcs are not longer than $1/4$ of the length of the circle $\mathbb S^1$. Then the angle between consecutive chords, i.e. between straight line segments $a_ia_{i+1}$ and $a_{i+1}a_{i+2}$, is not smaller than  $\pi/2$.  For $k\in \mathbb Z$, let $\mathscr S_k$ \index{$\mathscr S_k$} be the sector of the plane obtained as the union of all rays starting at the center of the circle $\mathbb S^1$ and going through the points of the circle arc $\wideparen{a_k a_{k+1}}$, see Figure~\ref{fig:rhombusZ}.

We choose a (so called `big') rhombus $BR_0\subseteq \mathscr S_0 \subseteq \pi_0$  \index{big rhombus} \index{$BR_0$} such that the straight line segment $a_0a_1$ is one of its diagonal, the other diagonal being shorter.\footnote{The strange notation $BR$ means a `big rhombus', in contrast with `small rhombuses' which will appear later. The small rhombuses will be subsets of big rhombuses and the notation $SR$ (with some indices) will be used for them.} Then the angles of this big rhombus at the vertices $a_0$ and $a_1$ are smaller than $\pi/2$. Note that the only points of $BR_0$ lying in the boundary of the sector $\mathscr S_0$ are $a_0$ and $a_1$.

Let $\mathscr H_0 \subseteq BR_0$  \index{$\mathscr H_0$} be a Cook continuum (which is not homeomorphic to any of the Cook continua used below as bricks in the snake) containing the points $a_0$ and $a_1$ and containing no other points in the boundary of $BR_0$. The points $a_0$ and $a_1$ will be called the \emph{extremal points} of $\mathscr H_0$, namely the \emph{first point} and the \emph{last point} of $\mathscr H_0$, respectively. So,
\begin{equation}\label{Eq:H0inIntBR0}
\text{$\mathscr H_0 \subseteq BR_0$ with all non-extremal points in $\Int BR_0$}
\end{equation}
and the distance between extremal points of $\mathscr H_0$ equals the diameter of $\mathscr H_0$ (and this equals the diameter of $BR_0$).

For every integer $k$, let $Z_k$ \index{$Z_k$} be the  direct similitude (similarity transformation)\index{similitude} defined in the plane $\pi_0$, i.e. the composition of a homothety and a direct euclidean motion, which maps the straight line segment $a_0a_1$ onto the straight line segment $a_ka_{k+1}$, with $Z_k (a_0) = a_k$, $Z_k(a_1)=a_{k+1}$ (in particular, $Z_0$ is the identity). Then the big rhombus $BR_0$ containing $a_0, a_1$ is mapped by $Z_k$ onto some rhombus, again call it a big rhombus,  $BR_k$ containing $a_k, a_{k+1}$ and the Cook continuum $\mathscr H_0 \subseteq BR_0$ is mapped by $Z_k$ onto a (homeomorphic) Cook continuum $\mathscr H_k \subseteq BR_k$. So,
\begin{equation}\label{Eq:mathscrHk}
	BR_k := Z_k(BR_0) \quad \text{and} \quad	\mathscr H_k := Z_k(\mathscr H_0).
\end{equation} \index{$BR_k$} \index{$\mathscr H_k$}
The points  $a_k$ and $a_{k+1}$ are said to be the extremal points, or, more precisely, the first point and the last point, respectively, of both $BR_k$ and $\mathscr H_k$. Due to~\eqref{Eq:H0inIntBR0},
\begin{equation}\label{Eq:HkinIntBRk}
\text{$\mathscr H_k \subseteq BR_k$ with all non-extremal points in $\Int BR_k$.}
\end{equation}
So, all the continua $\mathscr H_k$ are copies of $\mathscr H_0$, obtained from it by just `zooming and moving' it. Recall that the angle between two consecutive chords $a_ia_{i+1}$ and $a_{i+1}a_{i+2}$ is at least $\pi/2$, while the vertex angles of our big rhombuses at those vertices which belong to the set $A$ are smaller than $\pi/2$ (for $BR_0$ this is by construction and note that the similitudes preserve angles). Hence the big rhombuses $BR_k$, $k\in \mathbb Z$, are pairwise disjoint, except that $BR_k\cap BR_{k+1} =\{a_{k+1}\}$, $k\in \mathbb Z$. The same is true for the Cook continua $\mathscr H_k$, $k\in \mathbb Z$.

Clearly, $Z_0$ is an isometry and since the arcs $\wideparen{a_0 a_{1}}$ and $\wideparen{a_{-1}a_{0}}$ have the same lengths, also $Z_{-1}$ is an isometry. However,
\begin{equation}\label{Eq:decr}
	\text{for $k\neq 0, -1$, the similitudes $Z_k$ decrease distances.}
\end{equation}
Note also that, given $i,j\in \mathbb Z$, for the homeomorphism (in fact a similitude defined in the whole plane $\pi_0$)
\begin{equation}\label{Eq:mapsZij}
	Z_{i,j} = Z_j \circ (Z_i)^{-1}
\end{equation}\index{$Z_{i,j}$}
we have $Z_{i,j}(a_i) =a_j$, $Z_{i,j}(a_{i+1}) =a_{j+1}$ and
\begin{equation}\label{Eq:HkHk+1}
BR_{j} = Z_{i,j}(BR_i) \quad \text{and} \quad \mathscr H_{j} = Z_{i,j}(\mathscr H_i),
\end{equation}
see Figure~\ref{fig:rhombusZ}.
\begin{figure}
	\centering
	\includegraphics[width=10cm]{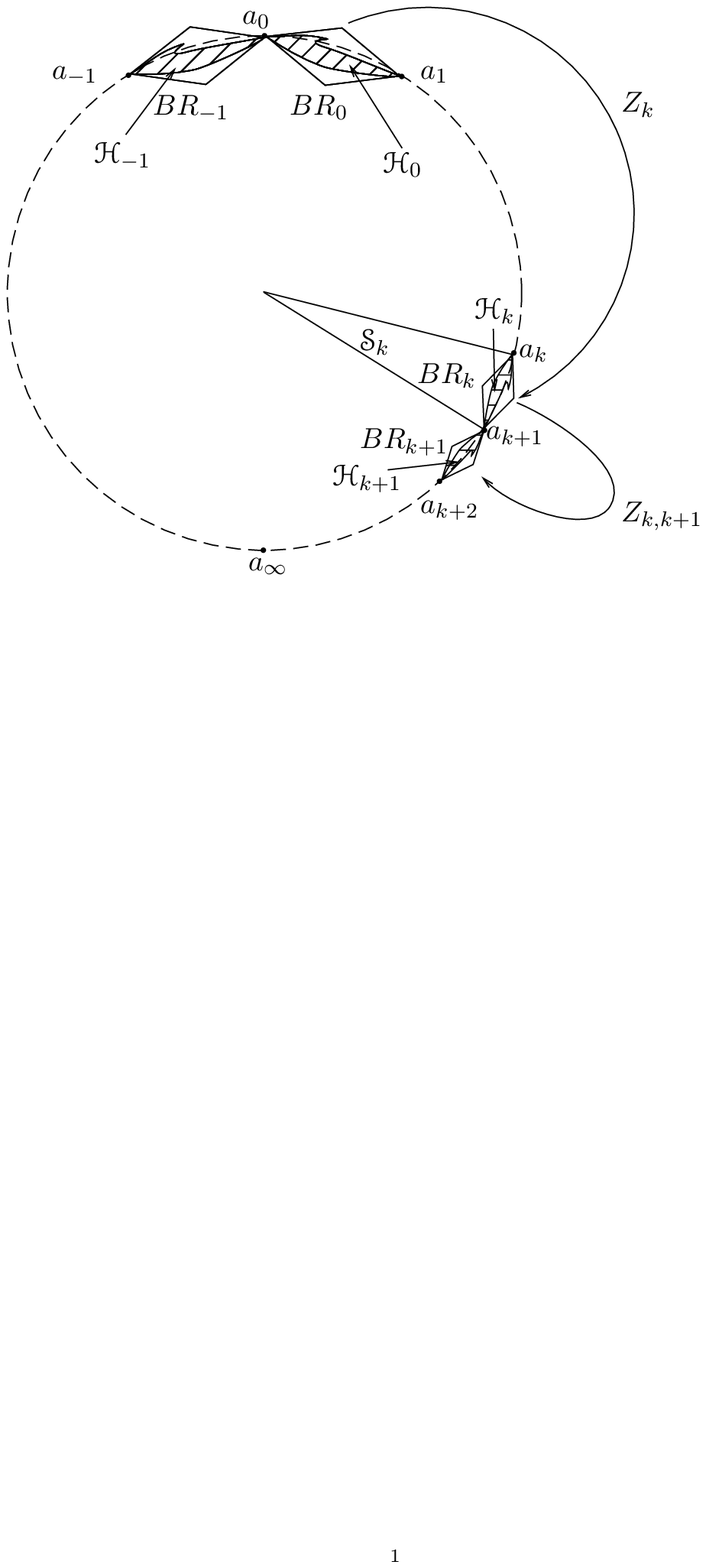}\\
	\caption{Big rhombuses and maps $Z_k$}\label{fig:rhombusZ}
\end{figure}

The \emph{head} \index{head} \index{$\mathscr A_0$} of our space $X$ will then be the set
\begin{equation}\label{Eq:headlog2}
	\mathscr A_0=A \cup \bigcup_{i \in \Z} \mathscr H_i = \{a_{\infty}\} \cup \bigcup_{i \in \Z} Z_i(\mathscr H_0) \subseteq \pi_0
\end{equation}
which looks like a necklace of infinitely many copies of the Cook continuum $\mathscr H_0$, together with the point $a_{\infty}$. Clearly, it is a continuum in the vertical plane $\pi_0$.

In Section~\ref{S:X} (i.e. in the case $S(X)=\{0,\infty\}$) we had planar Cook continua $\mathscr K^m_i$, called bricks,
and the sets $D_1^*, D_2^*, \dots$ made of the bricks, which were used to join the points of the trajectory
$x_1, x_2,\dots$. The union of all sets $D_i^*$ was the snake of~$X$.
Now (i.e. in the case $S(X)=\{0,\log 2\}$) the snake is defined similarly as in Section~\ref{S:X} (and, as before, $X$ is the union of the head and the snake) with only two differences.
\begin{itemize}
\item First, a small technical difference is that now we start with the trajectory $x_0, x_1,\dots$, rather than $x_1, x_2,\dots$, therefore we `shift' the indices in the notations of the bricks and the sets $D_m$. Now a continuum joining $x_m$ and $x_{m+1}$ will be in fact denoted by $D(m)$ \index{$D(m)$}, $m=0,1,\dots$. So,
\begin{equation}\label{Eq:DmDm*} \index{$D^*(m)$}
D(m)= \bigcup_{i=1}^{\infty} \mathscr K_i^m \cup \{x_{m+1}\} \quad \text{ and }
\quad D^*(m) = \bigcup_{i=1}^{\infty} \mathscr K_i^m= D_m \setminus \{x_{m+1}\},
\end{equation}
and
\begin{alignat}{1}
 D^*(0)         & \, = \, \text{copy of } \mathscr K_1 \cup \text{copy of }
                                                           \mathscr K_2 \cup \dots~,   \notag \\
 D^*(1)     & \, = \, \text{copy of } \mathscr K_2 \cup \text{copy of } \mathscr K_4 \cup \dots~, \label{Eq:every2nd}  \\
 D^*(2) & \, = \, \text{copy of } \mathscr K_4 \cup \text{copy of } \mathscr K_8 \cup \dots~,    \notag \\
                & \dots~.   \notag
\end{alignat}
Recall that $\{\mathscr K_1, \mathscr K_2, \dots\}$ is a family of planar non-homeomorphic Cook continua, which are also non-homeomorphic to Cook continua $\mathscr H_k$ used in the head. In the usual sense we speak on the extremal points, i.e. on the first point and the last point, of the copies of $\mathscr K_i$ used to built the snake. As in Subsection~\ref{SS:construction of X}, in every $\mathscr K_i$ we choose two points whose distance equals the diameter of $\mathscr K_i$ and we call them the \emph{extremal points} of $\mathscr K_i$, or the \emph{first point} and the \emph{last point} of $\mathscr K_i$. When considering a homeomorphic copy of $\mathscr K_i$, the images of these two points are still called the first point and the last point, respectively, of this copy. We will call them the \emph{extremal points} of this copy; later it will be seen that, in our construction, these two points will really be extremal points of the copy with respect to the euclidean metric. We will also speak on the extremal points of the set $D(m)$, the first point of $D(m)$ being defined as the first point of the first brick in $D(m)$ and the last point of $D(m)$ being defined as the unique point of $D(m)\setminus D^*(m)$. Again, they will be extremal in the sense of the metric. In $D^*(0)$, see~\eqref{Eq:every2nd}, the last point of the copy of $\mathscr K_i$ coincides with the first point of the copy of $\mathscr K_{i+1}$, $i=1,2,\dots$. The sets $D^*(m)$, $m>0$, are built analogously.

\item Second, while in Section~\ref{S:X} we joined $x_m$ and $x_{m+1}$ by placing the set $D_m$ along the straight line segment $x_mx_{m+1}$, see Figures~\ref{F:solids}-\ref{F:in solid}, now the straight line segments are replaced by (polygonal) arcs. We have to choose these arcs very carefully and to place the sets $D(m)$ into sufficiently small neighborhoods of them. In fact, otherwise the snake could have cluster points in the plane $\pi_0$ also outside the head $\mathscr A_0$ and so $X$ would not be compact. To be sure that the snake does not `produce' cluster points in $\pi_0\setminus \mathscr A_0$, as well as to avoid some other potential problems which will be discussed later, we proceed as shown in the next subsection.
\end{itemize}

Let us remark that not only the mentioned change in notation ($D(m)$, $m=0,1,\dots$ instead of $D_m$, $m=1,2,\dots$) is unimportant, but in fact the whole snake is even a \emph{homeomorphic copy} of the snake from Section~\ref{S:X}, compare Figure~\ref{F:3Dm} and ~\eqref{Eq:every2nd}. However, the two snakes have different positions in $\mathbb R^3$, meaning that they are not equivalently embedded in $\mathbb R^3$, i.e. there is no homeomorphism $\mathbb R^3 \to \mathbb R^3$ mapping one snake onto the other one (note that the two heads are not homeomorphic).

\subsection{More details of the construction of $X$}\label{SS:more details}
We first introduce some terminology.

The trajectory of $x_0$ performs infinitely many jumps in the sense described in Section \ref{S:X1-T1-log2} (formally, a jump is a pair of points).
So, we can use them to separate $\{x_0,x_1,x_2,\dots\}$ into a disjoint union of sets (finite sequences) $J(m)$ \index{$J(m)$}, $m=0,1,\dots$, as follows:
\begin{itemize}
	\item $J(0)=\{x_0,x_1,\dots, x_{j_0}\}=\{x_{k_0},x_{k_0+1},\dots, x_{k_0+j_0}\}$ where $k_0=0$ and $j_0$ is the smallest positive integer such that $x_{j_0}$ is the first point of a jump (then  $x_{j_0+1}$ is the last point of this jump). In Figure~\ref{fig:p11}, $j_0=3$.
	\item $J(1)=\{x_{k_1},x_{k_1+1},\dots, x_{k_1+j_1}\}$ where $k_1=j_0+1$ and $j_1$ is the smallest positive integer such that
	$x_{k_1+j_1}$ is the first point of a jump. In Figure~\ref{fig:p11}, $k_1=4$ and in Figure~\ref{fig:ig11} we see $x_{k_1+j_1}$ as the point on the right hand side of $a_{\infty}$, from which the trajectory jumps to the left hand side of $a_{\infty}$ (to be precise, in the figure we see the $P_2$-projections of those points).
	\item In general, for $m\geq 1$, $J(m)=\{x_{k_m},x_{k_m+1},\dots, x_{k_m+j_m}\}$ where $k_m=k_{m-1}+j_{m-1}+1$, and $j_m$ is the smallest positive integer such that
	$x_{k_m+j_m}$ is the first point of a jump.
\end{itemize}
So, these sets are maximal intervals in $\{x_0,x_1,x_2,\dots\}$ containing no jump. A jump can only be performed from the last point of
$J(m)$  to the first point of $J(m+1)$. The set $J(m)$ is called the \emph{$m$-th jump level} \index{jump level}. Note that it contains exactly one pair of points whose $P_2$-projections are $a_0$ and $a_1$.

\medskip

Fix $m\geq 0$. Recall that $P_2(J(m))$ is a block of consecutive points of $\{a_i:\, i\in \mathbb Z\}$. Let $a_k$ and $a_{k+1}$ be two neighboring points in  $P_2(J(m))$. They are $P_2$-projections of two neighboring points from $J(m)$ which we denote by $x^{[m]}_k$ \index{$x^{[m]}_k$} and $x^{[m]}_{k+1}$, respectively (the latter one is closer to $\pi_0$ than the former one).\footnote{Note that, given $m\geq 0$, $x^{[m]}_k$ exists only for finitely many integers $k$, because $J(m)$ is finite. Given an integer $k$, $x^{[m]}_k$ exists for all sufficiently large $m\geq 0$.}  Among all planes containing the points $x^{[m]}_k$ and $x^{[m]}_{k+1}$, let $\mathscr P ^{[m]}_k$ \index{$\mathscr P ^{[m]}_k$} be the one whose angle with $\pi_0$ is the same as the angle of the line $x^{[m]}_k x^{[m]}_{k+1}$ with $\pi_0$, i.e. the plane whose intersection with $\pi_0$ is the line perpendicular to the line $x^{[m]}_k x^{[m]}_{k+1}$.

If $m$ is given and $\mathscr P^{[m]}_k$ exists, then for every set $M \subseteq \pi_0$ lying in the sector $\mathscr S_k$ we define its lift to $\mathscr P ^{[m]}_k$, or the \emph{$[m]$-lift of $M$} \index{$M^{[m]}$}, by
\begin{equation}\label{Eq:mlift}
M^{[m]} = M^{[m]}_k := \{x\in \mathscr P^{[m]}_k :\, P_2(x) \in M\}~.
\end{equation}
The homeomorphism $Z_k \colon \pi_0 \to \pi_0$ can be lifted to a homeomorphism $Z^{[m]}_k\colon \mathscr P^{[m]}_0 \to \mathscr P^{[m]}_k$ \index{$Z^{[m]}_k$}. It is defined by the equality
\[
P_2\circ Z^{[m]}_k = Z_k \circ P_2
\]
(in particular, $Z^{[m]}_0$ is the identity).

If $\mathscr P^{[m]}_k$ is defined (hence $Z^{[m]}_k$ is defined), then the $Z^{[m]}_k$-image of a set (i.e., a $Z^{[m]}_k$-copy of a set) is sometimes less precisely said to be a \emph{$Z^{[m]}$-copy of that set} (i.e. without specifying $k$). A family of sets is said to be \emph{$Z^{[m]}$-homeomorphic} \index{$Z^{[m]}$-homeomorphic} if each of them is a $Z^{[m]}_k$-copy of the same set (the values of $k$ are different for different sets in the family, while $m$ is the same). Clearly, such sets are pairwise homeomorphic.

\medskip

Now let $H \subseteq \pi_0$ be a polygonal arc with the first point $c_1$ and the last point $c_k$, say the union of straight line segments $[c_1,c_2],[c_2,c_3],\dots, [c_{k-1},c_k]$. If they are maximal straight line segments in $H$, they are called the \emph{links} of $H$ and the points $c_2, \dots, c_{k-1}$ are called the \emph{turning points}\index{turning points} of $H$. The first point $c_1$ and the last point $c_k$ of $H$ are also called the \emph{extremal points} of $H$. We define the \emph{distance of any points $a,b\in H$ along $H$},  denoted by $l_H(a,b)$ \index{$l_H(a,b)$}, as follows:
\begin{itemize}
  \item If $a,b \in [c_i,c_{i+1}]$ for some $i$, then $l_H(a,b)=d(a,b)$, where $d$ is the euclidean metric in $\pi_0$.
  \item If $a \in [c_i, c_{i+1}]$, $b \in [c_j, c_{j+1}]$ with $j>i$, then
  $$
  l_H(a,b)=d(a,c_{i+1})+\sum_{t=i+1}^{j-1}d(c_t,c_{t+1}) +d(c_j,b).
  $$
\end{itemize}
The case $i>j$ is covered by requiring $l_H(a,b)=l_H(b,a)$. Notice that if $a,b \in H$ then $d(a,b)\leq l_H(a,b)$. Given different points $a,b\in H$, let $[a,b]_H$ \index{$[a,b]_H$} be the polygonal arc which is a subset of $H$ and has endpoints $a$ and $b$; so, $l_H(a,b)$ is the length of the arc $[a,b]_H$.

\medskip

Now we are going to choose, in the plane $\pi_0$, some polygonal arcs `close' to $\mathscr H_0$ and joining $a_0$ and $a_1$, then some polygonal arcs `close' to $\mathscr H_k$, $k\neq 0$, and joining $a_k$ and $a_{k+1}$, and finally the polygonal arcs, in this case in fact straight line segments, `along' the $P_2$-projections of jumps (meaning that such a straight line segment joins the $P_2$-projections of the points forming a jump).

\bigskip

First consider $\mathscr H_0$. Let $\mathscr V_0^m$ be the open $(1/2^m)$-neighbourhood of $\mathscr H_0$ in $\pi_0$, $m=0,1,\dots$. Then, since $\mathscr H_0$ is closed,
\begin{equation}\label{Eq:V-H0}
\mathscr V_0^0 \supseteq \mathscr V_0^1 \supseteq \dots \quad \text{and} \quad \bigcap_{m=0}^{\infty} \mathscr V_0^m =\mathscr H_0.
\end{equation}
By Lemma~\ref{L:arc in nbhd}, for any nonnegative $m$, the points $a_0$ and $a_1$ can be joined by a polygonal arc $H_0^m \subseteq \mathscr V_0^m$. \index{$\mathscr V_0^m$} Moreover, since $\mathscr H_0$ is a subset of the big rhombus $BR_0$, there exists even a polygonal arc $H_0^m$  \index{$H_0^m$ } joining $a_0$ and $a_1$ such that
\begin{equation}\label{Eq:HinR}
H_0^m \subseteq \mathscr V_0^m \cap BR_0~.
\end{equation}
In fact, for every point $X_{\rm o} \in \pi_0\setminus BR_0$ there exists its projection into $BR_0$, i.e. a unique point $X_{\rm b}$ in the boundary of $BR_0$ which is closest to $X_{\rm o}$ among all points of $BR_0$. When we replace those points of $H_0^m$ which are outside $BR_0$ by their projections into $BR_0$, we obtain a new polygonal arc  joining $a_0$ and $a_1$, already satisfying~\eqref{Eq:HinR}. Since $\mathscr V_0^m$ is open and $H_0^m$ is compact, we may even assume that, in~\eqref{Eq:HinR}, the polygonal arc $H_0^m$ lies in the interior of $BR_0$ with the exception of its extremal points which coincide with the endpoints of the longer diagonal of $BR_0$. Thus,
\begin{equation}\label{Eq:HinRint}
\text{$H_0^m \subseteq \mathscr V_0^m \cap BR_0$ with all non-extremal points of $H_0^m$ in $\Int (BR_0)$~.}
\end{equation}

If a set $A$ is in the open $\varepsilon$-neighbourhood of a set $B$, we say that $A$ is \emph{$\varepsilon$-close} \index{$\varepsilon$-close} to $B$.\footnote{Note the lack of symmetry: If $A$ is $\varepsilon$-close to $B$ then $B$ need not be $\varepsilon$-close to $A$.} The fact that $H_0^m$ is in the open $(1/2^m)$-neighbourhood $\mathscr V_0^m$ of $\mathscr H_0$ thus means that
\begin{equation}\label{Eq:H0mclosetoH0}
\text{$H_0^m$ is $(1/2^m)$-close to $\mathscr H_0$~.}
\end{equation}
Figure~\ref{fig:inrhombus} illustrates the situation.

\begin{figure}[h]
	\centering
	\includegraphics[width=8cm]{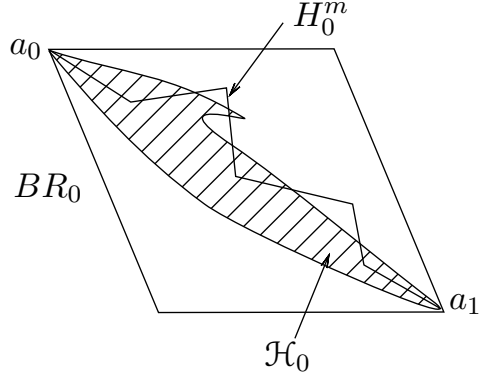}\\
	\caption{The Cook continuum $\mathscr H_0$ and the polygonal arc $H^m_0$ in the big rhombus $BR_0$, with only the extremal points $a_0$ and $a_1$ in its boundary. The arc $H_0^m$ is $(1/2^m)$-close to the continuum $\mathscr H_0$.}\label{fig:inrhombus}
\end{figure}

We choose finitely many points in $H_0^m$ satisfying the following `equidistant points conditions'. \index{equidistant points conditions}

\begin{itemize}
  \item [(EP1)] The points $a_0, (b_0^m)_1, (b_0^m)_2, \dots, (b_0^m)_{n_0^m}, a_1 \in H_0^m$ are equidistant with respect to $l_{H_0^m}$, i.e.
  \[
  l_{H_0^m}(a_0, (b_0^m)_1)= l_{H_0^m}((b_0^m)_1, (b_0^m)_2)= \dots = l_{H_0^m}((b_0^m)_{n_0^m},a_1)= :r_0^m~.
  \]
 This family is said to be the \emph{(distinguished) equidistant family} \index{distinguished equidistant family} in $H_0^m$ or the (distinguished) family of equidistant points in $H_0^m$. The number $r_0^m$ \index{$r_0^m$} is said to be the \emph{equidistance constant} \index{equidistance constant} of $H_0^m$.\footnote{More precisely, we should speak on the equidistance constant of $H_0^m$ with respect to that family of points or on the equidistance constant of the considered family of points with respect to $H_0^m$.} The considered equidistant points divide $H_0^m$ into finitely many (in fact $n_0^m +1$) subarcs. In the sequel we will call them \emph{equi-subarcs} of $H_0^m$. The family of them is naturally ordered, the first equi-subarc being the one containing $a_0$.
 \item [(EP2)] $n_0^m$ is so large, i.e. $r_0^m$ is so small, that each equi-subarc of $H_0^m$ contains at most one turning point of $H_0^m$ in its interior (in the topology of $H_0^m$ inherited from the plane~$\pi_0$).
 \item [(EP3)] $n_0^m \geq 3$ is so large that the first two equi-subarcs and the last equi-subarc of $H_0^m$ are straight line segments.
 \item [(EP4)] $n_0^m$ is so large that
 \[
 [a_0, (b_0^m)_1]_{H_0^m} = \overline{B}_{d}(a_0, r_0^m) \cap H^m_0 \quad \text{ and } \quad [(b_0^m)_{n_0^m}, a_1]_{H_0^m} = \overline{B}_{d}(a_1, r_0^m)\cap H^m_0
 \]
 where $\overline{B}_{d}(x, \delta)$ is the closed ball, with respect to the euclidean metric $d$, with centre $x$ and radius~$\delta$.
 Notice that, by (EP3), on the left hand sides we have straight line segments and not just polygonal arcs.
 \item  [(EP5)] The sequence $(n_0^m)_{m=0}^\infty$ increases so fast that, for any nonnegative $m$,
  \[
  0<r_0^{m+1}<(1/2)r_0^m~, \text{ hence } r_0^m \searrow 0 \text{ as } m\to \infty~.
  \]
\end{itemize}

\medskip

 For $m \geq 0$, the homeomorphism $Z_k$ sends
\begin{itemize}
\item the open neighbourhood $\mathscr V^m_0$ of $\mathscr H_0$ to the open neighbourhood  $\mathscr V^m_k := Z_k(\mathscr V^m_0)$ of $\mathscr H_k$, with (see~\eqref{Eq:V-H0})
\begin{equation}\label{Eq:V-Hk}
\mathscr V_k^0 \supseteq \mathscr V_k^1 \supseteq \dots \quad \text{and} \quad  \bigcap_{m=0}^{\infty}\mathscr V^m_k = \mathscr H_k~;
\end{equation}
\item the polygonal arc $H_0^m$ joining $a_0$ and $a_1$ to the polygonal arc $H_k^m:= Z_k(H_0^m)$ \index{$H_k^m$} joining $a_k$ and $a_{k+1}$ (in particular, the turning points of $H_0^m$ are mapped to the turning points of $H_k^m$), where, due to~\eqref{Eq:HinRint},
\begin{equation}\label{Eq:HkinRint}
\text{$H_k^m \subseteq \mathscr V_k^m \cap BR_k$ with all non-extremal points of $H_k^m$ in $\Int (BR_k)$}
\end{equation}
and, since the similitudes $Z_k$ do not increase distances, due to~\eqref{Eq:H0mclosetoH0} we have that
\begin{equation}\label{Eq:HkmclosetoHk}
\text{$H_k^m$ is $(1/2^m)$-close to $\mathscr H_k$~;}
\end{equation}
\item the points $a_0$, $(b_0^m)_1$, $(b_0^m)_2$, $\dots$, $(b_0^m)_{n_0^m}$, $a_1 \in H_0^m$ with the equidistance constant $r_0^m$ to the same number of points
$a_k$, $(b_k^m)_1$, $(b_k^m)_2$, $\dots$, $(b_k^m)_{n_0^m}$, $a_{k+1} \in H_k^m$.
\end{itemize}

\medskip

Since $Z_k$ is a similitude, we have the following analogues of (EP1)-(EP5).

\begin{itemize}
\item [(EP6)] The family of points $a_k$, $(b_k^m)_1$, $(b_k^m)_2$, $\dots$, $(b_k^m)_{n_0^m}$, $a_{k+1} \in H_k^m$, being the $Z_k$-image of the equidistant family of points in $H_0^m$, is an equidistant family of points in $H_k^m$, with some equidistance constant $r_k^m \leq r_0^m$ (with equality only for $k=0,-1$,  see~\eqref{Eq:decr}) with respect to $l_{H_k^m}$.
\item [(EP7)] Each equi-subarc of $H_k^m$ contains at most one turning point of $H_k^m$ in its interior.
\item [(EP8)] The first two equi-subarcs and the last equi-subarc of $H_k^m$ are straight line segments.
\item [(EP9)] $[a_k, (b_k^m)_1]_{H_k^m} = \overline{B}_{d}(a_k, r_k^m) \cap H^m_k$ and $[(b_k^m)_{n_0^m}, a_{k+1}]_{H_k^m} = \overline{B}_{d}(a_{k+1}, r_k^m)\cap H^m_k$.
\item  [(EP10)] $0<r_k^{m+1}<(1/2)r_k^m$,  hence $r_k^m \searrow 0 \text{ as } m\to \infty$.
\end{itemize}

\medskip

 Finally, for any nonnegative $m$ consider the $m$-th jump, i.e. the jump from $J(m)$ to $J(m+1)$, and let $L^m \subseteq \pi_0$ \index{$L^m$} be the straight line segment whose endpoints are $P_2(x_{k_m + j_m}), P_2(x_{k_{m+1}}) \in A$, i.e. the $P_2$-projections of the two points forming the jump (here $k_{m+1}= k_m + j_m +1$). Let $\mathscr U^m$ \index{$\mathscr U^m$} be the open $(1/2^m)$-neighbourhood of $L^m$  in $\pi_0$ (for each $L^m$ we consider just one neighbourhood $\mathscr U^m$). In an obvious sense,
\begin{equation}\label{Eq:U-ainfty}
\text{the neighbourhoods $\mathscr U^m$ of $L^m$ converge to the point $a_{\infty}$ as $m\to \infty$}.
\end{equation}
We choose finitely many points in $L^m$ satisfying the following conditions.
\begin{itemize}
  \item [(EP11)] The points $P_2(x_{k_m + j_m}),  (c_L^m)_1, (c_L^m)_2, \dots, (c_L^m)_{n_0^m}, P_2(x_{k_{m+1}}) \in L^m$
  are equidistant with respect to (both $d$ and) $l_{L^m}$, with the equidistance constant $r^m_L$. Note that the number of these points is the same as the number of the chosen distinguished points in $H_0^m$ or in $H_k^m$. Notice that the analogues of (EP2)-(EP4), or (EP7)-(EP9), are trivial.
  \item [(EP12)] Since the number of points in (EP11) tends to infinity and the length of $L^m$ converge to zero, we get $r_L^m \to 0$ as $m\to \infty$.
\end{itemize}

\medskip

Now consider the trajectory
$$
\underbrace{x_0,x_1,\dots, x_{j_0}}_{J(0)}, \quad \underbrace{x_{k_1},x_{k_1+1},\dots, x_{k_1+j_1}}_{J(1)}, \quad
\underbrace{x_{k_2},\dots}_{J(2)}, \quad \dots
$$
and also the sequence of $P_2$-projections of these points. These two sequences can be seen in the first and the second row of Table~\ref{T:xP2}, respectively. In the table we of course have $p_0=j_0$. All the integers $n_i$ and $p_i$ are positive. It follows from the construction of the trajectory of $x_0$ that $n_i \to \infty$ and $p_i \to \infty$.
\begin{table}[h]
	\begin{footnotesize}
		\begin{center}
			\begin{tabular}{|c|c|c|c|c?c|c|c|c|c?c|c|}
				\hline
				$x_0$ & $x_1$ & $\dots$ & $x_{j_0-1}$ & $x_{j_0}$ & $x_{k_1}$  & $x_{k_1+1}$ & $\dots$ & $x_{k_1+j_1-1}$
				& $x_{k_1+j_1}$ & $x_{k_2}$  & $\dots$ \\
				\hline
				$a_0$ & $a_1$ & $\dots$ & $a_{p_0-1}$ & $a_{p_0}$ & $a_{-n_1}$ & $a_{-n_1+1}$
				& $\dots$ & $a_{p_1-1}$ & $a_{p_1}$ & $a_{-n_2}$ & $\dots$ \\
				\hline
			\end{tabular}
		\end{center}
	\end{footnotesize}
	\caption{Trajectory of $x_0$ and its $P_2$-projection}\label{T:xP2}
\end{table}

\noindent While the trajectory of $x_0$ is injective, its $P_2$-projection in the second row contains for instance the periodic subsequence $a_0, a_1, a_0, a_1, \dots$. Consider the following sequence of polygonal arcs which join the con\-se\-cu\-ti\-ve points of the sequence $P_2(x_0) =~a_0$, $P_2(x_1) = a_1, \dots$:
\begin{equation}\label{Eq:HLpn}
\underbrace{H_0^0, H_1^0, \dots, H_{p_0-1}^0}_{H_k^0 = Z_k(H_0^0)}, \quad L^0, \quad \underbrace{H_{-n_1}^1, H_{-n_1+1}^1,\dots, H_{p_1-1}^1}_{H_k^1= Z_k(H_0^1)}, \quad L^1, \quad H_{-n_2}^2, \dots~.
\end{equation}
Notice that this sequence contains a subsequence of polygonal arcs $H^0_0, H^1_0, H^2_0, \dots$ which are used to join the pairs of points $a_0$ and $a_1$ from the above mentioned periodic sequence.
In the second row of the following Table~\ref{T:bounds} we summarize the equidistance constants for the arcs from~\eqref{Eq:HLpn}.  By (EP5), (EP10), (EP12) and the fact that $r_k^m \leq r_0^m$ we know that they converge to zero.
\begin{table}[h]
	\begin{footnotesize}
		\begin{center}
			\begin{tabular}{|c|c|c|c?c?c|c|c|c?c?c|c|}
				\hline
				$H_0^0$ & $H_1^0$ & $\dots$ & $H_{p_0-1}^0$ & $L^0$ & $H_{-n_1}^1$  & $H_{-n_1+1}^1$ & $\dots$ & $H_{p_1-1}^1$
				& $L^1$ & $H_{-n_2}^2$  & $\dots$ \\
				\hline
				$r_0^0$ & $r_1^0$ & $\dots$ & $r_{p_0-1}^0$ & $r_L^0$ & $r_{-n_1}^1$ & $r_{-n_1+1}^1$
				& $\dots$ & $r_{p_1-1}^1$ & $r_L^1$ & $r_{-n_2}^2$ & $\dots$ \\
				\hline
			\end{tabular}
		\end{center}
	\end{footnotesize}
	\caption{Polygonal arcs joining the points of the $P_2$-projection of the trajectory of $x_0$ and their equidistance constants (converging to zero)}\label{T:bounds}
\end{table}

Now we are going to lift the polygonal arcs from~\eqref{Eq:HLpn} (together with the distinguished families of equidistant points in them), in order to join the consecutive points of the trajectory $x_0, x_1, \dots$. We use lifts as follows.
\begin{itemize}
	\item For each arc from~\eqref{Eq:HLpn} which is of the form $H_i^m$ (and so is a subset of the big rhombus $BR_i$ in the sector $\mathscr S_i$) we use its $[m]$-lift, see~\eqref{Eq:mlift}.
	\item For each straight line segment $L^m \subseteq \pi_0$ joining the points $a_{p_m}$ and $a_{-n_{m+1}}$, see Table~\ref{T:xP2}, we use, as its lift, the straight line segment joining the points $x_{k_m+j_m}$ and $x_{k_{m+1}}$, i.e. the last point of $J(m)$ and the first point of $J(m+1)$.
\end{itemize}
For  each polygonal arc $H$ from~\eqref{Eq:HLpn} we have thus defined, in a unique way, its \emph{lift} \index{lift}. We denote this lift by $\mathscr L(H)$ and we also call it the \emph{regular polygonal arc above $H$}. However, if $\mathscr L(H)$ has endpoints $x_h$ and $x_{h+1}$, then a `natural' notation for this lift is
$\mathscr L(h)$ \index{$\mathscr L(h)$}, so we have the sequence $\mathscr L(0), \mathscr L(1), \mathscr L(2), \dots$. In Table~\ref{T: L(H)},
each of the considered regular polygonal arcs is placed in the third row just below its first point (the next point to the right is its last point). In the fourth row we write the mentioned `natural' notation. Any two arcs are of course homeomorphic. Some of the considered arcs are even $Z^{[m]}$-homeomorphic for some $m$, as shown in the last row of the table. By the construction, see also \eqref{Eq:HinR} and \eqref{Eq:HkinRint}, the lifts  in the table intersect only in their common endpoints.
\begin{table}[h]
	\begin{footnotesize}
		\begin{center}
			\begin{tabular}{|c|c|c|c|c|c|c|c|c|c|}
		    	\hline
				\multicolumn{4}{|c|}{$J(0)$} &	\multicolumn{4}{c|}{$J(1)$} & \multicolumn{2}{c|}{$\dots$} \\
				\hline
				$x_0$ & $\dots$ & $x_{j_0-1}$ & $x_{j_0}$ & $x_{k_1}$  & $\dots$ & $x_{k_1+j_1-1}$
				& $x_{k_1+j_1}$ & $x_{k_2}$  & $\dots$ \\
				\hline \hline
				$\mathscr L(H_0^0)$ & $\dots$ & $\mathscr L(H_{p_0-1}^0)$ & $\mathscr L(L^0)$ & $\mathscr L(H_{-n_1}^1)$
				& $\dots$ & $\mathscr L(H_{p_1-1}^1)$ & $\mathscr L(L^1)$ & $\mathscr L(H_{-n_2}^2)$ & $\dots$ \\
				\hline
				$\mathscr L(0)$ & $\dots$ & $\mathscr L(j_0-1)$ & $\mathscr L(j_0)$ & $\mathscr L(j_0+1)$ & $\dots$ & $\mathscr L(j_0+j_1)$
				& $\mathscr L(j_0+j_1+1)$ & $\mathscr L(j_0+j_1+2)$  & $\dots$ \\
				\hline
			    \multicolumn{3}{|c|}{$Z^{[0]}$-copies of $\mathscr L(H^0_0)$} &  & \multicolumn{3}{c|}{$Z^{[1]}$-copies of $\mathscr L(H^1_0)$} &  & \multicolumn{2}{c|}{$\dots$}\\
			    \hline
			    \end{tabular}
		\end{center}
	\end{footnotesize}
	\caption{The regular polygonal arcs joining the elements of the trajectory of $x_0$}\label{T: L(H)}
\end{table}

\noindent The distinguished equidistant family of points in a polygonal arc $H$ from~\eqref{Eq:HLpn} is lifted to the family of points in $\mathscr L(H)$ with the same cardinality but, in general, this lifted family of points is not equidistant.

The polygonal arcs from~\eqref{Eq:HLpn} can be viewed as the $P_2$-projections of their lifts. The first and the second row of the following Table~\ref{T:bounds2} show these arcs written in two alternative ways. The third row of the table contains their equidistance constants, but redenoted; this new notation will be more convenient than the one from Table~\ref{T:bounds}.

\begin{table}[h]
	\begin{footnotesize}
		\begin{center}
			\begin{tabular}{|c|c|c?c?c|c|c?c?c|}
				\hline
				$H_0^0$ &  $\dots$ & $H_{p_0-1}^0$ & $L^0$ & $H_{-n_1}^1$  &  $\dots$ & $H_{p_1-1}^1$
				& $L^1$ &  $\dots$ \\
				\hline
				$P_2(\mathscr L(0))$ & $\dots$ & $P_2(\mathscr L(j_0-1))$ & $P_2(\mathscr L(j_0))$ & $P_2(\mathscr L(j_0+1))$  & $\dots$ & $P_2(\mathscr L(j_0+j_1))$
				& $P_2(\mathscr L(j_0+j_1+1))$ &  $\dots$ \\
				\hline
				$r(0)$ &  $\dots$ & $r(j_0-1)$ & $r(j_0)$ & $r(j_0+1)$
				& $\dots$ & $r(j_0+j_1$ & $r(j_0+j_1+1)$ &  $\dots$ \\
				\hline
			\end{tabular}
		\end{center}
	\end{footnotesize}
	\caption{Polygonal arcs joining the points of the $P_2$-projection of the trajectory of $x_0$, an alternative notation. The equidistance constants $r(j) \to 0$.}\label{T:bounds2}
\end{table}

Recall~\eqref{Eq:H0mclosetoH0} and \eqref{Eq:HkmclosetoHk}. Further, since $L^m \subseteq \mathscr U^m$, \eqref{Eq:U-ainfty} shows that, for some $\delta(m)$,
\begin{equation}\label{Eq:L-ainfty}
\text{$L^m$ is $\delta(m)$-close to $a_{\infty}$ and $\delta(m) \to 0$ as $m\to \infty$}.
\end{equation}
The situation is shown in Table~\ref{T:close to head}.

\begin{table}[h]
	\begin{footnotesize}
		\begin{center}
			\begin{tabular}{|c|c|c?c?c|c|c?c?c|}
				\hline
				$H_0^0$ &  $\dots$ & $H_{p_0-1}^0$ & $L^0$ & $H_{-n_1}^1$  &  $\dots$ & $H_{p_1-1}^1$
				& $L^1$ &  $\dots$ \\
				\hline
				$P_2(\mathscr L(0))$ & $\dots$ & $P_2(\mathscr L(j_0-1))$ & $P_2(\mathscr L(j_0))$ & $P_2(\mathscr L(j_0+1))$  & $\dots$ & $P_2(\mathscr L(j_0+j_1))$
				& $P_2(\mathscr L(j_0+j_1+1))$ &  $\dots$ \\
				\hline
				\multicolumn{3}{|c?}{$(1/2^0)$-close to} & $\delta(0)$-close & \multicolumn{3}{c?}{$(1/2^1)$-close to} & $\delta(1)$-close  &  $\dots$ \\
				$\mathscr H_0$ &  $\dots$ & $\mathscr H_{p_0-1}$ & to $\{a_{\infty}\}$ & $\mathscr H_{-n_1}$  &  $\dots$ & $\mathscr H_{p_1-1}$  & to $\{a_{\infty}\}$ &  $\dots$ \\
				\hline
			\end{tabular}
		\end{center}
	\end{footnotesize}
	\caption{Polygonal arcs $P_2(\mathscr L(j)) \subseteq \pi_0$ approach the head. Here $\delta(m) \to 0$.}\label{T:close to head}
\end{table}

We are going to finish the definition of the snake. For every $j\geq 0$, the continuum $D(j)$ joining the points $x_j$ and $x_{j+1}$ will lie in a plane $\mathscr P(j)$:
\[
x_j, x_{j+1} \in D(j) \subseteq \mathscr P(j)~.
\]
With few exceptions (corresponding to the jumps), $\mathscr P(j)$ \index{$\mathscr P(j)$} will be one of the planes $\mathscr P_k^{[m]}$ introduced earlier. To define $\mathscr P(j)$, fix $j$ and take into account that $x_j\in J(m)$ for some $m\geq 0$. There are two possibilities.
\begin{enumerate}
	\item [(i)] If also $x_{j+1}$ belongs to $J(m)$, the $P_2$-projections of $x_j$ and $x_{j+1}$ are some consecutive points $a_k$ and $a_{k+1}$ from $A$. Then put $\mathscr P(j):=\mathscr P_k^{[m]}$.
	\item [(ii)] If $x_j$ is the last point of $J(m)$, then $x_{j+1}$ is the first point of $J(m+1)$ and so these two points form the $m$-th jump. Then let $\mathscr P(j)$ be that plane containing $x_j$ and $x_{j+1}$ whose intersection with $\pi_0$ is the line perpendicular to the line containing $x_j$ and $x_{j+1}$ (recall that $x_{j+1}$ is closer to $\pi_0$ than $x_j$).
\end{enumerate}

So, we have
\begin{equation}\label{Eq:LjinPj}
x_j, x_{j+1} \in \mathscr L(j)\subseteq \mathscr P(j), \quad j=0,1,\dots .
\end{equation}
For each $j$,
\[
\text{we place $D(j)$ in a `nice way' along the regular polygonal arc $\mathscr L(j)$}
\]
(see~\eqref{Eq:DmDm*} and~\eqref{Eq:every2nd}). By this we first of all mean that we choose homeomorphic copies of the sets $D(j)$, but still denoted by the same symbols $D(j)$, satisfying the following three `placing conditions' for every $j$. \index{placing conditions}
\begin{itemize}
	\item [(PL1)] The first point and the last point of $D(j)$ coincide with the first point $x_j$ and the last point $x_{j+1}$ of the polygonal arc $\mathscr L(j)$, respectively.
	
	\item [(PL2)] The set $E(j)$ \index{$E(j)$} consisting of the extremal points of $D(j)$ and the extremal points of all bricks in $D(j)$ is a subset of the polygonal arc $\mathscr L(j)$ and is ordered along $\mathscr L(j)$ in the `natural' way, i.e. as in Section~\ref{S:X} (where, however, the role of $\mathscr L(j)$ is played by a straight line segment): the first point of $D(j)$ = $x_j$ = the 1st point of the 1st brick, then the last point of the 1st brick = the first point of the 2nd brick, then the last point of the 2nd brick, etc., and finally the last point of $D(j)$.
	
	\item [(PL3)] The set $E(j) \subseteq \mathscr L(j)$ contains, among others,
		\begin{itemize}
			\item [$\blacktriangleright$] the set $E_{eq}(j)$ \index{$E_{eq}(j)$} of all the points of $\mathscr L(j)$ which are the lifts of the points from the distinguished equidistant family of points in the polygonal arc $P_2(\mathscr L(j))$,\footnote{In spite of this and in spite of the notation, note that $E_{eq}(j)$ is in general not equidistant with respect to $\ell_{\mathscr L(j)}$.} and also
			\item [$\blacktriangleright$] all the points of $\mathscr L(j)$ which are the lifts of the turning points of the polygonal arc $P_2(\mathscr L(j))$.
		\end{itemize}
\end{itemize}		
		\medskip
			
	Since $E(j)$ is by definition infinite, it has to contain also some other points from  $\mathscr L(j)$. In other words, (PL1)--PL(3) admit many ways of placing the set $D(j)$ along $\mathscr L(j)$  --- the points of the two kinds mentioned in (PL3) divide $\mathscr L(j)$ into finitely many \emph{straight line segments}, and we have to choose how many consecutive bricks of $D(j)$ will be placed along the individual segments. Of course, along each of them we have to place at least one brick and along the last of them we have to place a `tail' consisting of infinitely many bricks. Though the particular way how we do that will be specified below, in Subsection~\ref{SS:geometry}, for a moment suppose that we have already chosen the extremal points of all the bricks in the snake. This assumption allows us to describe in more details how we place the sets $D(j)$ along $\mathscr L(j)$.
	
\medskip
	
Fix $j\geq 0$. By~\eqref{Eq:every2nd} we can write $D^*(j)$ \index{$D^*(j)$} as the union of bricks, call them now $B_{j,n}$, $n=1,2,\dots$\index{$B_{j,n}$} :
\begin{equation}\label{Eq:DviaB}
D^*(j)= \bigcup_{n=1}^{\infty} B_{j,n}~.
\end{equation}
Let $f_{j,n}$ and $\ell_{j,n}$ be the first and the last point of $B_{j,n}$, respectively. The polygonal arc $P_2(\mathscr L(j))$ is one of the polygonal arcs $H^m_k$ or one of the straight line segments $L^m$, see e.g. Table~\ref{T:close to head}.

First suppose that $P_2(\mathscr L(j)) = H^m_k$ for some $m\geq 0$ and $k\in \mathbb Z$. Recall that $H^m_k$ is a subset of the big rhombus $BR_k$ and that even~\eqref{Eq:HkinRint} holds. Due to (PL3), for every $n$ we have that $[f_{j,n}, \ell_{j,n}]_{\mathscr L(j)}$ is a straight line segment in $\mathscr L(j)$ and so $P_2([f_{j,n}, \ell_{j,n}]_{\mathscr L(j)}) \subseteq H^m_k$ is a straight line segment with the endpoints $P_2(f_{j,n})$ and $P_2(\ell_{j,n})$. For every $n$, we choose a
\emph{`small' rhombus}  $SR_{j,n}$ \index{small rhombus} \index{$SR_{j,n}$}
such that the longer diagonal of $SR_{j,n}$ is the straight line segment $P_2([f_{j,n}, \ell_{j,n}]_{\mathscr L(j)})$ and the other diagonal is short enough so that the small rhombuses $SR_{j,n}$ for $n=1,2,\dots$ are pairwise disjoint, except that the consecutive rhombuses have one vertex in common. Moreover, due to~\eqref{Eq:HkinRint} we may choose this `chain' of small rhombuses in such a way that it completely lies in the big rhombus $BR_k$:
\[
\bigcup_{n=1}^{\infty} SR_{j,n} \subseteq BR_k~.
\]
\noindent By lifting the small rhombuses $SR_{j,n}$ to $\mathscr P(j)$ we get quadrilaterals $Q_{j,n} \subseteq \mathscr P(j)$ \index{$Q_{j,n}$}(in fact  parallelograms, but not necessarily rhombuses).\footnote{We do not introduce any notation for the lifts of the big rhombuses $BR_k$.} They are pairwise disjoint except that two consecutive quadrilaterals have one vertex in common, because the small rhombuses have such a property. See Figure~\ref{fig:SRinBR}.

\begin{figure}[h]
	\centering
	\includegraphics[width=14cm]{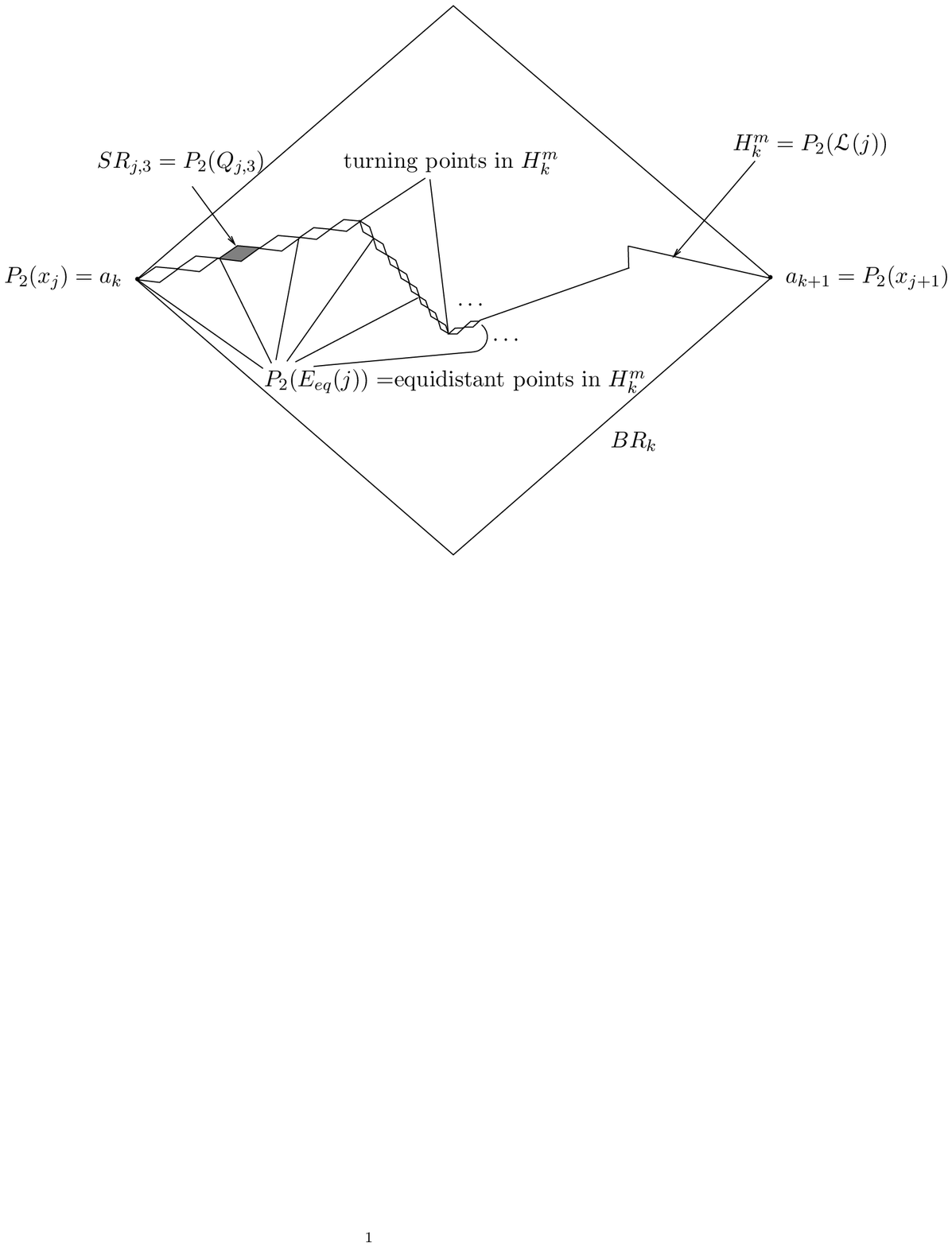}\\
	\caption{The chain of small rhombuses $SR_{j,n}$, $n=1,2,\dots$, placed along the po\-ly\-go\-nal arc $P_2(\mathscr L(j)) = H^m_k$ in the big rhombus $BR_k$. The lifts of $SR_{j,n}$
		to $\mathscr P(j)$ are some quadrilaterals $Q_{j,n}$ placed along the polygonal arc $\mathscr L(j)$ joining the points $x_j$ and $x_{j+1}$. Notice that, by (PL3) and the definition of small rhombuses, every turning point of $H^m_k$ is a vertex of a small rhombus and the same is true for every point from the distinguished family of equidistant points in $H^m_k$.}\label{fig:SRinBR}
\end{figure}

Now suppose that $P_2(\mathscr L(j)) = L^m$  for some $m\geq 0$ (we say that $D(j)$ is going to be placed along the straight line segment $\mathscr L(j)$ or along the $m$-th jump). Then we proceed analogously. Again we choose, now along the straight line segment $L^m$,  small rhombuses $SR_{j,n}$ which are pairwise disjoint, except that two consecutive rhombuses have one point in common. Here $L^m$ corresponds to a jump, so this chain of small rhombuses is between two chains of small rhombuses corresponding to the neighbouring sets $D(j-1)$ and $D(j+1)$, and so placed along polygonal arcs in some big rhombuses as already explained above. Since the jump is almost horizontal (in the sense explained below the formula~\eqref{eq:qth jump}), we can choose the small rhombuses along $L^m$  in such a way that they are disjoint from those corresponding to $D(j-1)$ and $D(j+1)$. Again, by lifting $SR_{j,n}$ to $\mathscr P(j)$ we get quadrilaterals $Q_{j,n} \subseteq \mathscr P(j)$ (one can see that in this case they are in fact rhombuses) which are pairwise disjoint except that two consecutive quadrilaterals have one vertex in common. See Figure~\ref{fig:SRjump}.

\begin{figure}[h]
	\centering
	\includegraphics[width=14cm]{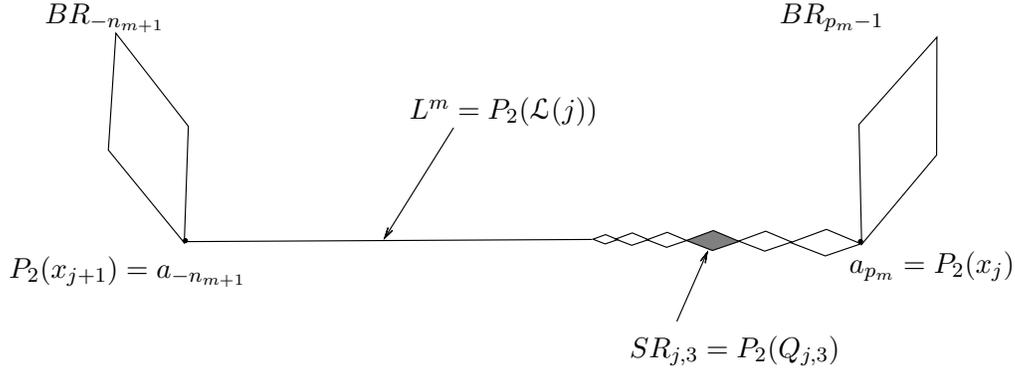}\\
	\caption{The chain of small rhombuses along $P_2(\mathscr L(j))$. It is between two big rhombuses containing their own chains of small rhombuses.
		See Table~\ref{T:xP2} for the notation of the endpoints of $L_m$.}\label{fig:SRjump}
\end{figure}

\medskip

Thus, for every $j\geq 0$, $n\geq 1$ and corresponding brick $B_{j,n}$ as in~\eqref{Eq:DviaB}, we choose a small rhombus $SR_{j,n} \ \subseteq \pi_0$ and a quadrilateral $Q_{j,n} \subseteq \mathscr P(j)$. The consecutive points of the trajectory $x_0, x_1, \dots$ are joined by chains of quadrilaterals, so we have in fact a chain of chains of quadrilaterals. According to our choice of them, and the fact that the distances of the points $x_j$ from the plane $\pi_0$ strictly decrease, it is obvious that all the quadrilaterals are pairwise disjoint except that if two of them are consecutive (and so lie in one of those chains of quadrilaterals) then they have one point in common.\footnote{The family of all small rhombuses $SR_{j,n}$ in $\pi_0$ of course does not have this property. For instance, every point $a_k$, $k\in \mathbb Z$ belongs to infinitely many small rhombuses.}

Finally, we are ready to formulate the fourth `placing condition'.
\begin{itemize}
	\item [(PL4)] Every brick $B_{j,n}$, see~\eqref{Eq:DviaB}, is placed in the corresponding quadrilateral $Q_{j,n}\subseteq \mathscr P(j)$ in such a way that the first point and the last point of $B_{j,n}$ coincide with the first point and the last point of $Q_{j,n}$, respectively (these are the lifts of the first point and the last point of $SR_{j,n}$, respectively).
\end{itemize}	

Note that, by (PL4), all the bricks in the snake are pairwise disjoint with the exception that two neighboring bricks intersect at one point (because the quadrilaterals containing them have such a property).

Clearly, for every $j$ and $n$, $\diam SR_{j,n} = \diam P_2(B_{j,n}) = d(P_2(f_{j,n}), P_2(\ell_{j,n}))$. If the small rhombuses $SR_{j,n}$ are chosen narrow enough, i.e. with shorter diagonals short enough, then also $\diam Q_{j,n} = \diam B_{j,n} = d(f_{j,n},\ell_{j,n})$ and  even the following is true. We formulate it as the fifth `placing condition'.

\begin{itemize}	
	\item [(PL5)] We choose small rhombuses narrow enough to ensure the following. If, for some $j\geq 0$, $\mathscr D$ is a union of (finitely or infinitely many) consecutive bricks of $D^*(j)$, $\mathscr {SR_D}$ and $\mathscr {Q_D}$ are the unions of small rhombuses and quadrilaterals, respectively, corresponding to the bricks in $\mathscr D$, and $\mathscr {E_D}$ is the set of extremal points of the bricks in $\mathscr D$, then
	\[
	\diam \mathscr {SR_D} = \diam P_2(\mathscr D) = \diam P_2(\mathscr {E_D})
	\]
	and
	\[
	\diam \mathscr {Q_D} = \diam \mathscr D = \diam \mathscr {E_D}~.
	\]
\end{itemize}	
Indeed, no matter whether we are in the plane $\pi_0$ or in the plane $\mathscr P(j)$ (i.e. no matter whether we want to prove the first formula or the second one), a simple geometrical argument shows that the diameter of a polygonal arc is realized as the diameter of the finite set of the endpoints of all its links. Then one can see that if the small rhombuses are narrow enough, both formulas in (PL5) work.

\medskip

Having described the head and the snake, we finally define our space, a candidate for the equality $S(X)=\{0,\log 2\}$, by
\begin{equation}\label{X-definition-log2}
X= \text{head } \sqcup \text{ snake} = \mathscr A_0 \sqcup \bigcup_{m=0}^{\infty}D(m)= \mathscr A_0 \sqcup \bigcup_{m=0}^{\infty}D^*(m).
\end{equation}

\begin{lem}\label{L:Xcpct}
The space $X\subseteq \mathbb R^3$ defined by~\eqref{X-definition-log2} is a one-dimensional continuum.	
\end{lem}

\begin{proof}
Due to the `closedness conditions' which can be seen in Table~\ref{T:close to head} and taking into account~\eqref{Eq:approachpi0}, our construction ensures that all those cluster points of the snake which are in $\pi_0$ belong to the head $\mathscr A_0$. Hence $X$ is compact.\footnote{We do not claim that the closure of the snake is the \emph{whole} $X$,
	as it was the case in Section~\ref{S:X}, since we do not know whether the snake approaches \emph{every} point in every $\mathscr H_k$. However, with a little care we could modify the
	construction to get this -- instead of just using Lemma~\ref{L:arc in nbhd} for joining the endpoints of $\mathscr H_k$, we would combine it with Lemma~\ref{L:chains}.}
Moreover, $X$ is one-dimensional and connected for the same reasons as the space $X$ in Section~\ref{S:X}.
\end{proof}

\subsection{A map $G\colon X\to X$ and a need to specify the geometry of $X$}\label{SS:defG}
The proof that  $S(X)=\{0,\log 2\}$ will be similar to that in Section~\ref{S:X} (though we will still need to specify the geometry of our space $X$ from~\eqref{X-definition-log2} in more details). This means that
\begin{itemize}
	\item we construct a map $G\colon X\to X$, in fact a continuous extension of the map $T \colon X_1\to X_1$ with $h^*(T)=\log 2$, such that also $h^*(G)=\log 2$, and
	\item we show that, for every continuous map $F\colon X\to X$, either $h^*(F)=0$ or $h^*(F)=h^*(G)$.
\end{itemize}
We define, for the space $X$ from~\eqref{X-definition-log2} and the map $T$ from~\eqref{Eq:T1log2}, the map $G\colon X \rightarrow X$ as follows:
\begin{itemize}
		\item $G|_{X_1}=T$;
		\item for every $k\in \mathbb Z$, $G|_{\mathscr H_k}$ is the unique homeomorphism $\mathscr H_k \to \mathscr H_{k+1}$ (i.e.  $G|_{\mathscr H_k} = Z_{k,k+1}|_{\mathscr H_k} = (Z_{k+1}\circ Z_k^{-1})|_{\mathscr H_k}$);
		\item for every $m \geq 0$, $G|_{D(m)}$ is the unique continuous surjective map $D(m) \to D(m+1)$ (see Lemma~\ref{L:mapDD} and Figure~\ref{F:sur}).
\end{itemize}

In Section~\ref{S:X} the map $G$ was trivially continuous and $h^*(G)=\infty$. Contrary to this, now we are facing the following two problems.
\begin{itemize}
\item While our map $G$, just defined above, is obviously continuous at all points of the snake, we still have not proved that it is  continuous also at all points of the head. (The restrictions of $G$ to the head and to the snake are clearly continuous, but we have not proved that these two continuous parts of $G$ fit together to produce a continuous map on $X$. In Section~\ref{S:X} the continuity was trivial because $G$ on the head was the identity and the diameters of $D_m$ tended to zero. Now $G$ on the head is not identity and the diameters of $D(m)$ do not tend to zero.)
\item While it is trivial that $h^*(G)\geq h^*(T)=\log 2$, this inequality alone does not imply that $h^*(G) = \log 2$. (In Section~\ref{S:X}, see Lemma~\ref{L:existsG}, we were in better situation, since $h^*(G)\geq h^*(T)= \infty$ trivially implies that $h^*(G) = \infty$.)
\end{itemize}

To cope with these two problems, we are going to specify the geometry of $X$ in more details, by adding some additional requirements to our construction of $X$.

\subsection{Specification of geometry of $X$}\label{SS:geometry}
We have adopted placing conditions (PL1)--(PL5) for placing $D(j)$ along $\mathscr L(j)$. But still there is much freedom for doing that. Consider the sets $E_{eq}(j) \subseteq E(j)$ in $\mathscr L(j)$ from (PL3). The set $E_{eq}(j)$ is already chosen by our construction, the larger set $E(j)$ will depend on how we now place the bricks from $D(j)$ along $\mathscr L(j)$.\footnote{Here and below we for simplicity sometimes speak on $D(j)$, though it would perhaps be more precise to speak on $D^*(j)$.}

Since the set $E_{eq}(j)$ is obtained as a lift of the distinguished family of equidistant points in $P_2(\mathscr L(j))$, we will sometimes call it a distinguished family of points in $\mathscr L(j)$ (we do not say ``equidistant" points, since the family need not be equidistant in $\mathscr L(j)$). The set $E_{eq}(j)$ divides $\mathscr L(j)$ into finitely many distinguished (polygonal) subarcs. In the sequel we will call them \emph{lequi-subarcs of $\mathscr L(j)$} \index{lequi-subarcs of $\mathscr L(j)$} (meaning ``lifts of equi-subarcs of $P_2(\mathscr L(j))$"). We order this family of lequi-subarcs in the natural way, the first of them being the one containing $x_j$, the first point of $\mathscr L(j)$.
It follows from (EP2) and (EP7) that each of these lequi-subarcs contains at most one turning point. Note the following, which follows from (PL3).
\begin{enumerate}
	\item [(PL3a)] If a lequi-subarc \emph{does} contain a (unique) turning point, in view of (PL3) we will have to place \emph{at least two bricks} of $D(j)$ along this subarc. Moreover, we will have to do it in such a way that this turning point will coincide with the last point of some brick and with the first point of the next brick. However, we will not need to specify this pair of consecutive bricks --- in case of placing more than two bricks along this subarc we have thus some freedom in choosing this pair.
	\item [(PL3b)] On the other hand, if a lequi-subarc contains no turning point, (PL3) gives us a freedom --- we are allowed to place either only one brick or more bricks along this subarc.
\end{enumerate}

We are going to specify the positions, along the polygonal arcs $\mathscr L(j)$, of the bricks from $D(j)$. When speaking on positions of bricks, we are not interested in all the details related to them. We in fact have in mind the positions of the corresponding quadrilaterals (basically just the positions of the extremal points of their longer diagonals), see (PL4). The reason is that everything what is important for us about the position of a brick is determined by the quadrilateral containing it. We hope that no misunderstanding will arise if we speak just on positions of bricks.

We start with $j=0$, i.e. with $\mathscr L(0)$ and $D(0)$. The distinguished family $E_{eq}(0)$ of points in $\mathscr L(0)$ divides $\mathscr L(0)$ into finitely many lequi-subarcs.
We place the first brick of $D(0)$ along the first of these lequi-subarcs. For each of the other lequi-subarcs except the last one, let the number of bricks placed along that lequi-subarc be the same as the total number of bricks placed along all the previous lequi-subarcs of $\mathscr L(0)$. So, the numbers of bricks placed along the lequi-subarcs, except the last one, are powers of two, starting with $1,1,2,4$. All the infinitely many remaining bricks of $D(0)$ are placed along the last lequi-subarc of $\mathscr L(0)$. This is in accordance with (PL3a) and (PL3b) because the first two lequi-subarcs, being lifts of straight line segments by (EP3), are straight line segments.

We will shortly say that the bricks of $D(0)$ are placed along $\mathscr L(0)$ according to \emph{the rule ``1, sums, $\infty$"}. We divide  $D(0)$ into three parts: $D^{\rm front}(0)$, $D^{\rm middle}(0)$ and $D^{\rm end}(0)$ as follows: \index{$D^{\rm front}(0), D^{\rm middle}(0), D^{\rm end}(0)$}
\begin{itemize}
  \item $D^{\rm front}(0)$ is the union of bricks (in fact one brick) placed along the first lequi-subarc of $\mathscr L(0)$,
  \item $D^{\rm end}(0)$ is the union of (infinitely many) bricks placed along the last lequi-subarc of $\mathscr L(0)$, together with the last point of $\mathscr L(0)$,
  \item $D^{\rm middle}(0)$ is the union of (finitely many) other bricks in $D(0)$.
\end{itemize}
Notice that these sets are not pairwise disjoint, because $D^{\rm middle}(0)$ intersects each of the other two sets in one point.

Recall that $r(0)=r_0^0$ (see Table~\ref{T:bounds} and the alternative notation in Table~\ref{T:bounds2}). We may assume that
\begin{equation}\label{Eq:P2front0}
P_2(D^{\rm front}(0)) = \overline{B}_{d}(a_0, r(0)) \cap P_2(D(0)) \quad \text{and} \quad
P_2(D^{\rm end}(0)) = \overline{B}_{d}(a_1, r(0)) \cap P_2(D(0)),
\end{equation}
though (EP1), (EP3) and (EP4) do not automatically ensure this. If necessary, we just choose $n_0^0$ larger enough, i.e. $r_0^0$ smaller enough, to get the first and the last equi-subarcs in (EP3) sufficiently short to get these properties.
Finally, denote by $N_0$ the number of bricks in $D^{\rm front}(0) \cup D^{\rm middle}(0)$.

\medskip

Now consider $0<j<j_0$. We have $\mathscr L(j) = Z_j^{[0]} (\mathscr L(0))$, see Table~\ref{T: L(H)}. The same is true for the distinguished families of points in $\mathscr L(0)$ and $\mathscr L(j)$, i.e. $E_{eq}(j) = Z_j^{[0]} (E_{eq}(0))$, and so the cardinalities of these families are the same (the number of lequi-subarcs in $\mathscr L(0)$ is the same as in $\mathscr L(j)$ for $0<j<j_0$). Following the pattern from the case $j=0$, we place $D(j)$ along the polygonal arc $\mathscr L(j)$ according to the rule ``1, sums, $\infty$". Here $D(j)$ is not a copy of $D(0)$, it is given by~\eqref{Eq:every2nd}.\footnote{This is an exception from the `rule' that for passing from $j=0$ to $0<j<j_0$ one just needs to apply the similitude $Z_j$ or its lift $Z_j^{[0]}$.} In the same way as above for $j=0$, we define the sets $D^{\rm front}(j)$, $D^{\rm middle}(j)$ and $D^{\rm end}(j)$. We may also assume (see (EP8), (EP9) and~\eqref{Eq:P2front0}) that
\begin{equation}\label{Eq:P2frontj}
P_2(D^{\rm front}(j)) = \overline{B}_{d}(a_j, r(j)) \cap P_2(D(j)) \quad \text{and} \quad
P_2(D^{\rm end}(j)) = \overline{B}_{d}(a_{j+1}, r(j)) \cap P_2(D(j)).
\end{equation}

\medskip

Now consider $j=j_0$. Then $\mathscr L(j_0)$ corresponds to the $0$-th jump, see Table~\ref{T: L(H)}, and it is a straight line segment with the cardinality of $E_{eq}(j_0)$ equal to the cardinality of $E_{eq}(j)$ for $0\leq j < j_0$ (the number of lequi-subarcs in $\mathscr L(j_0)$ is the same as in $\mathscr L(j)$ for $0\leq j<j_0$).  Still according to the rule ``1, sums, $\infty$", we place $D(j_0)$, see~\eqref{Eq:every2nd}, along $\mathscr L(j_0)$. In the same way as above, we define the sets $D^{\rm front}(j_0)$, $D^{\rm middle}(j_0)$ and $D^{\rm end}(j_0)$. Since $\mathscr L(j_0)$ is a straight line segment, an analogue of~\eqref{Eq:P2front0} again holds.

\medskip

Notice that for every $j$ considered so far, i.e. for $0\leq j \leq j_0$, the set $D^{\rm front}(j) \cup D^{\rm middle}(j)$ consists of $N_0$ bricks, as in case $j=0$. The situation is summarized in Table~\ref{D_0 along} (where the lequi-subarcs are called just subarcs).
 \begin{table}[h]
 	\begin{footnotesize}
 		\begin{center}
 			\begin{tabular}{|c|c|c|}
 				\hline
 				\multicolumn{3}{|c|}{$D(j)$, \quad $0\leq j \leq j_0$}  \\
 				\hline
 				$D^{\rm front}(j)$ & $D^{\rm middle}(j)$ & $D^{\rm end}(j)$ \\
 				\hline
 				1 brick along 1st lequi-subarc     &  finitely many bricks   &   $\infty$ many bricks along last lequi-subarc                  \\
 				\hline
 				\multicolumn{2}{|c|}{$N_0$ bricks}    \\
 				\cline {1-2}
 			\end{tabular}
 		\end{center}
 	\end{footnotesize}
 	\caption{$D(j)$ along the polygonal arc $\mathscr L(j)$ for $j$ such that $x_j\in J(0)$. The rule is ``1, sums, $\infty$".} \label{D_0 along}
 \end{table}

We have thus put restrictions on the geometry of the continua $D(j)$ joining the consecutive points of the $0$-th jump level $J(0)$ and the $0$-th jump.

\medskip

Now we are going to work with the $1$-st jump level $J(1)$ and the $1$-th jump. We start with placing $D(t_1)$ where $x_{t_1},x_{t_1+1}$ is the unique pair of points in $J(1)$ with $P_2(x_{t_1})=a_0$ and $P_2(x_{t_1+1})=a_1$.
Similarly as we have distributed the bricks of $D(0)$ along $\mathscr L(0)$ by the rule ``1, sums, $\infty$", we proceed with $D(t_1)$, though with a different rule. The family $E_{eq}(t_1)$  divides $\mathscr L(t_1)$ into lequi-subarcs. We place the bricks of $D(t_1)$ along them, but now according to \emph{the rule ``$N_0$, sums, $\infty$"}. This means the following. We place the first $N_0$ bricks of $D(t_1)$ along the first of these lequi-subarcs. Further, for each of the other lequi-subarcs except the last one, let the number of bricks placed along that lequi-subarc be the same as the total number of bricks placed along all the previous lequi-subarcs of $\mathscr L(t_1)$. All the infinitely many remaining bricks of $D(t_1)$ are placed along the last lequi-subarc of $\mathscr L(t_1)$. We also divide $D(t_1)$ into three parts: $D^{\rm front}(t_1)$, $D^{\rm middle}(t_1)$ and $D^{\rm end}(t_1)$ as follows:
\begin{itemize}
  \item $D^{\rm front}(t_1)$ is the union of bricks (in fact $N_0$ bricks) placed along the first lequi-subarc of $\mathscr L(t_1)$,
  \item $D^{\rm end}(t_1)$ is the union of (infinitely many) bricks placed along the last lequi-subarc of $\mathscr L(t_1)$, together with the last point of $\mathscr L(t_1)$,
  \item $D^{\rm middle}(t_1)$ is the union of (finitely many) other bricks of $D(t_1)$.
\end{itemize}
Recall that $r(t_1)=r_0^1$. For analogous reasons as in case $j=0$, we may assume that
\begin{equation}\label{Eq:P2frontt1}
P_2(D^{\rm front}(t_1)) = \overline{B}_{d}(a_0, r(t_1)) \cap P_2(D(t_1)) \quad \text{and} \quad
P_2(D^{\rm end}(t_1)) = \overline{B}_{d}(a_1, r(t_1)) \cap P_2(D(t_1)).
\end{equation}
Finally, denote by $N_1$ the number of bricks in $D^{\rm front}(t_1) \cup D^{\rm middle}(t_1)$.

\medskip

In the same way, according to the rule ``$N_0$, sums, $\infty$" we place $D(j)$ along the polygonal arc $\mathscr L(j)$ not only for $j=t_1$, but also for all other  $j_0+1 \leq  j < j_0+j_1+1$ (so, $N_0$ bricks are placed along the 1st lequi-subarc etc.). For each such $j$, we define the sets $D^{\rm front}(j)$, $D^{\rm middle}(j)$ and $D^{\rm end}(j)$ \index{$D^{\rm front}(j), D^{\rm middle}(j), D^{\rm end}(j)$} in the same way as above for $j=t_1$. We also have an analogue of~\eqref{Eq:P2frontt1}:
\begin{equation}\label{Eq:P2frontakot1}
P_2(D^{\rm front}(j)) = \overline{B}_{d}(P_2(x_j), r(j)) \cap P_2(D(j)) \quad \text{and} \quad
P_2(D^{\rm end}(j)) = \overline{B}_{d}(P_2(x_{j+1}), r(j)) \cap P_2(D(j)).
\end{equation}
Further, still according to the rule ``$N_0$, sums, $\infty$", we place  $D(j_0+j_1+1)$ along the straight line segment $\mathscr L(j_0+j_1+1)$ corresponding to the $1$-st jump (i.e. $N_0$ bricks placed along the 1st lequi-subarc etc.), we define the sets $D^{\rm front}(j_0+j_1+1)$, $D^{\rm middle}(j_0+j_1+1)$ and $D^{\rm end}(j_0+j_1+1)$ and we again have an analogue of~\eqref{Eq:P2frontt1}.

\medskip

Notice that for every $j_0+1 \leq  j \leq j_0+j_1+1$, the set $D^{\rm front}(j) \cup D^{\rm middle}(j)$ consists of $N_1$ bricks.
The situation is shown in Table~\ref{D(t_1) along}.
\begin{table}[h]
	\begin{footnotesize}
		\begin{center}
			\begin{tabular}{|c|c|c|}
				\hline
				\multicolumn{3}{|c|}{$D(j)$, \quad $j_0+1 \leq  j \leq j_0+j_1+1$}  \\
				\hline
				$D^{\rm front}(j)$ & $D^{\rm middle}(j)$ & $D^{\rm end}(j)$ \\
				\hline
				$N_0$ bricks along 1st lequi-subarc     &  finitely many bricks   &   $\infty$ many bricks  along last lequi-subarc                 \\
				\hline
				\multicolumn{2}{|c|}{$N_1$ bricks}    \\
				\cline {1-2}
			\end{tabular}
		\end{center}
	\end{footnotesize}
	\caption{$D(j)$ along the polygonal arc $\mathscr L(j)$ for $j$ such that $x_j\in J(1)$. The rule is ``$N_0$, sums, $\infty$".} \label{D(t_1) along}
\end{table}

We have thus put restrictions on the geometry of the continua $D(j)$ joining the consecutive points of the $1$-st jump level $J(1)$ and the $1$-st jump.

\medskip

We continue by induction. Suppose that, for $s\geq 2$, we have already put the restrictions on the geometry of the continua $D(j)$ joining the consecutive points of the jump levels $J(0), \dots, J(s-1)$ and the $0$-th, $\dots$, $(s-1)$-st jumps. Denote by $x_{t_s},x_{t_s+1}$ the unique pair of points in $J(s)$ with $P_2(x_{t_s})=a_0$ and $P_2(x_{t_s+1})=a_1$, and consider the polygonal arc $\mathscr L(t_s)$ joining $x_{t_s}$ and $x_{t_s+1}$.
The family $E_{eq}(t_s)$ divides $\mathscr L(t_s)$ into lequi-subarcs. We place $D(t_s)$ along $\mathscr L(t_s)$ according to the rule ``$N_{s-1}$, sums, $\infty$", we define the sets $D^{\rm front}(t_s)$, $D^{\rm middle}(t_s)$, $D^{\rm end}(t_s)$ and we require also an analogue of~\eqref{Eq:P2frontt1}. Finally, denote by $N_s$ the number of bricks in $D^{\rm front}(t_s) \cup D^{\rm middle}(t_s)$.

\medskip

In the same way, according to the rule ``$N_{s-1}$, sums, $\infty$" we place $D(j)$ along the polygonal arc $\mathscr L(j)$ not only for $j=t_s$, but also for any other  $k_s \leq  j < k_s +j_s$ (so, $N_{s-1}$ bricks are placed along the 1st lequi-subarc etc.). For each such $j$, we define the sets $D^{\rm front}(j)$, $D^{\rm middle}(j)$ and $D^{\rm end}(j)$ in the same way as above for $j=t_s$. We have also the analogues of~\eqref{Eq:P2frontt1} for
$P_2(D^{\rm front}(j))$ and $P_2(D^{\rm end}(j))$, with $a_0$ and $a_1$ replaced by the first point and the last point of $P_2(\mathscr L(j))$, respectively, and with radius $r(j)$.
Further, still according to the rule ``$N_{s-1}$, sums, $\infty$", we place  $D(k_s+j_s)$ along the straight line segment $\mathscr L(k_s+j_s)$ corresponding to the $s$-th jump (i.e. $N_{s-1}$ bricks placed along the 1st lequi-subarc etc.), we define the sets $D^{\rm front}(k_s+j_s)$, $D^{\rm middle}(k_s+j_s)$ and $D^{\rm end}(k_s+j_s)$ and we require also an analogue of~\eqref{Eq:P2frontt1}.

Notice that for every $k_s \leq  j \leq k_s +j_s$, the set $D^{\rm front}(j) \cup D^{\rm middle}(j)$ consists of $N_s$ bricks.
The situation is illustrated in Table~\ref{D(t_s) along}.
\begin{table}[h]
	\begin{footnotesize}
		\begin{center}
			\begin{tabular}{|c|c|c|}
				\hline
				\multicolumn{3}{|c|}{$D(j), \quad k_s \leq  j \leq k_s +j_s$}  \\
				\hline
				$D^{\rm front}(j)$ & $D^{\rm middle}(j)$ & $D^{\rm end}(j)$ \\
				\hline
				$N_{s-1}$ bricks along 1st lequi-subarc     &  finitely many bricks   &   $\infty$ many bricks along last lequi-subarc                  \\
				\hline
				\multicolumn{2}{|c|}{$N_{s}$ bricks}    \\
				\cline {1-2}
			\end{tabular}
		\end{center}
	\end{footnotesize}
	\caption{$D(j)$ along the polygonal arc $\mathscr L(j)$ for $j$ such that $x_j\in J(s)$. The rule is ``$N_{s-1}$, sums, $\infty$".} \label{D(t_s) along}
\end{table}

We have thus put restrictions on the geometry of the continua $D(j)$ joining the consecutive points of the $s$-th jump level $J(s)$ and the $s$-th jump. The induction is finished.

\subsection{Properties of the map $G$}\label{SS:propertiesG}
In Subsection~\ref{SS:defG} we have defined a map $G\colon X\to X$. Recall that $G|_{X_1}=T$, for every integer $k$ the map $G|_{\mathscr H_k}$ is the unique homeomorphism $\mathscr H_k \to \mathscr H_{k+1}$, i.e. $Z_{k,k+1}|_{\mathscr H_k}$, and, for every $m\geq 0$, $G|_{D^*(m)}$ is the unique \emph{surjective} continuous map $D^*(m) \to D^*(m+1)$.

By~\eqref{Eq:every2nd}, the set $D^*(m)$ is the union of some bricks $B_1, B_2, B_3, B_4, \dots$ and $D^*(m+1)$ is the union of copies of $B_2, B_4, B_6, B_8, \dots$. The map $G$ maps the 1st brick of $D^*(m)$ to the first point of the 1st brick of $D^*(m+1)$, the 2nd brick of $D^*(m)$ onto the 1st brick of $D^*(m+1)$ and, in general,
\begin{equation}\label{Eq:nthBrickTo}
G(\text{$n$-th brick of $D^*(m)$}) \subseteq \text{$\left \lfloor \frac{n+1}{2}\right \rfloor$-th brick of $D^*(m+1)$}
\end{equation}
(if $n$ is even we have equality and if $n$ is odd then the left side is a singleton, namely the first point of the brick on the right side).

We introduce the following notations:
\begin{itemize}
	\item $B^{(j)}_i$ is the union of all bricks placed along the $i$-th lequi-subarc of $\mathscr L(j)$,
	\item ${b}^{(j)}_i$ is the number of all bricks placed along the $i$-th lequi-subarc of $\mathscr L(j)$,
	\item $Q^{(j)}_i$ is the union of all quadrilaterals placed along the $i$-th lequi-subarc of $\mathscr L(j)$,
	\item $SR^{(j)}_i$ is the union of all small rhombuses placed along the $i$-th equi-subarc of $P_2(\mathscr L(j))$.
\end{itemize}

Let $j\geq 0$ be such that both $x_j$ and $x_{j+1}$ belong to $J(s)$, $s\geq 0$. Then the numbers of lequi-subarcs in $\mathscr L(j)$ and in $\mathscr L(j+1)$ are the same, denote them by $t$. Moreover, if we put $N_{-1}=1$, we have (see Table~\ref{D(t_s) along})
\[
{b}^{(j)}_1 = {b}^{(j)}_2 = N_{s-1}, \, {b}^{(j)}_3 = 2N_{s-1}, \, \dots , \, {b}^{(j)}_{t-1} = 2^{t-3}N_{s-1} \, \text{ and }\, {b}^{(j)}_{t} = \infty
\]
and exactly the same equalities for ${b}^{(j+1)}_i$, $i=1,\dots,t$.
A simple computation, using this and~\eqref{Eq:nthBrickTo}, shows that for such $j$ we have
\begin{alignat}{1}
G(B^{(j)}_1 \cup B^{(j)}_2)  & \, = \, B^{(j+1)}_1, \notag \\
G(B^{(j)}_k)    & \, = \, B^{(j+1)}_{k-1}, \quad k=3,4,\dots, t-1, \label{Eq:GAL}  \\
G(B^{(j)}_t) & \, = \,  B^{(j+1)}_{t-1} \cup B^{(j+1)}_{t}~.  \notag
\end{alignat}

\begin{lem}\label{L:fronttofront}
     Let $x_j$ and $x_{j+1}$ belong to the same set $J(s)$. Then
    \begin{alignat}{1}
     G(D^{\rm front}(j)) \,  & \, \subseteq D^{\rm front}(j+1), \notag \\
     G(D^{\rm middle}(j)) \,  & \, \subseteq D^{\rm front}(j+1) \cup D^{\rm middle}(j+1), \notag  \\
     G(D^{\rm end}(j)) \,  & \, \subseteq D^{\rm middle}(j+1) \cup D^{\rm end}(j+1)~.  \notag
     \end{alignat}
\end{lem}

\begin{proof}
The number of lequi-subarcs in $\mathscr L(j)$ and in $\mathscr L(j+1)$ are the same and, as mentioned above~\eqref{Eq:GAL}, the number of bricks placed along the i-th lequi-subarc is the same for $\mathscr L(j)$ and for $\mathscr L(j+1)$. Now use~\eqref{Eq:GAL}.	
\end{proof}

\begin{lem}\label{L:G-log2-Dm-3part}
	Let $x_m, x_{m+1}$ either belong to $J(i)$ or form the $i$-th jump, and let $x_k, x_{k+1}$ either belong to $J(s)$ or form the $s$-th jump for some $s>i$. Let
	$$
	x \in D^{\rm front}(m)\cup D^{\rm middle}(m)
	$$
	and let $n$ be the positive integer for which $G^n(x) \in D^*(k)$. Then
	$$
	G^n(x) \in D^{\rm front}(k) \text{ \ and so \ } d(P_2(G^nx),P_2(x_k)) \le r(k).
	$$
\end{lem}
\begin{proof}
	Since $s>i$, the number of bricks in $D^{\rm front}(k)$ is greater than or equal to the number of bricks in $D^{\rm front}(m)\cup D^{\rm middle}(m)$ (if $s=i+1$, they are equal). Since $G$ preserves the order of bricks in the snake,\footnote{We speak on non-strict order since two consecutive bricks can be mapped to one brick, one of those two bricks being mapped to a point.} the assumption $x\in D^{\rm front}(m)\cup D^{\rm middle}(m)$ implies $G^n(x) \in D^{\rm front}(k)$. The inequality in the lemma then follows from~\eqref{Eq:P2front0} or one of its  analogues mentioned below it, e.g. \eqref{Eq:P2frontj}, \eqref{Eq:P2frontt1}, \eqref{Eq:P2frontakot1}.
\end{proof}

From the previous two lemmas we immediately get the following corollary.

\begin{cor}\label{C:fronttofront}
	For every $j\geq 0$, $G(D^{\rm front}(j)) \subseteq D^{\rm front}(j+1)$.
\end{cor}

\begin{lem}\label{L:Gprecont}
	Let $x_j, x_{j+1}, x_{j+2} \in J(s)$ for some $j\geq 0$ and $s\geq 0$ and so, for some $k\in \mathbb Z$, $P_2(x_j)=a_k$, $P_2(x_{j+1})=a_{k+1}$, $P_2(x_{j+2})=a_{k+2}$.
	Let $y\in D^*(j)$ and so $G(y)\in D^*(j+1)$. Then
	\[
	d(Z_{k,k+1}(P_2(y)), P_2(G(y))) \leq 2 \cdot r(j+1).
	\]
\end{lem}

\begin{proof} The arc $\mathscr L(j)$ joins $x_j$ and $x_{j+1}$, the arc $P_2(\mathscr L(j))$ joins $a_k$ and $a_{k+1}$. Similarly, $\mathscr L(j+1)$ joins $x_{j+1}$ and $x_{j+2}$,
	$P_2(\mathscr L(j+1))$ joins $a_{k+1}$ and $a_{k+2}$. Each of the four arcs is divided into the same number of distinguished subarcs (lequi-subarcs or equi-subarcs).
	
	Assume that $y\in {B}^{(j)}_i$, i.e. $y$ belongs to some brick placed along the $i$-th distinguished subarc of $\mathscr L(j)$. By~\eqref{Eq:GAL},
	\begin{equation}\label{Eq:Gy}
	G(y) \in B^{(j+1)}_{i-1} \cup B^{(j+1)}_{i}
	\end{equation}
	where $B^{(j+1)}_{i-1}$ should be replaced by the empty set provided $i=1$. Suppose that $i >1$ (for $i=1$ the proof is almost the same).
	
	Since $y\in {B}^{(j)}_i$, we have $y \in {Q}^{(j)}_i$. Hence $P_2(y) \in {SR}^{(j)}_i$ and so $Z_{k,k+1}(P_2(y)) \in {SR}^{(j+1)}_i$.
	On the other hand, by~\eqref{Eq:Gy} we have $G(y) \in Q^{(j+1)}_{i-1} \cup Q^{(j+1)}_{i}$ and so $P_2(G(y)) \in SR^{(j+1)}_{i-1} \cup SR^{(j+1)}_{i}$.
	Thus,
	$$
	d(Z_{k,k+1}(P_2(y)), P_2(G(y))) \leq \diam \left ( SR^{(j+1)}_{i-1} \cup SR^{(j+1)}_{i} \right ) \leq \diam SR^{(j+1)}_{i-1} + \diam SR^{(j+1)}_{i}.
	$$
	The equidistance constant of $P_2(\mathscr L(j+1))$ is $r(j+1)$, see e.g. Table~\eqref{T:bounds2}. Further, the euclidean distance of points on a polygonal arc is less than or equal to their distance along the arc. Therefore, by (PL5), each of the diameters on the right-hand side is less than or equal to $r(j+1)$.
\end{proof}

\begin{lem}\label{L:Gcont}
The map $G\colon X\to X$ is continuous.
\end{lem}

\begin{proof}
The restrictions of $G$ to the head and to the snake are both continuous by the construction of~$G$. Since the snake is open in $X$ and the head is not, it remains to prove that $G\colon X\to X$ is continuous at every point of the head. So, let $z$ be a point from the head and let $y_k \in D^*(i_k)$, $k=0, 1, \dots$  be a sequence of points from the snake converging to $z$. We want to show that $G(y_k) \to G(z)$.

First suppose that $z=a_{\infty}$. Then, by the construction, the sequence $G(D^*(i_k)) = D^*(i_k +1)$ converges to $a_{\infty}$. Hence $G(y_k)\to a_{\infty}$. Since $a_{\infty}$ is a fixed point of $G$ , the continuity of $G$ at $a_{\infty}$ follows.

Now assume that $z\in \mathscr H_m$ for some $m\in \mathbb Z$ and that $z$ is not an extremal point of $\mathscr H_m$. Then $z$ is an interior point of $\mathscr H_m$ in the topology of the head. It follows that the extremal points $x_{i_k}, x_{i_k+1}$ of $D(i_k)$ are $P_2$-projected onto the extremal points $a_m, a_{m+1}$ of $\mathscr H_m$ for all sufficiently large $k$. We may of course assume that this is the case for all $k$. Moreover, as the sequence of jumps converges in an obvious sense to $a_{\infty}$, for all sufficiently large $k$ we have that neither the set $D(i_k+1)$ is placed along a jump. Again, we may assume that this is the case for every $k$. So, we conclude that for every $k$ there is $s_k$ such that $x_{i_k}, x_{i_k+1}, x_{i_k+2} \in J(s_k)$. Then the $P_2$-projections of these three points are $a_m, a_{m+1}, a_{m+2}$, respectively. Since $y_k \in D^*(i_k)$ and so $G(y_k) \in D^*(i_k+1)$, we may now use Lemma~\ref{L:Gprecont} to get
\begin{equation}\label{Eq:Gcont}
d(Z_{m,m+1}(P_2(y_k)), P_2(G(y_k))) \leq 2 \cdot r(i_k+1).
\end{equation}
On the other hand, since $z\in \mathscr H_m$ and so $G(z)=Z_{m,m+1}(z)$, our task is to show that $G(y_k) \to Z_{m,m+1}(z)$. However, $Z_{m,m+1}(z) \in \pi_0$ and the distance between $G(y_k)$ and $\pi_0$ tends to zero as $k\to \infty$. Therefore it is sufficient to show that $P_2(G(y_k)) \to Z_{m,m+1}(z)$. Using~\eqref{Eq:Gcont}, we get the estimate:
\begin{align}
d(P_2(G(y_k)), Z_{m,m+1}(z)) & \leq d(P_2(G(y_k)), Z_{m,m+1}(P_2(y_k))) + d(Z_{m,m+1}(P_2(y_k)), Z_{m,m+1}(z)) \notag \\
                             & \leq 2 \cdot r(i_k+1) + d(Z_{m,m+1}(P_2(y_k)), Z_{m,m+1}(z))~. \notag
\end{align}
For $k \to \infty$ we have $r(i_k+1) \to 0$ and since $P_2(y_k) \to P_2(z) = z$, just use the continuity of $Z_{m,m+1}$ to see that the right-hand side tends to zero.

Finally, let $z\in \mathscr H_m$ be an extremal point of $\mathscr H_m$, say $z\in  \mathscr H_m \cap \mathscr H_{m+1}$. Then we get $G(y_k)\to G(z)$ by dividing, if necessary, the sequence $y_k$ into two subsequences and then applying the above argument.
\end{proof}

\begin{lem}\label{L:OmegaG}
	The set $\Omega (G)$ of nonwandering points of $G$ equals $A = \{a_i\colon i\in \mathbb Z\} \cup \{a_{\infty}\}$.
\end{lem}
\begin{proof}
	Each point of the snake is clearly wandering for $G$. Further, for $i\in \mathbb Z$, infinitely many points of the trajectory $x_0, x_1, \dots $ are $P_2$-projected onto $a_i$. Since this subsequence of the trajectory converges to $a_i$, we have $a_i \in \Omega(G)$.  The point $a_{\infty}$ is also nonwandering, even for $G|_A$.
	
	Fix  $x\in \mathscr A_0 \setminus A$. To finish the proof we need to show that $x$ is wandering. Since $x\in \mathscr H_k \setminus \{a_k, a_{k+1}\}$ for some $k\in \mathbb Z$ and $r(j)\to 0$ as $j\to \infty$, there is $n_0\in \mathbb N$ such that the three open $r(n_0)$-balls centered at the points $x$, $a_k$ and $a_{k+1}$ are pairwise disjoint and we may also assume that the $r(n_0)$-disk $B_d(x, r(n_0))\cap \pi_0$ centered at $x$ lies in the sector $\mathscr S_k$. Hence there is $\varepsilon >0$ such that the open set $U_0:=(-\varepsilon, \varepsilon)\times (B_d(x, r(n_0))\cap \pi_0) \subseteq \mathbb R^3$ satisfies, for every $m\geq 0$, the following implication:
	\begin{equation}\label{Eq:DmcapU}
	D(m) \cap U_0 \neq \emptyset \,\, \Rightarrow \,\, D(m) \cap U_0 \subseteq D^{\rm middle}(m) \setminus D^{\rm front}(m)
	\end{equation}
	(see~\eqref{Eq:P2front0} and its analogues mentioned below it, e.g. \eqref{Eq:P2frontj}, \eqref{Eq:P2frontt1}, \eqref{Eq:P2frontakot1}).
	Note that $U_0$ is an open set containing $x$. To prove that $x$ is wandering, fix any $z\in U_0\cap X$. If $z\in U_0\cap {\rm head}$ then $z$ will never visit $U_0$ again, because $G(\mathscr H_i)=\mathscr H_{i+1}$, for every integer $i$. If $z\in U_0\cap {\rm snake}$, say $z\in U_0\cap D(m)$ with $x_m, x_{m+1} \in J(i)$, then we claim that again all the points $G(z), G^2(z), \dots$ are outside $U_0$. In fact, if $j> 0$ is small enough then $G^j(z)\in D(s)$ with $x_s, x_{s+1}$ either belonging to the same set $J(i)$ or forming the $i$-th jump, whence $G^j(z) \notin U_0$. Otherwise use ~\eqref{Eq:DmcapU} and Lemma~\ref{L:G-log2-Dm-3part} to get that $G^j(z) \notin U_0$. It follows that $x$ is a wandering point of $G$.
\end{proof}

We embark on the proof that $h^*(G)=\log 2$. Of course, $h^*(G)\geq h^*(T) = \log 2$.
By Proposition~\ref{P}(c), when we are looking for IN-tuples of the map $G$, it is sufficient to consider the set $\Omega (G) = A$. Recall also how we proved that $h^*(T)=\log 2$. First we realized that it was sufficient to prove the conditions (1), (2) and (3) below the formula~\eqref{Eq:approachpi0}, cf. Theorem~\ref{final-thm}.

\begin{lem} \label{L:G-log2-2case}
For the map $G$ defined above, the following holds.
	\begin{enumerate}
		\item $(a_i,a_j)$ is not an IN-pair for any $i, j \in \mathbb Z$ with $|j-i| \ge 2$.
		\item $(a_i,a_{\infty})$ is not an IN-pair for any $i \in \Z$.
	\end{enumerate}
\end{lem}
\begin{proof}
  $(1)$ Fix $i, j \in \mathbb Z$ such that $|j -i|\geq 2$. We may assume that $i>j$.  By Theorem \ref{final-thm}, $(a_0,a_{i-j})$ is not an IN-pair for $(X_1,T)$. In view of Proposition~\ref{P}(a) and Proposition~\ref{P2}(b) this implies that
  \begin{equation}\label{Eq:aiajnotT1}
  \text{$(a_i,a_j)$ is not an IN-pair for $(X_1,T)$.}
  \end{equation}
  We are going to show that such a pair is neither an IN-pair for $(X,G)$.

   By~\eqref{Eq:aiajnotT1}, there is $N \in \N$ and a pair of neighborhoods of the points $a_i,a_j$ such that it has no independence set of times of length $N$ in $(X_1, T)$. Of course, neither smaller neighbourhoods have such an independence set. Therefore one can fix $\varepsilon_i >0$, $\varepsilon_j >0$ and $r>0$ such that the following holds.
     \begin{enumerate}
      \item[(i)] The closed balls (in the plane $\pi_0$) $\overline B(a_{j-1},r)$, $\overline B(a_j,r)$ and $\overline B(a_{j+1},r)$ are pairwise disjoint,
      \item[(ii)] The closed balls $\overline B(a_{i-1},r)$, $\overline B(a_i,r)$ and $\overline B(a_{i+1},r)$ are pairwise disjoint,
      \item[(iii)] For the open balls $B(a_i,r), B(a_j,r)$ in $\pi_0$ we have that, for $U(a_i)= [0,\varepsilon_i) \times B(a_i,r)$ and $U(a_j)= [0,\varepsilon_j) \times B(a_j,r)$,
      \begin{equation}\label{Eq:UnoNforT1}
      \text{the pair $(U(a_i), U(a_j))$ has no independence set of times of length $N$ in $(X_1,T)$.}
      \end{equation}
     \end{enumerate}
  For other restrictions on the choice of  $\varepsilon_i, \varepsilon_j$ and $r$ see below. (In (iii), the shape of $U(a_i)$ and $U(a_j)$ is due to the fact that $X$ lives in $[0,1] \times \pi_0$; to be precise, we should speak on the pair of sets $(U(a_i)\cap X_1, U(a_j)\cap X_1)$, but we hope that no misunderstanding can arise.) We may also assume that all the six balls from (i) and (ii) are pairwise disjoint (except for the case $i=j+2$ when $B(a_{j+1},r)$ and $B(a_{i-1},r)$ coincide) and also disjoint with the closed $r$-ball centered at any other point from $A$.
  Since we have (EP5), (EP10), (EP12) as well as~\eqref{Eq:P2front0} and its analogues~\eqref{Eq:P2frontj},~\eqref{Eq:P2frontt1},~\eqref{Eq:P2frontakot1}, ..., we may also assume that $r$ is small enough so that the following holds.
    \begin{enumerate}
    	\item[(iv)]  If $z\in D(t) \cap U(a_i)$ then $a_i$ is the $P_2$-projection of the first or the last point of $D(t)$. In the former case  $z\notin D^{\rm end}(t)$ and in the latter case $z\notin D^{\rm front}(t)$. The same is true for $U(a_j)$.
    \end{enumerate}
  By (iv), if $D^{\rm front}(t)$ intersects $U(a_i)$, then $a_i=P_2(x_t)$. Since the point $x_{t+1}$ is closer to $\pi_0$ than $x_t$, it could happen that though some point $z\in D^{\rm front}(t)$ is in $U(a_i)$, the point $x_t$ itself is outside $U(a_i)$ (though $x_t \in [0, \infty)\times B(a_i,r)$). A similar undesirable effect can occur for $U(a_j)$. However, if $D^{\rm front}(t)$ intersects $U(a_i)$ and $x_t \notin U(a_i)$, notice that such $t$ exists only one (if $t'>t$ and again $D^{\rm front}(t')$ intersects $U(a_i)$ then already $x_{t'} \in U(a_i)$). Similarly for $U(a_j)$. Therefore, by an appropriate choice of $\varepsilon_i$, $\varepsilon_j$ and $r$ we may assume the following.

   \begin{enumerate}
  	\item[(v)] If $D^{\rm front}(t)$ intersects $U(a_i)$ then $x_t \in U(a_i)$. Similarly for $U(a_j)$.
  \end{enumerate}
  Again, by (iv), if $D^{\rm end}(t)$ intersects $U(a_i)$, then $a_i=P_2(x_{t+1})$. Since the point $x_{t+1}$ is closer to $\pi_0$ than $x_t$, we then have the following.
  \begin{enumerate}
	\item[(vi)] If $D^{\rm end}(t)$ intersects $U(a_i)$ then $x_{t+1} \in U(a_i)$. Similarly for $U(a_j)$.
\end{enumerate}

\medskip

Now we are ready to prove that $(a_i,a_j)$ is not an IN-pair for $(X,G)$. Suppose on the contrary that this is not the case. Then the pair $(U(a_i), U(a_j))$ has arbitrarily long finite independence sets of times in $(X,G)$. We get a contradiction with~\eqref{Eq:aiajnotT1} by proving, for any fixed $N\in \mathbb N$, the following claim.

\medskip

{\noindent \bf Claim: } If the pair of sets $(U(a_i), U(a_j))$ has an independence set of times of length $2N+4$ in $(X,G)$, then it has an independence set of times of length $N$ in $(X_1,T)$.

{\noindent \bf Proof of Claim: } Suppose that $0\leq l_{-1}<l_0<\dots<l_{2N+2}$ and	$\{l_{-1},l_0,\dots, l_{2N+2}\}$ is an independence set of times of length $2N+4$ for $(U(a_i), U(a_j))$ in $(X,G)$. We are going to show that then $(U(a_i), U(a_j))$ has an independence set of times of length $N$ in $(X_1,T)$.

Fix $t_0 \in \{i,j\}^{\{1,2,\dots,N\}}$. For any $s \in \{i,j\}^{\{1,2,\dots,N\}}$ we can consider
\begin{equation}\label{Eq:2N+3}
(i,i, s(1),s(2),\dots,s(N),i,i,t_0(1),t_0(2),\dots,t_0(N))\in \{i,j\}^{\{-1,0,1,2,\dots,2N+2\}}
\end{equation}
    and denote
    \begin{eqnarray*}
	I(s,t_0)&=&(G^{-l_{-1}}U(a_i)) \cap (G^{-l_{0}}U(a_i)) \cap (\bigcap_{k=1}^{N}G^{-l_k}U(a_{s(k)})) \\
            & & \cap (G^{-l_{N+1}}U(a_i))\cap (G^{-l_{N+2}}U(a_i))\cap(\bigcap_{k=1}^{N}G^{-l_{N+2+k}}U(a_{t_0(k)})).
	\end{eqnarray*}
By the assumption, $I(s,t_0)\neq \emptyset$. In other words,
	\begin{equation}\label{Eq:Ist0nonempty}
	\text{for every $s$ there exists $z_{s,t_0} \in X$ such that $z_{s,t_0} \in I(s,t_0)$.}
	\end{equation}
Notice also that 	
	\begin{equation}\label{Eq:Ist0snake}
    I(s,t_0) \subseteq \text{snake}
	\end{equation}	
because, apart from the fixed point $a_{\infty}$, the dynamics in the head is `clock-wise' and so a point from the head cannot visit $U(a_i)$ (at least) 4-times, as required by~\eqref{Eq:2N+3}.	

Clearly, $\{i,j\}^{\{1,2,\dots,N\}} = \mathfrak S_1(t_0) \sqcup \mathfrak S_2(t_0)$ where
\[
\mathfrak S_1(t_0) := \{s \in \{i,j\}^{\{1,2,\dots,N\}}:\,  I(s,t_0) \cap X_1 \neq \emptyset\}
\]
and
\[
\mathfrak S_2(t_0) := \{s \in \{i,j\}^{\{1,2,\dots,N\}}:\,  I(s,t_0) \subseteq \text{snake} \setminus X_1\}.
\]

If $s\in \mathfrak S_2(t_0)$ and $z_{s,t_0} \in I(s,t_0)$, then there exists a unique $\widetilde m$ such that $z_{s,t_0} \in D(\widetilde m) \setminus X_1$. Then, for $m=\widetilde m +l_{-1}$, using also the definition of $I(s,t_0)$ and the fact that the set $D(m)$ has the first point $x_m$ and the last point $x_{m+1}$, we get two disjoint possibilities:
\begin{enumerate}
	\item [(a)] $G^{l_{-1}}z_{s,t_0} \in D(m) \cap U(a_{i})$, $a_i=P_2(x_{m+1})$, or
	\item [(b)] $G^{l_{-1}}z_{s,t_0} \in D(m) \cap U(a_{i})$, $a_i=P_2(x_m)$
\end{enumerate}
(of course, in (a) and (b) the number $m$ depends on $z_{s,t_0}$ and $P_2(G^{l_{-1}}z_{s,t_0}) \in B(a_i,r)$). Consider the sets
\[
\mathfrak S_{2a}(t_0) := \{s \in \mathfrak S_2(t_0):\, \forall z_{s,t_0} \in I(s,t_0) \text{ the condition (a) holds}\}
\]
and
\[
\mathfrak S_{2b}(t_0) := \{s \in \mathfrak S_2(t_0):\, \exists z_{s,t_0} \in I(s,t_0) \text{ such that the condition (b) holds}\}.
\]

Let $s\in \mathfrak S_{2a}(t_0)$ and $z_{s,t_0} \in I(s,t_0)$. In view of the condition (a) and the definition of $I(s,t_0)$, we have $G^{l_{N+1}}z_{s,t_0} \in  D(m+l_{N+1}-l_{-1}) \cap U(a_i)$ for some $m$. There are two disjoint possibilities:
\begin{enumerate}
	\item [(c)] $G^{l_{N+1}}z_{s,t_0} \in D^{\rm end}(m+l_{N+1}-l_{-1})\cap U(a_i)$, or
	\item [(d)] $G^{l_{N+1}}z_{s,t_0} \in (D(m+l_{N+1}-l_{-1}) \setminus D^{\rm end}(m+l_{N+1}-l_{-1}))\cap U(a_i)$.
\end{enumerate}
Consider the sets
\[
\mathfrak S_{2ac}(t_0) := \{s \in \mathfrak S_{2a}(t_0):\, \forall z_{s,t_0} \in I(s,t_0) \text{ the condition (c) holds}\}
\]
and
\[
\mathfrak S_{2ad}(t_0) := \{s \in \mathfrak S_{2a}(t_0):\, \exists z_{s,t_0} \in I(s,t_0) \text{ such that the condition (d) holds}\}.
\]

Clearly, for every $t_0 \in \{i,j\}^{\{1,2,\dots,N\}}$ we have
\[
\{i,j\}^{\{1,2,\dots,N\}} = \mathfrak S_1(t_0) \sqcup \mathfrak S_{2ac}(t_0) \sqcup \mathfrak S_{2ad}(t_0) \sqcup \mathfrak S_{2b}(t_0).
\]
We consider four cases.

\medskip

\emph{Case 1: There exists $t_0$ such that $\{i,j\}^{\{1,2,\dots,N\}} = \mathfrak S_1(t_0)$.}

In this case $\{l_1,\dots, l_N\}$ is an independence set of times of length $N$ for $(U(a_i), U(a_j))$ in $(X_1, T)$ and we are done.

\medskip

\emph{Case 2: There exists $t_0$ such that $\{i,j\}^{\{1,2,\dots,N\}} = \mathfrak S_1(t_0) \sqcup \mathfrak S_{2b}(t_0)$ with $\mathfrak S_{2b}(t_0) \neq \emptyset$.}

Fix such $t_0$. Let $s\in \mathfrak S_{2b}(t_0)$. Choose $z_{s,t_0} \in I(s,t_0)$ such that (b) holds for some $m$. For such $s$ and $z_{s,t_0}$ we get, by (iv), that $G^{l_{-1}}z_{s,t_0} \in D^{\rm front}(m)\cup  D^{\rm middle}(m)$. Using Lemma \ref{L:G-log2-Dm-3part},
\[
G^{l_{0}}z_{s,t_0}=G^{l_{0}-l_{-1}}G^{l_{-1}}z_{s,t_0} \in D^{\rm front}(m+l_0-l_{-1}).
\]
Then, by Corollary~\ref{C:fronttofront}, the points $G^{l_{0}}z_{s,t_0}$ and $x_{m+l_0-l_{-1}}$ travel together through the front parts, i.e. for every $p\geq 0$ both $G^p(G^{l_{0}}z_{s,t_0})$ and $G^{p}(x_{m+l_0-l_{-1}})$ belong to $D^{\rm front}(m+l_0-l_{-1}+p)$. In particular,
\[
G^{l_{k}}z_{s,t_0}=G^{l_{k}-l_{0}}G^{l_{0}}z_{s,t_0} \in D^{\rm front}(m+l_k-l_{-1}), \qquad k=1,\dots,N.
\]
Since $z_{s,t_0} \in I(s,t_0)$, we also have $G^{l_k}z_{s,t_0} \in U(a_{s(k)})$, $k=1,\dots,N$.  Thus, $G^{l_k}z_{s,t_0}$ is a point of $D^{\rm front}(m+l_k-l_{-1})$ belonging to $U(a_{s(k)})$ and then, by (v),
\[
x_{m+l_k-l_{-1}}\in U(a_{s(k)}),  \qquad k=1,\dots,N.
\]
Then
\begin{equation}\label{Eq:sinS2bviax}
G^{l_k-l_0}x_{m+l_0-l_{-1}} = x_{m+l_k-l_{-1}} \in U(a_{s(k)}), \qquad k=1,\dots,N,
\end{equation}
whence
\begin{equation}\label{Eq:sinS2b}
X_1 \cap \bigcap_{k=1}^{N}G^{-(l_{k}-l_0)}U(a_{s(k)}) \neq \emptyset.
\end{equation}

We have shown that~\eqref{Eq:sinS2b} holds for all $s\in \mathfrak S_{2b}(t_0)$. However, if $s\in \mathfrak S_1(t_0)$ then~\eqref{Eq:sinS2b} is also true, because by definition of $\mathfrak S_1(t_0)$ there is a point $z_{s,t_0}\in I(s, t_0) \cap X_1$ and then the point $G^{l_0}(z_{s,t_0})$ belongs both to the $G$-invariant set $X_1$ and to the set $\bigcap_{k=1}^{N}G^{-(l_{k}-l_0)}U(a_{s(k)})$.
Thus we have~\eqref{Eq:sinS2b} for all $s\in \mathfrak S_1(t_0) \sqcup \mathfrak S_{2b}(t_0) = \{i,j\}^{\{1,2,\dots,N\}}$. It follows that
$\{l_1-l_0, \dots, l_N-l_0\}$ is an independence set of times of length $N$ for $(U(a_i), U(a_j))$ in $(X_1, T)$ and so we are done also in this case.

\medskip

\emph{Case 3: There exists $t_0$ such that $\{i,j\}^{\{1,2,\dots,N\}} = \mathfrak S_1(t_0) \sqcup \mathfrak S_{2ac}(t_0) \sqcup \mathfrak S_{2b}(t_0)$ with $\mathfrak S_{2ac}(t_0) \neq \emptyset$.}

Fix such $t_0$. Let $s\in \mathfrak S_{2ac}(t_0)$. By the condition (a) and by (iv), $G^{l_{-1}}z_{s,t_0} \in (D(m) \setminus D^{\rm front}(m)) \cap U(a_{i})$ for some~$m$. Hence $G^{l_{-1}}z_{s,t_0} \in (D^{\rm middle}(m) \cup D^{\rm end}(m)) \cap U(a_{i})$.
On the other hand, by condition (c) we have $G^{l_{N+1}}z_{s,t_0} \in D^{\rm end}(m+l_{N+1}-l_{-1})\cap U(a_i)$.
Then by Lemma~\ref{L:fronttofront} and Lemma~\ref{L:G-log2-Dm-3part},
\[
G^{l_{k}}z_{s,t_0} \in D^{end}(m+l_{k}-l_{-1}), \qquad k=-1,0,1,2,\dots,N.
\]
Further, since $z_{s,t_0}\in I(s,t_0)$, we have
\[
G^{l_{k}-l_{-1}}G^{l_{-1}}z_{s,t_0} = G^{l_{k}}z_{s,t_0}\in U(a_{s(k)}), \qquad k=1,2,\dots,N.
\]
Thus, the point $G^{l_{-1}}z_{s,t_0} \in D^{end}(m)$ has the $G^{l_{k}-l_{-1}}$-image (for $k=1,2,\dots,N$) both in $D^{end}(m+l_{k}-l_{-1})$ and in $U(a_{s(k)})$. By (vi),
$x_{m+l_k-l_{-1}+1} \in U(a_{s(k)})$. Hence
\[
G^{l_{k}-l_{-1}}x_{m+1} = x_{m+l_k-l_{-1}+1} \in U(a_{s(k)}), \qquad k=1,2,\dots,N.
\]
It follows that
\begin{equation}\label{Eq:sinS2ac}
X_1 \cap \bigcap_{k=1}^{N}G^{-(l_{k}-l_{-1})}U(a_{s(k)}) \neq \emptyset.
\end{equation}

Let us summarize. For any $s \in  \{i,j\}^{\{1,2,\dots,N\}}$ there are the following $3$ possibilities.
\begin{itemize}
	\item If $s\in \mathfrak S_{2ac}(t_0)$ then, as we have just shown, \eqref{Eq:sinS2ac} holds.
	\item If $s\in \mathfrak S_1(t_0)$ then there is $z_{s,t_0} \in I(s,t_0)\cap X_1$. Then $G^{l_{-1}} z_{s,t_0} \in X_1$ and, by the definition of $I_{s,t_0}$, $G^{l_k-l_{-1}} (G^{l_{-1}} z_{s,t_0}) = G^{l_k}(z_{s,t_0}) \in U(a_{s(k)})$ for $k=1,\dots, N$. Hence, again we have~\eqref{Eq:sinS2ac}.
	\item If $s\in \mathfrak S_{2b}(t_0)$ then, as shown in~\eqref{Eq:sinS2bviax} in Case 2, for $k=1,\dots,N$ we have $G^{l_k-l_0}x_{m+l_0-l_{-1}} \in U(a_{s(k)})$. Since $G^{l_k-l_0}x_{m+l_0-l_{-1}} = G^{l_k-l_0} (G^{l_0-l_{-1}}(x_{m})) = G^{l_k-l_{-1}}(x_{m})$, we get~\eqref{Eq:sinS2ac} again.
\end{itemize}
Since~\eqref{Eq:sinS2ac} holds for every $s \in  \{i,j\}^{\{1,2,\dots,N\}}$, $\{l_1-l_{-1}, \dots, l_N-l_{-1}\}$ is an independence set of times of length $N$ for $(U(a_i), U(a_j))$ in $(X_1, T)$. The proof of Claim in Case 3 is finished.

\medskip

It remains to consider the last case, which is the negation of the logical disjunction of the first three cases. It can be formulated as follows.

\medskip

\emph{Case 4: For every $t_0 \in \{i,j\}^{\{1,2,\dots,N\}}$ there exists $s_0 \in  \{i,j\}^{\{1,2,\dots,N\}}$ such that $s_0 \in \mathfrak S_{2ad}(t_0)$.}

Fix any $t_0$ and choose $s_0 \in \mathfrak S_{2ad}(t_0)$. Then there exists $z_{s_0,t_0}$ such that the following holds:
\begin{enumerate}
	\item[($4_1$)] $z_{s_0,t_0} \in I(s_0,t_0)$.
	\item[($4_2$)] There exists $m \in \N$ with $P_2(x_{m+1})=a_{i}$ such that $G^{l_{-1}}z_{s_0,t_0} \in D(m)\cap U(a_{i})$. Moreover, by (iv), $G^{l_{-1}}z_{s_0,t_0} \notin D^{ front}(m)$.
	\item[($4_3$)] $G^{l_{N+1}}z_{s,t_0} \in (D(m+l_{N+1}-l_{-1}) \setminus D^{\rm end}(m+l_{N+1}-l_{-1}))\cap U(a_i)$.
\end{enumerate}
By ($4_3$),
\[
G^{l_{N+1}}z_{s_0,t_0} \in  (D^{\rm front}(m+l_{N+1}-l_{-1})\cup D^{\rm middle}(m+l_{N+1}-l_{-1})) \cap U(a_i)
\]
and so $G^{l_{N+2}}z_{s_0,t_0} \in D(m+l_{N+2}-l_{-1})\cap U(a_i)$. Since $x_{m+l_{N+1}-l_{-1}}$ and $x_{m+l_{N+2}-l_{-1}}$
are not in the same jump level, by Lemma  \ref{L:G-log2-Dm-3part} we in fact have
\begin{equation}\label{eq:11to5-1}
     G^{l_{N+2}}z_{s_0,t_0} \in D^{\rm front}(m+l_{N+2}-l_{-1}) \cap U(a_i).
\end{equation}
Hence, by Corollary~\ref{C:fronttofront}, the points $G^{l_{N+2}}z_{s_0,t_0}$ and $x_{m+l_{N+2}-l_{-1}}$ travel together through the front parts. In particular,
\[
G^{l_{N+2+k}}z_{s_0,t_0}=G^{l_{N+2+k}-l_{N+2}}(G^{l_{N+2}}z_{s_0,t_0}) \in  D^{\rm front}(m+l_{N+2+k}-l_{-1}), \qquad k=1,2,\dots,N.
\]
Further, by ($4_1$) and the definition of $I(s_0,t_0)$ we get
\[
G^{l_{N+2+k}}z_{s_0,t_0} \in U(a_{t_0(k)}), \qquad k=1,2,\dots, N.
\]
By the last two inclusions, (v) gives that $x_{m+l_{N+2+k}-l_{-1}} \in U(a_{t_0(k)})$, $k=1,2,\dots, N$. Equivalently,
\[
G^{l_{N+2+k}-l_{N+2}}x_{m+l_{N+2}-l_{-1}} \in U(a_{t_0(k)}), \qquad k=1,2,\dots,N
\]
and so
\[
X_1 \cap \bigcap_{k=1}^{N}G^{-(l_{N+2+k}-l_{N+2})}U(a_{t_0(k)}) \neq \emptyset.
\]
Since $t_0\in \{i,j\}^{\{1,2,\dots, N\}}$ was arbitrary, $\{l_{N+3}-l_{N+2},l_{N+4}-l_{N+2},\dots,l_{2N+2}-l_{N+2}\}$ is an independence set of times of length $N$ for
$(U(a_i), U(a_j))$ in $(X_1,T)$.

\medskip	

$(2)$ The proof is very similar to that of (1). For completeness, we give an outline of it, emphasizing the differences when compared with the proof of (1). To make the comparison easier, we will use an analogous notation as in $(1)$. Now we will of course have $U(a_{\infty})$ instead of $U(a_j)$. The main difference is that (iv) will be replaced by (iv'$_a$)-(iv'$_d$).

By Theorem \ref{final-thm}, $(a_0,a_{\infty})$ is not an IN-pair for $(X_1,T)$. In view of Proposition~\ref{P}(a) and Proposition~\ref{P2}(b) this implies that
\begin{equation}\label{Eq:aiajnotT1-infty}
\text{$(a_i,a_{\infty})$ is not an IN-pair for $(X_1,T)$.}
\end{equation}
We are going to show that such a pair is neither an IN-pair for $(X,G)$.

By~\eqref{Eq:aiajnotT1-infty}, there is $N \in \N$ and a pair of neighborhoods of the points $a_i,a_{\infty}$ such that it has no independence set of times of length $N$ in $(X_1, T)$. Therefore one can fix $\varepsilon_i >0$, $\varepsilon_{\infty} >0$ and $r>0$ such that the following holds.
\begin{enumerate}
	\item[(i')] The closed balls (in the plane $\pi_0$) $\overline B(a_{i-1},r)$, $\overline B(a_i,r)$ and $\overline B(a_{i+1},r)$ are pairwise disjoint (and also disjoint with the closed $r$-ball centered at any other point $a_j\in A$).
	\item[(ii')] The closed ball $\overline B(a_{\infty},r)$ is disjoint with the three closed balls in (i') and there exists $K\in \N$
	such that $a_j~\in~B(a_{\infty},r)$ if and only is $|j|\geq K$.
	\item[(iii'$_a$)] For the open balls $B(a_i,r), B(a_{\infty},r)$ in $\pi_0$ we have that, for $U(a_i)= [0,\varepsilon_i) \times B(a_i,r)$ and $V(a_{\infty})= [0,\varepsilon_{\infty}) \times B(a_{\infty},r)$,
	\begin{equation}\label{Eq:VnoNforT1-infty}
	\text{the pair $(U(a_i), V(a_{\infty}))$ has no independence set of times of length $N$ in $(X_1,T)$.}
	\end{equation}
\end{enumerate}
For other restrictions on the choice of  $\varepsilon_i, \varepsilon_{\infty}$ and $r$ see below.

	Recall that, by our construction, all the jumps are `almost horizontal' and the jump numbers corresponding to the starting points of the jumps tend to infinity very fast.  Therefore we may assume that $r$ is chosen in such a way that $K$ is  not a jump number, i.e. no jump starts or ends in a point whose $P_2$-projection is $a_K$ or $a_{-K}$, respectively. Again, taking into account that all jumps are `almost horizontal', one can see that by shrinking $B(a_{\infty},r)$ appropriately we get an open neighborhood (not necessarily a ball) $\tilde{B}(a_{\infty})\subseteq B(a_{\infty},r)$ of $a_{\infty}$ such that still $a_j \in \tilde{B}(a_{\infty})$ if and only is $|j|\geq K$ and, moreover,
\[
U(a_{\infty}):= [0,\varepsilon_{\infty}) \times \tilde{B}(a_{\infty})\subseteq V(a_{\infty})
\]
has the following property.
\begin{enumerate}
	\item[(iii'$_b$)] If $D(t)$ is placed along a jump and $D(t) \cap U(a_{\infty}) \neq \emptyset$, then $D(t)\subseteq U(a_{\infty})$.
\end{enumerate}
Since $U(a_{\infty})\subseteq V(a_{\infty})$, it is clear that
\begin{equation}\label{Eq:UnoNforT1-infty}
\text{the pair $(U(a_i), U(a_{\infty}))$ has no independence set of times of length $N$ in $(X_1,T)$.}
\end{equation}

Just as we do in $(1)$, we may also assume that $r$, $\varepsilon_{\infty}$ and $U(a_{\infty})$ are chosen properly so that the following holds.
\begin{enumerate}
	\item[(iv'$_a$)]  If $z\in D(t) \cap U(a_i)$ then $a_i$ is the $P_2$-projection of the first or the last point of $D(t)$. In the former case  $z\notin D^{\rm end}(t)$ and in the latter case $z\notin D^{\rm front}(t)$.
	\item[(iv'$_b$)]  If $z\in D(t) \cap U(a_{\infty})$ and $P_2(x_{t})=a_{-K}$,  then $z \notin D^{\rm end}(t)$.
	\item[(iv'$_c$)]  If $z\in D(t) \cap U(a_{\infty})$ and $P_2(x_{t+1})=a_{K}$, then $z \notin D^{\rm front}(t)$.
	\item[(iv'$_d$)]  If $z\in D(t) \cap U(a_{\infty})$ and  $P_2(x_{t+1})=a_{j}$ where $j>K$ or $j \le -K$, then $ D(t) \subseteq U(a_{\infty})$.
\end{enumerate}

\medskip

By (iv'$_a$), if $D^{\rm front}(t)$ intersects $U(a_i)$, then $a_i=P_2(x_t)$. Since the point $x_{t+1}$ is closer to $\pi_0$ than $x_t$, it could happen that the point $x_t$ itself is outside $U(a_i)$. However, as in (1), we may assume the following.
 \begin{enumerate}
 	\item[(v'$_a$)] If $D^{\rm front}(t)$ intersects $U(a_i)$ then $x_t \in U(a_i)$.
 \end{enumerate}
Similarly, we can claim the following.
 \begin{enumerate}
 	\item[(v'$_b$)] If $D^{\rm front}(t)$ intersects $U(a_{\infty})$ then $x_t \in U(a_{\infty})$.
 \end{enumerate}
Indeed, if $D(t)$ is placed along a jump, then this follows from (iii')(b). Now suppose that $D(t)$ is not placed along a jump, put $P_2(x_{t+1}) = a_j$ and consider all the possible cases as follows. If $j>K$ or $j \le -K$, then $x_t \in U(a_{\infty})$ by (iv')(d). If $j=-(K-1)$, we have $P_2(x_t)=a_{-K}\in \tilde{B}(a_{\infty})$ and then $x_t\in U(a_{\infty})$ due to a slight change of $\varepsilon _{\infty}$ if necessary (see the corresponding discussion in (1), just above (v)). The case $j=K$ is impossible because our assumption that $D^{\rm front}(t)$ intersects $U(a_{\infty})$ contradicts (iv')(c). It remains the case $-(K-1) <j < K$ which implies that both $j$ and $j-1$ have absolute values less than $K$ and so both $a_j, a_{j-1} \notin \tilde{B}(a_{\infty})$, whence $x_{t+1}, x_t \notin U(a_{\infty})$ and so $D(t)\cap U(a_{\infty})=\emptyset$, a contradiction.

\medskip

By (iv'$_a$), if $D^{\rm end}(t)$ intersects $U(a_i)$, then $a_i=P_2(x_{t+1})$. Since the point $x_{t+1}$ is closer to $\pi_0$ than $x_t$, we have the following.
\begin{enumerate}	
 	\item[(vi'$_a$)] If $D^{\rm end}(t)$ intersects $U(a_i)$ then $x_{t+1} \in U(a_i)$.
\end{enumerate}
Similarly, we can claim the following.
\begin{enumerate}	
	\item[(vi'$_b$)] If $D^{\rm end}(t)$ intersects $U(a_{\infty})$ then $x_{t+1} \in U(a_{\infty})$.
\end{enumerate}
Indeed, if $D(t)$ is placed along a jump, then this follows from (iii')(b). Now suppose that $D(t)$ is not placed along a jump, put $P_2(x_{t+1}) = a_j$ and consider all the possible cases as follows. If $j>K$ or $j \le -K$, then $x_{t+1} \in U(a_{\infty})$ by (iv')(d). If $j=-(K-1)$, we have $P_2(x_t)=a_{-K}$ and (iv')(b) shows that this case is impossible. If $j=K$, then $P_2(x_{t+1})=a_{K}\in \tilde{B}(a_{\infty})$ and since $x_{t+1}$ is closer to $\pi_0$ than $x_t$ (and $D^{\rm end}(t)$ intersects $U(a_{\infty})$), this implies $x_{t+1}\in U(a_{\infty})$. Finally, if $-(K-1) <j < K$ then both $j$ and $j-1$ have absolute values less than $K$. Thus both $a_j, a_{j-1} \notin \tilde{B}(a_{\infty})$, whence $x_{t+1}, x_t \notin U(a_{\infty})$ and so $D(t)\cap U(a_{\infty})=\emptyset$, a contradiction.

\medskip

Now we are ready to prove that $(a_i,a_{\infty})$ is not an IN-pair for $(X,G)$. Suppose on the contrary that this is not the case. Then the pair $(U(a_i), U(a_{\infty}))$ has arbitrarily long finite independence sets of times in $(X,G)$. We get a contradiction with~\eqref{Eq:aiajnotT1-infty} by proving, for any fixed $N\in \mathbb N$, the following claim.

\medskip

{\noindent \bf Claim': } If the pair of sets $(U(a_i), U(a_{\infty}))$ has an independence set of times of length $2N+4$ in $(X,G)$, then it has an independence set of times of length $N$ in $(X_1,T)$.

{\noindent \bf Proof of Claim': } Suppose that $0\leq l_{-1}<l_0<\dots<l_{2N+2}$ and	$\{l_{-1},l_0,\dots, l_{2N+2}\}$ is an independence set of times of length $2N+4$ for $(U(a_i), U(a_{\infty}))$ in $(X,G)$. We are going to show that then $(U(a_i), U(a_{\infty}))$ has an independence set of times of length $N$ in $(X_1,T)$.

Fix $t_0 \in \{i,\infty\}^{\{1,2,\dots,N\}}$. For any $s \in \{i,\infty\}^{\{1,2,\dots,N\}}$ we can consider
\begin{equation}\label{Eq:2N+3-infty}
(i,i, s(1),s(2),\dots,s(N),i,i,t_0(1),t_0(2),\dots,t_0(N))\in \{i,\infty\}^{\{-1,0,1,2,\dots,2N+2\}}
\end{equation}
and also denote
\begin{eqnarray*}
	I(s,t_0)&=&(G^{-l_{-1}}U(a_i)) \cap (G^{-l_{0}}U(a_i)) \cap (\bigcap_{k=1}^{N}G^{-l_k}U(a_{s(k)})) \\
	& & \cap (G^{-l_{N+1}}U(a_i))\cap (G^{-l_{N+2}}U(a_i))\cap(\bigcap_{k=1}^{N}G^{-l_{N+2+k}}U(a_{t_0(k)})).
\end{eqnarray*}
By the assumption,
\begin{equation}\label{Eq:Ist0nonempty-infty}
\text{for every $s$ there exists $z_{s,t_0} \in X$ such that $z_{s,t_0} \in I(s,t_0)$}
\end{equation}
and for the same reason as in (1) we have 	
\begin{equation}\label{Eq:Ist0snake-infty}
I(s,t_0) \subseteq \text{snake}.
\end{equation}	

Clearly, $\{i,\infty\}^{\{1,2,\dots,N\}} = \mathfrak S_1(t_0) \sqcup \mathfrak S_2(t_0)$ where
\[
\mathfrak S_1(t_0) := \{s \in \{i,\infty\}^{\{1,2,\dots,N\}}:\,  I(s,t_0) \cap X_1 \neq \emptyset\}
\]
and
\[
\mathfrak S_2(t_0) := \{s \in \{i,\infty\}^{\{1,2,\dots,N\}}:\,  I(s,t_0) \subseteq \text{snake} \setminus X_1\}.
\]

We further define $\widetilde m$, $m$ and also $\mathfrak S_{2a}(t_0)$, $\mathfrak S_{2b}(t_0)$ and $\mathfrak S_{2ac}(t_0)$, $\mathfrak S_{2ad}(t_0)$ as in (1). Then, for every $t_0 \in \{i,\infty\}^{\{1,2,\dots,N\}}$ we have
\[
\{i,\infty\}^{\{1,2,\dots,N\}} = \mathfrak S_1(t_0) \sqcup \mathfrak S_{2ac}(t_0) \sqcup \mathfrak S_{2ad}(t_0) \sqcup \mathfrak S_{2b}(t_0).
\]

To finish the proof, it is sufficient to consider exactly those four cases as in (1). In each of them, the proof is the same as in the corresponding case in (1). Of course, instead of $\{i,j\}^{\{1,2,\dots,N\}}$ and $(U(a_i), U(a_j))$ one has to write $\{i,\infty\}^{\{1,2,\dots,N\}}$ and $(U(a_i), U(a_{\infty}))$, respectively. The list of other differences is:
\begin{itemize}
	\item in Case 2, instead of (iv) use (iv'$_a$) and instead of (v) use (v'$_a$) and (v'$_b$),
	\item in Case 3, instead of (iv) use (iv'$_a$) and instead of (vi) use (vi'$_a$) and (vi'$_b$),
	\item in Case 4, instead of (iv) use (iv'$_a$) and instead of (v) use (v'$_a$).
\end{itemize}
\end{proof}

Recall that in Section~\ref{S:X} dealing with the case $S(X)=\{0,\infty\}$, namely above Lemma~\ref{L:space}, we introduced the sets
$\langle \langle a, b \rangle \rangle$, $\langle \langle \mathscr K_0, b \rangle \rangle$ and $((\mathscr K_0, b \rangle \rangle$. The sets of the form $((\mathscr K_0,
f^m_i\rangle \rangle$ were said to be sub-snakes. Now, when we are working with the case $S(X)=\{0,\log 2\}$, the head $\mathscr K_0$ is replaced by the head $\mathscr A_0$ and the snake is homeomorphic with the snake from Section~\ref{S:X}. Therefore we can again use these notations, with $\mathscr K_0$ replaced by $\mathscr A_0$.

With the help of the above Lemma~\ref{L:G-log2-2case}, we get the following analogue of Lemma~\ref{L:existsG}.	
\begin{lem}\label{L:existsG-log2}
For the map $G: X\to X$ defined above, in Subsection~\ref{SS:defG}, as a continuous extension of the map $T\colon X_1\to X_1$, the following properties hold.
	\begin{itemize}
		\item [(a)] $h^{*}(G)=\log 2$.
		\item [(b)] For every $r\in \N$, the set $S^{\Join}_r= \langle \langle
		\mathscr A_0, x_r  \rangle \rangle$ is $G$-invariant
		and $h^*(G|_{S^{\Join}_r})=\log 2$.
		\item [(c)] For every $m\in \N$, $G(D_m^*) =D_{m+1}^*$.
	\end{itemize}
\end{lem}

\begin{proof}
Trivially, $h^*(G)\geq h^*(T) = \log 2$. By Proposition~\ref{P}(c), the elements of IN-tuples lie in $\Omega (G) = A$ and it follows from Lemma \ref{L:G-log2-2case} that there is no IN-triple for $G$. Hence (a). By the construction of $G$ we obviously have (c) and also the fact that $S^{\Join}_r= \langle \langle \mathscr A_0, x_r  \rangle \rangle$ is $G$-invariant. Since $h^*(G)= h^*(T) = \log 2$  and clearly also $h^*(T|_{S^{\Join}_r\cap X_1})=\log 2$, we get $h^*(G|_{S^{\Join}_r})=\log 2$, which finishes the proof of (b).
\end{proof}

\subsection{Properties of continuous selfmaps of $X$ and proof that $S(X)=\{0, \log 2\}$}\label{SS:prooflog2}

As in Section~\ref{S:X}, our continuum $X$ is the union of bricks (now with the point $a_{\infty}$ added). The snake is homeomorphic to that from Section~\ref{S:X}, but the heads are substantially different. In Section~\ref{S:X} the head was just one brick, namely the Cook continuum $\mathscr K_0$. Now the head $\mathscr A_0$ is the necklace of homeomorphic Cook continua $\mathscr H_k$ , together with the point $a_{\infty}$.  The list of all bricks is: $\mathscr \mathscr H_k$ ($k\in \mathbb Z$), $\mathscr K_1^0, \mathscr K_2^0, \dots,
\mathscr K_1^1, \mathscr K_2^1, \dots, \dots$. While the construction of $X$ and the proofs of the properties of $G$ were now much more complicated than in the case $S(X)=\{0,\infty\}$, fortunately the analogues of the results from Subsection~\ref{SS:cont-bricks} on properties of continuous selfmaps of $X$ are true, with almost the same proofs (for the results involving only the snake this is trivial, but many of the results involve explicitly or implicitly also the head, i.e. the ifluence of the head on their validity is not apriori excluded). There are only few differences due to the fact that $\mathscr K_0$ is replaced by $\mathscr A_0$. We are going to describe them.

	\begin{quotation}
		{\bf Standing notation for the rest of Section~\ref{S:cont zero-log2}:}
		In the rest of the section, $X$ denotes the space $X$ constructed above
		in~\eqref{X-definition-log2} and $F$ denotes a continuous map $X\to X$.
	\end{quotation}

\medskip

The analogue of Lemma~\ref{L:Kim to Kjn} is clearly true, with the continuum $\mathscr K_0$ in the parts (b) and (c) replaced by any of the continua $\mathscr H_k$. What is really important is that the following complete analogue of Lemma~\ref{L:map} holds.

\begin{lem}\label{L:map-log2}
	If $B$ is a  brick then $F(B)$ is either a singleton or a brick homeomorphic to $B$.		
\end{lem}

\begin{proof}
The proof is basically the same as the proof of Lemma~\ref{L:map}, with one exception. The Case 1 is now more complicated, because the head is more complicated. So, let $B$ be a brick, $F|_B=G$ and $G(B)$ be a non-degenerate sub-continuum of the head. We distinguish two possibilities.

If $G(B) \subseteq \mathscr H_i$ for some $i\in \mathbb Z$, , we have a brick $B$ (in the snake or in the head) mapped onto a non-degenerate continuum $G(B)$ in $\mathscr H_i$. Due to the just discussed analogue of Lemma~\ref{L:Kim to Kjn}(b), $B$ is not a brick in the snake. So, $B= \mathscr H_j$ for some $j\in \mathbb Z$ and since $\mathscr H_j$ and $\mathscr H_i$ are copies of the same Cook continuum, we necessarily have $G(B)=\mathscr H_i$ and we are done.

Now asssume that
\begin{equation}\label{Eq:GBnotinHi}
\text{$G(B)\subseteq$ head is not a subset of any of the bricks $\mathscr H_i$~.}
\end{equation}
We are going to show that this assumption leads to a contradiction.

Since we assume that $G(B)$ is non-degenerate, it intersects the interior of a brick in the head (here we speak on the interiors in the topology of the head, so the interior $\mathscr H_i^\circ$ of  $\mathscr H_i$ is obtained from  $\mathscr H_i$ by removing the two extremal points). Fix $k\in \mathbb Z$ such that
\[
\text{$G(B)\cap \mathscr H_k^\circ\,$ contains a point $a$.}
\]
By~\eqref{Eq:GBnotinHi}, there is $l\neq k$ such that also $G(B)\cap \mathscr H_l^\circ \neq \emptyset$.
Then, since $G(B)$ is connected, it necessarily contains at least one extremal point of $\mathscr H_k$.
The set $U:=G^{-1}(\mathscr H_k^\circ)$ is an open subset of $B$ containing a point $a^*$ with $G(a^*)=a$. Since $G(B)$ contains an extremal point of $\mathscr H_k$ and $G(U)$ does not, $U$ is a proper subset of $B$. Let $K$ be the component of $\overline U$ containing $a^*$. Since $K\subseteq \overline U$, we have $G(K) \subseteq \mathscr H_k$. By the Boundary Bumping Theorem, see Theorem~\ref{T:BBT}, $K$ contains a point from the boundary of $U$ and since this point does not belong to the open set $U$, its $G$-image is necessarily an extremal point of $\mathscr H_k$. Thus $K\subseteq B$ and $G(K)\subseteq \mathscr H_k$ are non-degenerate continua, the latter one containing both $a$ and an extremal point of $\mathscr H_k$. Now distinguish two cases.

If $B$ is a brick in the snake then, by our construction of bricks, $K$ and $G(K)$ are copies of disjoint non-degenerate subcontinua of one planar Cook continuum, which gives a contradiction.

If $B=\mathscr H_s$ for some $s\in \mathbb Z$, we get a contradiction as follows. First realize that $G(B)$ contains also an extremal point of $\mathscr H_k$. This is because, by the argument above, $B$ contains a non-degenerate continuum $K$ such that $G(K) \subseteq \mathscr H_k$ contains both a non-extremal point and an extremal point of $\mathscr H_k$. By~\eqref{Eq:mathscrHk},  the sets $B=\mathscr H_s$ and $\mathscr H_k$ are copies, even under similitudes, of the same Cook continuum $\mathscr H_0$ (recall also that these similitudes preserve extremal points). Since $\mathscr H_0$ is Cook, one non-degenerate sub-continuum of $\mathscr H_0$ is a continuous image of another non-degenerate sub-continuum of $\mathscr H_0$ if and only if the two sub-continua coincide. Therefore the fact that $G(K) \subseteq \mathscr H_k$ contains an extremal point of $\mathscr H_k$ implies that also $K\subseteq B=\mathscr H_s$ contains an extremal point of $B$.

We have shown that if $G(B)$ intersects $\mathscr H_k^\circ$ (and we know that in our situation there are at least two such indices $k$) then at least one extremal point of $B$ is mapped by $G$ to an extremal point of $\mathscr H_k$. As a consequence we get that $G(B)$ cannot intersect the interiors of three pairwise disjoint continua $\mathscr H_i$, because $B$ has only two extremal points. Hence there exists $m\in \mathbb Z$ such that $G(B)\cap \mathscr H_m^\circ = \emptyset$. Fix again  $k\in \mathbb Z$ with $G(B)\cap \mathscr H_k^\circ \neq \emptyset$ and denote by $r$ the monotone retract $\mathscr A_0 \setminus \mathscr H_m^\circ \to \mathscr H_k$ sending all the points from
$(\mathscr A_0 \setminus \mathscr H_m^\circ)\setminus \mathscr H_k$ to the extremal points of $\mathscr H_k$. Then $\Phi:=r\circ G: B\to \mathscr H_k$ is continuous and non-constant (because $G(B)$ contains a non-degenerate continuum $G(K)$ in $\mathscr H_k$, as discussed above).
Hence $\Phi$ is the unique homeomorphism from $B = \mathscr H_s$ onto $\mathscr H_k$, in fact $\Phi =  Z_k \circ Z_s^{-1}$, see~\eqref{Eq:mathscrHk}. It follows that for $B^\circ = \mathscr H_s^\circ$ we have $\Phi (B^\circ)= \mathscr H_k^\circ$ and so obviously $G(B^\circ)= \mathscr H_k^\circ$, even $\Phi|_{B^\circ} = G|_{B^\circ}$. Then also $\Phi|_{B} = G|_{B}$ and so $G(B)= \mathscr H_k$, a contradiction with~\eqref{Eq:GBnotinHi}.
\end{proof}

Due to this lemma, also all the other results from Subsection~\ref{SS:cont-bricks} basically work for our candidate $X$ for the equality $S(X)=\{0,\log 2\}$.  It is straightforward to check that the following claims work for our space $X$ with the same proofs, up to easy modifications, as in Section~\ref{S:X}: Corollary~\ref{C:union}, Corollary~\ref{C:brick-brick}, Corollary~\ref{C:hull}, Lemma~\ref{L:const 1} (with a small modification in the proof of (a) and with $\mathscr A_0$ instead of $\mathscr K_0$ in (f)), Lemma~\ref{L:metalemma} (just remember that now $m=0,1,\dots$ and not $m=1,2,\dots$, so in (P5) we have now $m\geq 1$). Besides Lemma~\ref{L:map-log2}, next crucial lemma is the following analogue of Lemma~\ref{L:const 2}.

\begin{lem}\label{L:const 2-log 2}
		\begin{itemize}
			\item [(a)] If the snake is not $F$-invariant then $F$ is constant.
			\item [(b)] If the set $\Sigma$ is not $F$-invariant then $F$ is constant.
			\item [(c)] If $F(\mathscr A_0) = \{z_0\}$ for some $z_0\in \mathscr A_0$,
			then $F(X)=\{z_0\}$ and so $F$ is constant.
		\end{itemize}
\end{lem}

\begin{proof}
The proof is almost the same as that of Lemma~\ref{L:const 2} and so we comment only the differences. First, where the proof of Lemma~\ref{L:const 2} refers to some lemmas from Subsection~\ref{SS:cont-bricks}, now we use the discussed analogues of them. In particular, the proof of (c) does not require any other modifications.

At the very end of the proof of (a), now it is not true that $X$ is the closure of the snake. However, the whole set $A=\{a_i:\, i\in \mathbb Z\} \cup \{a_{\infty}\}$ is in the closure of the snake. Therefore, once we know that the snake is mapped by $F$ to the point $z_0 \in \mathscr A_0$, we also have $F(A)=\{z_0\}$. In particular, every brick $\mathscr H_k$ in the head contains at least two points mapped by $F$ to $z_0$ and so, by the analogue of Lemma~\ref{L:const 1}(c), the whole $\mathscr H_k$ is mapped to $z_0$. Hence $F(X)=\{z_0\}$.

Exactly the same argument as above, can be used at the end of the proof of (b) to get $F(X)=\{z\}$.
\end{proof}

Further, the analogue of Lemma~\ref{L:ord-pres} works with the same proof. Moreover, we also have Lemma~\ref{L:mapDD}.

Also the results from Subsection~\ref{SS:induced} can be carried over to our space $X$. Since now the head is $\mathscr A_0$, in~\eqref{Eq:partition} we replace $\mathscr K_0$ by $\mathscr A_0$. Then Lemmas~\ref{L:antiwell2} and~\ref{L:existsFhat} obviously still work. In particular, we can consider the induced function $\widehat F \colon \widehat X \to \widehat X$ defined by  $\widehat F(A)=B$ when $F(A)\subseteq B$ for $A, B \in \widehat X$.

Recall that just after the proof of Lemma~\ref{L:existsFhat} we remark that $\widehat F(A)=B$ means that $F(A)$ is a singleton in $B$ or coincides with $B$. This is still true, though the case $A=B=\mathscr A_0$ needs an explanation. The following lemma shows slightly more.

\begin{lem}\label{L:phionhead}
	Let $\varphi: \mathscr A_0 \to \mathscr A_0$ is a continuous map. Then there are the following possibilities.
	\begin{itemize}
		\item [(1)] If $\varphi (a_{\infty}) \in \mathscr H_k^{\circ}$ for some $k\in \mathbb Z$, then $\varphi$ is constant.
		\item [(2)] If $\varphi (a_{\infty}) = a_i$ for some $i\in \mathbb Z$, then $\varphi$ is constant.
		\item [(3)] If $\varphi (a_{\infty}) = a_{\infty}$ then $\varphi$ is either constant or surjective.
	\end{itemize}
	So, $\varphi$ is either constant or surjective and the space $\mathscr A_0$ has fixed point property.	
\end{lem}

\begin{proof}
(1) Put $z= \varphi (a_{\infty}) \in \mathscr H_k^{\circ}$.  By continuity, the whole bricks $\mathscr H_i$ with $|i|$ very large are also mapped into $\mathscr H_k^{\circ}$, and so, being copies of the Cook continuum $\mathscr H_k$, they are necessarily mapped to the point $z$. Then, for analogous reasons, also the remaining finitely many bricks are mapped to $z$. 	

(2) The bricks $\mathscr H_i$ with $|i|$ very large are mapped to a small neighbourhood of $a_i$ which does not contain any whole brick. Hence they are mapped to $a_i$. The finitely many other bricks of $\mathscr A_0$ form a continuum whose two `endpoints' are mapped to $a_i$. Each of those bricks is mapped either to a point or onto a brick with sending the first and the last point onto the first and the last point, respectively. It follows that the only possibility is that they all are mapped to $a_i$.

(3) Let $\varphi (a_{\infty}) = a_{\infty}$. Again, each of the bricks is mapped either to a point or onto a brick, with sending the first and the last point onto the first and the last point, respectively. This rule implies that if $\varphi$ is not constant, then $\varphi(\mathscr A_0)$ covers all the bricks in $\mathscr A_0$.
\end{proof}

Let us also remark the difference: If $\varphi (\mathscr K_0)=\mathscr K_0$, then all the points in $\mathscr K_0$ are fixed. Now, with $\mathscr K_0$ replaced by $\mathscr A_0$ this is not automatically the case, in general we have only $\Fix (\varphi) \subseteq \mathscr A_0$.

Checking next results, we see that Lemma~\ref{L:how br to br}, Corollary~\ref{C:gemini}, Lemma~\ref{L:2xFix} and Lemma~\ref{L:Darboux} work without any changes. Also Corollary~\ref{C:arrow} works, with $\mathscr K_0$ replaced by $\mathscr A_0$. Then we have the following complete analogue of Lemma~\ref{L:FixFhat}.

\begin{lem}\label{L:FixFhat-log2}
	The set $\Fix (\widehat F)$	is nonempty, has the smallest element and the
	largest element, and is connected. Moreover, if $\Fix (\widehat F)$ has more
	than one element, then $\Fix (F) = \bigcup \Fix (\widehat F)$.
\end{lem}

The proof is the same, just at the end of the proof, instead of saying that if $F$ is identity on the sub-snake $((\mathscr K_0, f^m_i \rangle
\rangle$ then it is identity also on $\langle \langle \mathscr K_0, f^m_i \rangle \rangle$, we use the following argument: If $F$ is identity on the sub-snake $((\mathscr A_0, f^m_i \rangle \rangle$, it is identity also on the set $A$, because $A$ is in the closure of that sub-snake. Thus, every brick $\mathscr H_k$ in the head contains at least two fixed points. Then obviously all the points in the Cook continua $\mathscr H_k$ are fixed (formally, one can use the analogue of Lemma~\ref{L:const 1}).

The last result in Subsection~\ref{SS:induced}, Corollary~\ref{C:FixF}, has to be modified as follows.

\begin{cor}\label{C:FixF-log2}
		If $F$ is not constant, then the set $\Fix (F)$ is either a subset of $\mathscr A_0$ or of
		the form $\langle \langle \mathscr A_0, f^m_i \rangle \rangle$ or  $\langle
		\langle  f^n_j, f^m_i \rangle \rangle$ for some $f^n_j \preccurlyeq
		f^m_i$.
\end{cor}

The proof differs from that of Corollary~\ref{C:FixF} only in the beginning: If $F (\mathscr
A_0)\subseteq \mathscr A_0$ then, by Lemma~\ref{L:phionhead}, $F$ has a fixed point in $\mathscr A_0$.
Then, combining (the analogue of) Corollary~\ref{C:gemini} and Lemma~\ref{L:FixFhat-log2}  we get that either $\Fix (F)
\subseteq  \mathscr A_0$ or $\Fix (F) = \langle \langle \mathscr A_0, f^m_i \rangle
\rangle$ for some $m$ and $i$. The rest of the proof is the same as in the proof of Corollary~\ref{C:FixF}.

\medskip
Let us also remark that in Corollary~\ref{C:FixF-log2}, instead of ``$\Fix (F)$ is either a subset of $\mathscr A_0$" one could in fact write ``$\Fix (F)$ is either $\{a_{\infty}\}$". This follows from Lemma~\ref{L:jumps-log2}(c) below.

\medskip

We have thus shown that all the results from Subsections~\ref{SS:cont-bricks} and~\ref{SS:induced} do work, at least in slightly modified forms, also for our space $X$ which is a candidate for $S(X)=\{0,\log 2\}$. We show that this is sufficient for finishing the proof that really $S(X)=\{0,\log 2\}$.

The analogue of Lemma~\ref{L:left} works with almost the same  proof. (The only modification is needed in the first paragraph of the proof: Now it is not true that the head is in the closure of the snake. However, the set $A$ is, and so every brick in the head contains points from the closure of the sub-snake. It follows that if the sub-snake is mapped by $F^2$ to $p$, in view of Lemma~\ref{L:map-log2} so is the head.)
Further, the analogue of Lemma~\ref{L:right} works even with the same proof. The same is true for the analogue of Lemma~\ref{L:jumptoanotherblock}.

Next lemma is quite obvious but for clarity we state it explicitly.

\begin{lem}\label{L:fromheadtosnake}
	\begin{enumerate}
	\item [(1)] If there exists $a \in \mathscr A_0$ such that $F(a) \notin \mathscr A_0$, then
			$F(\mathscr A_0)=F(a)$.
	\item [(2)] If there exists $b$ in the snake such that $F(b) \in  \mathscr A_0$,
			then $F(X)=F(b)$.
	\end{enumerate}
\end{lem}
	
\begin{proof}
(2) is just a reformulation of Lemma~\ref{L:const 2-log 2}(a). To prove (1), let $\mathscr H_k$ be the brick containing the point $a$. Since the bricks in the snake are not homeomorphic to those in the head, Lemma~\ref{L:map-log2} gives $F(\mathscr H_k) = \{F(a)\}$. For the same reason, the neighboring bricks $\mathscr H_{k-1}$ and $\mathscr H_{k+1}$ are also mapped to $F(a)$. By induction, this is true for all the bricks in the head and the result follows.
\end{proof}

Then we get the following analogue of Lemma~\ref{L:jumps}.

\begin{lem}\label{L:jumps-log2}
	Assume that $F$ is not constant and has no fixed point in the snake.
	\begin{itemize}
		\item [(a)] For every $m$, $\jump(m+1) \in \{\jump(m)-1, \ \jump(m)\}$.
		\item [(b)] The sequence $\jump(1), \jump(2), \dots$ is eventually constant.
		\item [(c)] There exist positive integers $r$ and $N$ such that on
		$S^{\Join}_r= \langle \langle \mathscr A_0, x_r  \rangle \rangle$ we have
		$$
		F|_{S^{\Join}_r} = G^N|_{S^{\Join}_r}
		$$
		where $G$ is the map from Lemma~\ref{L:existsG-log2}.
	\end{itemize}
\end{lem}

The proof differs from that of Lemma~\ref{L:jumps} only at the very end of it. When we already get that $F$ and $G^N$ coincide on
the sub-snake $(( \mathscr A_0, x_r  \rangle \rangle$, we first deduce that, by continuity, they coincide also on the set $A$ which is in the closure of the sub-snake. Hence, by Lemma~\ref{L:fromheadtosnake}(1), $F(\mathscr A_0)\subseteq \mathscr A_0$. For every $k\in \mathbb Z$, $G^N(\mathscr H_k) = \mathscr H_{k+N}$. Since $F$ coincides with $G^N$ at the extremal points of  $\mathscr H_k$, by Lemma~\ref{L:map-log2} also $F(\mathscr H_k) = \mathscr H_{k+N}$. By Lemma~\ref{L:Cook-properties}(3), $F$ coincides with $G^N$ on $\mathscr H_k$.  Hence $F$ and $G^N$ coincide on $S^{\Join}_r$ as required.

\medskip

Finally, we have the following analogue of Proposition~\ref{P:zero inf}.

\begin{prop}\label{P:zero log2}
	For the one-dimensional continuum $X$ constructed above in~(\ref{X-definition-log2}) we have
	$S(X)=\{0,\log 2\}$. Moreover, if $F: X\to X$ is a continuous map then $h^*(F) = \log 2$ if $F$ is
	non-constant and has
	no fixed point in the snake, otherwise $h^*(F) = 0$.
\end{prop}

\begin{proof}
	Let $F: X\to X$ be a continuous map. If $F$ is constant then $h^*(F)=0$. Now let
	$F$ be non-constant.	
	
	First assume that $F$ has a fixed point in the snake. By Corollary~\ref{C:FixF-log2},
	$\Fix (F)$ is either of the form $\langle \langle \mathscr K_0, f^m_i \rangle
	\rangle$ or  $\langle \langle  f^n_j, f^m_i \rangle \rangle$ for some $f^n_j
	\preccurlyeq f^m_i$. Then, by Lemmas~\ref{L:left} and~\ref{L:right}, there
	exists a positive integer $N$ such that $F^N(X)=\Fix(F)$. This clearly implies
	that $h_A(F)=0$ for any sequence $A$ and so $h^*(F)=0$ (alternatively, use
	Proposition~\ref{P2}(a)).
	
	Now assume that $F$ has no fixed point in the snake. Then, by
	Lemma~\ref{L:jumps-log2}(c), there exist positive integers $r$ and $N$ such that on
	$S^{\Join}_r= \langle \langle \mathscr A_0, x_r  \rangle \rangle$ we have
	$$
	F|_{S^{\Join}_r} = G^N|_{S^{\Join}_r}
	$$
	where $G$ is the map from Lemma~\ref{L:existsG-log2}. So,
	\[
	h^*(F) \geq h^*(F|_{S^{\Join}_r}) = h^*(G^N|_{S^{\Join}_r}) = h^*(G|_{S^{\Join}_r})=\log 2~.
	\]
	On the other hand, the jumps in Lemma~\ref{L:jumps-log2} are positive integers and so, for some positive integer $M$, $F^M(X) \subseteq S^{\Join}_r$.  Then repeated use of Proposition~\ref{P2}(a) gives
	\[
	h^*(F)= h^*(F|_{F^M(X)}) \leq h^*(F|_{S^{\Join}_r}) = \log 2~.
	\]

	We have shown that, for every continuous map $F$ on $X$, either $h^*(F)=0$ or
	$h^*(F)=\log 2$ and so the proposition is proved.
\end{proof}


\section{Generalization from $\log 2$ to $\log m$}\label{S:generalization}
In this section, we will extend the result of Section~\ref{S:X1-T1-log2} and Section~\ref{S:cont zero-log2}. We show
that for any $m \ge 3$, there exists a space $X$ with $S(X)=\{0, \log m\}$.

\subsection{A map $T: X_1 \to X_1$ with $h^*(T)=\log m$ for $m\ge 3$}\label{S:X1-T1-logm}
To get the system $(X_1,T)$ with $h^*(T)=\log m$, $m\ge 3$, one can repeat the construction from Section~\ref{S:X1-T1-log2}
almost word by word, with just small modifications. Therefore we comment only differences here. For terminology, the reader
is referred to Section~\ref{S:X1-T1-log2}.

Just as in Section \ref{S:X1-T1-log2}, $X_1=A\sqcup Y $ where $A=\{a_{i}:i \in \Z\}\cup \{a_{\infty}\}$
and $Y=\{x_0,x_1,\dots\}$ is the trajectory of $x_0$. Again, $T(a_i)=a_{i+1},
T{a_{\infty}}=a_{\infty}$ and $Tx_i=x_{i+1}$, $i=0,1,\dots$.
But this time, the trajectory of $x_0$ is chosen in such a way that
the following requirements are fulfilled:
\begin{enumerate}
	\item [(1')] $(a_0,a_1,a_2,\dots,a_{m-1})$ is an IN-tuple of length $m$.
	
	\item [(2')] For $|j| \geq m$, $(a_0,a_j)$ is not an IN-pair.
	
	\item [(3')] $(a_0,a_{\infty})$ is not an IN-pair.	
\end{enumerate}
Similarly as in the beginning of Section~\ref{S:X1-T1-log2} (see the text between~\eqref{Eq:approachpi0} and~\eqref{Eq:block-outgap}),
one can see that to fulfill these three requirements, it is sufficient to fulfill the following two requirements:
\begin{enumerate}
	\item [(R1')] For every $k$, the tuple $(U^k(a_0),
	U^k(a_1),\dots,U^{k}(a_{m-1}))$ has an independence set of times of cardinality
	$k+1$ .
	\item [(R2')] The tuple $(U^1(a_0), U^1(a_j))$ does not have an independence
	set of times of cardinality $5$ whenever $|j|\geq m$ or $j=\infty$.
\end{enumerate}

Similarly as before, the sequence $x_0, x_1, \dots$ is described as:
\begin{equation}\label{Eq:block-outgap-logn}
x_0, x_1, \dots = \text{1st block}, \,  \text{1st outer gap}, \, \text{2nd
	block}, \, \text{2nd outer gap}, \, \dots \notag
\end{equation}
and for every $k$, the $k$-th block is a concatenation of pieces and inner
gaps:
\begin{equation}\label{Eq:pieces-ingaps-logn}
\text{k-th block} = \text{1st piece}, \, \text{1st inner gap}, \, \text{2nd
	piece}, \, \dots, \text{$(m^{k+1}-1)$st inner gap}, \, \text{$(m^{k+1})$-th
	piece}. \notag
\end{equation}
A proper choice of pieces will ensure (R1') and proper choices of inner gaps and
outer gaps will ensure (R2').

In Section~\ref{S:X1-T1-log2}, i.e. in the case $h^{*}(T)=\log 2$,
to ensure (R1), for each $k$ there were $2^{k+1}$ different pieces. Now, to ensure (R1'),
for each $k$ we need $m^{k+1}$ pieces $P(k,l)$ \index{$P(k,l)$} and to separate these different pieces we
need more inner gaps $ig(k,j)$ in each block $B(k)$. Denote
\begin{align}\label{eq:notFk-logm}
\begin{split}
F(k)= & \{0,1,\dots,m-1\}^{\{0,1,2,\cdots,k\}}=\{s_l: 1 \le l \le m^{k+1}\}, \\
& s_1=(0,0,\dots,0),\, \dots, \, s_{m^{k+1}}=(m-1,m-1,\dots,m-1).
\end{split}
\end{align}
For each $k$ we fix, once and for all, a choice of $m^{k+1}$ functions $s_l$.
Say, let them be
ordered from $s_1$ to $s_{m^{k+1}}$ lexicographically (then $s_1$ and
$s_{m^{k+1}}$ are constant, as written above).\footnote{Since the functions
	$s_l$ depend both on $k$ nad $l$, we should in fact write $s_{(k,l)}$. We abuse
	notation here hoping that no misunderstanding will arise.}

For each $l$, the piece $P(k,l)$ will be a finite sequence of length $n_k^k+1$
of the form
\begin{equation}\label{eq:Pkl-logm}
P(k,l) = x_j, x_{j+1}, \dots, x_{j+n_1^k},  \dots, x_{j+n_2^k},  \dots,
x_{j+n_k^k}
\end{equation}
where $j=j(k,l) \geq 0$ will depend on both $k$ and $l$, but $n_1^k, \dots,
n_k^k$ only on $k$ and not on $l$,\footnote{Therefore we write $n_i^k$ rather
	than  $n_i^{(k,l)}$.} such that
\begin{equation}\label{eq: pieces requirment-logm}
x_j \in U^{k}(a_{s_l(0)}), \,\, T^{n_i^k}x_j=x_{j+n_i^k} \in
U^{k}(a_{s_l(i)}), \quad 1 \le i \le k~.
\end{equation}
Equivalently, if we put $n_0^k:=0$, $x_j \in
\bigcap_{i=0}^{k}T^{-n_i^k}U^k(a_{s_l(i)})$.
Since this will be true for each $s_l$, the set
\begin{equation}\label{eq:notNk-logm}
N(k)=\{n_0^k=0, n_1^k, \dots, n_k^k\}
\end{equation}
will be an independence set of times of length $k+1$ for $(U^k(a_0), U^k(a_1),\dots, U^k(a_{m-1}))$.

Just as in Section \ref{S:X1-T1-log2},
the piece $P(k,l)$ consists of $k$ shorter
sequences, called \emph{winds}:
\begin{equation}\label{eq:winds-logm}
W^{(k,l)}_1= x_j (= x_{j+n_0^k}), x_{j+1}, \dots, x_{j+n_1^k}, \quad  \dots,
\quad W^{(k,l)}_k = x_{j+n_{k-1}^k}, \dots, x_{j+n_k^k}~.
\end{equation}
The wind $W^{(k,l)}_i$ starts in $x_{j+n_{i-1}^k}\in U^k(a_{s_l(i-1)})$ and ends
in $x_{j+n_{i}^k}\in U^k(a_{s_l(i)})$.

The pieces and gaps are built analogously as in Section~\ref{S:X1-T1-log2}.
For proving that $h^*(T)$ is $\log m$ (and not higher), we specify appropriately
the \emph{independence set of times} $N(k)=\{n_0^k=0,n_1^k,n_2^k\cdots,
n_{k}^k\}$ \index{$N(k)$}, the \emph{outer gap length} $og(k)$ \index{$og(k)$} and the \emph{inner gap lengths}
$ig(k,j)$, $ 1 \le j \le m^{k+1}-1$\index{$ig(k,j)$}.
Inductively, we require the following (for $k=1,2,\dots$):
\begin{enumerate}
	\item [(L1')] $w^1_1 = 3m+2$ or, equivalently, $n_1^1 = 3m+1$.
	\item [(L2')] $|IG| \gg |\pre (IG)|$ whenever $IG$ is an inner gap.
	\item [(L3')] $|OG| \gg |\pre (OG)|$ whenever $OG$ is an outer gap.
	\item [(L4')] $|W| \gg  |\pre (W)|$	whenever $W$ is a wind in the first piece of
	a block.
	(Recall that the lengths of the winds in $P(k,l)$ are the same as those in
	$P(k,1)$.)
\end{enumerate}

For any $k$, the  inner gaps $IG(k,j), 1
\le j \le m^{k+1}-1$ \index{$IG(k,j)$} and outer gaps $OG(k)$ \index{$OG(k)$} can be chosen similarly as in
Section \ref{S:X1-T1-log2}, only here we need more inner gaps. Making no difference between
a (finite) sequence and its set of values, for the $k$-th block we have
\[
B(k)=\bigcup_{j=1}^{m^{k+1}-1}(P(k,j)\cup IG(k,j))\cup P(k,m^{k+1}).
\]
As in Section~\ref{S:X1-T1-log2}, the set $Y_k = B(k) \cup OG(k)$ is called the
\emph{$k$-th level} of the set
\[
Y=\bigcup_{k=1}^{\infty}(B(k) \cup OG(k)).
\]
The set $Y$, properly ordered, is the trajectory of the point $x_0$, so $Y=\{x_0,x_1,\dots\}$.

We  define $\tilde{P}(k,l)$ \index{$\tilde{P}(k,l)$}, the  \emph{$l$-th part in the $k$-th
level},  similarly as in Section \ref{S:X1-T1-log2},  but with some modification.
Recall that in Section~\ref{S:X1-T1-log2}, the core of the definition was to make sure that
if $x_i \in U^{1}(a_0)$ then $x_i$ was contained in some part. If we denote by $y_{(k,l),0}$
the first point of $P(k,l)$, then $\tilde{P}(k,l)$ is defined
as follows:
\begin{itemize}
	\item If $y_{(k,l),0} \in U^{1}(a_0)$, then set $\tilde{P}(k,l)=P(k,l) $;
	\item If $y_{(k,l),0} \notin U^{1}(a_0)$ (i.e. $s_{l}(0)\neq 0$), then set
          $$\tilde{P}(k,l)=P(k,l) \cup	\{x_t:\, j-s_l(0) \le t \le j-1 \}$$
          where $x_{j} =y_{(k,l),0}$.
\end{itemize}

We have, analogously as in~\eqref{eq:no0outsidepart}, that
\begin{equation}\label{eq:no0outsidepart-logm}
\text{if $x_i \in U^{1}(a_0)$, then $x_i \in \tilde{P}(k,l)$ for some $k$ and
	$l$.}
\end{equation}

As an analogue of Lemma~\ref{L: distance of two 0}, we have the following lemma. The main difference is in (2);
recall that in Lemma~\ref{L: distance of two 0}(2) we had $t \in \{n_c^k-n_d^k-1,n_c^k-n_d^k, n_c^k-n_d^k+1\}$.
\begin{lem}\label{L: distance of two 0-logm}
	Let $k>0,l\in \{1,\dots, m^{k+1}\}$.
	\begin{enumerate}
		\item The piece $P(k,l)$ is of the form
			$$
			P(k,l) = x_j, x_{j+1}, \dots, x_{j+n_1^k},  \dots, x_{j+n_2^k},  \dots,
			x_{j+n_k^k}~.
			$$
           	If $s_l\in F(k)$ is the function corresponding to $P(k,l)$, then we can write (here $n_0^k=0$)
		\begin{equation}\label{eq:listnew-logm}
		\tilde{P}(k,l) \cap U^{1}(a_0) =\{x_{j+n_0^k-s_l(0)}, x_{j+n_1^k -s_l(1)}, \dots, x_{j+n_k^k-s_l(k)}\}.		
		\end{equation}
		\item If two points in $\tilde{P}(k,l)\cap U^{1}(a_0)$ have iterative distance
		$t>0$, then
		\begin{equation*}
		t \in \{n_c^k-n_d^k-(m-1),n_c^k-n_d^k-(m-1)+1,\dots , n_c^k-n_d^k+(m-1)\}
		\end{equation*}
		for some $0\le d<c \le k$ and no other pair of points in $\tilde{P}(k,l)\cap U^{1}(a_0)$ has the same iterative distance~$t$.
	\end{enumerate}
\end{lem}
\begin{proof}
	(1) As in~(\ref{eq:Pkl-logm}), we have $P(k,l) = x_j, x_{j+1}, \dots, x_{j+n_1^k},  \dots, x_{j+n_2^k},  \dots,
	x_{j+n_k^k}.$ For each $c\in \{0,1,\dots, k\}$
	we have $x_{j+n^k_c} \in U^1(a_{s_l(c)})$ and so
	$$
	x_{j+n^k_c-s_l(c)} \in U^1(a_0).
	$$
	This means that the list of all elements of $\tilde{P}(k,l) \cap U^{1}(a_0)$ is~(\ref{eq:listnew-logm}).

	(2) So, if $t>0$ and $x_s, x_{s+t}$ are two points in $\tilde{P}(k,l)\cap
	U^{1}(a_0)$, there are	
	$0\le d<c \le k$ such that
    $$x_s=x_{j+n^k_d-s_l(d)} \quad \text{and} \quad x_{s+t}=x_{j+n^k_c-s_l(c)}$$
    note that $s_l(c), s_l(d) \in \{0,1,\dots, m-1\}$, hence
    $$t \in \{n_c^k-n_d^k-(m-1),n_c^k-n_d^k-(m-1)+1,\dots , n_c^k-n_d^k+(m-1)\}.$$

    If $p<q$ and $x_p, x_q$ is \emph{another}
	pair of points in $\tilde{P}(k,l)\cap U^{1}(a_0)$, i.e. in the
	list~(\ref{eq:listnew-logm}), then either
	$x_p \neq x_{j+n_d^k-s_l(d)}$ or $x_{q} \neq x_{j+n^k_c-s_l(c)}$. Since by (L4')
	the lengths of the winds in $P(k,l)$ satisfy the inequalities
	$$
	n^k_1+1 \ll n^k_2 -n^k_1 +1 \ll \dots \ll n^k_k-n^k_{k-1}+1~,
	$$
	the iterative distance $q-p$ of $x_q$ and $x_p$ is different from $t$.
\end{proof}

Similarly as in Section \ref{S:X1-T1-log2},
$$
E(k,l) = \{x_{j}, x_{j+n_1^k},  \dots, x_{j+n_2^k}\}
$$ \index{$E(k,l)$}
is the set of the \emph{endpoints of the winds} in $P(k,l)$.
But here we redefine the notion of \emph{almost coincidence}.
By saying that two points almost coincide \index{almost coincide}, we now mean that their iterative distance is at most $m-1$
(rather than at most $1$ as in Section~\ref{S:X1-T1-log2}).
By Lemma~\ref{L: distance of two 0-logm}(1), the points of $U^1(a_0) \cap \tilde{P}(k,l)$ \emph{almost coincide} with the endpoints of the winds.

The following analogue of Lemma \ref{L:notL-shiftable} works with the same proof.
\begin{lem}\label{L:notL-shiftable-logm}
		Let $s \ge 0$, $t>0$ and $x_s, x_{s+t} \in U^1(a_0)$. If the points $x_s$ and $x_{s+t}$ belong to different blocks,
		then they are not $U^1(a_0)$-left shiftable.
\end{lem}

Analogue to Lemma \ref{L:notL-shiftable2}, we have the following lemma.
\begin{lem}\label{L:notL-shiftable2-logm}
Let $s \ge 0$, $t>0$ and $x_s, x_{s+t} \in U^1(a_0) \cap \tilde{P}(k,l)$. Then there is no $h<0$ such that $x_{s+h}\in U^{1}(a_0) \cap \bigcup_{i=1}^{k-1}B(i)$ and $x_{s+t+h} \in U^{1}(a_0) \cap P(k,1)$.	
\end{lem}
\begin{proof}
The proof is almost the same as that of Lemma \ref{L:notL-shiftable2} but for completeness we give it here. Recall that when we say that two points almost coincide, it means that their iterative distance is at most $m-1$.

	Let $k\geq 2$, otherwise there is nothing to prove.
	Suppose, on the contrary, that there is $h<0$ with that property (hence $s+h\geq 0$).
	We have, as in~(\ref{eq:Pkl-logm}),
	$$
	P(k,l) = x_{j(l)}, x_{j(l)+1}, \dots, x_{j(l)+n_1^k},  \dots, x_{j(l)+n_2^k},  \dots, x_{j(l)+n_k^k},
	$$
	where the set $E(k,l) = \{x_{j(l)}, x_{j(l)+n_1^k},  \dots, x_{j(l)+n_k^k}\}$ is the set of the \emph{endpoints of the winds} in $P(k,l)$. Similarly,
	$$
	P(k,1) = x_{j(1)}, x_{j(1)+1}, \dots, x_{j(1)+n_1^k},  \dots, x_{j(1)+n_2^k},  \dots,	x_{j(1)+n_k^k},
	$$
	where $E(k,1) =\{x_{j(1)}, x_{j(1)+n_1^k},  \dots, x_{j(1)+n_k^k}\}$ is the set of the endpoints of the winds in $P(k,1)$.
	
	Since $x_s, x_{s+t} \in U^1(a_0) \cap\tilde{P}(k,l)$, by Lemma~\ref{L: distance of two 0-logm}(1) we know that they almost coincide with the $p$-th and the $q$-th elements in $E(k,l)$, for some $p<q$.  The left shift from $x_s, x_{s+t}$ to $x_{s+h}, x_{s+t+h}$ can be performed as the composition of two shorter left shifts. First, we shift $x_s, x_{s+t}$ to points $x_{s+\sigma}, x_{s+t+\sigma}$ which almost coincide with the $p$-th and the $q$-th elements in $E(k,1)$ (this is possible because the lengths of winds in $P(k,l)$ are the same as in $P(k,1)$). Then the point $x_{s+\sigma}$ is either in $P(k,1)$ or it is one of the last $m-1$ points of $OG(k-1)$. So we need to shift $x_{s+\sigma}, x_{s+t+\sigma}$ still to the left, now finally to $x_{s+h}, x_{s+t+h}$. Since $OG(k-1)$ does not contain points from $U^1(a_0)$, this shift has to be at least as long as it is the length of $OG(k-1)$, which is much larger than $m-1$. Since $x_{s+t+h}$ has to be in $U^1(a_0)$ and to the left of $x_{s+t+\sigma}$, this second shift (whose length is much larger than $m-1$) is of course at least as long as the iterative distance between $(q-1)$-st and $q$-th elements in $E(k,1)$ (see Lemma~\ref{L: distance of two 0-logm}(1)), meant in approximative sense, i.e. an error, now definitely not greater than $2(m-1)$, is possible when we claim this. This iterative distance is, by (L4'), much larger than $|\pre (W)|$ where $W$ is the wind whose endpoints are the $(q-1)$-st and $q$-th elements in $E(k,1)$. Since $x_{s+\sigma}$ almost coincides with the $p$-th element of $E(k,1)$ and $p\leq q-1$, we get that $s+h<0$, a contradiction.
\end{proof}	

The following is an analogue of Lemma~\ref{L:location}. The proof is the same as that of Lemma~\ref{L:location}, one just needs to replace the lemmas used in the proof by their analogues discussed above.
\begin{lem}\label{L:location-logm}
	Let $s \ge 0$, $t>0$ and let the points $x_s, x_{s+t} \in U^1(a_0)$ be
	$U^1(a_0)$-shiftable, i.e. there exists an $h \neq 0$  such that also
	$x_{s+h} \in U^{1}(a_0)$ and $x_{s+t+h} \in U^{1}(a_0)$. Then the following is true.
	\begin{enumerate}
		\item If $x_s, x_{s+t} \in \tilde{P}(k,i_1)$ for some $k$ and $i_1$, then $x_{s+h}, x_{s+t+h} \in \tilde{P}(k,i_2)$ for some $i_2\neq i_1$.
		\item If $x_s \in \tilde{P}(k,i_1)$, $x_{s+t} \in \tilde{P}(k,i_2)$
		for some $k$ and $i_1<i_2$, then
		$x_{s+h} \in \tilde{P}(k,i_1)$, $x_{s+t+h}\in \tilde{P}(k,i_2)$.
		\item The points $x_s, x_{s+t}, x_{s+h}, x_{s+t+h}$ belong to the same block $B(k)$, for some $k$.
		\item If  $x_s, x_{s+t} \in \tilde{P}(k,i)$ for some $k$ and $i$,
		then
		$x_s, x_{s+h}$ are in the ``similar positions"\index{similar positions}, meaning that if we write, as in~(\ref{eq:listnew-logm}), the point $x_s$ in the form
		$$
		x_s= x_{r+n^k_{c}-s_{i}(c)} \in \tilde{P}(k,i) \quad \text{for some} \quad 0 \le c \le k,
		$$
		then there exists $i'$ such that
		$$
		x_{s+h} =x_{r'+n^k_{c}-s_{i'}(c)} \in \tilde{P}(k,i') \quad \text{with the same} \quad 0 \le c \le k.
		$$
		Here $r,r' \ge 0$ are such that $x_{r}=y_{(k,i),0}$,
		$x_{r'}=y_{(k,i'),0}$ are the first points of the pieces $P(k,i), P(k,i')$, respectively, and $s_{i}, s_{i'} \in F(k)$ are
		the functions corresponding to the pieces $P(k,i), P(k,i')$, respectively.
		\item If $x_s \in \tilde{P}(k,i_1)$ and $x_{s+t} \in \tilde{P}(k,i_2)$
		for some $k$ and $i_1<i_2$,
		then $x_s, x_{s+t}$ are in the ``similar positions".
	\end{enumerate}
\end{lem}

Finally, as an analogue of Theorem~\ref{final-thm} we have the following theorem.
\begin{thm}\label{final-thm-logm} The system $(X_1, T)$ has the following properties.
	\begin{enumerate}
		\item [(1')] $(a_0,a_1,a_2,\dots,a_{m-1})$ is an IN-tuple of length $m$.
		
		\item [(2')] For $|j|\geq m$, $(a_0,a_j)$ is not an IN-pair.
		
		\item [(3')] $(a_0,a_{\infty})$ is not an IN-pair.
		
	\end{enumerate}
	Hence $h^*(T)=\log m$.
\end{thm}
\begin{proof}
The whole proof is almost the  same as that of Theorem~\ref{final-thm}. However, now $s_l(c) \in \{0,1,2,\dots,m-1\}$ and so we need
to modify the proof of $(2)$ Case $1$ (in a similar way we have modified the proof of Lemma~\ref{L: distance of two 0}(2) to get the proof of
Lemma~\ref{L: distance of two 0-logm}(2)). We are going to describe this modification.

(2') Just as in the proof of Theorem~\ref{final-thm}(2), it is sufficient to prove this for $j \geq m$. So fix $j\geq m$ and assume, on the contrary, that $(a_0,a_j)$ is an IN-pair. It follows that $(U^1(a_{0}),U^1(a_{j}))$ has an independence set of times of length $5$, i.e. there are pairwise distinct positive integers $l_{-1}<l_0 < l_1 < l_2 <l_3$ such that $\{l_{-1},l_0, l_1, l_2, l_3\}$ is an independence set of times for $(U^1(a_{0}),U^1(a_{j}))$. Then, in particular, there exist pairwise distinct $m_1, m_2, m_3, m_4 \in \N$ such that (notice that in the underlined inclusions we have $a_j$ rather than $a_0$)
	\begin{equation}\label{eq:r1-logm}
	x_{m_1+l_{-1}} \in U^1(a_{0}), \, x_{m_1+l_0} \in U^1(a_{0}), \, x_{m_1+l_1} \in U^1(a_{0}),\ x_{m_1+l_2} \in U^1(a_{0}), \ x_{m_1+l_3} \in U^1(a_{0}),
	\end{equation}
	\begin{equation}\label{eq:r2-logm}
	x_{m_2+l_{-1}} \in U^1(a_{0}), \, x_{m_2+l_0} \in U^1(a_{0}), \, x_{m_2+l_1} \in U^1(a_{0}),\ \underline{x_{m_2+l_2} \in U^1(a_{j})}, \ x_{m_2+l_3} \in U^1(a_{0}),
	\end{equation}
	\begin{equation}\label{eq:r3-logm}
	x_{m_3+l_{-1}} \in U^1(a_{0}), \, x_{m_3+l_0} \in U^1(a_{0}), \, \underline{x_{m_3+l_1} \in U^1(a_{j})},\ x_{m_3+l_2} \in U^1(a_{0}), \ \ x_{m_3+l_3} \in U^1(a_{0}),
	\end{equation}
	\begin{equation}\label{eq:r4-logm}
	x_{m_4+l_{-1}} \in U^1(a_{0}), \, \underline{x_{m_4+l_0} \in U^1(a_{j})}, \, x_{m_4+l_1} \in U^1(a_{0}),\ x_{m_4+l_2} \in U^1(a_{0}), \ \ x_{m_4+l_3} \in U^1(a_{0}).
	\end{equation}

   The same arguments as in the proof of Theorem~\ref{final-thm} show that the eight points in the first and last columns are in the same $B(k)$, hence all the points in this $4\times 5$ table are in $B(k)$. Further, in the right-upper $3\times 3$ sub-table, all three points in the first row are in the same part, or they are in three different parts. Then we again consider two cases. Only the proof of Case 1 needs a modification, so we discuss only this case.

   \emph{Case 1: $x_{m_1+l_1}, x_{m_1+l_2},  x_{m_1+l_3}
	\in \tilde{P}(k,i)$ for some $i$.}
	
	In this case, by Lemma~\ref{L:location-logm}$(1)$, we have
	\begin{alignat}{6}
		x_{m_1+l_1}, & \,\, x_{m_1+l_2},  & \,\, x_{m_1+l_3} & \in  & \,\, & \tilde{P}(k,i)   \label{3cases-1-logm} \\
		x_{m_2+l_1}, &                    & \,\, x_{m_2+l_3} & \in  & \,\, & \tilde{P}(k,i')  \label{3cases-2-logm} \\
	                 & \,\, x_{m_3+l_2},  & \,\, x_{m_3+l_3} & \in  & \,\, & \tilde{P}(k,i'') \label{3cases-3-logm}
	\end{alignat}
	where $i,i',i''$ are pairwise different (note that all these seven points are in $U^1(a_0)$).
	According to Lemma~\ref{L:location-logm}$(4)$,  $T^{m_2+l_1}x_0$ and
	$T^{m_1+l_1}x_0$ are in the ``similar positions", i.e.
	$$
	x_{m_1+l_1}= x_{r+n^k_{d}-s_{i}(d)} \quad \text{and} \quad x_{m_2+l_1}= x_{r'+n^k_{d}-s_{i'}(d)}
	$$
	for some $0\leq d\leq k$ (here $x_r$ and $x_{r'}$ are the first points of the pieces $P(k,i)$ and $P(k,i')$, and $s_i$ and $s_{i'}$ are the corresponding functions in $F(k)$). Further, since $x_{m_1+l_2} \in \tilde{P}(k,i)$, we have
	\begin{equation}\label{eq:m1l2-logm}
	x_{m_1+l_2}= x_{r+n^k_{c}-s_{i}(c)}
	\end{equation}
    for some $0\leq c \leq k$ (here $d < c$ and so $n_d^k<n_c^k$, because $m_1+l_1 < m_1+l_2$, but we do not use this property). Then
	\begin{equation}\label{Eq:l2-l1}
	l_2-l_1 = m_1+l_2 - (m_1 +l_1) =n_c^k-n_d^k-(s_i(c)-s_i(d)).
	\end{equation}
We are interested in the point $x_{m_2+l_2}$, so let us compute (here finally the modifications start)
	\begin{equation}\label{lastbut1-logm}
	\begin{split}
		m_2+l_2 & =   m_2+l_1+(l_2-l_1)= r'+n^k_{d}-s_{i'}(d) + (l_2-l_1) \\
		        & =   r'+n^k_c-s_{i'}(d)-(s_i(c)-s_i(d))  \\
		        & =   r'+n^k_c-s_{i'}(c)+(s_{i'}(c)-s_{i'}(d))-(s_i(c)-s_i(d))   \\
		        & \in \{r'+n^k_c-s_{i'}(c)+t: \, -2(m-1)\le t \le 2(m-1)\}.
	\end{split}	
	\end{equation}
	Note that $x_{r'+n^k_c} \in U^{1}(a_{s_{i'}(c)})$, whence
	$x_{r'+n^k_c-s_{i'}(c)} \in U^{1}(a_0)$. Therefore
	\begin{equation}\label{eq:|j|>2 not-logm}
	x_{m_2+l_2} \in \bigcup_{t=-2(m-1)}^{2(m-1)} U^1(a_{t}).
	\end{equation}
	On the other hand, by ~\eqref{eq:r2-logm}, $x_{m_2+l_2} \in U^1(a_j)$ and so $-2(m-1)\le j \le 2(m-1)$. Since in (2') we assume $j\geq m$, we in fact have that $j$ is an integer in the interval $[m, 2(m-1)]$.	

    Recall that by \eqref{eq:r2-logm}, $x_{m_2+l_2} \in U^1(a_j)$.	
	Then, using~\eqref{lastbut1-logm} and the fact that $x_{r'+n^k_c-s_{i'}(c)} \in U^1(a_0)$, we get
	$$(s_{i'}(c)-s_{i'}(d))-(s_i(c)-s_i(d))=j \ge m.$$
    However, we have
    $$
    s_{i'}(c)-s_{i'}(d), \, \, s_i(c)-s_i(d) \in \{-(m-1),-(m-1)+1,\dots, m-1 \}
    $$
    and so
	\begin{equation}\label{eq:l2-l1-logm}
	s_i(c)-s_i(d) \in \{-(m-1),-(m-1)+1,\dots,-1\}.
	\end{equation}
	Let $x_{r''}$ be the first point of $P(k,i'')$ and let $s_{i''}$ be the function from $F(k)$ which corresponds to $P(k,i'')$.
	Then by the construction of each pieces, see \eqref{eq:Pkl-logm} and \eqref{eq: pieces requirment-logm}, we have
	\begin{equation}\label{eq:xr-in-Ua0-logm}
	x_{r''+n^k_u-s_{i''}(u)} \in U^{1}(a_0) \text{ \ \ for \ \ }  0 \le u 	\le k.
	\end{equation}
	Recall that, by~\eqref{eq:r1-logm} and ~\eqref{eq:r3-logm},
	$$
	x_{m_1+l_2}, \, x_{m_1+l_3} \in  U^{1}(a_0) \text{ \ \ and \ \ }
	x_{m_3+l_2},  \, x_{m_3+l_3} \in U^{1}(a_0).
	$$
	Therefore, by Lemma~\ref{L:location-logm}$(4)$, $x_{m_1+l_2}$ and
	$x_{m_3+l_2}$ are in the ``similar positions". In view of ~\eqref{eq:m1l2-logm},
	$$
	x_{m_3+l_2}=x_{r''+n_c^k-s_{i''}(c)}.
	$$
	Using this and~\eqref{Eq:l2-l1} and~\eqref{eq:l2-l1-logm} and taking into account that the values of $s_{i''}$ are in $\{0,1,\dots,m-1\}$, we have
	\begin{equation*}
	\begin{split}
	m_3+l_1&=m_3+l_2-(l_2-l_1)\\
	&=r''+n_c^{k}-s_{i''}(c)-(n_c^k-n_d^k)+(s_i(c)-s_i(d))\\
	&=r''+n_d^{k}-s_{i''}(d)+(s_{i''}(d)-s_{i''}(c))+(s_i(c)-s_i(d))\\
	&\in \{r''+n^k_d-s_{i''}(d)+t: -2(m-1)\le t \le m-2\}.
	\end{split}
	\end{equation*}
	Since $x_{r''+n^k_d-s_{i''}(d)} \in U^1(a_0)$ by~\eqref{eq:xr-in-Ua0-logm}, it follows that
	$$
	x_{m_3+l_1} \in U^{1}(a_{-2(m-1)})\cup U^{1}(a_{-2(m-1)+1})\cup \dots \cup U^{1}(a_{m-2}).
	$$
	However, we have $j \ge m$ and so this contradicts the fact that, by~\eqref{eq:r3-logm}, $x_{m_3+l_1} \in U^{1}(a_j)$.
\end{proof}


\subsection{A continuum $X$ with $S(X)=\{0,\log m\}, m\ge 3$}\label{SS:cont logm}
To get a continuum $X$ with $S(X)=\{0,\log 2\}$ in Section~\ref{S:cont zero-log2}, we used an auxiliary system $(X_1,T)$ from Section~\ref{S:X1-T1-log2}. Now, to get a continuum $X$ with $S(X)=\{0,\log m\}$ for a given $m\ge 3$, we use the same construction, but we replace the auxiliary system $(X_1,T)$ by the system from Theorem~\ref{final-thm-logm}, with $h^*(T)=\log m$. By joining the consecutive points of the trajectory $x_0,x_1,\dots$ by the continua $D_m$ we then get the required continuum $X$ with $S(X)=\{0,\log m\}$. Indeed, to prove the following proposition, it is basically sufficient to repeat the proof of Proposition~\ref{P:zero log2} word by word.

\begin{prop}\label{P:zero logm}
	Let $m\geq 3$. For the one-dimensional continuum $X$ described just above we have
	$S(X)=\{0,\log m\}$. Moreover, if $F: X\to X$ is a continuous map then $h^*(F) = \log m$ if $F$ is 	non-constant and has no fixed point in the snake, otherwise $h^*(F) = 0$.
\end{prop}


\section{Proof of Main Theorem. Analogues in other settings}\label{S:proof}
 After a long preparation in previous sections, we are finally ready to prove our Main Theorem (it is repeated below as Theorem~\ref{T:main again}).
We also prove that the same conclusion holds when considering homeomorphisms rather than continuous maps, see Theorem~\ref{T:main-homeo} (Theorem B). The problem is addressed also for \emph{group actions}. In Theorem~\ref{T:action} (Theorem C) we show that the result works for the actions (by homeomorphisms) of the finitely generated groups which have $\mathbb Z$ as a quotient group. Without the assumption that the group is finitely generated, the result works with possible exceptions of some sets $A$. In full generality this problem remains open. Analogous results for \emph{semigroup actions} (by continuous maps) are also true, see Theorem~\ref{T:semigroup} (Theorem D).

\subsection{Proof of Main Theorem}

We already know, from Propositions~\ref{P:zero inf},~\ref{P:zero log2} and~\ref{P:zero logm},
that if $A$ has just two elements, one of them being of course zero, then there exists a
one-dimensional continuum $X_A \subseteq \mathbb R^3$ with $S(X_A)=A$. Moroever, from our
constructions of these continua $X_A$ (and from our results on them) we know the
following facts.
\begin{enumerate}
	\item [(F1)] If $A=\{0,\infty\}$ or $A=\{0, \log k\}$ for $k\in \{2,3,\dots\}$,
	the continuum $X_A$ consists of two parts. The first one (called a head) is a
	planar continuum. The second one (called a snake) is obtained from just one
	sequence $x^A_1, x^A_2, \dots$ (which approaches the first part) by joining any
	two consecutive points of the sequence by an infinite chain of appropriately
	chosen Cook continua. In what follows, the point
	$x_1^A$ will be called the \emph{first point of $X_A$}.
	\item [(F2)] If $T_A$ is a continuous selfmap of such a continuum $X_A$ and the
	first point $x_1^A$ of $X_A$ is fixed for $T_A$ then $h^*(T_A) = 0$ (we have in
	fact proved that if $T_A$ has a fixed point in the snake then $h^*(T_A) = 0$).
	\item [(F3)] In the construction of each of these sets $X_A$ we have used
	copies of some pairwise disjoint nondegenerate subcontinua of a planar Cook
	continuum $\mathscr Q$. Since we have only a countable family of such sets $X_A$
	(note that so far we have considered only sets $A$ with cardinality $2$), we may
	assume that an infinite family of pairwise disjoint nondegenerate subcontinua of
	$\mathscr Q$ was split into infinitely many infinite subfamilies and in the
	constructions of different sets $X_A$ we used different subfamilies. So we may
	assume that if $Q_1$ and $Q_2$ are subcontinua of $Q$ such that their copies are
	used in the constructions of $X_{A_1}$ and $X_{A_2}$ with $A_1\neq A_2$ then
	$Q_1$ and $Q_2$ are disjoint.
\end{enumerate}
The last fact is crucial for the following lemma.

\begin{lem}\label{L:FXA1-point}
	Let $A_1=\{0, \log k_1\}$ and $A_2=\{0, \log k_2\}$ for different $k_1, k_2 \in
	\N^*\setminus \{1\}$. Suppose that $X_{A_1}$ and $X_{A_2}$, with the first
	elements $x_1^{A_1}$ and $x_1^{A_2}$, respectively, are subspaces of a metric
	space $Y$ such that $X_{A_1} \setminus \{x_1^{A_1}\}$ and $X_{A_2} \setminus
	\{x_1^{A_2}\}$ are disjoint open subsets of $Y$. Let $F: Y\to Y$ be a continuous
	map and $F(X_{A_1}) \cap (X_{A_2} \setminus \{x_1^{A_2}\})\neq \emptyset$. Then
	$F(X_{A_1})$ is a singleton.
\end{lem}

\begin{proof}
	$F(X_{A_1})$ intersects $X_{A_2} \setminus \{x_1^{A_2}\}$ in a point $p$.
	Suppose, on the contrary, that $F(X_{A_1}) \neq \{p\}$. Then the continuum
	$F(X_{A_1})$ contains a nondegenerate subcontinuum lying in $X_{A_2} \setminus
	\{x_1^{A_2}\}$. This nondegenerate subcontinuum has cardinality $\mathfrak c$, therefore
	there is a brick in $X_{A_1}$ mapped onto some nondegenerate subcontinuum lying
	in $X_{A_2} \setminus \{x_1^{A_2}\}$. It follows that there is a nondegenerate
	subcontinuum in a brick of $X_{A_1}$ which is mapped onto a nondegenerate
	subcontinuum in a brick of $X_{A_2}$. However, the Cook continuum $Q$ contains
	disjoint homeomorphic copies of these two subcontinua, which contradicts
	Lemma~\ref{L:Cook-properties}.
\end{proof}

We are ready to prove our ultimate result announced in Introduction as Theorem A (Main Theorem).

\begin{thm}[{\bf Main Theorem}]\label{T:main again}
	For every set $\{0\} \subseteq A \subseteq \log \mathbb N^*$ there exists a
    one-dimensional con\-ti\-nuum $X_A\subseteq \mathbb R^3$ with $S(X_A) = A$.	
\end{thm}

\begin{proof}
	The case $A=\{0\}$ is trivial, just choose any rigid continuum in $\mathbb
	R^3$. If $A$ has cardinality $2$, the existence of $X_A$ has already been
	proved.
	
	So, from now on let $A$ have cardinality at least $3$. We are going to describe
	the continuum $X_A$ in this case.
	For every $k\in \N^*\setminus \{1\}$ put $A(k) = \{0, \log k\}$ and consider
	the continuum $X_{A(k)}$ we have already constructed in the previous sections.
	Now choose, in $\mathbb R^3$, homeomorphic copies of these continua $X_{A(k)}$
	(without changing the notation, i.e. still denoting these copies by $X_{A(k)}$)
	such that their diameters tend to zero and are pairwise disjoint except of the
	common first point $x_1^{A(2)}= x_1^{A(3)} = \dots = x_1^{A(\infty)}=:z$. What
	we get is a compact space looking like a `flower' with infinitely many smaller
	and smaller `petals' (this will be  our name for the sets $X_{A(k)}$). To obtain
	$X_A$, keep just those petals which correspond to the set $A$, i.e. denote
	$I=\{k\in \N^*\setminus \{1\}: \, \log k \in A\}$ and put	
	$X_A = \bigcup _{k\in I}X_{A(k)}$. This is a flower with finitely many or
	infinitely many petals (there are at least two petals, because we assume that
	$A$ has at least three elements). Since the union of countably many closed
	one-dimensional sets is one-dimensional, also $X_A$ is one-dimensional.
	Clearly, it is a continuum.
	
	We need to show that $S(X_A)= A = \{0\} \cup \{\log k: k\in I\}$. Since for
	each $k\in I$ the petal $X_{A(k)}$ is a retract of $X_A$, we clearly have
	$S(X_A) \supseteq A$. To prove the converse inclusion, fix a continuous map $F:
	X_A \to X_A$ and show that $h^*(F) = 0$ or $\log k$ for some $k\in I$.

	Let $I_{\rm inv}$ and $I_{\rm not}$ be the sets of all $k\in I$ for which
	$X_{A(k)}$ is $F$-invariant or is not $F$-invariant, respectively.
	First assume that $I_{\rm not}=\emptyset$. Then $F(z)=z$. We are going to show
	that in this case $h^*(F)=0$. Suppose on the contrary that for some $x\neq y$ in
	$X_A$, the pair $( x, y )$ is an IN-pair for the whole map $F: X_A
	\to X_A$. Since at least one of the points $x$ and $y$ is different from $z$, we
	may assume that $x\in X_{A(i)}\setminus \{z\}$ (here of course $i\in I_{\rm
		inv}$). Choose open neighbourhoods $U_x$ and $U_y$ (in the topology of $X_A$) of
	$x$ and $y$, respectively, such that $U_x \subseteq X_{A(i)}\setminus \{z\}$.
	Though it is possible that $U_y$ intersects the complement of $X_{A(i)}$, the
	$F$-orbits of points from this complement do not intersect $U_x$. Therefore
	$( x, y )$, an IN-pair for $F$, is an IN-pair even for the
	restriction $F_i$ of $F$ to the set $X_{A(i)}$. Hence $h^*(F_i) > 0$, which
	contradicts the fact (F2).
	
	Now assume that $I_{\rm not} \neq \emptyset$. Denote $X_A^{\rm inv} = \bigcup
	_{i\in I_{\rm inv}} X_{A(i)}$ and $X_A^{\rm not} = \bigcup _{n\in I_{\rm not}}
	X_{A(n)}$. By Lemma~\ref{L:FXA1-point}, $F(X_{A(n)})$ is a singleton whenever
	$n\in I_{\rm not}$. Then, since all the sets $X_{A(n)}$ intersect, $F(X_A^{\rm
		not})$ is a singleton $\{q\}$.  If $q\in A(n_0)$ for some $n_0\in I_{\rm not}$
	then $X_{A(n_0)}$ would be $F$-invariant, a contradiction. Therefore $q\in
	X_A^{\rm inv}$, i.e. $I_{\rm inv} \neq \emptyset$. So, $X_A$ is the union of two
	nonempty sets  $X_A^{\rm inv}$ and $X_A^{\rm not}$ having the point $z$ in
	common, the first of them is $F$-invariant and the second one is mapped by $F$
	to the point $q\in X_A^{\rm inv}$.  By Theorem~\ref{P}(e),
	$h^*(F)=h^*(F_{\rm inv})$ where $F_{\rm inv}$ is the restriction of $F$ to the
	set $X_A^{\rm inv}$. If $I_{\rm inv}$ contains just one element $k$, we have
	$X_A^{\rm inv} = X_{A(k)}$ and then $h^*(F)\in \{0, \log k\}$ and we are done.
	If $I_{\rm inv}$ contains at least two elements, the set $X_A^{\rm inv}$ is a
	flower with at least two petals and by the case considered above (the case
	$I_{\rm not}=\emptyset$) we get that $h^*(F)=0$.
\end{proof}

\subsection{Analogue of Main Theorem for homeomorphisms}\label{S:main-homeo}
As an analogue of the set $S(X)$ from~\eqref{Eq:defSX}, for any compact metric space $X$ put
\[
\Shom (X)=\{h^{*}(T)\colon T\text{ is a homeomorphism } X \rightarrow X \}.
\]
\index{$\Shom (X)$}Our Main Theorem has the following analogue for homeomorphisms (it is in fact Theorem B from Introduction).  The proof needs a construction, due to Hanfeng Li \cite{L}, which did not appear in the proof of Main Theorem.

\begin{thm}[{\bf Theorem B}]\label{T:main-homeo}
For every set $\{0\} \subseteq A \subseteq \log \mathbb N^*$ there exists a one-dimensional continuum $\tilde{X}_A\subseteq \mathbb R^3$ with $\Shom(\tilde{X}_A) = A$.
\end{thm}	

\begin{proof}  The proof is to large extent similar to that of Main Theorem.
	
The case $A=\{0\}$ is trivial, any rigid continuum in the plane can be chosen as $\tilde{X}_A$. We are going to consider the nontrivial cases.

\medskip

\emph{Case 1: $A$ has cardinality $2$.}

\medskip

First consider the set $A=\{0, \infty\}$. We begin with the continuum $X$ from Lemma~\ref{L:Xiscont}, see Figure~\ref{F:spaceX} and Proposition~\ref{P:zero inf}. For our purposes we modify this continuum $X$ as follows.
\begin{itemize}
	\item The continuum $X$ has the `first' point $x_1$. We add a converging sequence of points $x_0,x_{-1},x_{-2},\dots$ with
	\[
	P_2(x_{j}) = P_2(x_1), \,\, j=0,-1,-2,\dots \text{ and } P_1(x_1) < P_1(x_0) < P_1(x_{-1}) < P_1(x_{-2})< \dots~.
	\]
	We denote the limit of this sequence by $x_{-\infty}$ and we add also this point to $X$.
	\item For $m\in \mathbb N$, the points $x_m$ and $x_{m+1}$ in $X$ are joined by the continuum $D_m$ from~\eqref{Eq: DD*}, see also Figure~\ref{F:3Dm}.
	We replace the (non-homeomorphic) continua $D_m$, $m\in \mathbb N$, by \emph{homeomorphic} continua $\tilde D_m$, $m\in \mathbb N$, each of them being a homeomorphic copy of, say, $D_1$. Further, also for $m=0,-1,-2, \dots$ we join $x_m$ and $x_{m+1}$ by a continuum $\tilde D_m$ which is again a homeomorphic copy of $D_1$. So, now all the continua $\tilde D_m$, $m\in \mathbb Z$, are homeomorphic and they are copies of $D_1$. We may of course assume that the continua $\tilde D_m$ are pairwise disjoint, except that two consecutive continua have one point in common. This can of course be done in such a way that we still have properties analogous to those in Lemma~\ref{L:block}. In particular, the sequence of Cook continua which forms $\tilde D_m$  `goes' from $x_m$ to $x_{m+1}$ and not in the opposite direction, and the diameters of $\tilde D_m$ tend to zero as $|m|\to \infty$.
\end{itemize}
These modifications yield a one-dimensional continuum $\tilde X_{\{0,\infty\}}$, see Figure~\ref{F:XAi}. The point $x_{-\infty}$ will be called the \emph{starting point} of $\tilde X_{\{0,\infty\}}$.\index{starting point}

	\begin{figure} [h]
		\includegraphics[width=12cm]{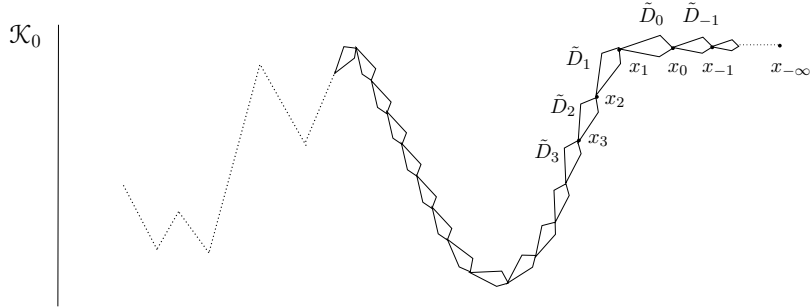}\\
		\caption{$\tilde X_{\{0,\infty\}}=\mathscr K_0 \sqcup \bigcup_{m \in \Z} \tilde{D}_m \sqcup \{x_{-\infty}\}.$}\label{F:XAi}
	\end{figure}

Denote by $\tilde G_{\{0,\infty\}} \colon \tilde{X}_{\{0,\infty\}} \to \tilde{X}_{\{0,\infty\}}$ the homeomorphism which maps $\tilde D_m$ onto $\tilde D_{m+1}$, $m\in \mathbb Z$ (and so is the identity on $\mathscr K_0 \cup \{x_{-\infty}\}$). In particular, $\tilde G_{\{0,\infty\}}$ restricted to the set $\mathscr K_0 \cup \{x_1, x_2, x_3, \dots \}$ coincides with the map $T$ from Subsection~\ref{SS:cont+orbit} (provided $\mathscr C_0$ is chosen to be $\mathscr K_0$). Since we have shown there that $h^*(T)=\infty$, we necessarily have $h^*(\tilde G_{\{0,\infty\}}) =\infty$.

On the other hand, let $G\colon \tilde{X}_{\{0,\infty\}} \to \tilde{X}_{\{0,\infty\}}$ be any homeomorphism. Taking into account the structure of $\tilde X_{\{0,\infty\}}$, $G$ is the identity on $\mathscr K_0$. Suppose it is not the identity on the whole of $\tilde X_{\{0,\infty\}}$.  Then, since each $\tilde D_m$ is homeomorphic to $D_1$ and $D_1$ is the closure of a `concatenation' of pairwise non-homeomorphic Cook continua, there exist $k\in \mathbb N$,  $m\in \mathbb Z$ and $n \in \mathbb Z\setminus \{0\}$ such that some interior point of the $k$-th Cook continuum in $\tilde{D}_m$ is mapped by $G$ to an interior point of the $k$-th Cook continuum in $\tilde{D}_{m+n}$. Hence $G$ maps the $k$-th Cook continuum in $\tilde{D}_m$ onto the $k$-th Cook continuum in $\tilde{D}_{m+n}$. Then it is not difficult to show that $G(\tilde{D}_m) = \tilde{D}_{m+n}$ and in fact that $G$ is the $n$-th iterate of $\tilde G_{\{0,\infty\}}$ on the whole of $\tilde X_{\{0,\infty\}}$. We have thus shown that the only homeomorphisms on $\tilde{X}_{\{0,\infty\}}$ are just all the iterates of $\tilde G_{\{0,\infty\}}$. This together with $h^*(\tilde G_{\{0,\infty\}}) =\infty$ and the fact that $h^*(\varphi^n)=h^*(\varphi)$ whenever $\varphi$ is a homeomorphism and $n\neq 0$, imply that $\Shom(\tilde{X}_{\{0,\infty\}})=\{0,\infty\}$.

Notice that for every homeomorphism on $\tilde{X}_{\{0,\infty\}}$ the starting point $x_{-\infty}$ of $\tilde{X}_{\{0,\infty\}}$ is a fixed point.

Now consider the set $A=\{0, \log m\}$ where $m\geq 2$. The construction of $\tilde{X}_{\{0,\log m\}}$ is the same as above
construction of $\tilde{X}_{\{0,\infty\}}$, the only difference is that now we begin with the continuum $X$ from Section~\ref{S:cont zero-log2} if $m=2$ or
Section~\ref{S:generalization} if $m\geq 3$. For this $X$ we have $S(X)=\{0, \log m\}$ and we modify it similarly as above, to get $\tilde{X}_{\{0,\log m\}}$.
Again, all the homeomorphisms of $\tilde{X}_{\{0,\log m\}}$ are just the iterates of one distinguished homeomorphism $\tilde G_{\{0,\log m\}}$ which maps $\tilde D_m$ onto $\tilde D_{m+1}$, $m\in \mathbb Z$. It follows from Section~\ref{S:cont zero-log2} if $m=2$ or
Section~\ref{S:generalization} if $m\geq 3$ that the restriction of this homeomorphism to the union of the head and the `forward' part of the snake (lying between $x_0$ and the head) has the supremum topological sequence entropy equal to $\log m$. Using Proposition~\ref{P}(c), one can see that adding the `backward' part of the snake does not change it, i.e. $h^*(\tilde G_{\{0,\log m\}}) =\log m$.
Then the proof that $\Shom(\tilde{X}_{\{0,\log m\}}) = \{0,\log m\}$ is completely analogous to the above proof that $\Shom(\tilde{X}_{\{0,\infty\}})=\{0,\infty\}$.

Again, remember that similarly as above, for every homeomorphism on $\tilde{X}_{\{0,\log m\}}$ the naturally defined starting point of $\tilde{X}_{\{0,\log m\}}$ is a fixed point.

Finally, before going to Case 2, note that in Case 1 we have constructed countably many continua $\tilde{X}_{\{0,\log m\}}$, $m\in \{2,3,\dots \} \cup \{\infty\}$, and for the construction of each of them we have used only countably many Cook continua. Therefore we may assume that we have chosen countably many disjoint subcontinua of a Cook continuum in the plane, then we have divided them into countably many disjoint countable families, and different families have been used to construct different spaces $\tilde{X}_{\{0,\log m\}}$,  still keeping one of those families unused, i.e. as a reservoir of Cook continua for further use. To express this fact we will, though not very precisely, just say that the spaces $\tilde{X}_{\{0,\log m\}}$ are constructed by using different families of Cook continua.

\medskip
\medskip

\emph{Case 2: $A$ has cardinality $\geq 3$.}\footnote{The construction in this case was suggested by Hanfeng Li.}

For every $m\in \{2,3,\dots\}\cup \{\infty\}$ consider the continuum $\tilde X_{\{0, \log m\}}$ we have already constructed in Case~1 above. Since we will work with many such continua simultaneously, we will now use double indices; for every $m$ we will write
\[
\tilde X_{\{0, \log m\}} = \underbrace{\{x_{m,-\infty}\} \, \cup \,\bigcup_{k<0} \tilde{D}_{m,k}}_{\tilde{D}_{m}^{-}} \, \cup \,\, \tilde{D}_{m,0} \, \cup \, \underbrace{\bigcup_{k>0} \tilde{D}_{m,k} \, \cup \,\, \text{head}_m}_{\tilde{D}_{m}^{+}}.
\]
\index{$\tilde{D}_{m}^{-}$, $\tilde{D}_{m}^{+}$}Here $x_{m,-\infty}$ is the starting point of $\tilde X_{\{0, \log m\}}$, $\text{head}_m$ is the head of $\tilde X_{\{0, \log m\}}$ (such as the set $\mathscr K_0$ in Figure~\ref{F:XAi}, where $m=\infty$) and recall that  the continuum $\tilde{D}_{m,0}$ with first point $x_{m,0}$ and last point $x_{m,1}$ is the union of a sequence of different Cook continua, together with one limit point (which is the last point $x_{m,1}$). The other sets $\tilde{D}_{m,k}$, $k\in \Z \setminus \{0\}$, are homeomorphic copies of $\tilde{D}_{m,0}$. The first point of $\tilde{D}_{m, k+1}$ coincides with the last point of $\tilde{D}_{m, k}$ for all $k$. Recall also that the homeomorphism group of $\tilde X_{\{0, \log m\}}$ is isomorphic to $\Z$, with the generator $\tilde G_{\{0,\log m\}}$ sending $D_{m, k}$ to $D_{m,k+1}$ for all $k$. For every homeomorphism on
$\tilde X_{\{0, \log m\}}$, the starting point $x_{m,-\infty}$ is a fixed point and the set $\text{head}_m$ is invariant.

Consider numbers $0=c_1<c_2<c_3<\dots$ with $\lim_{m\to\infty} c_m =1$. The sets $[0,1]^2 \times (c_m, c_{m+1})$ will be called  \emph{layers}\index{layer}, see Figure~\ref{F:layers}. For every $m\in \{\infty\} \cup \{2,3,\dots\}$ we choose a copy of $\tilde X_{\{0, \log m\}}$, still denoted by the same symbol, and a continuum $C_m$, satisfying the following conditions, see Figure~\ref{F:spacein1layer}:
\begin{itemize}
	\item $\tilde X_{\{0, \log m\}}$ lies in the layer $[0,1]^2 \times (c_m, c_{m+1})$ if $m$ is finite, $\tilde X_{\{0, \infty\}}$ lies in the layer $[0,1]^2 \times (c_1, c_2)$;
	\item the projections, into the $x,y$-plane, of the first point $x_{m,0}$ and the last point $x_{m,1}$ of $\tilde{D}_{m,0}$ coincide with the points $[1,1]$ and $[0,1]$, respectively;
	\item the continuum $\tilde{D}_{m,0}$ is in the $\varepsilon_m$-neighborhood of the straight line segment whose endpoints are the first point and the last point of $\tilde{D}_{m,0}$, and $\varepsilon_m \to 0$ as $m\to \infty$;
	\item the diameters of $\tilde{D}_{m}^{-}$ and $\tilde{D}_{m}^{+}$ tend to zero as $m\to \infty$;
	\item $C_m$ is a (copy of a) Cook continuum from the reservoir mentioned at the end of Case 1, and for different $m$'s these continua are different (non-homeomorphic);
	\item $C_m$ contains the starting point $x_{m,-\infty}$ of $\tilde X_{\{0, \log m\}}$ and the point $w=[0,0,1]$, and lies in the $\varepsilon_m$-neighborhood of the straight line segment joining these two points, with $\varepsilon_m \to 0$ as $m\to \infty$;
	\item for $m\in \{\infty\} \cup  \{2,3,\dots\}$, the continua
	\[
	L_m = C_m \cup \tilde X_{\{0, \log m\}}
	\]
	are pairwise disjoint, except that each of them contains the point $w$.
\end{itemize}

\begin{figure}[h]
	\begin{minipage}[t]{0.5\linewidth}
		\centering
		\includegraphics[width=8cm]{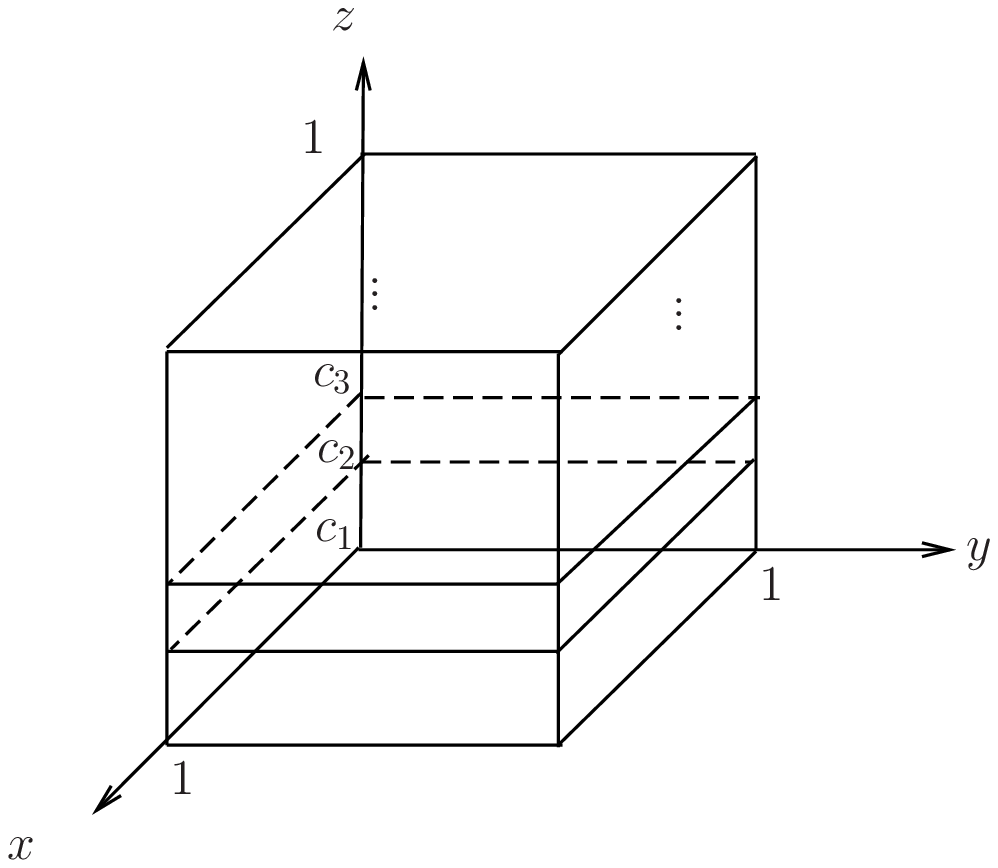}
		\caption{A layer.}\label{F:layers}
	\end{minipage}%
	\begin{minipage}[t]{0.5\linewidth}
		\centering
		\includegraphics[width=9cm]{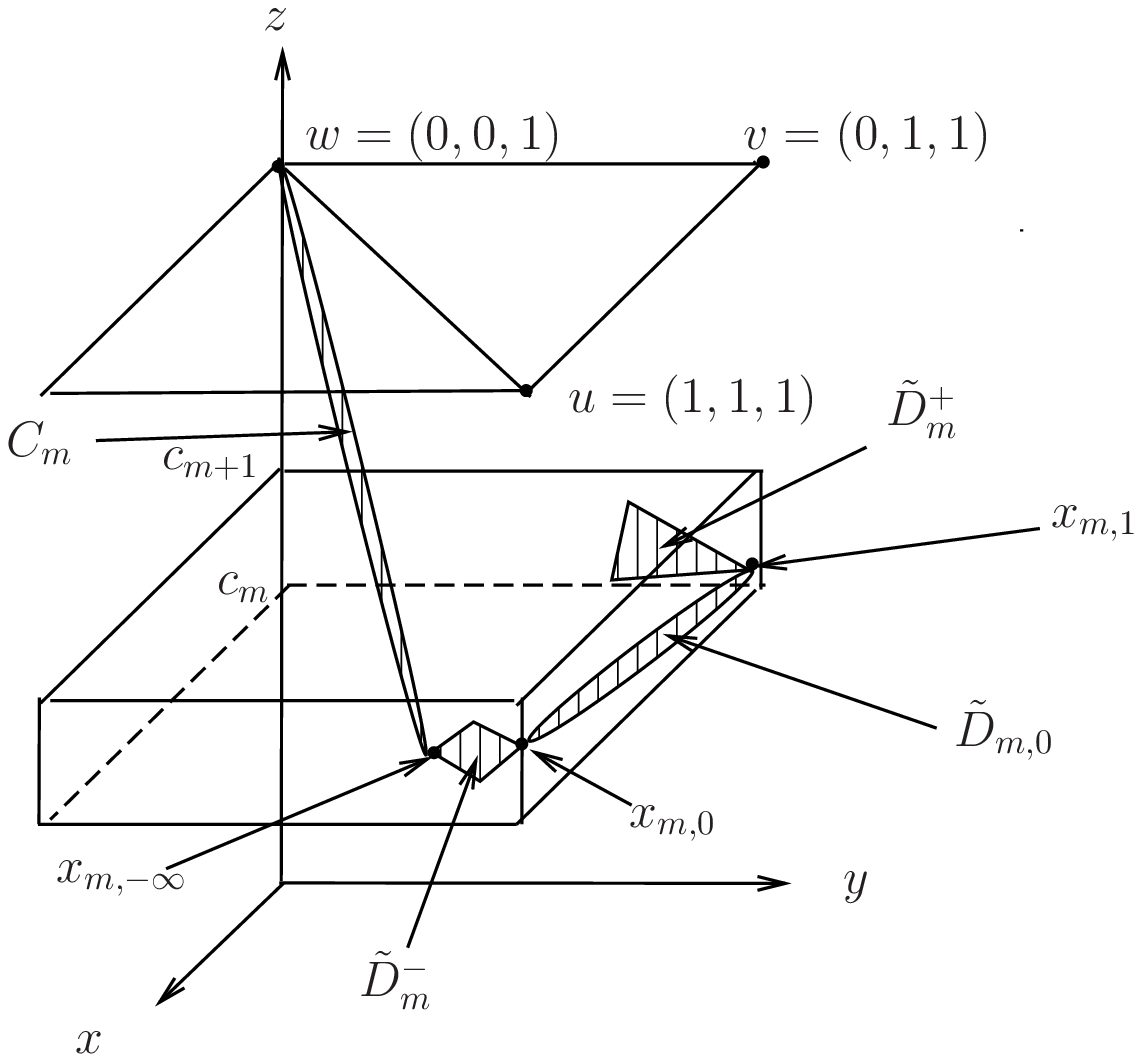}
		\caption{The continuum $L_m$.}\label{F:spacein1layer}
	\end{minipage}
\end{figure}

Consider further the points $u=[1,1,1]$ and $v=[0,1,1]$ and denote by $[u,v]$ the straight line segment from $u$ to $v$. Similarly, let $[w,u]$ be the straight line segment from $w$ to $u$. Notice that, in the Hausdorff metric, the continua $C_m$ converge to $[w,u]$ and the continua $\tilde{D}_{m,0}$, as well as the continua $\tilde X_{\{0, \log m\}}$, converge to $[u,v]$.

Now finally define
\begin{equation}\label{Eq:unionL}
\tilde{X}_A = \begin{cases} \, \bigcup \{L_m \colon \, 2\leq m \leq \infty \text{ and } \log m \in A \}, & \text{if $A$ is finite}, \\
\, \bigcup \{L_m \colon \, 2\leq m \leq \infty \text{ and } \log m \in A \} \, \cup \, [w,u] \, \cup \, [u,v], & \text{if $A$ is infinite}.
\end{cases}
\end{equation}	
Clearly, $\tilde{X}_A$ is a one-dimensional continuum. Notice that it is still a kind of a `flower', with central point $w$ and `petals' $[w,u] \cup [u,v]$ (if $A$ is infinite) and $L_m$ (for $2 \leq m \leq \infty$ such that $\log m \in A$), but now it is a more complicated flower than that from the proof of Theorem~\ref{T:main again}, since in the case of infinite $A$ the petals $L_m$ converge to the petal $[w,u] \cup [u,v]$. Notice also that $\tilde{X}_A$ has at least two petals, since $A$ has cardinality at least three.

We need to show that $\Shom(\tilde{X}_{A})= A$. Trivially, $\Shom(\tilde{X}_{A}) \ni 0$. Fix $m\in \{2,3,\dots\}\cup \{\infty\}$ with $\log m \in A$ and a homeomorphism $G_m$ on $\tilde{X}_{A}$ such that it is the identity on $\tilde{X}_{A} \setminus \tilde X_{\{0, \log m\}}$ and on $\tilde X_{\{0, \log m\}}$ coincides with the distinguished homeomorphism $\tilde G_{\{0,\log m\}}$, the generator of the homeomorphism group of $\tilde X_{\{0, \log m\}}$. Note that $x_{m,-\infty}$ is a fixed point for this generator, so $G_m$ is well defined. Since $h^*(\tilde G_{\{0,\log m\}}) = \log m$, we obviously have $h^*(G_m) = \log m$. We have thus shown that $\Shom(\tilde{X}_{A}) \supseteq A$.

To prove that also $\Shom(\tilde{X}_{A}) \subseteq A$, let $H\colon \tilde{X}_{A} \to \tilde{X}_{A}$ be any homeomorphism different from the identity. We are going to show that $h^*(H) \in A$ (for identity it is trivial). Since the petals $L_m$ of the space $\tilde{X}_{A}$ are constructed by using different families of Cook continua, each of them is $H$-invariant (hence, if $A$ is infinite, also $[w,u]\cup [u,v]$ is $H$-invariant). Moreover, $H$ is obviously the identity on $C_m$ (hence, if $A$ is infinite, also on $[w,u]$) and is a $k_m$-th iterate of $\tilde G_{\{0,\log m\}}$ on $\tilde X_{\{0, \log m\}}$.

Suppose that the set of those $m$'s for which $k_m > 0$ is infinite (i.e. there is a sequence of layers $L_m$ with $k_m>0$, which converges to $[w,u]\cup [u,v]$). For every such $m$, the set $\tilde{D}_{m,-k_m}$ (which is a subset of $\tilde{D}_{m}^{-}$) is mapped by $H$ onto the set  $\tilde{D}_{m,0}$ whose diameter is at least $1$. Since diameters of $\tilde{D}_{m}^{-}$ tend to zero as $m\to \infty$, we have a contradiction with uniform continuity of $H$. (Alternately, one can argue as follows. If $k_m>0$ then $H(x_{m,0}) \in \tilde{D}_{m}^{+}$ and since we have infinitely many such $m$'s, the fact that the diameters of $\tilde{D}_{m}^{+}$ tend to zero and the continuity of $H$ imply that $H(u)=v$. However, we have already shown that $H(u)=u$, a contradiction.) Similarly, we get a contradiction with uniform continuity if $k_m < 0$ for infinitely many $m$'s. (Or, one can show that in such a case we would have $H(v)=u$ which is a contradiction because $H(u)=u$ and $H$ is a homeomorphism.) Thus we conclude that $H$ is identity on the whole space $\tilde{X}_{A}$ except of finitely many petals $L_m$; more precisely, except of \emph{finitely many} continua $\tilde X_{\{0, \log m\}}$ (there is at least one such continuum, because $H$ is different from the identity). Since these continua are pairwise disjoint and $\Shom(\tilde{X}_{\{0,\log m\}}) = \{0,\log m\}$, we get that $h^*(H)$ equals the maximum of $\log m$ for such $m$'s. Hence $h^*(H) \in A$.
\end{proof}

\begin{rem}\label{R:bad case}
 The reader might be curious why, in Case 2, we did not use the simpler construction with a `standard' flower with petals, exactly as in the proof of Theorem~\ref{T:main again}. In fact, it is obvious that such a flower can be used if $A$ is finite. One can show that it works also when $A$ si infinite and closed with respect to the supremum. However, it does not work in other cases, i.e. when $A$ is infinite and $\infty \notin A$. Indeed, let $A$ consist of $0$ and \emph{numbers} $\log k_i$ for some $2\leq k_1 <k_2< \dots$ (note that $\sup_{i=1,2,\dots} \log k_i = \infty \notin A$). Consider the flower $F_{A}$ whose petals are $X_{k_i} = \tilde{X}_{\{0,\log k_i\}}$, $i=1,2,\dots$. Then $F_{A}$ admits a homeomorphism $H$ such that it fixes the central point of the flower, every petal $X_{k_i}$ is $H$-invariant and the restriction $H_i$ of $H$ to that petal has $h^*(H_i)=\log k_i$.  By~\eqref{Eq:hstarT-IN} we easily get $h^*(H)=\infty$ and so
\[
\Shom(F_{A}) = A \cup \{\infty\} \neq A.
\]
So, due to such sets $A$ a more sophisticated construction was needed.
\end{rem}

	\subsection{Analogues of Main Theorem for group actions\index{group action} and semigroup actions}\label{S:group actions}
	
	By a \emph{dynamical system} we now mean a triple $(X, G, \Phi)$, where $X$ is a compact metric space,
	$G$ is a topological group and $\Phi$ is an \emph{action of $G$ on $X$}, i.e. a continuous map
	$\Phi \colon G\times X\rightarrow X$ such that $\Phi(s,\Phi(t,x)) = \Phi(st,x)$ for all $s,t\in G$ and
	$\Phi(e,x)=x$ for every $x\in X$, where $e$ is the neutral element of the group $G$.
	For each $s\in G$, the \emph{acting map}\index{acting map} $\Phi_s \colon X\rightarrow X$ defined by $\Phi_s(x):=\Phi(s,x)=:sx$ can easily be proved to be a
	homeomorphism. To define $\Phi$ is the same as to define all $\Phi_s$. We also say that $\Phi$ is a \emph{$G$-action}
	(on the space $X$). In the sequel, when speaking on $G$-actions, we always assume that $G$ has \emph{discrete topology}.
	Then, to check the continuity of a map $\Phi \colon G\times X\rightarrow X$, it is sufficient to check the continuity
	of all $\Phi_s$.

	Note that the topological sequence entropy is developed in the
	literature for a single continuous map, but we can also work in the framework of a
	dynamical system $(X, G, \Phi)$.
	Following Goodman \cite{Good1} and Kerr-Li \cite{KL}, for a sequence $\sigma = \{s_n\}_{n\in\Z_+}$ in $G$ we define the topological
	sequence entropy of $(X, G,\Phi)$ with respect to $\sigma$ and a finite open cover $\mathcal{U}$ of $X$ by
	\begin{equation}\label{Eq:quantity}
	h^{\sigma}(G, \mathcal{U},\Phi) = \limsup_{n\rightarrow \infty}\frac{1}{n}\log \mathcal{N}\left (\bigvee_{i=0}^{n-1} s_n^{-1}\mathcal{U} \right ),
	\end{equation}
	where $\mathcal{N}(\cdot)$ denotes the minimal cardinality of a subcover. Then the sequence entropy with respect to $\sigma$ is defined by
	$h^{\sigma}(X,G,\Phi)=\sup_{ \mathcal{U}}h^{\sigma}(G, \mathcal{U},\Phi)$, where $\mathcal{U}$ runs over all finite open covers of $X$. Again, by \cite[Theorem~A.1]{HYad}, we can define the \emph{supremum topological sequence entropy}
	\begin{equation}\label{Eq:stse-def}
	h^*(X,G,\Phi) = \sup\{h^{\sigma}(X,G,\Phi) \colon \, \text{$\sigma = \{s_n\}_{n\in\Z_+}$ is a sequence in $G$}\}.
	\end{equation}
	
   \begin{de} Let $(X,G,\Phi)$ be a dynamical system and $\tilde{A}=( A_1,\ldots,A_k )$
		be a tuple of subsets of $X$. We say that a subset $J\subseteq G$ is an
		independence set for $\tilde{A}$ (or that  $\tilde{A}$ has the independence set
		$J$)
		if for any nonempty finite subset $I\subseteq J$, we have
		$$
		\bigcap_{s\in I}s^{-1}A_{\sigma(s)}\not=\emptyset
		$$
		for any function  $\sigma: I \rightarrow \{1,\ldots,k\}$. If such a set $J$ is finite and has $p$ elements, we also say that it is an independence set,  or independence set of times, of length~$p$.
	\end{de}

   \begin{de}\label{D:IN-G}
   Consider a tuple $\tilde{x}=(x_1,\dots, x_k)\in X^k$. If for every product neighbourhood $U_1 \times \dots \times U_k$ of $\tilde{x}$ the tuple
   $(U_1, \dots, U_k)$ has arbitrarily long finite independence sets, then the tuple $\tilde{x}$ is called an IN-tuple in $(X,G,\Phi)$.
   \end{de}	
	
	As before we may define the sequence entropy tuples, see \cite{KL} for details.
	We have the following proposition proved in \cite[Theorem 5.9]{KL}.
	
	\begin{prop}\label{P:G-tuples}
		Let $(x_1, \ldots, x_k )$ be a tuple in $X^k\setminus \Delta_k$ with $k\ge 2$. Then $(x_1, \ldots, x_k )$
		is a sequence entropy tuple if and only if it is an IN-tuple.
	\end{prop}
	
	Now for a dynamical system $(X,G,\Phi)$ we have, by \cite[Theorem A.3]{HYad}, that
	\begin{equation}\label{Eq:hstarT-G}
	h^*(X, G,\Phi)=\sup\{\log n \colon \text{there is an intrinsic sequence entropy tuple of
		length}\ n\}
	\end{equation}
	and, in view of Proposition~\ref{P:G-tuples}, we can also write
	\begin{equation}\label{Eq:hstarT-G-IN}
		h^*(X, G,\Phi)=\sup\{\log n \colon \text{there is an intrinsic IN-tuple of
			length}\ n\}.
	\end{equation}
	
	Finally, for a fixed topological group $G$ and a fixed compact metric space $X$ we put
	\begin{equation}\label{def-g}
	S_G(X)=\{h^*(X, G, \Phi) \colon \, \Phi\ \text{is an action of $G$ on $X$} \}.
	\end{equation}
	We emphasize that in the left side of the definition (\ref{def-g}), $G$ is a fixed group.

    In the sequel we will use simple but useful observations from the following two lemmas.
	\begin{lem}\label{L:1-1 corr}
		Consider the well known one-one correspondence between homeomorphisms on $X$ and $\mathbb Z$-actions on $X$; the $\mathbb Z$-action $\Phi^{\varphi}$ with the acting maps $(\Phi^{\varphi})_m = \varphi^m$, $m\in \mathbb Z$, corresponds to the homeomorphism~$\varphi$. Then
		\begin{equation}\label{Eq:seq-homeo-action}
		h^*(\varphi) = h^*(X, \mathbb Z, \Phi^{\varphi}).
		\end{equation}	
	\end{lem}
	
	\begin{proof}
     We show that every $IN$-tuple in $(X, \mathbb Z, \Phi^{\varphi})$ in the sense of Definition~\ref{D:IN-G} is an $IN$-tuple in $(X,\varphi)$ in the sense of Definition~\ref{D:IE-IT-IN} (the converse is trivial since $\Z_{+} \subseteq \Z$).

     Assume that $\tilde{x}=(x_1,\dots,x_k)$ is an IN-tuple in $(X,\mathbb Z,\Phi^{\varphi})$. Let $U_1 \times \dots \times U_k$ be the product neighbourhood of $\tilde{x}$ and let $J=\{l_1,l_2,\dots, l_t\}\subseteq \mathbb Z$ be an independence set of length $t$ for $(U_1, \dots, U_k)$. If $l_j \ge 0$ for $j=1,2,\dots,t$, then $J$ is also the independence set of length $t$ for $(U_1, \dots, U_k)$ in $(X,\varphi)$. If some of the integers $l_j$ is negative, there exists $m>0$ such that $l_j+m \ge 0$ for all $j =1,2,\dots, t$.
     Let $\tilde{l}_j=l_j+m$ and $\tilde{J}=\{\tilde{l}_j: j=1,2,\dots,t\}$. For any function $\sigma: J \rightarrow \{1,\ldots,k\}$,
     we define $\tilde{\sigma}: \tilde{J} \rightarrow \{1,\ldots,k\}$ by $\tilde{\sigma}(\tilde{l_j})=\sigma(l_j)$, $j=1,2,\dots,t$. Since
     $$	\bigcap_{j=1}^{t}\varphi^{-l_j} U_{\sigma(l_j)}=\bigcap_{j=1}^{t}(\Phi^{\varphi})_{-l_j} U_{\sigma(l_j)}\not=\emptyset,$$
     we also have
     $$\bigcap_{j=1}^{t}\varphi^{-\tilde{l}_j} U_{\tilde{\sigma}(\tilde{l}_j)}=\bigcap_{j=1}^{t}\varphi^{-(l_j+m)} U_{\sigma(l_j)}=\varphi^{-m} \left (\bigcap_{j=1}^{t}\varphi^{-l_j} U_{\sigma(l_j)} \right )\not=\emptyset,$$
	 this implies that $\tilde{J}$ is an independence set of length $t$ for $(U_1, \dots, U_k)$ in $(X,\varphi)$. Hence $\tilde{x}=(x_1,\dots,x_k)$ is an IN-tuple in $(X,\varphi)$.
    \end{proof}

	\begin{lem}\label{L:2actions}
		Let $(X, G_1, \Phi)$ and $(X, G_2, \Psi)$ be dynamicall systems.
		\begin{enumerate}
			\item [(1)] Assume that for every $s\in G_1$ there exists $t\in G_2$ with $\Phi_s = \Psi_t$. Then $h^*(X,G_1,\Phi) \leq h^*(X,G_2,\Psi)$.
			\item [(2)] Assume that for every $s\in G_1$ there exists $t\in G_2$ with $\Phi_s = \Psi_t$ and for every $t\in G_2$ there exists $s\in G_1$ with $\Psi_t = \Phi_s$. Then $h^*(X,G_1,\Phi) = h^*(X,G_2,\Psi)$.
		\end{enumerate}
	\end{lem}
	
	\begin{proof}
		Due to symmetry, it is sufficient to prove (1). By the assumption, for every sequence $\sigma = \{s_n\}_{n\in\Z_+}$ in $G_1$ there is a sequence $\tau = \{t_n\}_{n\in\Z_+}$ in $G_2$ such that $\Phi_{s_n} = \Psi_{t_n}$ for every $n$. Then $h^{\sigma}(X,G_1,\Phi) = h^{\tau}(X,G_2,\Psi)$. It follows that $h^*(X,G_1,\Phi) \leq h^*(X,G_2,\Psi)$.	
	\end{proof}
	
	Based on the above observations, we get the following proposition.
	\begin{prop}\label{G-Z}
		Let $X$ be a compact metric space.
		\begin{enumerate}
			\item [(a)] $\Shom (X)=S_{\Z}(X)$.
			\item [(b)] Let $G_1, G_2$ be topological groups such that there is a surjective group homomorphism $\kappa \colon G_2\rightarrow G_1$. Then $S_{G_2}(X) \supseteq  S_{G_1}(X)$.
			\item [(c)] Let $G$ be a topological group such that there is a surjective group homomorphism $G\rightarrow \mathbb{Z}$. Then for every set $\{0\} \subseteq A \subseteq \log \mathbb N^*$ there exists a one-dimensional continuum $\tilde{X}_A\subseteq \mathbb R^3$ with $S_G(\tilde{X}_A) \supseteq A$.
		\end{enumerate}	
		
	\end{prop}	
	\begin{proof}
		(a) This follows from~\eqref{Eq:seq-homeo-action} and the one-one correspondence discussed above in Lemma~\ref{L:1-1 corr}.
		
		(b) Fix $h^*(X, G_1, \Phi) \in S_{G_1}(X)$. To prove that $h^*(X, G_1, \Phi) \in S_{G_2}(X)$,
it is sufficient to find a $G_2$-action $\Psi$ on $X$ with $h^*(X, G_1, \Phi) = h^*(X, G_2, \Psi)$.
		
		The acting maps of $\Phi$ are homeomorphisms $\Phi_s \colon X\to X$, $s\in G_1$. To define $\Psi$,
for $t\in G_2$ put $\Psi_t = \Phi_{\kappa (t)}$. Since the topology in $G_2$ is discrete, $\kappa$ is continuous.
Since also $\Phi$ is continuous, we get that $\Psi$ is continuous. Since $\kappa$ is a homomorphism,
it is straightforward to check that then $\Psi$ is a $G_2$-action on $X$.
		
		Since $\kappa$ is surjective, for every $s\in G_1$ there is $t\in G_2$ with $\kappa (t)=s$
and so $\Psi_{t} = \Phi_{\kappa (t)} = \Phi(s)$.
		On the other hand, for every $t\in G_2$ we have $s:= \kappa (t) \in G_1$ and $\Psi_{t}
= \Phi_{\kappa (t)} = \Phi_{s}$. By Lemma~\ref{L:2actions} then $h^*(X, G_1, \Phi) = h^*(X, G_2, \Psi)$.

		(c) For a given set $\{0\} \subseteq A \subseteq \log \mathbb N^*$, let $\tilde{X}_A$ be the
space defined in the proof of Theorem~\ref{T:main-homeo}. Then $\Shom (\tilde{X}_A) = A$. To finish the
proof, use (b) and (a) to get $S_G(\tilde{X}_A) \supseteq S_{\Z}(\tilde{X}_A) = \Shom (\tilde{X}_A)$.
	\end{proof}

The following theorem on group actions is an analogue of our Main Theorem and Theorem~\ref{T:main-homeo} but it deals only with special kinds of groups (it is Theorem C from Introduction).

 \begin{thm}[{\bf Theorem C}]\label{T:action}
 	Let $G$ be a topological group such that there is a surjective group homomorphism $G\rightarrow \mathbb{Z}$. Then
 	for every set $\{0\} \subseteq A \subseteq \log \mathbb N^*$ with $A$ finite or $\infty\in A$, there exists a one-dimensional continuum
 	$\tilde{X}_A\subseteq \mathbb R^3$ with $S_G(\tilde{X}_A)= A$. If in addition $G$ is also finitely generated, then such a continuum exists
 	for every set $\{0\} \subseteq A \subseteq \log \mathbb N^*$.
 \end{thm}

	\begin{proof}  Fix a set $A$ such that $\{0\} \subseteq A \subseteq \log \mathbb N^*$. Let $\tilde{X}_A$ be the
space defined in the proof of Theorem~\ref{T:main-homeo} and used also in the proof of
		Proposition~\ref{G-Z}(c). In view of this proposition it remains to check whether $S_G(\tilde{X}_A)\subseteq A$.

		Let  $\Phi\colon G \times \tilde{X}_A \rightarrow \tilde{X}_A$ be a $G$-action on $\tilde{X}_A$. It can be viewed
as a group homomorphism from $G$ into the group $\mathcal H (\tilde{X}_A)$ of all self-homeomorphisms of $\tilde{X}_A$;
the set $\{\Phi_s\colon s \in G\}$ is a subgroup of $\mathcal H (\tilde{X}_A)$. So, when computing $h^*(\tilde{X}_A,G,\Phi)$,
we can view $G$ as a subgroup of $\mathcal H (\tilde{X}_A)$.
		More precisely, by Lemma~\ref{L:2actions}(2) we have
		\begin{equation}\label{Eq:canview}
			h^*(\tilde{X}_A,G,\Phi) = h^*(\tilde{X}_A, \{\Phi_s\colon s \in G\}, \Psi)
		\end{equation}
		where $\Psi$ is an action of the group $\{\Phi_s\colon s \in G\}$ on the space $\tilde{X}_A$ defined by
$\Psi (\Phi_s, x) = \Phi (s,x) = \Phi_s(x)$, i.e. $\Psi_{\Phi_s} = \Phi_s$.
		
		\medskip
		
		\emph{Case 1: $A$ has cardinality $2$.}
		
		\medskip
		
		Aassume that $A=\{0,\log m\}$, $m \in \mathbb N^*$, i.e. $\tilde{X}_{A}=\tilde{X}_{\{0,\log m\}}$. Then the total
homeomorphism group $\mathcal H (\tilde{X}_A) = \{ ( \tilde{G}_{\{0,\log m\}} )^n \colon n\in \Z \}$,
 where $\tilde G_{\{0,\log m\}}$ is the distinguished homeomorphism from the proof of Theorem~\ref{T:main-homeo}.
 Since $\{\Phi_s\colon s \in G\}$ is a subgroup of this (infinite cyclic) group, there are two possibilities.
		\begin{enumerate}
			\item [(1)] The first possibility is that $\{\Phi_s \colon s \in G\}= \{\Id|_{\tilde{X}_A}\}$. Then $h^*(\tilde{X}_A,G,\Phi) =0 \in A$.
			\item [(2)] The second possibility is that  $\{\Phi_s \colon s \in G\}= \{  ( \tilde{G}_{\{0,\log m\}}^d  )^n \colon n\in \Z \}$
for some integer $d>0$. Then, by Lemma~\ref{L:2actions}(2),
			\begin{equation}\label{Eq:canview-2}
		h^*(\tilde{X}_A, \{\Phi_s\colon s \in G\}, \Psi)  =  	h^*(\tilde{X}_A, \{  ( \tilde{G}_{\{0,\log m\}}^d  )^n\colon n \in \Z\}, \Psi)
= h^*(\tilde{X}_A, \Z, \Theta)
		\end{equation}
		with $\Theta$ being a $\Z$-action on $\tilde{X}_A$ defined by $\Theta (n, x) =\Psi ((\tilde{G}_{\{0,\log m\}}^d  )^{n}, x)
= \Psi (\Phi_s, x)$, i.e. $\Theta_{n} = \Psi_{\Phi_s}$, where $s$ is any of those elements of $G$ for which $\Phi_s = (\tilde{G}_{\{0,\log m\}}^d  )^{n}$.
        Moreover, by~\eqref{Eq:seq-homeo-action} we then have
		\begin{equation}
		h^*(\tilde{X}_A, \Z, \Theta) =h^*(\tilde{G}_{\{0,\log m\}}^d) = h^*(\tilde{G}_{\{0,\log m\}}) = \log m.
		\end{equation}
		Combining this with~\eqref{Eq:canview} and~\eqref{Eq:canview-2}, we finally get $h^*(\tilde{X}_A,G,\Phi) = \log m$.
		\end{enumerate}
		We have thus shown that if $A=\{0,\log m\}$, $m \in \mathbb N^*$, then $S_G(\tilde{X}_A)=\{0,\log m\}=A$.
		
		\medskip
		
		\emph{Case 2: $A$ has cardinality $\geq 3$.}
		
		\medskip
		
		So, $\tilde{X}_{A}$ is a `flower' with finitely or infinitely many `petals', namely with the petals $L_m$ for all
		$2 \leq m~\leq~\infty $ such that $\log m \in A$ and, if $A$ is infinite, then also with the `limit' petal $[w,u]\cup [u,v]$.
		
		Recall that if $T$ is a homeomorphism on $\tilde{X}_A$, then every petal is $T$-invariant (so, every petal is
invariant for the action $\Phi$) and $T$ is different from the identity only on the union of finitely many pairwise disjoint sets $\tilde{X}_{\{0,\log k\}}$.
		
		\smallskip
		
		\emph{Subcase 2a: $A=\{0\}\cup\{\log k_i: i=1,2,\dots,l\}$ is a finite set (possibly containing $\infty$).}
		
		\smallskip
		
 Then $\tilde{X}_{A}=\bigcup_{i=1}^{l} L_{k_i}$. For the given action $\Phi$, there are two possibilities.
 \begin{enumerate}
   \item[(1)] The first possibility is that for every $s\in G$, $\Phi_s$ is the identity on $\tilde{X}_{A}$. Then $$h^*(\tilde{X}_A,G,\Phi) =0\in A.$$
   \item[(2)] The second possibility is that there exist some petals $L_{j_1}, \dots, L_{j_t}$ with $\{j_1,\dots,j_t\}\subseteq \{k_1,\dots,k_l\}$
   such that only for  $j\in \{j_1,\dots,j_t\}$ we have $\{\Phi_s|_{L_j}: s\in G\} \neq \{\Id|_{L_j}\}$. We conclude that if $(x_1,\dots,x_n)$ is an IN-tuple for $\Phi$, then $x_1,\dots,x_n \in L_{m}$ for  some $m \in  \{j_1,\dots,j_t\}$. Moreover, by the construction we know that in such a case even $x_1,\dots,x_n \in \tilde{X}_{\{0,\log m\}}$.
   Further, using~\eqref{Eq:canview} and the fact that $S_G(\tilde{X}_{\{0, \log m\}}) = \{0,\log m\}$, we obtain
   $$
   h^*(L_m,G,\Phi|_{L_m}) = h^*(L_m, \{\Phi_s|_{L_m}: s \in G\}, \Psi|_{L_m})=\log m.
   $$
   It follows that
    $h^*(\tilde{X}_A, G,\Phi)=\max_{i=1}^{t}h^*(L_{j_i}, \{\Phi_s|_{L_{j_i}}: s \in G\}, \Psi|_{L_{j_i}})=\max_{i=1}^{t}\log j_i \in A.$
  \end{enumerate}
We have shown that in this subcase $S_G(\tilde{X}_A) \subseteq  A$, as required.

		\smallskip
		
		\emph{Subcase 2b: $A$ is an infinite set with $\infty \in A$.}
		
		\smallskip

It follows from the construction of $\tilde X_A$ and the fact that the petals are invariant for the action $\Phi$, that
\begin{itemize}
	\item each of the IN-tuples of $\Phi$ lies in one of the petals,
\end{itemize}
i.e. in $L_m$ with  $\log m \in A$ or in the limit petal $[w,u]\cup [u,v]$. One can say more. Since every homeomorphism on $\tilde{X}_A$ is identity on $\bigcup_{\log m \in A} C_m \cup [w,u]$ (see Figure~\ref{F:spacein1layer}),
\begin{itemize}
	\item the segment $[w,u]$ does not contain any IN-pair.
\end{itemize}
Further, we claim that
\begin{itemize}
	\item the segment $[u,v]$ does not contain any IN-tuple of length $\geq 4$.
\end{itemize}
To see this, suppose on the contrary that this is not the case. Then there are two different points $c_1, c_2$ in the `interior' of $[u,v]$ which form an IN-pair. However, every homeomorphism $T\colon \tilde X_{A} \to \tilde X_{A}$ is identity on $[u,v]$ and if it sends an interior point of some set $D_{m,0}$ to an interior point of this set (note also that it cannot be mapped into another petal) then $T$ is necessarily the identity on $D_{m,0}$. This already shows that points from a small neighborhood of $c_1$ cannot visit a small neighborhood of $c_2$ under the action $\Phi$. This is a contradiction with the the assumption that $(c_1,c_2)$ is an IN-pair.

Suppose that for the fixed action $\Phi$ we have $h^*(\tilde{X}_A,G,\Phi)=\log t \notin A$. Hence $t$ is finite and $t\geq 2$. Let $( x_1,\ldots, x_t )$ be an intrinsic IN-tuple for $\Phi$. It lies in a petal.

If there exists  $m$ with $\log m \in A$ such that this tuple lies in $L_m$, then obviously all the points $x_1,\ldots, x_t$ belong to the same set $\tilde{X}_{\{0,\log m\}}$. Then we have $t\leq m$ because $S_G(L_m) = S_G(\tilde{X}_{\{0, \log m\}}) = \{0,\log m\}$ and we have $t\neq m$ because $\log t \notin A$. Hence $m>t$. It follows that $h^*(\tilde{X}_A,G,\Phi) \ge \log m > \log t$, a contradiction with our assumption.

If the points $x_1,\ldots, x_t \in [w,u]\cup [u,v]$, then they all have to lie in $[u,v]$ and so $t<4$. By the definition of an IN-tuple and the construction of $\tilde{X}_A$ (see Figure~\ref{F:spacein1layer} and note that $A$ is infinite now), we can find some petal $L_m$ with $ \log m \in A$ and $m>4$ such that $\{\Phi_s|_{L_m}: s \in G\}\neq \{\Id|_{L_m}\}$. Then we get
$$
h^*(\tilde{X}_A,G,\Phi) \ge h^*(L_{m}, \{\Phi_s|_{L_{m}}: s \in G\}, \Psi|_{L_{m}})=\log m >\log 4 > \log t,
$$
and this is again a contradiction with our assumption. This finishes the proof that $S_G(\tilde{X}_A) \subseteq  A$.

\medskip

Finally, assume that $G$ is finitely generated.  We just need to consider the case when $A$ is infinite with $\infty \not\in A$.
Assume that $g_1,g_2,\dots,g_n \in G$ is the list of all generators of $G$. For the given action $\Phi$,  there are two possibilities.
		\begin{enumerate}
			\item [(1)] The first possibility is that $\Phi_{g_i}=\Id|_{\tilde{X}_A}$, $i=1,2,\dots,n$. This implies that
             $\{\Phi_s : s \in G\}= \{\Id|_{\tilde{X}_A}\}$ and then $h^*(\tilde{X}_A,G,\Phi) =0\in A$.
            \item [(2)] The second possibility is that there exists a nonempty subset $J \subseteq \{1,2,\dots,n\}$ such that for every $j\in J$ we have $\Phi_{g_j} \neq \Id|_{\tilde{X}_A}$. However, as we know, each of these finitely many homeomorphisms $\Phi_{g_j}$, $j\in J$ is different from the identity only on the union of finitely many pairwise disjoint sets $\tilde{X}_{\{0,\log k\}}$. It follows that
            there exists a finite set $M \subseteq \N^*$ with $\log M \subseteq A$ such that
            $$
            \{\Phi_s|_{L_j}: s\in G\} \neq \{\Id|_{L_j}\} \, \, \text{ if and only if } \, \,  j\in M.
            $$
            Using the same argument as above (when $A$ is a finite set), we obtain
            $$
            h^*(\tilde{X}_A, G,\Phi)=\max_{i\in M}\log i \in \log M \subseteq A.
            $$
        \end{enumerate}
        Hence $S_G(\tilde{X}_A) \subseteq A$. This ends the proof of the theorem.
    	\end{proof}

\medskip

We add some remarks.

Let $G$ be a finite group. Then, for any action $\Phi$, the set $\{\Phi_{s}\colon s\in G\}$ is finite and so, for any  sequence $\sigma = \{s_n\}_{n\in\Z_+}$ in $G$ and any finite open cover $\mathcal{U}$, the quantity in~\eqref{Eq:quantity} is zero. Hence $S_G(X)=\{0\}$ for any space $X$.

	If we wish to consider groups consisting of elements of finite order, i.e. torsion groups, the following example is instructive.
	
	\begin{example}
		Let $X$ be the space in Figure~\ref{F:block1}. Define $f_1\colon X \to X$ by
		$f_1(x_i)=x_{i+1}, i=1,2,\dots,k_1-1$, $f_1(x_{k_1})=x_1$ and $f_1|_{X \setminus\{x_1,x_2,\dots,x_{k_1}\}}=\Id|_{X \setminus\{x_1,x_2,\dots,x_{k_1}\}}$.
		So, the first block is the `natural' periodic orbit of period $k_1$ (where `natural' means that every point $x_i$ in this block is
		mapped to its `successor' $x_{i+1}$, modulo the number of points in the block), all the other points of $X$ are fixed for $f_1$.
		Similarly, also for any integer $n\geq 2$ we can define a homeomorphism $f_n \colon X \to X$ such that the $n$-th block is
		the `natural' periodic orbit of period  $k_n-k_{n-1}$ and $f_n$ is identity on the rest part of the space $X$.
		Now let $G$ be the group generated by $\{f_1,f_2,f_3,\dots\}$ and $\Phi$ be the natural action. It follows from the construction of the space $X$ and the maps $f_i$ that $(X,G,\Phi)$ has $IN$-tuples of arbitrary lengths; in fact every finite subset of the countable infinite set $\{e_0^1,e_0^2,\dots \}$
        forms an IN-tuple. Thus $h^{*}(X,G,\Phi)=\infty$.
	\end{example}

	\medskip
	
	Finally, consider semigroups rather than groups.

	 Let $P$ be a topological \emph{semigroup with identity}. One can consider $P$-actions on $X$.\index{semigroup action} The difference with the group actions is that now the acting maps are just continuous maps (not necessarily homeomorphisms). This gives more freedom in constructing $P$-actions, therefore it is not surprising that if we repeat the above considerations for $P$-actions rather than for group actions, we get analogous results. Without repeating basically the same definitions, notations and arguments as above, we just say here that now $S_{\mathbb Z_{+}}(X)=S(X)$ (compare this with Proposition~\ref{G-Z}(a)) and that we have the following theorem (it is Theorem D from Introduction).

	 \begin{thm}[{\bf Theorem D}]\label{T:semigroup}
	 Let $P$ be a topological semigroup with identity such that there is a surjective semigroup homomorphism $P\rightarrow \mathbb{Z}_+$. Then
  for every set $\{0\} \subseteq A \subseteq \log \mathbb N^*$ with $A$ finite or $\infty\in A$, there exists a one-dimensional continuum
  $X_A\subseteq \mathbb R^3$ with $S_P(X_A)= A$. If in addition $P$ is also finitely generated, then such a continuum exists for every set $\{0\} \subseteq A \subseteq \log \mathbb N^*$.
 	\end{thm}

    To prove this theorem, the space $X_A$ can be the space defined in the proof of Theorem~\ref{T:main again}, i.e. the flower with a central point, whose petals are continua $X_{\{0,\log m\}}$ with $2\leq m \leq \infty$ such that $\log m \in A$.


\section{Questions}\label{S:questions}
To end the paper we formulate some open problems.

\addtocontents{toc}{\protect\setcounter{tocdepth}{-1}}

\subsection{The set $S(X)$ for the pseudoarc and the
	pseudocircle}\label{SS:pseudo}
In the continuum theory, the pseudoarc and the pseudocircle are very important
examples of planar continua. The dynamics on them is also more and more studied.
Let us mention at least the interesting question whether every continuous  map of the
pseudo-arc  has either infinite entropy or zero entropy, see~\cite{Mouron}, cf. Barge's question Q19 in~\cite{Wayne}.

\begin{ques}
	What is $S(X)$ if $X$ is the pseudoarc or the pseudocircle?
\end{ques}

\subsection{The set $S(X)$ if $X$ admits a positive entropy
	map}\label{SS:posent}
We have constructed a continuum $X$ with $S(X)=\{0,\infty\}$. Notice that this
continuum admits only continuous selfmaps with zero topological entropy.

\begin{ques}
	Can a continuum/space $X$ with $S(X)=\{0,\infty\}$ admit a continuous selfmap
	with positive topological entropy?
\end{ques}

\subsection{The set $\Smin(X)$ for the 2-torus and other continua}\label{SS:Smin}
If a compact metric space $X$ admits a minimal map (i.e. a continuous map with
every orbit dense), put
$$
\Smin (X)=\{h^*(T)| \,\, T\colon X\ra X\ \text{is minimal}\}.
$$
If $X$ is finite, then $\Smin(X)=\{0\}$. If $X$ is a Cantor set
then for each $n\in\N$
it admits a minimal selfmap $T$ with $h^*(T)=\log n$, see \cite[Example 2]{MS},
and it is also known that a Cantor set
admits minimal systems with positive entropy; therefore $\Smin(X)= \log \N^*$.
If $X$ is a circle then $\Smin(X)=\{0\}$ since
the minimal maps on the circle are just the maps topologically conjugate to
irrational rotations, and the
topological sequence entropy is an invariant of topological conjugacy. However,
already for the torus the question
is nontrivial.

\begin{ques} Is it true that if $\mathbb T^2$ is the 2-dimensional
	torus then $\Smin(\mathbb T^2)=\log \N^*$?
\end{ques}

\subsection{Continua $X_A$ as attractors?}\label{SS:attr}

Recall that, due to Handel~\cite{Hand}, the pseudocircle is known to be an attracting minimal set
for a $C^{\infty}$ diffeomorphism in the plane. Our continua $X_A$ do not admit minimal maps. However,
the following question essentially suggested by Benjamin Weiss is still challenging.

\begin{ques}
	If a continuum $X_A$ is embedded in a manifold $M$, does there exists a dynamics on $M$
	such that $X_A$ is an attractor? Is it at least true that a Cook continuum in the plane can be an attractor
	for a continuous selfmap of the plane?
\end{ques}

\subsection{Possible sets of values of topological entropy}\label{SS:topent}

The present paper deals with possible sets of values of supremum topological sequence entropy. One can consider an analogous problem for the usual topological entropy.
If $X$ is a nonempty compact metric space, consider the set
\[
\Ent (X) = \{h(T):\, T \text{ is a continuous map } X\to X\}.
\]
Of course, the set $\Ent(X)$ always contains $0$. It is also closed with respect to multiples, meaning that if $\alpha \in \Ent (X)$ then $n\alpha \in \Ent (X)$ for $n=1,2\dots$. Apparently, the techniques from our paper can be useful for answering, at least partially, the following problem.

\begin{ques}\label{Q:ent}
What are the possibilities for $\Ent (X)$?	Is it true that for every set $\{0\} \subseteq A \subseteq [0,\infty]$ which is closed with respect to multiples there exists a compact metric space $X$ with $\Ent (X) = A$?
\end{ques}

\subsection{Group/semigroup actions case}\label{SS:actions case}

Our Theorems~\ref{T:action} and~\ref{T:semigroup} deal only with special kinds of groups and semigroups, respectively. We would like to know the answer to the following question (and the answer to its analogue for semigroup actions).

\begin{ques}
	How to determine all groups $G$ such that for every set $\{0\} \subseteq A \subseteq \log \mathbb N^*$, there exists a
	space/a continuum $X_A$ with $S_G(X_A) = A$?
\end{ques}

\end{document}